\documentclass[12pt, titlepage]{article}

\usepackage{tikz}
\usepackage{pgfplots}

\usepackage{hyperref}

\usepackage{tikz-cd}

\usepackage[utf8]{inputenc}
\usepackage{remreset}
\usepackage[T1]{fontenc}
\usepackage[amsmath]{}
\usepackage[ngerman]{babel}
\usepackage{amsmath}
\usepackage{amsthm}
\usepackage{amssymb}
\usepackage{amsfonts}
\usepackage{amssymb}
\usepackage{graphicx}
\usepackage{color}
\usepackage[a4paper,lmargin={4cm},rmargin={2cm},tmargin={2.5cm},bmargin={2.5cm}]{geometry}
\usepackage{stmaryrd}

\newcommand{\uproman}[1]{\uppercase\expandafter{\romannumeral#1}}

\newtheorem{definition}{Definition}[section]
\newtheorem{proposition}{Satz}[section]
\newtheorem{theorem}{Theorem}[section]
\newtheorem{lemma}{Lemma}[section]
\newtheorem{remark}{Bemerkung}[section]
\newtheorem{example}{Beispiel}[section]
\newtheorem{corollary}{Folgerung}[section]

\begin{document}

\usetikzlibrary{angles,quotes, intersections}

\pagestyle{empty}

\begin{titlepage}
\title{Stationäre Kurven auf endlichdimensionalen Mannigfaltigkeiten}
\date{2024}
\author{Tobias Starke}
\maketitle
\end{titlepage}

\thispagestyle{empty}

\newpage 
\thispagestyle{empty}
\quad 
\newpage

\begin{center}
    {Danksagung}
\end{center}
An dieser Stelle möchte ich mich kurz bei all denjenigen bedanken, die mich während der Anfertigung dieser Arbeit unterstützten. 

Als Erstes möchte ich mich bei meinem Betreuer Herrn Professor Michael Breuß bedanken, welcher mir während der Verfassung dieser Arbeit stets mit vielen hilfreichen Anregungen und konstruktiver Kritik zur Seite stand. 

Weiter möchte ich mich ganz herzlich bei Herrn Professor Friedrich Sauvigny bedanken, welcher mit seiner an der BTU vom Wintersemester 2021 bis Sommersemester 2023 gehaltenen Vorlesungsreihe über Differentialgeometrie überhaupt erst mein Interesse an differentialgeometrischen Fragestellungen weckte. 

Abschließend möchte ich mich auch bei meiner Mutter und meinem Großvater für die Unterstützung während des gesamten Studiums bedanken.

\newpage

\newpage 
\thispagestyle{empty}
\quad 
\newpage

\pagestyle{plain}
\thispagestyle{empty}

\tableofcontents

\newpage

\newpage 
\thispagestyle{empty}
\quad 
\newpage

\setcounter{page}{1}


\section{Einleitung}

In der hier vorliegenden Arbeit wollen wir uns mit sogenannten stationären Kurven des Längenfunktionals, sogenannten schwachen Geodätischen, auf Mannigfaltigkeiten beschäftigen. Ausgangspunkt dieser Arbeit sind die zahlreichen Anwendungen derartiger Kurven im Bereich der Physik, aber auch im Bereich der Bildverarbeitung. 

So wird beispielsweise in der \textit{Allgemeinen Relativitätstheorie} die Raumzeit als vierdimensionale metrikkompatible Lorentzmannigfaltigkeit modelliert, sodass die beobachteten gravitativen Phänomene über die innere Krümmung beziehungsweise die inneren Geometrie der Mannigfaltigkeit erklärt werden \cite{kriele2003spacetime}. 

Das erreicht man dadurch, dass man vorgibt, dass sich Teilchen, die sich lediglich unter dem Einfluss der Gravitation befinden, sich auf den (schwachen) Geodätischen der Raumzeitmannigfaltigkeit bewegen, wobei die (schwachen) Geodätischen im Allgemeinen von der inneren Geometrie der Raumzeitmannigfaltigkeit abhängig sind. Die innere Geometrie der Raumzeitmannigfaltigkeit wird hingegen über die Einsteinschen Feldgleichungen durch die Energiedichteverteilung im Universum determiniert.  

Im Bereich der Bildverarbeitung und dort im speziellen im sogenannten Bild-Matching nutzt man hingegen (schwache) Geodätische auf die folgende Weise: Wir beginnen zuerst damit, zwei Bilder, die wir mit $P_1$ und $P_2$ bezeichnen wollen, mit Funktionen $\mathcal{I}_1, \mathcal{I}_2 : \Omega \longrightarrow \mathbb{R}$ zu assoziieren, wobei $\Omega \subseteq \mathbb{R}^d$ als die (beschränkte) Menge der Bildpixel verstanden wird. Das können wir darüber realisieren, indem wir beispielsweise jedem Pixel seinen Färbungsgrad, dargestellt durch eine reelle Zahl, zuordnen. Dabei nehmen wir an, dass beide Bilder die selbe \textit{Sache} zeigen sollen. Die Bilder sind dabei im Allgemeinen optisch aber nicht identisch, da diese durch verschiedene Verfahren erzeugt sein können. 

Als ganz konkretes Beispiel können wir die bildgebenden Verfahren zur Visualisierung des menschlichen Gehirns anführen. Um ein Bild vom Gehirn zu erhalten, können wir heute auf verschiedene Verfahren zurückgreifen, wie die Magnetresonanztomographie (MRT), die Positronen-Emissions-Tomographie (PET) oder die Computertomographie (CT). Während beispielsweise bei ersterem Verfahren auf starke Magnetfelder zurückgegriffen und ausgenutzt wird, dass sich der Spin der Atomkerne im Gehirn nach diesem Magnetfeld ausrichten werden, so wird beim PET-Verfahren hingegen eine schwach radioaktive Substanz verwendet, um über diese eine visuelle Darstellung des Gehirns zu erhalten. 

Aufgrund der Verschiedenheit der Verfahren sind die Informationen über das Gehirn, die auf den durch die verschiedenen Verfahren erzeugten Bilder codiert sind, im Allgemeinen verschieden. So liefert das MRT-Verfahren eine gute bildliche Darstellung der Struktur des Gehirns, während beispielsweise das PET-Verfahren eine gute Darstellung der Aktivität des Gehirns liefert. Betrachten wir zwei konkrete bildgebende Verfahren für das menschliche Gehirn, so können die durch die Verfahren erzeugte Bilder $P_1$ und $P_2$ sich nicht nur hinsichtlich der auf ihnen gespeicherten Informationen unterschieden, sondern auch relativ zueinander verzerrt wirken. Diese beiden möglich auftretenden Effekte führen dazu, dass die beiden Bilder sich optisch oft stark voneinander unterscheiden und nicht leicht miteinander identifiziert werden können. Letzteres bedeutet dabei, dass es im Allgemeinen nicht sofort offensichtlich ist, welche Region im Bild $P_1$ genau welcher Region im Bild $P_2$ entspricht.  

Es kann aber durchaus vorkommen, dass man beide Bilder im obigen Sinne miteinander identifizieren können möchte, um die Informationen, die in den jeweiligen Bildern gespeichert sind, auf einem Bild sinnvoll zusammenzuführen. Gesucht ist damit also eine sogenannte Matching-Funktion zwischen den mit den Bildern $P_1$ und $P_2$ assoziierten Funktionen $\mathcal{I}_1$ und $\mathcal{I}_2$, welche wir mit $\varphi$ bezeichnen wollen. Diese Matching-Funktion gibt an, welche Punkte auf dem einem Bild genau welchen Punkten in dem anderen Bild entsprechen. Beim Matching handelt es sich also um eine mindestens bijektive Funktion der Form $\Omega \longrightarrow \Omega$, die 
\begin{align}
    \mathcal{I}_1 = \mathcal{I}_2 \circ \varphi
\end{align}
erfüllt, wobei es sich bei $\circ$ um die gewöhnliche Funktionenkomposition handelt. Die Frage, die sich damit aber unmittelbar stellt ist, wie wir ein derartiges Matching (möglichst effizient) finden können.

Ein möglicher differentialgeometrischer Ansatz ist es, die Matchingfunktion $\varphi$ als glatten Diffeomorphismus auf $\Omega$ aufzufassen und im Raum aller glatter Diffeomorphismen auf $\Omega$, bezeichnet als $\textit{Diff}(\Omega)$, nach dem Matching $\varphi$ zu suchen \cite{miller2006geodesic}. Es zeigt sich, dass die Menge $\textit{Diff}(\Omega)$ eine sogenannte unendlichdimensionale Fréchet-Mannigfaltigkeit ist \cite{golubitsky2012stable}, auf welcher wir einen Längenbegriff für Kurven einführen können, was es uns wiederrum erlaubt, den Begriff einer (schwachen) Geodätischen auf $\textit{Diff}(\Omega)$ zu definieren. Der grundlegende Gedanke ist nun, das wir uns ausgehend von der Identitätsabbildung $\textit{id}_\Omega \in \textit{Diff}(\Omega)$, definiert durch
\begin{align}
    \textit{id}_\Omega (x) = x \:\:\:\:\:\:\:\: \forall x \in \Omega,
\end{align}
auf einer (schwachen) Geodätischen hin zum Matching $\varphi$ bewegen wollen. Dieser Grundgedanke liefert die Grundlage des sogenannten \textit{Geodesic shooting}-Algorithmus, mit derem Hilfe man die gesuchte Matchingfunktion $\varphi$ approximieren kann.

Damit wir verstehen, was wir uns konkret unter dem Begriff einer (schwachen) Geodätischen vorzustellen haben, welcher grundlegend für beide oben erwähnte Anwendungen ist, wollen wir uns zuerst mit einigen grundlegenden Konzepte aus der Topologie und der endlichdimensionalen Differentialgeometrie auseinandersetzen, und dabei unter anderem klären, was genau eine Topologie und was eine endlichdimensionale Mannigfaltigkeit ist. Darüberhinaus werden wir Kurven auf endlichdimensionalen Mannigfaltigkeiten betrachten und auch erklären, wie man diesen eine sinnvolle Länge zuordnen kann.

Das primäre Ziel dieser Arbeit wird es dabei sein, einige der grundlegenden Begriffe, die nötig sind, um die obigen Anwendungen rigoros verstehen zu können, zu entwickeln. Allen voran wollen wir damit letzlich auch rigoros klären, was eine (schwache) Geodätische genau ist. Wir werden in dieser Arbeit dabei vornehmlich sogenannte endlichdimensionale Mannigfaltigkeiten betrachten. Das hat konkret zwei Gründe:

Zum einen wollen wir ein gutes geometrisches und anschauliches Verständnis vieler verschiedener Konzepte aus der Differentialgeometrie erhalten. Am besten lassen sich dabei diese zu diskutierenden Konzepte und Ideen im Kontext der endlichdimensionalen Differentialgeometrie motivieren. Haben wir erst einmal die endlichdimensionale Theorie verstanden, so können wir, ausgehend von dem Verständis, welches wir durch die endlichdimensionale Differentialgeometrie erworben haben, auch die wesentlich abstrakteren Begriffe der zugehörigen unendlichdimensionalen Theorie leicht motivieren.
 
Zum anderen tauchen viele Begrifflichkeiten aus der endlichdimensionalen Theorie auch wieder im unendlichdimensionalen auf. Das heißt, um gewisse Konstruktionen in der unendlichdimensionalen Differentialgeometrie verstehen zu können, müssen wir uns so oder so mit dem endlichdimensionalen Fall auseinander setzen.

Auch wenn sich diese Arbeit lediglich mit der endlichdimensionalen Theorie beschäftigt, werden wir uns am Ende dieser Arbeit, im Kapitel $9$, noch etwas genauer mit den Konzepten der unendlichdimensionalen Theorie auseinandersetzen. Wir werden dabei grob angeben, wie sich die Konzepte der endlichdimensionalen Differentialgeometrie in das Unendlichdimensionale verallgemeinern lassen. So werden wir im Kapitel $9$ grob skizzieren, was wir uns genau unter einer unendlichdimensionalen Mannigfaltigkeit vorzustellen haben und wie sich der Längenbegriff für Kurven sinnvoll in das Unendlichdimensionale überführen lässt.   

Auch wollen wir uns am Ende dieser Arbeit noch etwas ausführlicher mit der Anwendung des Begriffes der (schwachen) Geodätischen im Bereich des Bild-Matchings beschäftigen, wobei wir uns dabei primär auf \cite{miller2006geodesic} beziehen werden. 

\newpage
\newpage 

\quad 
\newpage


\section{Topologie}

Die in diesem Kapitel angegebenen Definitionen orientieren sich im wesentlichen an \cite{joshi1983introduction}.

Da, wie in der Einleitung bereits angemerkt, die stationären Kurven eines sogenannten Längenfunktionals (was auch immer das zu diesem Zeitpunkt sein mag) für uns von großem Interesse sind, müssen wir erst einmal den Kurvenbegriff an sich einführen. Es lohnt sich nun unsere Intution zu bemühen, um grob zu erklären, was wir genau unter einer Kurve in einer abstrakten Menge $X$ verstehen wollen. 

Sei dazu $X$ also eine beliebige Menge und $x_1,x_2 \in X$ zwei Punkte in dieser Menge. Anschaulich dürfte eine Kurve $\mathcal{F}$ von $x_1$ nach $x_2$ in $X$ nun eine eindimensionale Teilmenge von $X$ sein, die die Punkte $x_1$ und $x_2$ miteinander verbindet. Mathematisch codieren wir das über die Funktion oder Abbildung 
\begin{align}
     \gamma : [a,b] \longrightarrow X,
\end{align}
wobei $a,b \in \mathbb{R}$ sind mit $- \infty < a < b < + \infty$. Per Definiton einer Funktion ordnet also $\gamma$ jedem $c \in [a,b]$ genau ein $x \in X$ zu. Dabei gelte
\begin{align}
     \mathcal{F} = \{x \in X \: | \: \exists \:  c \in [a,b] : \gamma(c) = x \}, \label{2}
\end{align}
dass heißt $\gamma(c) = x$ ist ein Punkt auf der Kurve $\mathcal{F}$, und 
\begin{align}
     \gamma(a) = x_1, \: \: \gamma(b) = x_2. \label{3}
\end{align}
Warum stellt $\gamma$ nun eine sinnvolle Modellierung der Kurve $\mathcal{F}$ dar? Zum einen wird die oben nur grob erwähnte Eindimensionalität der Kurve durch die Parametrisierung von $\mathcal{F}$ über das Intervall $[a,b]$ eingefangen, und zum anderen wird durch \eqref{3} die Tatsache codiert, das die Endpunkte der Kurve $\mathcal{F}$, die in der gewählten Modellierung das Bild von $\gamma$ ist, in $x_1$ beginnt, und in $x_2$ endet. 

Eine weitere Bedingung, die wir für den Kurvenbegriff noch dringend benötigen, die wir aber bisher nicht erwähnt haben, ist die Stetigkeit! Wir erwarten natürlich, dass $\mathcal{F}$ in irgend einem Sinne ein zusammenhängendes Gebilde ist und das wir mit dem durchlaufen des Intervalls $[a,b]$ auch $\mathcal{F}$ auf ähnliche Weise durchlaufen. Anschaulich bedeutet das, dass wenn $c,d \in [a,b]$ \textit{nah} beieinander liegen, auch $\gamma(c)$ und $\gamma(d)$ \textit{nah} beieinander liegen. Hier stoßen wir nun auf ein erstes Problem: Was bedeutet Nähe auf dem Intervall $[a,b]$ und auf $X$? 

Bisher haben wir, bezogen auf $[a,b]$ und $X$, nur eine \textit{nackte} Menge vor uns, auf der noch kein Begriff der Nähe existiert. Auch wenn wir beispielsweise bezüglich dem Intervall $[a,b]$ eine intuitive Vorstellung von Nähe haben, so rührt diese Vorstellung immer von einer auf der Menge zusätzlich definierten Struktur her, die den Nähebegriff auf der Menge erst erklärt. 

Der intuitive Nähebegriff auf $[a,b]$ kommt dabei von der Betragsfunktion $f(\cdot) = |\cdot|$, die über die Funktion $d(x,y) := f(x-y) = |x-y|$ mit $x,y \in [a,b]$ einen Abstandsbegriff und damit einen Nähebegriff einführt. Das sich zwei Zahlen $x,y \in [a,b]$ nun nahe sein sollen, bedeutet, dass $d(x,y) < \epsilon$ ist, wobei $\epsilon > 0$ eine vorgegebene (kleine) Zahl ist. Wir sagen auch, dass $y$ näher an $x$ liegt, als alle Punkte in $[a,b]$, die einen Abstand größer gleich $\epsilon$ zu $x$ haben. Um also überhaupt den intuitiven Nähebegriff auf $[a,b]$ etablieren zu können, mussten wir erst einmal die eben erwähnte Abbildung $d : [a,b] \times [a,b] \longrightarrow [0, \infty )$ auf $[a,b]$ definieren.

In Bezug auf unsere Menge $[a,b]$ haben wir nun, wie wir noch sehen werden, eigentlich schon mehr getan, als wir eigentlich hätten tun müssen, denn um den intuitiven Nähebegriff auf $[a,b]$ einzuführen, haben wir einen Abstandsbegriff definiert. D.h. wir haben eine Funktion $d$ eingeführt, mit deren Hilfe wir den \textit{Abstand} zwischen zwei beliebigen Punkten $x,y \in [a,b]$ vermessen können. 

Damit der durch $d$ erklärte Abstandsbegriff wirklich ein sinnvoller ist, muss natürlich die Funktion $d$, die wir später auch als \textit{Metrik} bezeichnen wollen, gewisse Eigenschaften erfüllen, denn nicht jede beliebige Funktion $\Tilde{d}$ auf $[a,b] \times [a,b]$ erklärt einen sinnvollen Abstandsbegriff. So muss beispielsweise für eine Abstandsfunktion $\Tilde{d} : [a,b] \times [a,b] \longrightarrow [0, \infty)$ gelten, dass 

\begin{itemize}
    \item[\textit{i)}]   $\Tilde{d}(x,y) \geq 0$ und $\Tilde{d}(x,y) = 0 \: \iff \: x=y \:\:\:\: \forall x,y \in [a,b]$
    \item[\textit{ii)}]  $\Tilde{d}(x,y) = \Tilde{d}(y,x) \:\:\:\: \forall x,y\in [a,b]$
    \item[\textit{iii)}] $\Tilde{d}(x,z) \leq \Tilde{d}(x,y) + \Tilde{d}(y,z) \:\:\:\: \forall x,y,z \in [a,b]$
\end{itemize}

Die Bedingung \textit{i)} bedeutet dabei, dass der Abstand zweier Punkte in $[a,b]$ immer größer als Null ist, wenn $x$ und $y$ verschieden sind und der Abstand zwischen zwei Punkten genau dann Null ist, wenn die Punkte identisch sind. Das macht auch intuitiv sind, da der Abstand von einem Punkt zu sich selbst natürlich Null betragen sollte. Bedingung \textit{ii)} meint dagegen, dass es egal ist, ob man den Abstand von $x$ nach $y$ oder den Abstand von $y$ nach $x$ misst. In beiden Fällen sollte das selbe heraus kommen. Und Bedingung \textit{iii)} bedeutet schließlich, dass der \textit{direkte Weg} zwischen zwei Punkten $x$ und $y$ immer kürzer ist, als ein Umweg über einen Punkt $y$.  

Die oben definierte Funktion $d(x,y) = |x-y|$ erfüllt natürlich alle diese Eigenschaften, weshalb diese auch einen sinnvollen Abstandsbegriff erklärt. Betrachten wir aber nun die abstrakte Menge $X$, so stellt sich die Frage, wie sich dort eine derartige Abstandsfunktion erklären lässt. Unter Umständen kann sich das, abhängig von der Menge $X$, als große Herausforderung herausstellen, besonders wenn man gewisse zusätzliche Forderungen an den aus der Abstandsfunktion resultierenden Nähebegriff stellt.

Weiterhin müssen wir uns der Tatsache vergegenwärtigen, dass wir \textit{nur} einen Nähebegriff auf $X$ (und $[a,b]$) suchen, und keinen Abstandsbegriff. Ein Abstandsbegriff führt zwar auf einen Nähebegriff, aber die Umkehrung muss nicht unbedingt gelten. 

Um also so allgemein wie möglich zu bleiben, suchen wir stattdessen nach einer alternativen Struktur, die sich auf $[a,b]$ und $X$ einführen lässt und dann dort einen Nähebegriff einführt. Diese allgemeinere Struktur nennen wir \textit{Topologie}. Wir werden nun im Folgenden zuerst die abstrakte Definition einer Topologie betrachten und uns dann davon überzeugen, dass mittels einer Topologie tatsächlich ein (sinnvoller) Nähebegriff induziert wird. 

\begin{definition}
Sei $X$ eine beliebige, nichtleere Menge und sei $\mathcal{P}(X)$ die zugehörige Potenzmenge von $X$, das heißt die Menge aller Teilmengen von $X$. Dann ist eine Topologie $\tau$ eine Teilmenge von $\mathcal{P}(X)$, das heißt $\tau \subseteq \mathcal{P}(X)$, sodass die folgenden Eigenschaften gelten:
\begin{itemize}
    \item [\textit{i)}] $\emptyset, X \in \tau$
    \item [\textit{ii)}] Sei $\mathcal{I}$ eine beliebe Indexmenge und $\mathcal{O}_i \in \tau \:\: \forall i \in \mathcal{I}$. So folgt 
    \begin{align}
        \bigcup_{i \in \mathcal{I}} \mathcal{O}_i \in \tau.
    \end{align}

    \item [\textit{iii)}] Sei $\mathcal{I}$ eine endliche Indexmenge und $\mathcal{O}_i \in \tau \:\: \forall i \in \mathcal{I}$. So folgt 
    \begin{align}
        \bigcap_{i \in \mathcal{I}} \mathcal{O}_i \in \tau.
    \end{align}

\end{itemize}
Die Elemente $\mathcal{O} \in \tau$ nennen wir dabei im Folgenden offene Mengen. Das Tupel $(X, \tau)$ bezeichnen wir als topologischen Raum.
\end{definition}

Wie erklärt nun eine Topologie $\tau$ auf $X$ einen Nähebegriff auf selbiger? Eine erste Verplausibiliserung liefert die Tatsache, dass sich, wie wir gleich sehen werden, mit einer Topologie ein Umgebungs- und ein Konvergenzbegriff erklären lässt. Man fragt sich nun vielleicht, inwiefern die Tatsache, dass sich mit einer Topologie ein Umgebungs- und Konvergenzbegriff definieren lässt, ein Konzept der Nähe auf der Menge $X$ eingeführt wird. Betrachten wir dazu beispielsweise den Konvergenzbegriff und erinnern wir uns einmal daran, was es eigentlich bedeutet, dass eine Folge $(x_n)_{n \in \mathbb{N}} \subseteq X$ gegen ein $x \in X$ konvergiert. 

Anschaulich heißt das ja gerade, dass sich die Folgenglieder $x_n$ für immer größer werdendes $n \in \mathbb{N}$ im wesentlichen immer mehr dem Element $x \in X$ \textit{annähern}. Damit ist in einem Konvergenzbegriff immer auch ein Nähebegriff codiert, da ansonsten das Konzept der Konvergenz keinen Sinn ergibt. Wie also lässt sich ein Umgebungs- und Konvergenzbegriff mit einer Topologie $\tau$ auf $X$ erklären?  

Die grundlegende Idee ist die Folgende: Sei $(X, \tau)$ ein topologischer Raum. Sei weiter ein $x \in X$ gegeben. Zuerst wollen wir intuitiv erklären, was eine sogenannte offene Umgebung des Punktes $x$ ist und wie dieser Umgebungsbegriff einen Nähebegriff induziert.

\begin{definition}
    Jedes $\mathcal{O} \in \tau$ mit der Eigenschaft $x \in \mathcal{O}$ heißt (offene) Umgebung von $x \in X$. Eine (offene) Umgebung $\mathcal{O}$ von $x$ mit der Eigenschaft $\mathcal{O} \in \tau \setminus \{X\}$ nennen wir nichttriviale Umgebung von x. Ein Punkt $y \in X$ liegt damit in einer \textit{nichttrivialen Umgebung} von $x$, falls sich ein $\mathcal{O} \in \tau \setminus \{X\}$ mit $x \in \mathcal{O}$ finden lässt, sodass $y \in \mathcal{O}$ ist.
\end{definition}

Mit dieser Umgebungsdefinition auf Basis der Topologie $\tau$ lässt sich bereits ein durch die Topologie induzierter Nähebegriff erahnen, denn wir könnten nun natürlich sagen, dass $y_1 \in X$ \textit{näher} als $y_2 \in X$ an $x$ liegt, wenn es zwei nichttriviale Umgebungen $\mathcal{O}_1, \mathcal{O}_2 \in \tau$ von $x$ gibt, sodass $y_1 \in \mathcal{O}_1, y_2 \in \mathcal{O}_2$ und $\mathcal{O}_1 \subset \mathcal{O}_2$, sowie $y_2 \notin \mathcal{O}_1$ gilt. Eine Topologie als Teilmenge von $\mathcal{P}(X)$ kann damit salopp als Vorschrift verstanden werden, die einem sagt, wie und vorallem wie stark die Punkte aus $X$ miteinander verklebt werden sollen. So sagt man beispielsweise auch, dass die Punkte $y_1$ und $y_2$ nichtrivial nah zueinander sind, wenn es eine Menge $\mathcal{O} \in \tau$ mit $\mathcal{O} \subset X$ gibt, sodass $y_1, y_2 \in \mathcal{O}$ gilt. 

Am besten lässt sich aber das Konzept der Nähe auf Basis einer Topologie mittels der Einführung eines Konvergenzbegriffes verplausibilisieren, da, wie oben bereits erwähnt, eine Folge $(x_n)_{n \in \mathbb{N}} \subseteq X$ anschaulich ja gerade dann konvergiert, wenn sich die Folgenglieder $x_n$ im Schnitt für immer größer werdenden Index $n \in \mathbb{N}$ immer stärker einem $x \in X$ \textit{annähern}. 

\begin{definition}
    Sei $(x_n)_{n \in \mathbb{N}} \subseteq X$ eine Folge in $X$ und $x \in X$. Wir sagen, dass $(x_n)_{n \in \mathbb{N}}$ gegen $x \in X$ konvergiert, in Zeichen $x_n \xrightarrow{n \rightarrow \infty} x$, wenn gilt:
    \begin{align}
        \forall \mathcal{O} \in \tau \:\: \textit{mit} \:\: x \in \mathcal{O} : \exists N \in \mathbb{N} : x_n \in \mathcal{O} \:\: \forall n \geq N
    \end{align}
    In diesem Fall sagen wir, dass $x$ ein Grenzwert von $(x_n)_{n \in \mathbb{N}}$ ist. 
\end{definition}

Anschaulich bedeutet diese Definition, dass für jede beliebige Umgebung $\mathcal{O} \in \tau$ von $x$ sich immer eine natürliche Zahl $N \in \mathbb{N}$ finden lässt, sodass alle Folgenglieder der Folge $(x_n)_{n \in \mathbb{N}}$ mit $n \geq N$ in $\mathcal{O}$ liegen. Umgekehrt bedeutet das, dass nur endlich viele Folgenglieder für eine fest gewählte Umgebung $\mathcal{O}$ von $x$ außerhalb von $\mathcal{O}$ liegen können. 

Mit der obigen Bemerkung zum vom Umgebungsbegriff abgeleiteten Nähebegriff erhalten wir damit das anschauliche Bild der Konvergenz im Rahmen topologischer Räume.

Wir haben nun also den Begriff der Topologie kennen gelernt und verstehen, wie dieser einen Nähebegriff auf der Menge $X$ induziert. Es stellen sich nun aber mindestens zwei Fragen: 

\begin{itemize}
    \item [\textit{i)}] Ist der Grenzwert einer Folge mit dem obigem Begriff der Konvergenz eindeutig?
    \item [\textit{ii)}] Induziert nicht jede beliebige Teilmenge $\rho \subseteq \mathcal{P}(X)$ im obigem Sinne einen Nähebegriff? Wozu braucht es die Forderungen $\textit{i)}$, $\textit{ii)}$ und $\textit{iii)}$ aus der obigen Definition der Topologie?
\end{itemize}

Die Beantwortung der ersten Frage ist wichtig, wenn es darum geht, mit einer auf einer Menge $X$ gegebenen Topologie zu arbeiten. Insbesondere wird die Beantwortung dieser Frage von großer Bedeutung sein, wenn wir uns später mit sogenannten normierten Räumen und der Analysis auf diesen beschäftigen wollen.

Die Antwort auf die zweite Frage ist hingegen wichtig, um zu verstehen, warum wir eine Topologie überhaupt so erklärt haben, wie wir sie erklärt haben. Wie sich nämlich leicht nachprüfen lässt, ist es tatsächlich so, dass jede Teilmenge $\rho$ der Potenzmenge von $X$ einen Nähebegriff im obigem Sinne erklärt. 

Das lässt sich darüber einsehen, dass wir in unseren bisherigen Definitionen nicht die Eigenschaften einer Topologie ausgenutzt haben. Wie wir im Kapitel über $\textit{metrische Räume}$ noch sehen werden, streben wir nämlich tatsächlich einen speziellen Nähebegriff an, der zum einen stark genug ist, um ordentlich mit diesen arbeiten zu können, und zum anderen noch möglichst nah am herkömmlichen Nähebegriff liegt, sodass dieser beispielsweise die Konzepte der gewöhnlichen Konvergenz und Stetigkeit aus der reellen Analysis verallgemeinert.

Zuerst wollen wir uns aber mit der Beantwortung der ersten Frage auseinandersetzen. Dazu betrachten wir das folgende Beispiel:

\begin{example}
    Betrachten wir eine Menge $X$, die wir mit der Topologie $\tau = \{ \emptyset , X\}$ austatten. Es lässt sich dabei leicht einsehen, dass es sich bei $\tau$ tatsächlich um eine Topologie handelt. Sei nun weiter $(x_n)_{n \in \mathbb{N}} \subseteq X$ eine Folge in $X$ und $x \in X$ ein beliebiges Element in $X$. Dann gilt, dass $(x_n)_{n \in \mathbb{N}}$ gegen $x$ konvergiert, denn die einzige (offene) Umgebung von $x$ ist $X \in \tau$, und für diese gilt trivialerweise, dass $x_n \in X \:\: \forall n \in \mathbb{N}$. 

    Das heißt, dass in $(X, \tau)$ jede einzelne Folge konvergiert, unzwar gegen jedes einzelne Element aus $X$. Aus diesem Grund wird die hier gewählte Topologie $\tau$ auch oft als \textit{triviale} oder \textit{chaotische} Topologie bezeichnet. 
    
    Damit folgt, dass im Allgemeinen die Grenzwerte einer Folge bezüglich einer beliebigen Topologie nicht eindeutig sein müssen.
\end{example}

Aus dem obigem Beispiel haben wir also gelernt, dass im Allgemeinen der Grenzwert einer Folge in einem topologischen Raum nicht eindeutig sein muss. In Vorbereitung auf die noch kommenden Kapitel wollen wir noch eine spezielle Klasse von topologischen Räumen kennen lernen: die sogenannten \textit{Hausdorffräume}.

\begin{definition}
    Sei $(X, \tau)$ ein topologischer Raum. Wir nennen $(X, \tau)$ \textit{hausdorffsch}, wenn für alle $x, y \in X$ mit $x \neq y$ gilt:
    \begin{align}
        \forall x,y \in X \:\: \textit{mit} \:\: x \neq y \:\: \exists \mathcal{O}_x, \mathcal{O}_y \in \tau \:\: \textit{mit} \:\: x \in \mathcal{O}_x, y \in \mathcal{O}_y \:\: \textit{und} \:\: \mathcal{O}_x \cap \mathcal{O}_y = \emptyset
    \end{align}
\end{definition}

Anschaulich lassen sich also unterschiedliche Punkte aus $X$ mit Umgebungen aus $\tau$ trennen. Eine wichtige Eigenschaft von Hausdorffräumen ist, dass in diesen die Grenzwerte konvergenter Folgen eindeutig sind: 

\begin{proposition}
    Sei ($X$, $\tau$) ein Hausdorffraum und $(x_n)_{n \in \mathbb{N}} \subseteq X$ eine konvergente Folge. Dann ist der Grenzwert der Folge $(x_n)_{n \in \mathbb{N}}$ eindeutig bestimmt. Weiter wollen wir diesen eindeutigen Grenzwert unter anderem mit $\lim_{n \longrightarrow \infty} x_n$ bezeichnen. 
\end{proposition}

\begin{proof}
    Sei $(x_n)_{n \in \mathbb{N}}$ eine Folge in $X$. Angenommen es gilt $x_n \xrightarrow{n \rightarrow \infty} x$ und $x_n \xrightarrow{n \rightarrow \infty} y$ mit $x \neq y$, d.h. 
    \begin{align}
        \forall \mathcal{U} \in \tau \:\: \textit{mit} \:\: x \in \mathcal{U} : \exists N \in \mathbb{N} : x_n \in \mathcal{U} \:\: \forall n \geq N
    \end{align}
    und 
    \begin{align}
        \forall \mathcal{V} \in \tau \:\: \textit{mit} \:\: y \in \mathcal{V} : \exists M \in \mathbb{N} : x_n \in \mathcal{V} \:\: \forall n \geq M.
    \end{align}
    Aufgrund der Hausdorff-Eigenschaft existieren offene Mengen $U,V \in \tau$ mit $x \in U$, $y \in V$ und $U \cap V = \emptyset$. Es gilt nach den beiden obigen Ausdrücken, dass ein $N_U \in \mathbb{N}$ und ein $M_V \in \mathbb{N}$ existiert, sodass 
    \begin{align}
        x_n \in U \:\: \forall n \geq N_U
    \end{align}
    und 
    \begin{align}
        x_n \in V \:\: \forall n \geq M_V
    \end{align}
    Wir setzen $K := \textit{max}\{ N_V, M_U \}$. Es folgt damit 
    \begin{align}
        x_n \in U \cap V \:\: \forall n \geq K,
    \end{align}
    im Widerspruch zu $U \cap V = \emptyset$. Damit folgt, dass der Grenzwert von $(x_n)_{n \in \mathbb{N}}$ eindeutig sein muss.
\end{proof}

Diese Eigenschaft von Hausforffräumen wird sich noch als sehr nützlich erweisen, wenn wir später die Theorie der normierten Vektorräume diskutieren werden.

Zum Abschluss dieses Kapitels wollen wir noch eine nützliche Methode kennenlernen, mit welcher man eine Topologien, die auf einer Menge $Y$ gegeben ist, intuitiv gesprochen auf eine andere Menge $X$ mittels einer Funktion $f : X \longrightarrow Y$ vererben kann. Später in dieser Arbeit werden wir noch auf diese Methode zurückkommen.

\begin{lemma}
    Sei $X$ eine beliebige Menge, $(Y, \tau)$ ein topologischer Raum und $f : X \longrightarrow Y$ eine beliebige Abbildung. Dann definiert das Mengensystem 
    \begin{align}
        \tau_f := \{ f^{-1}(\mathcal{U}) \:\: | \:\: \mathcal{U} \in \tau \} \subseteq \mathcal{P}(X)
    \end{align}
    mit $f^{-1}(\mathcal{U}) := \{ x \in X \:\: | \:\: f(x) \in \mathcal{U} \}$, dem sogenannten Urbild von $\mathcal{U}$ unter $f$, eine Topologie auf $X$. $\tau_f$ wird als die Pullback-Topologie unter der Abbildung $f$ bezeichnet.
\end{lemma}

\begin{proof}
    Um zu zeigen, dass es sich bei $\tau_f$ tatsächlich um eine Topologie handelt, überprüfen wir die Topologieaxiome aus der Definition $2.1.$.
    \begin{itemize}
    
        \item[\textit{i)}] Da $\tau$ eine Topologie auf $Y$ ist folgt, dass $\emptyset$ und $Y$ in $\tau$ liegen. Weiter gilt, dass $f^{-1}(\emptyset) = \emptyset$ und $f^{-1}(Y) = X$ ist, woraus folgt, dass $\emptyset$ und $X$ in $\tau_f$ liegen.

        \item[\textit{ii)}] Sei $\mathcal{I}$ eine beliebige Indexmenge und $\mathcal{U}_i \in \tau_f$ für alle $i \in \mathcal{I}$ beliebig gegeben. Für jedes $i \in \mathcal{I}$ gibt es per Definition von $\tau_f$ ein $\mathcal{V}_i \in \tau$ mit $\mathcal{U}_i = f^{-1}(\mathcal{V}_i)$. Dann gilt
        \begin{align}
            \bigcup_{i \in \mathcal{I}} \mathcal{U}_i =& \: \bigcup_{i \in \mathcal{I}} f^{-1}(\mathcal{V}_i) \nonumber \\ =& \: \bigcup_{i \in \mathcal{I}} \{ x \in X \:\: | \:\: f(x) \in \mathcal{V}_i \} \nonumber \\ =& \: \bigg\{ x \in X \:\: \bigg| \:\: f(x) = \bigcup_{i \in \mathcal{I}} \mathcal{V}_i \bigg\} \nonumber \\ =& \: f^{-1} \bigg( \bigcup_{i \in \mathcal{I}} \mathcal{V}_i \bigg). \label{_______}
        \end{align}
        Da $\tau$ eine Topologie auf $Y$ ist und für alle $i \in \mathcal{I}$ gilt, dass $\mathcal{V}_i \in \tau$ ist, folgt, dass $\bigcup_{i \in \mathcal{I}} \mathcal{V}_i \in \tau$ ist. Daraus folgt wiederum, dass per Definition von $\tau_f$ die rechte Seite von \eqref{_______} in $\tau_f$ liegt, woraus wiederum folgt, dass $\bigcup_{i \in \mathcal{I}} \mathcal{U}_i$ in $\tau_f$ ist.

        \item[\textit{iii)}] Sei $\mathcal{I} := \{1,...,n\}$ mit $n \in \mathbb{N}$. Seien weiter $\mathcal{U}_i \in \tau_f$ für alle $i \in \mathcal{I}$ beliebig gegeben. Für jedes $i \in \mathcal{I}$ gibt es per Definition von $\tau_f$ ein $\mathcal{V}_i \in \tau$ mit $\mathcal{U}_i = f^{-1}(\mathcal{V}_i)$. Dann gilt
        \begin{align}
            \bigcap_{i \in \mathcal{I}} \mathcal{U}_i =& \: \bigcap_{i \in \mathcal{I}} f^{-1}(\mathcal{V}_i) \nonumber \\ =& \: \bigcap_{i \in \mathcal{I}} \{ x \in X \:\: | \:\: f(x) \in \mathcal{V}_i \} \nonumber \\ =& \: \bigg\{ x \in X \:\: | \:\: f(x) \in \bigcap_{i \in \mathcal{I}} \mathcal{V}_i \bigg\} \nonumber \\ =& \: f^{-1} \bigg( \bigcap_{i \in \mathcal{I}} \mathcal{V}_i \bigg). \label{'''''''}
        \end{align}
        Da $\tau$ eine Topologie auf $Y$ ist, $\mathcal{I}$ eine endliche Indexmenge darstellt und für alle $i \in \mathcal{I}$ gilt, dass $\mathcal{V}_i \in \tau$ ist, folgt, dass $\bigcap_{i \in \mathcal{I}} \mathcal{V}_i \in \tau$ ist. Per Definition von $\tau_f$ folgt damit, dass die rechte Seite von \eqref{'''''''} in $\tau_f$ liegt, woraus wiederum folgt, dass $\bigcap_{i \in \mathcal{I}} \mathcal{U}_i$ in $\tau_f$ ist.
        
    \end{itemize}
    Damit ist gezeigt, dass es sich bei der Pullback-Topologie $\tau_f$ um eine Topologie auf $X$ handelt.
\end{proof}


\newpage

\section{Metrische Räume}

In diesem Kapitel wollen wir uns nun mit der Frage beschäftigen, warum wir eine Topologie so definiert haben, wie wir sie definiert haben. Die hier angegebenen Definitionen orientieren sich dabei im wesentlichen an \cite{o2006metric}.

Wie wir im letzten Kapitel bereits angemerkt haben, ist der Grund der recht speziell anmutenden Definition einer Topologie damit zu begründen, dass wir unter anderem einen Nähebegriff erklären wollen, der die Begriffe der gewöhnlichen Konvergenz in $\mathbb{R}$ und der Stetigkeit von Funktion $f : \mathbb{R} \longrightarrow \mathbb{R}$ sinnvoll auf allgemeinere Räume, nämlich gerade den topologischen Räumen, verallgemeinert.

Beginnen wir damit, die anschaulichen und gewöhnlichen Definitionen der Konvergenz in $\mathbb{R}$ und der Stetigkeit einer auf $\mathbb{R}$ definierten reellwertigen Funktion zu wiederholen, um diese dann in die Sprache der metrischen Räume zu übersetzen. Wir starten mit der Definition des gewöhnlichen Konvergenzbegriffes in $\mathbb{R}$:

\begin{definition}
    Sei $(x_n)_{n \in \mathbb{N}} \subseteq \mathbb{R}$. Wir sagen, dass $(x_n)_{n \in \mathbb{N}}$ gegen ein $x \in \mathbb{R}$ konvergiert, falls gilt
    \begin{align}
        \forall \epsilon > 0 \: \exists N \in \mathbb{N} : |x_n - x| < \epsilon \:\: \forall n \geq N \label{11}
    \end{align}
\end{definition}

Wir fahren fort mit der Definition der gewöhnlichen Stetigkeit für reellwertige Funktionen der Form $f : \mathbb{R} \longrightarrow \mathbb{R}$:

\begin{definition}
    Sei $f : \mathbb{R} \longrightarrow \mathbb{R}$. Wir nennen $f$ stetig auf $\mathbb{R}$ genau dann, wenn für alle $x \in \mathbb{R}$ gilt, dass
    \begin{align}
        \forall \epsilon > 0 \:\: \exists \delta > 0 : \forall y \in \mathbb{R} \:\: \textit{mit} \:\: |x - y| < \delta \implies |f(x) - f(y)| < \epsilon \label{12}
    \end{align}
\end{definition}

Diese beiden Definitionen, welche beide augfrund der in den Definitionen auftauchenden Betragsfunktion einen anschaulichen Begriff der Nähe codieren, wollen wir nun auf sogenannte metrische Räume verallgemeinern. Dazu müssen wir zuerst den Begriff des metrischen Raumes definieren: 

\begin{definition}
    Sei $X$ eine Menge und $d : X \times X \longrightarrow [0, \infty)$ eine Funktion. Wir nennen $d$ eine Metrik auf $X$, falls 
    \begin{itemize}
    \item[\textit{i)}]   $d(x,y) \geq 0$ und $d(x,y) = 0 \: \iff \: x=y \:\:\:\: \forall x,y \in X$
    \item[\textit{ii)}]  $d(x,y) = d(y,x) \:\:\:\: \forall x,y\in X$
    \item[\textit{iii)}] $d(x,z) \leq d(x,y) + d(y,z) \:\:\:\: \forall x,y,z \in X$
\end{itemize}
Wir bezeichnen dabei das Tupel $(X,d)$ als metrischen Raum.
\end{definition}

Aus dem obigem Kapitel über die Topologie wissen wir bereits, dass jede Funktion $d : X \times X \longrightarrow [0, \infty)$, die den obigen Metrikaxiomen \textit{i)}, \textit{ii)} und \textit{iii)} aus der Definition einer Metrik genügt, als abstrakte Abstandsfunktion aufgefasst werden kann. 

In einem metrischen Raum $(X,d)$ sind wir also in der Lage, mittels $d$ einen Abstand $d(x,y)$ zwischen zwei Punkten $x,y \in X$ zu erklären. Die Metrikaxiome stellen dabei sicher, dass es sich bei der Zahl $d(x,y)$ auch tatsächlich um einen sinnvollen Abstand zwischen den Punkten $x$ und $y$ handelt. 

Mittels der abstrakten Metrik $d$ in einem metrischen Raum, lassen sich nun beide obigen Definitionen der gewöhnlichen Konvergenz in $\mathbb{R}$ und der gewöhnlichen Stetigkeit reellwertiger Funktionen auf $\mathbb{R}$ auf allgemeine metrische Räume $(X,d)$ verallgemeinern:

\begin{definition}
    Sei $(X,d)$ ein metrischer Raum und $(x_n)_{n \in \mathbb{N}}$ eine Folge in $X$. Wir sagen, dass $(x_n)_{n \in \mathbb{N}}$ gegen ein $x \in X$ konvergiert, falls 
    \begin{align}
        \forall \delta > 0 \:\: \exists N \in \mathbb{N} : d(x_n, x) < \delta \:\: \forall n \geq N. \label{13}
    \end{align}
    In diesem Fall sagen wir, dass $x \in X$ der Grenzwert der Folge $(x_n)_{n \in \mathbb{N}}$ ist.
\end{definition}

Wählen wir $X = \mathbb{R}$ und $d = | \cdot |$, so liefert obige Definition gerade die Definition der gewöhnlichen Konvergenz in $\mathbb{R}$. 

Die Verallgemeinerung der Definition der Stetigkeit für eine Funktion $f : X \longrightarrow Y$ zwischen zwei metrischen Räumen $(X, d_X)$ und $(Y, d_Y)$ ist gegeben durch:

\begin{definition}
    Seien $(X, d_X)$ und $(Y, d_Y)$ zwei metrische Räume und $f : X \longrightarrow Y$ eine Funktion. Wir nennen $f$ stetig auf $X$, falls für alle $x \in X$ gilt, dass
    \begin{align}
        \forall \epsilon > 0 \:\: \exists \delta > 0 : \forall y \in X \:\: \textit{mit} \:\: d_X (x,y) < \delta \implies d_Y (f(x), f(y)) < \epsilon \label{14}
    \end{align}
\end{definition}

Auch hier lässt sich wieder leicht einsehen, dass es sich bei dieser Definition um eine direkte Verallgemeinerung der Definition der gewöhnlichen Stetigkeit für Funktionen der Form $f : \mathbb{R} \longrightarrow \mathbb{R}$ handelt. Man setze dazu einfach $X = Y = \mathbb{R}$ und $d = | \cdot |$.

Als Nächstes wollen wir uns mit dem Begriff der von einer Metrik induzierten Topologie auseinandersetzen: Wir betrachten einen metrischen Raum $(X,d)$. Weiter betrachten wir einen beliebigen Punkt $x \in X$ und definieren die sogenannte $\delta$-Umgebung $\mathcal{U}_{\delta}(x)$ von $x$ durch 
\begin{align}
    \mathcal{U}_{\delta}(x) := \{ y \in X \:\: | \:\: d(x,y) < \delta\}.
\end{align}
Wir definieren nun ein Mengensystem $\tau_d \subseteq \mathcal{P}(X)$ durch

\begin{align}
    M \in \tau_d \: \iff \: \forall x \in M \:\: \exists \delta > 0 : \mathcal{U}_{\delta}(x) \subseteq M. \label{metrische Topologie...}
\end{align}
Wir zeigen nun, dass es sich bei diesem so definierten Mengensystem um eine Topologie handelt:

\begin{itemize}
    \item[\textit{i)}] Es gilt $\emptyset \in \tau_d$, da kein $x \in \emptyset$ existiert, was die Bedingung verletzen könnte. Desweiteren gilt $X \in \tau_d$, da für alle $x \in X$ gilt, dass per Definition $\mathcal{U}_{\delta}(x) \subseteq X$ ist.
    \item[\textit{ii)}] Sei $\mathcal{I}$ eine beliebige Indexmenge und $M_i \in \tau_d$ für alle $i \in \mathcal{I}$ beliebig gegeben. Sei $x \in \bigcup_{i \in \mathcal{I}} M_i$. Daraus folgt, dass es ein $j \in \mathcal{I}$ geben muss, sodass $x \in M_j$ ist. Da $M_j \in \tau_d$ ist, folgt, dass es ein $\delta_j > 0$ geben muss, sodass $\mathcal{U}_{\delta_j} (x) \subseteq M_j$. Daraus folgt aber, dass $\mathcal{U}_{\delta_j} (x) \subseteq \bigcup_{i \in \mathcal{I}} M_i$ ist. Da $x$ beliebig aus $\bigcup_{i \in \mathcal{I}} M_i$ gewählt war, folgt $\bigcup_{i \in \mathcal{I}} M_i \in \tau_d$.
    \item[\textit{iii)}] Sei $\mathcal{I} := \{1,...,n\}$ mit $n \in \mathbb{N}$ eine endliche Indexmenge. Seien weiter $M_i \in \tau_d$ für alle $i \in \mathcal{I}$ beliebig gegeben. Sei $x \in \bigcap_{i \in \mathcal{I}} M_i$ beliebig gewählt. Es folgt, dass $x \in M_i$ für alle $i \in \mathcal{I}$ ist. Daraus folgt nun wiederum, dass es für alle $i \in \mathcal{I}$ Zahlen $\delta_i > 0$ gibt, sodass $\mathcal{U}_{\delta_i}(x) \subseteq M_i$. Wir definieren $\delta := \{ \delta_1, ..., \delta_n\}$. Dann gilt $\mathcal{U}_\delta (x) \subseteq \bigcap_{i \in \mathcal{I}} M_i$, denn es gilt $\mathcal{U}_\delta (x) \subseteq \mathcal{U}_{\delta_i} (x) \subseteq M_i$ für alle $i \in \mathcal{I}$. Daraus folgt, dass $\bigcap_{i \in \mathcal{I}} M_i \in \tau_d$ ist.
\end{itemize}

\begin{remark}
    Eine wichtige Eigenschaft von $\tau_d$ ist es nun, dass für alle $x \in X$ und beliebige $\delta > 0$ gilt, dass $\mathcal{U}_\delta (x) \in \tau_d$ ist. Das lässt sich folgendermaßen einsehen: Sei $y \in \mathcal{U}_\delta (x)$ beliebig. Wir setzen $\delta_1 := \delta - d(y,x)$. Sei weiter $z \in \mathcal{U}_{\delta_1} (y)$ beliebig. Dann gilt $d(x,z) \leq d(z,y) + d(y,x) < (\delta - d(y,x)) + d(y,x) = \delta$. Daraus folgt, dass $z \in \mathcal{U}_\delta (x)$. Da $z \in \mathcal{U}_{\delta_1} (y)$ beliebig gewählt war, folgt, dass $\mathcal{U}_{\delta_1} (y) \subseteq \mathcal{U}_\delta (x)$. Da $y \in \mathcal{U}_\delta (x)$ beliebig gewählt war, folgt, dass $\mathcal{U}_\delta (x) \in \tau_d$ ist.
\end{remark}

Es sei anzumerken. dass, wenn wir uns im Spezialfall $X = \mathbb{R}$ mit $d(x,y) = |x - y|$, $\forall x,y \in \mathbb{R}$, befinden, wir die von $d$ induzierte Topologie $\tau_d$ als Standardtopologie (von $\mathbb{R}$) bezeichnen wollen. 

Eine weitere wichtige Eigenschaft metrischer Räume ist, dass, wenn wir diese mit der durch die Metrik $d$ induzierten Topologie $\tau_d$ ausstatten, der topologische Raum $(X, \tau_d)$ automatisch ein Hausdorffraum ist:

\begin{proposition}
    Sei $(X, d)$ ein metrischer Raum und $\tau_d$ die von der Metrik $d$ induzierte Topologie, so folgt, dass der topologische Raum $(X, \tau_d)$ ein Hausdorffraum ist.
\end{proposition}

\begin{proof}
    Wir zeigen, dass $(X, \tau_d)$ der Hausdorffeigenschaft genügt. Seien dazu zwei beliebige Elemente $x, y \in X$ mit der Eigenschaft $d(x,y) := d > 0$ gegeben. Aus $d(x,y) > 0$ folgt mittels der Metrikaxiome, dass $x \neq y$ ist. Wir betrachten nun die Mengen $\mathcal{U}_{\frac{d}{2}} (x)$ und $\mathcal{U}_{\frac{d}{2}} (y)$. Nach obiger Bemerkung gilt $\mathcal{U}_{\frac{d}{2}} (x), \mathcal{U}_{\frac{d}{2}} (y) \in \tau_d$. Wir zeigen, dass $\mathcal{U}_{\frac{d}{2}} (x) \cap \mathcal{U}_{\frac{d}{2}} (y) = \emptyset$: Angenommen das wäre nicht so, d.h. es gäbe ein $z \in \mathcal{U}_{\frac{d}{2}} (x) \cap \mathcal{U}_{\frac{d}{2}} (y)$. Dann gölte per Definition der $\frac{d}{2}$-Umgebungen von $x$ bzw. $y$, dass $d(x,z) < \frac{d}{2}$ und $d(y,z) < \frac{d}{2}$. Dann gilt aber 
    \begin{align}
        d = d(x,y) \leq d(x, z) + d(y,z) < \frac{d}{2} + \frac{d}{2} = d,
    \end{align}
    und damit $d < d$, was offensichtlich ein Widerspruch ist.
\end{proof}

Wir wollen nun zeigen, dass wir die Definitionen $3.4.$ und $3.5.$ auch äquivalent über die induzierten Topologien der beteiligten Metriken formulieren können. Wir beginnen mit dem Konvergenzbegriff von Folgen in metrischen Räumen:

\begin{proposition}
    Sei $(X,d)$ ein metrischer Raum und $\tau_d$ die zugehörige durch die Metrik $d$ induzierte Topologie. Sei weiter $(x_n)_{n \in \mathbb{N}} \subseteq X$ eine Folge in $X$ und $x \in X$. Dann sind folgende Aussagen äquivalent:

    \begin{itemize}
        \item[\textit{i)}] $(x_n)_{n \in \mathbb{N}}$ konvergiert bezüglich der Metrik $d$ gegen $x$, d.h. 
        \begin{align}
            \forall \delta > 0 \:\: \exists N > 0 : d(x_n,x) < \delta \:\: \forall n \geq N.
        \end{align}
        
        \item[\textit{ii)}] $(x_n)_{n \in \mathbb{N}}$ konvergiert bezüglich der Topologie $\tau_d$ gegen $x$, d.h. 
        \begin{align}
            \forall \mathcal{O} \in \tau_d \:\: \textit{mit} \:\: x \in \mathcal{O} : \exists N \in \mathbb{N} : x_n \in \mathcal{O} \:\: \forall n \geq N.
        \end{align}
           
    \end{itemize}
\end{proposition}

\begin{proof}
    \textit{i)} $\implies$ \textit{ii)}: Wir bemerken zuerst, dass wegen der Definition  von $\mathcal{U}_\delta (X)$ der Ausdruck in \textit{i)} äquivalent zu 
    \begin{align}
        \forall \delta > 0 \:\: \exists N > 0 : x_n \in \mathcal{U}_\delta (x) \:\: \forall n \geq N
    \end{align}
    ist. Sei nun $\mathcal{U} \in \tau_d$ mit $x \in \mathcal{U}$ beliebig gewählt. Dann existiert ein $\rho > 0$, sodass $\mathcal{U}_\rho (x) \subseteq \mathcal{U}$ ist. Aus dem äquivalenten Ausdruck zu \textit{i)} folgt nun aber, dass für dieses $\rho$ gerade gilt, dass ein $N \in \mathbb{N}$ existiert, sodass $x_n \in \mathcal{U}_\delta (x) \subseteq \mathcal{U}$ für alle $n \geq N$ gilt, was diese erste Richtung zeigt.

    \textit{ii)} $\implies$ \textit{i)}: Angenommen $\textit{ii)}$ gilt. Wir wählen in $\textit{ii)}$ für die Menge $\mathcal{O}$ die Menge $\mathcal{U}_\delta (x)$ mit beliebigem $\delta > 0$. Das dürfen wir tun, da $x \in \mathcal{U}_\delta (x)$ und nach obiger Bemerkung auch gilt, dass $\mathcal{U}_\delta (x) \in \tau_d$ ist. Da $\delta > 0$ beliebig gewählt wurde, folgt aus der Definition von $\mathcal{U}_\delta (x)$ die Aussage \textit{i)}.
\end{proof}

Mit diesem Satz haben wir nun unter anderem auch gezeigt, dass in einem metrischen Raum der Grenzwert einer konvergenten Folge, definiert über die Metrik $d$, eindeutig bestimmt ist, da wir im Satz $3.1.$ gezeigt haben, dass $(X, \tau_d)$ ein Hausdorffraum ist und wir aus Satz $2.1.$ wissen, dass konvergente Folgen in Hausdorffräume eindeutig bestimmte Grenzwerte haben.

Als Nächstes zeigen wir, dass der metrische Stetigkeitsbegriff einer Funktion $f : X \longrightarrow Y$ zwischen metrischen Räumen äquivalent über die zugehörigen induzierten Topologien erklärt werden kann:

\begin{proposition}
    Seien $(X, d_X)$ und $(Y, d_Y)$ zwei metrische Räume und sei $\tau_{d_X}$ die zur Metrik $d_X$ und $\tau_{d_Y}$ die zur Metrik $d_Y$ zugehörige induzierte Topologie. Sei weiter $f : X \longrightarrow Y$ eine Funktion zwischen $X$ und $Y$. Dann sind folgende Aussagen äquivalent:

    \begin{itemize}
        \item[\textit{i)}] $f$ ist eine stetige Funktion von $X$ nach $Y$ bezüglich der Metriken $d_X$ und $d_Y$, d.h. für alle $x \in X$ gilt 
        \begin{align}
            \forall \epsilon > 0 \:\: \exists \delta > 0 : \forall y \in X \:\: \textit{mit} \:\: d_X (x,y) < \delta \implies d_Y (f(x), f(y)) < \epsilon.
        \end{align}
        
        \item[\textit{ii)}] $f$ ist eine stetige Funktion von $X$ nach $Y$ bezüglich der Topologien $\tau_{d_X}$ und $\tau_{d_Y}$, d.h. 
        \begin{align}
            \forall \mathcal{V} \in \tau_{d_Y} : f^{-1}(\mathcal{Y}) \in \tau_{d_X}.
        \end{align}
        Dabei bezeichnet $f^{-1}(\mathcal{V}) := \{ x \in X \:\: | \:\: f(x) \in \mathcal{V} \}$ wieder das Urbild der Menge $\mathcal{V}$ unter $f$.
    \end{itemize}
\end{proposition}

\begin{proof}
    \textit{i)} $\implies$ \textit{ii)}: Sei $\mathcal{V} \in \tau_{d_Y}$ beliebig gewählt. Aus der Definition von $\tau_{d_Y}$ folgt, dass zu jedem $z \in \mathcal{V}$ ein $\epsilon_z > 0$ existieren muss, sodass für alle $z \in \mathcal{V}$ gilt, dass 
    \begin{align}
        \mathcal{V}^{Y} (\epsilon_z, z) := \{ w \in Y \:\: | \:\: d_Y(z, w) < \epsilon_z \} \subseteq \mathcal{V}
    \end{align}
    ist. Weiter gilt, dass 
    \begin{align}
        \mathcal{V} = \bigcup_{z \in \mathcal{V}} \mathcal{V}^{Y}(\epsilon_z, z),
    \end{align}
    denn sei $w \in \mathcal{V}$, so existiert zum einen ein $\epsilon_w > 0$, sodass $\mathcal{V}^{Y}(\epsilon_w, w) \subseteq \mathcal{V} \subseteq \bigcup_{z \in \mathcal{V}} \mathcal{V}^{Y} (\epsilon_z, z)$ gilt, während zum anderen $\bigcup_{z \in \mathcal{V}} \mathcal{V}^{-1}(\epsilon_z, z) \subseteq \mathcal{V}$ aus der Tatsache folgt, dass für alle $z \in \mathcal{V}$ die Menge $\mathcal{V}^{Y}(\epsilon_z, z) \subseteq \mathcal{V}$ ist. Wir müssen nun zeigen, dass $f^{-1}(\mathcal{V}) \in \tau_{d_X}$ ist: Angenommen $f^{-1}(\mathcal{V}) \neq \emptyset$ (andernfalls wäre nichts zu zeigen, da $\emptyset \in \tau_{d_X}$), so existiert ein $x \in f^{-1}(\mathcal{V})$, sodass $f(x) \in \mathcal{V}$ gilt. Wegen $\mathcal{V} = \bigcup_{z \in \mathcal{V}} \mathcal{V}^{Y}(\epsilon_z, z)$ folgt, dass $f(x) \in \mathcal{V}^{Y} (\epsilon_{f(x)}, f(x)) \subseteq \mathcal{V}$ ist. Aus \textit{i)} folgt, dass ein $\delta > 0$ existiert, sodass gilt: 
    \begin{align}
        \forall y \in X \:\:\textit{mit} \:\: d_X (x,y) < \delta \implies d_Y(f(x), f(y)) < \epsilon_z
    \end{align}
    oder anders ausgedrückt 
    \begin{align}
        \forall y \in X \:\:\textit{mit} \:\: y \in \mathcal{U}^{X}_\delta (x) \implies f(y) \in \mathcal{V}^{Y} (\epsilon_{f(x)}, f(x)) \subseteq \mathcal{V},
    \end{align}
    wobei $\mathcal{U}^{X}_\delta (x) := \{ u \in X \:\: | \:\: d_X(x,u) < \delta \}$ ist. Daraus folgt nun aber sofort $\mathcal{U}^{X}_\delta (x) \subseteq f^{-1}(\mathcal{V})$. Da $x \in f^{-1}(\mathcal{V})$ beliebig gewählt war folgt, dass $f^{-1}(\mathcal{V}) \in \tau_{d_X}$ ist, was die Aussage \textit{ii)} zeigt.

    \textit{ii)} $\implies$ \textit{i)}: Sei $x \in X$ und ein $\epsilon > 0$ beliebig gewählt. Aus obiger Bemerkung wissen wir, dass die Menge 
    \begin{align}
        \mathcal{V}^{Y}_\epsilon (f(x)) := \{ z \in Y \:\: | \:\: d_Y(f(x),z) < \epsilon \}
    \end{align}
    in der Topologie $\tau_{d_Y}$ liegt, d.h. $\mathcal{V}^{Y}_\epsilon (f(x)) \in \tau_{d_Y}$. Da \textit{ii)} gilt, folgt, dass 
    \begin{align}
        f^{-1}(\mathcal{V}^{Y}_\epsilon (f(x))) \in \tau_{d_X}
    \end{align}
    ist. Da $x \in f^{-1}(\mathcal{V}^{Y}_\epsilon (f(x)))$ ist, folgt mit der Definition der Topologie $\tau_{d_X}$, dass ein $\delta > 0$ existiert, sodass 
    \begin{align}
        \mathcal{U}^{X}_\delta (x) := \{ y \in X \:\: | \:\: d_X (x,y) < \delta\} \subseteq f^{-1}(\mathcal{V}^{Y}_\epsilon (f(x))).
    \end{align}
    Aus diesem Ausdruck folgt aber sofort, dass für alle $y \in X$ mit $d_X(x,y) < \delta$ gilt, dass $d_Y(f(x),f(y)) < \epsilon$ ist. Da $x$ und $\epsilon$ beliebig gewählt waren, folgt die Aussage.
\end{proof}

Damit haben wir gesehen, dass die beiden anschaulichen Begrifflichkeiten der Konvergenz einer Folge und der Stetigkeit einer Funktion in metrischen Räumen rein über (induzierte) Topologien definierbar und dabei äquivalent zu den zugehörigen \textit{metrischen} Definitionen sind. Da jede Metrik eine Topologie erzeugt, handelt es sich bei den topologischen Räumen um die allgemeinere Struktur. 

Wollen wir nun also die anschaulichen Begriffe der Konvergenz und der Stetigkeit einer Funktion auf sinnvolle Weise auf allgemeine topologische Räume verallgemeinern, so zeigen die letzten beiden Sätze, wie man das machen kann. Der Vollständigkeitshalber wollen wir noch kurz die Definition des Stetigkeitsbegriffes in allgemeinen topologischen Räumen angeben:

\begin{definition}
    Seien $(X, \tau_X)$ und $(Y, \tau_Y)$ zwei topologische Räume und $f : X \longrightarrow Y$ eine Funktion. Wir nennen $f$ $(\tau_X, \tau_Y)$-stetig oder auch einfach nur stetig (wenn sich die zugehörigen Topologien aus dem Kontext ergeben), falls gilt
    \begin{align}
        \forall \mathcal{V} \in \tau_Y : f^{-1}(\mathcal{V}) \in \tau_X
    \end{align}
\end{definition}

Betrachten wir nun noch kurz eine nützliche Eigenschaft stetiger Funktionen.

\begin{lemma}
    Seien $(X, \tau_X)$, $(Y, \tau_Y)$ und $(Z, \tau_Z)$ jeweils topologische Räume und $f : X \longrightarrow Y$ und $g : Z \longrightarrow X$ seien jeweils stetige Abbildungen. Dann gilt, dass die Funktion $h : Z \longrightarrow Y$, definiert durch 
    \begin{align}
        h(z) := f(g(z)) \:\:\:\: \forall z \in Z,
    \end{align}
    eine stetige Abbildung ist.
\end{lemma}

\begin{proof}
    Sei $\mathcal{V} \in \tau_Y$ beliebig. Da $f : X \longrightarrow Y$ stetig ist, folgt, dass $f^{-1}(\mathcal{V}) \in \tau_X$ ist. Da $g : Z \longrightarrow X$ stetig ist, folgt, dass $g^{-1}(f^{-1}(\mathcal{V})) \in \tau_Z$ ist. Weiter gilt
    \begin{align}
        g^{-1}(f^{-1}(\mathcal{V})) =& \: \{ z \in Z \:\: | \:\: g(z) \in f^{-1}(\mathcal{V}) \} \nonumber \\ =& \: \{ z \in Z \:\: | \:\: f(g(z)) \in \mathcal{V}\} \nonumber \\ =& \: \{ z \in Z \:\: | \:\: h(z) \in \mathcal{V} \} \nonumber \\ =& \: h^{-1}(\mathcal{V}),
    \end{align}
    woraus die Behauptung folgt.
\end{proof}

Fassen wir nun also kurz zusammen, warum wir eine Topologie so definiert haben, wie wir sie definiert haben: In einem uns vertrauten metrischen Raum, nämlich $(\mathbb{R}, d)$ mit $d(x,y) = |x-y|$ $\forall x,y \in \mathbb{R}$, lassen sich die intuitiven Begriffe der Konvegrenz von Folgen und der Stetigkeit reeller Funktionen mittels der Metrik $d$ erklären, welche einen natürlichen und intuitiven Nähebegriff einführt. Die Metrik $d$ induziert dabei ein spezielles Mengensystem, $\tau_d$, welches den Axiomen einer Topologie genügt. 

Mittels diesem Mengensystem lassen sich die intuitiven Definitionen der Konvergenz von Folgen und der Stetigkeit reeller Funktionen umformulieren, sodass die Metrik $d$ nicht mehr explizit auftaucht. Das heißt, dass wir diese Begriffe ohne expliziten Rückgriff auf einen Längenbegriff formulieren können und das nur dieses spezielle Mengensystem $\tau_d$, welches wir induzierte Topologie nennen, von nöten ist, um diese Begriffe zu definieren. 

Wollen wir nun zu allgemeineren Räumen $X$ übergehen, in welchen vorab keine Metrik definiert sein muss, in denen wir aber trotzdem über Begriffe wie die der Konvergenz sprechen wollen, so führen wir diese Begriffe über eine Teilmenge der Potenzmenge von $X$ ein. 

Wir nutzen dabei aber nicht beliebige Teilmengen der Potenzmenge, sondern beschränken uns auf sogenannte Topologien. Das heißt, dass, wenn wir einen Nähebegriff auf einer Menge einführen wollen, wir uns auf spezielle Teilmengen der Potenzmenge konzentrieren, die gewisse Eigenschaften des metrikinduzierten Nähebegriffes nachahmen. 

Der Vorteil davon ist, dass damit unser Nähebegriff auf einem topologischen Raum noch genügend Struktur besitzt, sodass wichtige und interessante Sätze für einen topologischen Raum gezeigt werden können, denn je mehr Struktur ein spezielles mathematisches Objekt hat, desto schönere Eigenschaften gehen damit einher. 

Eine gute Veranschaulichung davon bieten Hausdorffräume: Wollen wir die Eindeutigkeit des Grenzwertes einer Folge, so müssen wir die Menge der Topologien, welche wir betrachten wollen, einschränken. So konzentrieren wir uns in diesem Fall beispielsweise nur auf diejenigen Topologien, welche der Hausdorff-Eigenschaft genügen. Das heißt, wir geben der Topologie mehr Struktur, weil wir damit schönere Eigenschaften wie die Eindeutigkeit des Grenzwertes einer Folge zeigen können. 

Aus genau diesem Grund schränkem wir uns auch bei der Einführung eines Nähebegriffes auf einer Menge $X$ auf spezielle Teilmengen von $\mathcal{P}(X)$ ein. Da das Mengensystem $\tau_d$ im Falle metrischer Räume von großer Bedeutung war, liegt es nun natürlich Nahe, Eigenschaften von diesem Mengensystem als Forderung an unsere Teilmenge $\tau \subseteq \mathcal{P}(X)$ herzunehmen. 

Man kann sich abschließend nun noch fragen, warum wir ausgerechnet diese drei ganz bestimmten Eigenschaften vom Mengensystem $\tau_d$ als Axiome einer allgemeinen Topologie $\tau \subseteq \mathcal{P}(X)$ hernehmen, und nicht irgendwelche anderen Eigenschaften von $\tau_d$. Der Grund ist der Folgende: Da wir einen Nähebegriff über eine (spezielle) Teilmenge der Potenzmenge von $X$ einführen möchten, arbeiten wir in einem sehr mengentheoretischen Rahmen. Die fundamentalen Operartionen der Mengentheorie sind dabei die Mengenvereiningung $\cup$, der Mengenschnitt $\cap$ und die Komplementbildung. 

Es macht also zuallererst einmal Sinn, die Eigenschaften von $\tau_d$ bezüglich $\cup$ und $\cap$ als Axiome für $\tau$ zu nutzen. Beachte dabei, dass wir für eine induzierte Topologie $\tau_d$ nur beweisen konnten, dass endliche Mengenschnitte von Elementen aus $\tau_d$ wieder in $\tau_d$ liegen. Das im allgemeinen keine beliebige Mengenschnitte von Elementen aus $\tau_d$ wieder in $\tau_d$ liegen müssen, sieht man anhand des folgenden Beispiels: Sei der topologische Raum $(\mathbb{R}, \tau_d)$ mit $d(x,y) = |x-y|$ für alle $x,y \in \mathbb{R}$ gegeben. Wir betrachten die Menge $\{ (-\frac{1}{n}, \frac{1}{n}) \in \tau_d \:\: | \:\: n \in \mathbb{N} \}$. Dann gilt
\begin{align}
    \bigcap_{n \in \mathbb{N}} \bigg(-\frac{1}{n}, \frac{1}{n}\bigg) = \{ 0 \}
\end{align}
und $\{ 0 \}$ ist bzgl. $\tau_d$ nicht offen. Daraus folgt sofort, dass für eine allgemeine metrikinduzierte Topologie der sogenannte Abschluss bzgl. des Mengenschnittes nur für endlich viele Mengen aus $\tau_d$ besteht.

Weiter sollen alle Punkte in $X$ irgendwie zueinander in Beziehung stehen, d.h. irgendwie nah zueinander sein, weshalb es auch Sinn macht zu fordern, dass $X$ selbst in $\tau$ liegt, so wie $\mathbb{R} \in \tau_d$ gilt. 

Bezüglich der Komplementbildung wissen wir bezüglich $\tau_d$ ersteinmal nur, dass das Komplement von $\mathbb{R}$, also $\emptyset$, sicher in $\tau_d$ liegt. Ansonsten ist $\tau_d$ aber im allgemeinen nicht abgeschlossen über die Komplementbildung seiner Elemente. Das kann man wie folgt sehen: Wir betrachten abermals den topologischen Raum $(\mathbb{R}, \tau_d)$ mit $d(x,y) = |x-y|$ für alle $x,y \in \mathbb{R}$. Dann ist das Komplement der offenen Menge $(a,b) \in \tau_d$ nicht in $\tau_d$, denn das Komplement von $(a,b)$ ist die Menge $(-\infty, a] \cup [b, \infty) =: \mathcal{W}$. Da sich für die Punkte $a$ und $b$ kein $\epsilon > 0$ finden lässt, sodass die zugehörige $\epsilon$-Umgebung noch in der Menge $\mathcal{W}$ liegt, folgt, dass $\mathcal{W}$ nicht offen bzgl. $(\mathbb{R}, \tau_d)$ ist. Daher fordern wir für eine abstrakte Topologie $\tau$ auch nur, dass $\emptyset$ in $\tau$ liegt.

Damit haben wir nun also geklärt, warum Topologien so definiert sind, wie sie definiert sind, und wir haben damit auch gesehen, dass unser Nähebegriff auf einem topologischen Raum $(X, \tau)$ noch nah genug am Nähebegriff eines metrischen Raumes dran ist, sodass wir genügend Struktur zur Verfügung haben, um gut mit diesen arbeiten zu können.\\

Kehren wir nun also nach unseren kurzen Ausflug in die Theorie metrischer Räume zu unserer ursprünglichen Frage zurück, wie wir Kurven in einer Menge $X$ mathematisch modellieren. Im Kapitel $2$ hatten wir bereits begründet, warum eine Kurve $\mathcal{F} \subseteq X$, welche anschaulich in $x_1 \in X$ beginnt und in $x_2 \in X$ endet, über eine Funktion $\gamma : [a,b] \longrightarrow \mathcal{F}$ mit $\gamma (a) = x_1$ und $\gamma (b) = x_2$ modelliert werden sollte. 

Zusätzlich hatten wir im Kapitel $2$ noch angemerkt, dass wir bezüglich $\gamma$ auch eine Form der Stetigkeit fordern sollten. Wir wissen nun, wie wir diesen Begriff einführen: Zuerst statten wir $X$ und $[a,b]$ mit einer Topologie aus, sodass wir dann fordern können, dass die Funktion $\gamma$ bezüglich dieser Topologien stetig sein soll. 

Während wir dabei im Raum $X$ beliebige Topologien zu lassen wollen, möchten wir im folgenden eine spezielle Topologie auf $[a,b] \subseteq \mathbb{R}$ einführen, nämlich gerade die Standardtopologie von $\mathbb{R}$. Der Grund dafür ist, dass $[a,b]$ einer vertrauten Menge entspricht, auf der wir deshalb auch einen vertrauten Nähebegriff etablieren wollen. 

Es gibt nur ein Problem, welches uns daran hindert, diesen intuitiven Nähebegriff auf $[a,b]$ einzuführen: Die Standardtopologie von $\mathbb{R}$ ist auf ganz $\mathbb{R}$ definiert, während $[a,b]$ ein endliches Intervall in $\mathbb{R}$ ist. Wie also vererben wir eine Topologie, die auf $\mathbb{R}$ erklärt ist, auf eine Teilmenge $[a,b] \subseteq \mathbb{R}$? Die Beantwortung dieser Frage führt auf die folgende Definition:

\begin{definition}
    Sei $(X, \tau)$ ein topologischer Raum und $\mathcal{W} \subseteq X$. Wir erklären die sogenannte Teilraumtopologie von $\mathcal{W}$ als das Mengensystem 
    \begin{align}
        \tau_\mathcal{W} := \{ \mathcal{U} \cap \mathcal{W} \:\: | \:\: \mathcal{U} \in \tau  \}.
    \end{align}
\end{definition}

Bei der sogenannten Teilraumtopologie handelt es sich tatsächlich um eine Topologie:

\begin{lemma}
    Sei $(X, \tau)$ ein topologischer Raum und $\mathcal{W} \subseteq X$. Dann handelt es sich bei der Teilraumtopologie $\tau_\mathcal{W}$ um eine Topologie.
\end{lemma}

\begin{proof}
    Wir prüfen die Topologieaxiome:

    \begin{itemize}
        \item[\textit{i)}] Da $\emptyset \in \tau$ und $\emptyset \cap \mathcal{W} = \emptyset$ gilt, folgt, dass $\emptyset \in \tau_\mathcal{W}$ ist. Weiter gilt $X \in \tau$ und $X \cap \mathcal{W} = \mathcal{W}$, was impliziert, dass $\mathcal{W} \in \tau_\mathcal{W}$.
        
        \item[\textit{ii)}] Sei $\mathcal{I}$ eine beliebige Indexmenge und $M_i \in  \tau_\mathcal{W}$ für alle $i \in \mathcal{I}$ beliebig gegeben. Wegen $M_i \in \tau_\mathcal{W}$ für alle $i \in \mathcal{I}$ folgt, dass für alle $i \in \mathcal{I}$ ein $\mathcal{U}_i \in \tau$ existiert, sodass $M_i = \mathcal{U}_i \cap \mathcal{W}$ ist. Es gilt damit 
        \begin{align}
            \bigcup_{i \in \mathcal{I}} M_i = \bigcup_{i \in \mathcal{I}} (\mathcal{U}_i \cap \mathcal{W}) = \big(\bigcup_{i \in \mathcal{I}} \mathcal{U}_i\big) \cap \mathcal{W}, \label{a}
        \end{align}
        woraus per Definition sofort folgt, das $\bigcup_{i \in \mathcal{I}} M_i \in \tau_\mathcal{W}$ gilt, da $\bigcup_{i \in \mathcal{I}} \mathcal{U}_i \in \tau$ ist. Die letzte Gleichheit in \eqref{a} folgt dabei folgendermaßen: Sei $x \in \bigcap_{i \in \mathcal{I}} (\mathcal{U}_i \cap \mathcal{W})$, so existiert ein $j \in \mathcal{I}$, sodass $x \in \mathcal{U}_j \cap \mathcal{W} \subseteq \big(\bigcup_{i \in \mathcal{I}} \mathcal{U}_i \big) \cap \mathcal{W}$, da $\mathcal{U}_j \subseteq \big(\bigcup_{i \in \mathcal{I}} \mathcal{U}_i \big)$. Umgedreht gilt für $x \in \big(\bigcup_{i \in \mathcal{I}} \mathcal{U}_i \big) \cap \mathcal{W}$, dass $x \in \bigcup_{i \in \mathcal{I}} \mathcal{U}_i$ und $x \in \mathcal{W}$ ist. Daraus folgt, dass ein $j \in \mathcal{I}$ existiert, sodass $x \in \mathcal{U}_j$ und $x \in \mathcal{W}$. Das heißt aber gerade, dass $x \in \mathcal{U}_j \cap \mathcal{W} \subseteq \bigcup_{i \in \mathcal{I}} (\mathcal{U}_i \cap \mathcal{W})$ gilt.
        
        \item[\textit{iii)}] Sei $\mathcal{I} := \{1,...,n\}$ mit $n \in \mathbb{N}$. Seien weiter $M_i \in \tau_\mathcal{W}$ für alle $i \in \mathcal{I}$ beliebig gegeben. Wegen $M_i \in \tau_\mathcal{W}$ für alle $i \in \mathcal{I}$ folgt, dass für alle $i \in \mathcal{I}$ ein $\mathcal{U}_i \in \tau$ existiert, sodass $M_i = \mathcal{U}_i \cap \mathcal{W}$. Es gilt damit 
        \begin{align}
            \bigcap_{i \in \mathcal{I}} M_i = \bigcap_{i \in \mathcal{I}} (\mathcal{U}_i \cap \mathcal{W}) = \big( \bigcap_{i \in \mathcal{I}} \mathcal{U}_i \big) \cap \mathcal{W}, 
        \end{align}
        woraus sofort folgt, dass $\bigcap_{i \in \mathcal{I}} M_i \in \tau_\mathcal{W}$ gilt, da $\bigcap_{i \in \mathcal{I}} \mathcal{U}_i \in \tau$ ist. Die letzte Gleichheit folgt dabei folgendermaßen: Sei $x \in  \bigcap_{i \in \mathcal{I}} (\mathcal{U}_i \cap \mathcal{W})$, dass heißt $x \in \mathcal{U}_i \cap \mathcal{W}$ für alle $i \in \mathcal{I}$. Daraus folgt aber, dass $x \in \mathcal{U}_i$ für alle $i \in \mathcal{I}$ und $x \in \mathcal{W}$, woraus $x \in \big( \bigcap_{i \in \mathcal{I}} \mathcal{U}_i \big) \cap \mathcal{W}$ folgt. Umgedreht folgt aus $x \in \big( \bigcap_{i \in \mathcal{I}} \mathcal{U}_i \big) \cap \mathcal{W}$, dass $x \in \bigcap_{i \in \mathcal{I}} \mathcal{U}_i$ und $x \in \mathcal{W}$ gilt, woraus folgt, dass $x \in \mathcal{U}_i \cap \mathcal{W}$ für alle $i \in \mathcal{I}$, was schlussendlich die Gleichheit zeigt.
    \end{itemize}
\end{proof}

Betrachten wir die Definition der Teilraumtopologie mit dem Hintergrundwissen, dass es sich bei dieser auch tatsächlich um eine Topologie auf $\mathcal{W} \subseteq X$ handelt, so sehen wir auch, dass diese tatsächlich eine Vererbung der topologische Struktur auf $X$ auf die Teilmenge $\mathcal{W}$ darstellt, denn immerhin fügen wir der Topologie $\tau_\mathcal{W}$ keine vollkommen neuen Elemente hinzu. So liegen in $\tau_\mathcal{W}$ nur diejenigen offenen Mengen aus $\tau$. die onehin schon in $\mathcal{W}$ enthalten waren, und die Schnitte aller weiteren offenen Mengen mit $\mathcal{W}$. 

Mittels der Teilraumtopologie sind wir nun auch in der Lage, eine weitere wichtige Eigenschaft stetiger Funktionen zu formulieren und zu beweisen, die wir im späteren Verlauf dieser Arbeit noch brauchen, wenn wir uns mit Einschränkungen stetiger Kurven auf Teilintervalle beschäftigen werden.

\begin{lemma}
    Seien $(X, \tau_X)$ und $(Y, \tau_Y)$ zwei topolgische Räume, $A \subseteq X$ eine beliebige Teilmenge von $X$ und $f : X \longrightarrow Y$ eine stetige Funktion. Weiter definieren wir die sogenannte Einschränkung der Funktion $f$ auf die Menge $A$ als die Funktion $f |_A : A \longrightarrow Y$, definiert durch $f |_A (x) = f(x)$ für alle $x \in A$. Statten wir $A$ mit der durch $\tau_X$ induzierten Teilraumtopologie $\tau_{X,A}$ aus, so ist die Funktion $f |_A$ $(\tau_{X, A}, \tau_Y)$-stetig.
\end{lemma}

\begin{proof}
    Sei $\mathcal{V} \in \tau_Y$ beliebig gewählt. Wir müssen zeigen, dass die Menge $(f |_A)^{-1} (\mathcal{V})$ in der Topologie $\tau_{X, A}$ enthalten ist. Es gilt
    \begin{align}
        (f |_A)^{-1} (\mathcal{V}) =& \: \{ x\in A \:\: | \:\: f |_A (x) \in \mathcal{V} \} \nonumber \\ =& \: \{  x \in X \:\: | \:\: x \in A, \: f |_A (x) \in \mathcal{V} \} \nonumber \\ =& \: \{ x \in X \:\: | \:\: x \in A, \: f(x) \in \mathcal{V} \} \nonumber \\ =& \: A \cap \{ x \in X \:\: | \:\: f(x) \in \mathcal{V} \} \nonumber \\ =& \: A \cap f^{-1} (\mathcal{V}). 
    \end{align}
    Da $f$ stetig ist, folgt, dass $f^{-1}(\mathcal{V}) \in \tau_X$ ist. Daraus folgt aber per Definition der Teilraumtopologie $\tau_{X, A}$, dass $A \cap f^{-1}(\mathcal{V})$ in $\tau_{X, A}$ liegt. Daraus folgt, dass $f |_A$ $(\tau_{X,A}, \tau_Y)$-stetig ist.
\end{proof}

Insgesamt wissen wir mittels des Begriffes der Teilraumtopologie nun auch, wie wir eine Topologie, welche auf $\mathbb{R}$ erklärt ist, auf ein endliches Intervall $[a,b]$ vererben können, sodass $[a,b]$ selbst ein topologischer Raum wird. 

Fassen wir daher unsere bisherige Erkenntnisse in Bezug auf eine Kurve in einem topologischen Raum $X$ in einer Definition zusammen.

\begin{definition}
    Sei $(X, \tau)$ ein topologischer Raum, $x_1$ und $x_2$ aus $X$ mit $x_1 \neq x_2$ vorgegeben, sowie $[a,b] \subseteq \mathbb{R}$, mit $-\infty < a < b < + \infty$, ein Intervall, ausgestattet mit der Teilraumtopologie $\sigma_{[a,b]}$, wobei wir mit $\sigma$ die Standardtopologie auf $\mathbb{R}$ bezeichnen. Dann nennen wir die stetige Funktion
    \begin{align}
        \gamma : [a,b] \longrightarrow X,
    \end{align}
    mit $\gamma (a) = x_1$ und $\gamma (b) = x_2$ eine Kurve in $X$, die $x_1$ mit $x_2$ verbindet.
\end{definition}

Zu beachten ist, dass gemäß der obigen Definition die stetige Funktion $\gamma$ mit den Eigenschaften \eqref{2} und \eqref{3} als \textit{die} Kurve bezeichnen wird. Die Menge $\mathcal{F} \subseteq X$ bezeichnen wir stattdessen als die sogenannte Spur der Kurve $\gamma$. Der Grund dafür ist, dass wenn wir Kurven studieren wollen, wir im wesentlichen mit den Funktionen arbeiten, welche  die Menge $\mathcal{F}$ parametrisieren, statt mit der bloßen Menge $\mathcal{F}$. Im weiteren Verlauf dieser Arbeit werden wir auch sehen, warum das so ist. 

Damit wissen wir nun also, wie wir im allgemeinen den Begriff der Kurve zu definieren haben. Da wir uns später mit speziellen Kurven, sogenannten \textit{(schwachen) Geodätischen}, auf sogenannten \textit{Mannigfaltigkeiten} (speziellen topologischen Räumen) beschäftigen wollen, wird der Begriff der Kurve in dieser Arbeit noch eine sehr wichtige Rolle spielen. 

Bevor wir aber über (schwache) Geodätische sprechen können, müssen wir zuerst den Begriff der Mannigfaltigkeit einführen. Dazu wiederholen wir im Folgenden einige wichtige Begriffe aus der linearen Algebra, welche notwendig für die Definition des Mannigfaltigkeitenbegriffes sind.


\newpage

\section{Wichtige Begriffe aus der (linearen) Algebra}

In diesem Kapitel wollen wir uns unter anderem mit dem wichtigen Begriff des Prähilbertraumes und des topologischen Vektorraumes auseinandersetzen. Den ersteren Begriff werden wir dabei später benötigen, um den Begriff einer endlichdimensionalen riemannschen Mannigfaltigkeit definieren zu können, auf welchen wir später das nun schon oft erwähnte Längenfunktional erklären möchten. 

Den Begriff eines topologischen Vektorraumes werden wir hingegen später in Kapitel $9$ benötigen, in welchem wir unter anderem grob sehen werden, wie sich die Begrifflichkeiten aus dem endlichdimensionalen in das Unendlichdimensionale verallgemeinern lassen.

Viele der hier besprochenen Definitionen, Sätze und Beweise beziehen sich dabei auf \cite{fischer2003lineare}.


\subsection{Gruppentheorie}

Ein erster wichtiger algebraischer Begriff ist der einer mathematischen Gruppe, welcher die Basis vieler Konstruktionen sein wird, die wir noch in den kommenden Abschnitten diskutieren werden. 

\begin{definition}
    Sei $G$ eine Menge, zusammen mit einer binären Operation
    \begin{align}
        \circ : G \times G \longrightarrow& \: G, \nonumber\\
         (g,h) \longmapsto& \: \: g \circ h =: gh \in G
    \end{align}
    Wir nennen das Tupel $(G, \circ)$ eine Gruppe, falls die folgenden Bedingungen erfüllt sind:

    \begin{itemize}
        \item[\textit{i)}] $g \circ (h \circ k) = (g \circ h) \circ k \:\:\:\:\:\:\:\:\:\:\:\:\:\:\:\:\: \forall g, h, k \in G$
        
        \item[\textit{ii)}] $\exists e \in G : e \circ g = g \circ e = g \:\:\:\:\:\:\:\:\:\:\:\:\: \forall g \in G$ 
        
        \item[\textit{iii)}] $\forall g \in G \:\: \exists g^{-1} \in G: g \circ g^{-1} = g^{-1} \circ g = e$
    \end{itemize} 
    Die Bedingung \textit{i)} wird auch kurz als Assoziativität (der binären Operation) bezeichnet. Gilt außerdem noch, dass 
    \begin{align}
        g \circ h = h \circ g \:\:\:\:\:\:\:\: \forall g,h \in G,
    \end{align}
    so heißt die Operation $\circ$ kommutativ und die Gruppe $(G, \circ)$ abelsch.
\end{definition}

Die Bedingung \textit{ii)} in obiger Definition impliziert dabei die Existenz eines sogenannten \textit{neutralen Elements} $e \in G$, welches die Eigenschaft hat, dass, wenn man es mit einem beliebigem Element $g \in G$ der Gruppe verknüpft, als Ergebnis nur wieder $g$ herauskommt. Die Bedingung \textit{iii)} in der Definition der mathematischen Gruppen sagt stattdessen aus, dass zu jedem Element $g \in G$ ein sogenanntes \textit{inverses Element} $g^{-1} \in G$ existiert, sodass die Verknüpfung von $g$ mit seinem Inversen bezüglich $\circ$ das neutrale Element $e$ der Gruppe ergibt.

Auch wenn wir den Begriff einer Gruppe vornehmlich dafür einführen, um später den für uns wichtigen Begriff des Vektorraums einführen zu können, wollen wir uns kurz Zeit nehmen und den Begriff der Gruppe motivieren, um ein besseres Gefühl dafür zu bekommen, warum Gruppen im Allgemeinen wichtig sind und warum sich das Studium dieser lohnt. 

Eine erste Motivation obiger diese Definition erhält man, wie eben schon angerissen, darüber, dass (wie wir gleich noch sehen werden) viele bekannte, sehr wichtige mathematische Objekte auf natürliche Weise eine Gruppenstruktur besitzen, sodass wir die obige Definition als reine Abstraktion der Struktur ebendieser mathematischen Objekte verstehen können. 

Das Studium abstrakter Gruppen ließe sich dann unter anderem darüber rechtfertigen, dass eben so manches wichtiges mathematisches Objekt eine Gruppe darstellt und das Studium allgemeiner Gruppen uns auch beim Studium dieser wichtigen mathematischen Objekte weiterhilft. 

Eine andere, direktere Möglichkeit der Motivation des Gruppenbegriffes, ist die Betrachtung der Symmetrien eines geometrischen Objektes: Betrachten wir beispielsweise ein Quadrat in der Ebene, so können wir auf dieses verschiedenste Operationen anwenden. 

So können wir das Quadrat beispielsweise in der Ebene, in welcher es liegt, um einen gewissen Winkel drehen oder es entlang einer bestimmten Achse spiegeln. Wir können uns nun auf diejenigen Transformationen beschränken, die den Mittelpunkt des Quadrates fixieren (D.h., dass beispielsweise nichttriviale Translationen des Quadrates aus unserer Betrachtung herausfallen.) und uns dann fragen, welche Operationen die Erscheinung des Quadrates unverändert lassen. 

Beispiele für derartige Transformationen sind z.B. $\frac{\pi}{2}$-Rotationen um den Mittelpunkt des Quadrates. Man kann sich nun leicht überlegen, dass die Menge derartiger Operationen, die wir \textit{Symmetrietransformationen} des Quadrates nennen wollen, mit der Hintereinanderausführung von Operationen eine mathematische Gruppe bildet. 

So ist beispielsweise die Hintereinanderausführung zweier Symmetrietransformationen wieder eine Symmetrietransformation (Zwei $\frac{\pi}{2}$-Rotationen ergeben eine Rotation um den Winkel $\pi$, welche die Erscheinung des Quadrates abermals unverändert lässt). 

Weiter ist die Operation bei der wir nichts mit dem Quadrat machen ebenfalls eine Symmetrietransformation, welche das neutrale Element der Gruppe darstellt. Desweiteren können wir natürlich auch jede Symmetrietransformation umkehren, wobei die Inversen natürlich selbst Symmetrietransformationen darstellen (Das Inverse zur $\frac{\pi}{2}$-Rotation ist dann beispielsweise die Rotation um den Winkel $- \frac{\pi}{2}$, welche ebenfalls eine Symmetrietransformation des Quadrates darstellt.). 

Abschließend ist auch die Hintereinanderausführung von Operationen eine assoziative Verknüpfung, womit die Menge aller Symmetrietransformationen des Quadrates eine natürliche Gruppe bildet. 

Um also ein geometrisches Objekt richtig studieren zu können, muss man auch die zum Objekt zugehörige Gruppe an Symmetrien, die sogenannte $\textit{Symmetriegruppe}$, studieren. Um nun wiederum diese Symmetriegruppe ordentlich studieren zu können, lohnt es sich, die allgemeine Gruppentheorie zu betrachten.

Abschließend sei noch erwähnt, dass, wie oben bereits angedeutet, wir den Gruppenbegriff in dieser Arbeit primär deshalb einführen, weil wir ihn als Hilfsmittel zur Konstruktion bestimmter für uns wichtiger mathematischer Objekte benötigen. Nach den obigen Motivationen dürfte es nun aber weit weniger überraschen, warum Gruppen in der Mathematik recht oft als nützliches Hilfsmittel auftauchen.

Eine wichtige Eigenschaft von Gruppen, die wir im folgenden (beispielsweise in der noch kommenden Definition sogenannter Normfunktionen) immer wieder indirekt benutzen werden, ist die Folgende:

\begin{proposition}
    Sei $(G, \circ)$ eine Gruppe. Dann ist das neutale Element $e \in G$ der Gruppe eindeutig bestimmt.
\end{proposition}

\begin{proof}
    Angenommen das wäre nicht so, d.h. es gäbe noch ein weiteres Element $e' \in G$, sodass 
    \begin{align}
        e' \circ g = g \circ e' = g
    \end{align}
    für alle $g \in G$ gilt. Dann folgt, dass 
    \begin{align}
        e \circ g = g = e' \circ g.
    \end{align}
    Verknüpfen wir diese Gleichung mit $g^{-1}$ von rechts, so folgt 
    \begin{align}
        e \circ g \circ g^{-1} = g \circ g^{-1} = e' \circ g \circ g^{-1}
    \end{align}
    und damit 
    \begin{align}
        e \circ e = e = e' \circ e.
    \end{align}
    Daraus folgt aber, dass $e = e'$ und damit ist das neutrale Element einer Gruppe eindeutig bestimmt.
\end{proof}

Eine weitere wichtige Eigenschaft von Gruppen ist die Eindeutigkeit der Inversen:

\begin{proposition}
    Sei $(G, \circ)$ eine Gruppe mit neutralem Element $e \in G$ und $g \in G$ ein Element der Gruppe. Sei weiter $g^{-1}$ ein inverses Element von $g$. Dann ist $g^{-1}$ eindeutig bestimmt.
\end{proposition}

\begin{proof}
    Angenommen $g^{-1}$ wäre nicht eindeutig bestimmt, d.h. es gibt ein weiteres Element $h \in G$ mit $h \neq g^{-1}$ und $g \circ h = h \circ g = e$. Es gilt wegen $g \circ g^{-1} = g^{-1} \circ g = e$, dass $g \circ g^{-1} = g \circ h$ ist. Verknüpfen wir diesen Ausdruck von links mit $g^{-1}$ und nutzen wir die Assoziativität und die Eigenschaften des neutralen Elementes aus, so erhalten wir $g^{-1} = h$, im Widerspruch zur Annahme. Daraus folgt, dass das Inverse von $g \in G$ eindeutig bestimmt ist.
\end{proof}

Nach dem wir nun einige Worte zur Wichtigkeit des Gruppenbegriffes verloren und auch einige wichtige Eigenschaften von Gruppen gesehen haben, wollen wir als nächstes einige wichtige Beispiele von Gruppen betrachten.

\begin{example}
    Wichtige Beispiele von Gruppen sind:

    \begin{itemize}
        \item[\textit{i)}] Die reellen Zahlen bezüglich der gewöhnlichen Addition, d.h. $(\mathbb{R}, +)$, und die Menge $\mathbb{R} \setminus \{0\}$ zusammen mit der gewöhnlichen Multiplikation, d.h. $(\mathbb{R} \setminus \{0\}, \cdot)$, bilden jeweils eine Gruppe. In $(\mathbb{R}, +)$ ist dabei das neutrale Element gerade $0$, während in $(\mathbb{R} \setminus \{0\}, \cdot)$ das neutrale Element $1$ ist.
        \item[\textit{ii)}] Sei $\mathbb{C} := \{ re^{i \phi} \:\: | \:\: r \in [0, \infty), \:\: \phi \in [0, 2\pi), \:\: i^2 = -1 \}$ die Menge der komplexen Zahlen und $\mathbb{S}^1 := \{e^{i \phi}\ \:\: | \:\: \phi \in [0, 2 \pi)\} \subseteq \mathbb{C}$ die komplexe Einheitssphäre. Dann ist $(\mathbb{S}^1, \cdot)$ eine Gruppe, wobei $e^{i \phi} \cdot e^{i \eta} = e^{i(\phi + \eta)}$ ist.
        \item[\textit{iii)}] Sei $\Omega$ eine Menge und $G$ die Menge aller bijektiver Funktionen $f : \Omega \longrightarrow \Omega$. Dann bildet $G$ zusammen mit der Funktionenkomposition $\circ$, erklärt durch $(f \circ g)(x) := f(g(x)), \forall x \in \Omega$, eine Gruppe. Das neutrale Element der Gruppe ist dabei die Funktion $id_{\Omega}$, erklärt durch $id_\Omega (x) = x$ für alle $x \in \Omega$.
        \item[\textit{iv)}] Wir betrachten die Menge $SL(2, \mathbb{C}) := \{ M \in \textit{Mat}_{2} (\mathbb{C}) \:\: | \:\: det (M) = 1\}$, wobei $\textit{Mat}_{2} (\mathbb{C})$ die Menge aller $2 \times 2$-Matrizen mit komplexen Einträgen ist und die Funktion $det : \textit{Mat}_{2} (\mathbb{C}) \longrightarrow \mathbb{C}$ die sogenannte Determinantenfunktion darstellt. Die Menge $SL(2, \mathbb{C})$ wird auch spezielle lineare Gruppe vom Grad 2 über $\mathbb{C}$ genannt und die Elemente dieser Menge werden unimodulare Matrizen genannt. Statten wir diese Menge mit der gewöhnlichen Matrizenmultiplikation $\circ$ aus, so ist $(SL(2, \mathbb{C}), \circ)$ ebenfalls eine Gruppe.  
    \end{itemize}
\end{example}

Bevor wir nun den für uns wichtigen Begriff des Vektorraums erklären können, müssen wir noch kurz den Begriff des \textit{Körpers} einführen:

\begin{definition}
    Sei $\mathbb{K}$ eine Menge, ausgestattet mit den zwei folgenden binären Verknüfungen 
    \begin{align}
        + : \mathbb{K} \times \mathbb{K} \longrightarrow \mathbb{K}
    \end{align}
    und 
    \begin{align}
        \cdot : \mathbb{K} \times \mathbb{K} \longrightarrow \mathbb{K}.
    \end{align}
    Das Tripel $(\mathbb{K}, +, \cdot)$ heißt nun Körper, falls gilt:
    \begin{itemize}
        \item[\textit{i)}] $(\mathbb{K}, +)$ ist eine abelsche Gruppe, in welcher wir das neutrale Element mit $0$ bezeichnen.
        \item[\textit{ii)}] $(\mathbb{K} \setminus \{0\}, \cdot)$ ist eine abelsche Gruppe.
        \item[\textit{iii)}] $(a + b) \cdot c = a \cdot b + b \cdot c \:\:\:\:\:\:\:\: \forall a,b,c \in \mathbb{K}$
    \end{itemize}
\end{definition}

Die Eigenschaft \textit{iii)} in der obigen Definition eines Körpers nennt man auch \textit{Distributivität} und das neutrale Element der Gruppe $(\mathbb{K} \setminus \{0\}, \cdot)$ wollen wir im folgenden mit $1$ bezeichnen.

Beachte, dass aus der Definition eines Körpers trivialerweise folgt, dass $0 \neq 1$ ist. Ohne die Forderung, dass lediglich $(\mathbb{K} \setminus \{0\}, \cdot)$ und nicht $(\mathbb{K}, \cdot)$ eine Gruppe bildet, wird hingegen eine Situation eintreten, die wir für gewöhnlich zu vermeiden versuchen: Zum einen könnte ohne diese Forderung der Fall $1 = 0$ eintreten, waraus mittels der Distributivität folgen würde, dass 
\begin{align}
    \alpha = 1 \cdot \alpha = 0 \cdot \alpha = (0+0) \cdot \alpha = (1+1) \cdot \alpha = 1 \cdot \alpha + 1 \cdot \alpha = \alpha + \alpha
\end{align}
und damit $\alpha = 0$ für alle $\alpha \in \mathbb{K}$ ist. Andererseits würde das neutrale Element $0$ ohne obige Forderung notwendigerweise ein bezüglich $\cdot$ inverses Element $0^{-1}$ besitzen. Daraus würde aber folgen, dass 
\begin{align}
    1 = 0^{-1} \cdot 0 = 0^{-1} \cdot (0 + 0) = 0^{-1} \cdot 0 + 0^{-1} \cdot 0 = 0 + 0 = 0
\end{align}
und damit wieder $1 = 0$ ist, was wiederum implizieren würde, dass $\mathbb{K}$ trivial wäre. Um daher derartige triviale Situationen auszuschließen, fordern wir gleich per Definition, dass $0 \neq 1$ ist, woraus wiederum folgt, dass $0$ in $\mathbb{K}$ kein bezüglich $\cdot$ inverses Element besitzen kann.

Im weiteren Verlauf wollen wir das zu $a \in \mathbb{K}$ additiv inverse Element mit $-a$ bezeichnen und das zu $a \in \mathbb{K}$ multiplikativ inverse Element mit $\frac{1}{a}$. Desweiteren wollen wir für $a,b \in \mathbb{K}$ statt $a + (-b)$ im folgenden kurz $a-b$ schreiben.  

Wichtige Beispiele für mathematische Körper sind dabei die reellen Zahlen $\mathbb{R}$ und die komplexen Zahlen $\mathbb{C}$, beide jeweils ausgestattet mit der gewöhnlichen Addition und Multiplikation. Damit kann ein mathematischer Körper als Abstraktion unserer gewöhnlichen Zahlensysteme verstanden werden, da man in einem Körper wie gewohnt rechnen kann. 


\subsection{Vektorräume}

Nun sind wir in der Lage den für uns sehr wichtigen Begriff des Vektorraums zu erklären:

\begin{definition}
    Sei $\mathbb{K}$ ein Körper mit dem Einselement $1$ und $(V, \oplus)$ eine abelsche Gruppe. Weiter sei eine binäre Verknüpfung 
    \begin{align}
        \odot : \mathbb{K} \times V \longrightarrow V
    \end{align}
    mit den Eigenschaften

    \begin{itemize}
        \item[\textit{i)}] $\alpha \odot (v \oplus w) = \alpha \odot v \oplus \alpha \odot w \:\:\:\:\:\:\:\:\:\: \forall \alpha \in \mathbb{K} \:\: v, w \in V$
        \item[\textit{ii)}] $(\alpha + \beta) \odot v = \alpha \odot v \oplus \beta \odot v \:\:\:\:\:\:\:\:\:\:\:\: \forall \alpha, \beta \in \mathbb{K}, \:\: v \in V$
        \item[\textit{iii)}] $(\alpha \cdot \beta) \odot v = \alpha \cdot (\beta \odot v) \:\:\:\:\:\:\:\:\:\:\:\:\:\:\:\:\:\:\:\:\: \forall \alpha, \beta \in \mathbb{K}, \:\: v \in V $ 
        \item[\textit{iv)}] $1 \odot v = v \:\:\:\:\:\:\:\:\:\:\:\:\:\:\:\:\:\:\:\:\:\:\:\:\:\:\:\:\:\:\:\:\:\:\:\:\:\:\:\:\:\:\:\:\:\:\:\ \forall v \in V$
    \end{itemize}
    gegeben. Dann nennen wir $\odot$ Skalarmultiplikation und das Tripel $(V, \oplus, \odot )$ Vektorraum über $\mathbb{K}$ oder kurz $\mathbb{K}$-Vektorraum. Die Elemente in $V$ wollen wir Vektoren und den Körper $\mathbb{K}$ in Bezug auf $V$ den Skalarkörper oder Grundkörper von $V$ nennen.
\end{definition}

Bei einem $\mathbb{K}$-Vektorraum $V$ handelt es sich also anschaulich um eine abelsche Gruppe, deren Elemente mittels der Skalarmultiplikation beliebig um Elemente aus dem Skalarkörper $\mathbb{K}$ skaliert werden können, sodass dabei die üblichen Rechenregeln gelten und das Ergebnis dieser Skalierung am Ende wieder in $V$ liegt.

Betrachten wir einige Beispiele von Vektorräumen

\begin{example}
    Wichtige Beispiele von Vektorräumen sind:
    \begin{itemize}
        \item[\textit{i)}] Ein für uns sehr wichtiges Beispiel eines $\mathbb{K}$-Vektorraums mit $\mathbb{K} = \mathbb{R}$ und $d \in \mathbb{N}$ ist $(\mathbb{R}^d, \oplus, \odot)$, wobei 
        \begin{align}
            \mathbb{R}^d := \{ (x^1, x^2, ..., x^d) \:\: | \:\: x^{i} \in \mathbb{R} \:\: \forall i \in \{ 1, 2, ..., d \} \}
        \end{align} und die Vektoraddition $\oplus$ und die Skalarmultiplikation $\odot$ folgendermaßen erklärt sind:
        \begin{itemize}
            \item Seien $x = (x^1,...,x^d), \:\: y = (y^1,...,y^d) \in \mathbb{R}^d$. Dann ist $x \oplus y$ erklärt als 
            \begin{align}
                x \oplus y := (x^1 + y^1, ..., x^d + y^d) \in \mathbb{R}^d.
            \end{align}
            Dabei ist $+$ die gewöhnliche reelle Addtion.
             Sei $x = (x^1, ..., x^d) \in \mathbb{R}^d$ und $\alpha \in \mathbb{R}$. Dann ist $\alpha \odot x$ erklärt durch 
            \begin{align}
                \alpha \odot x := (\alpha \cdot x^1, ..., \alpha \cdot x^d) \in \mathbb{R}^d.
            \end{align}
            Dabei ist $\cdot$ die gewöhnliche reelle Multiplikation.
        \end{itemize}
        Der Spezialfall $d = 3$ ist dabei insbesondere von großer Bedeutung für die Physik, da sich oft der dreidimensionale physikalische Raum mittels des $\mathbb{R}^3$ modellieren lässt. Wie wir gleich noch sehen werden, kann die Menge $\mathbb{R}^3$, ausgestattet mit einer zusätzlichen Struktur, auch als Modell der gewöhnlichen Geometrie genutzt werden.
        \item[\textit{ii)}] Ein weiteres wichtiges Beispiel eines $\mathbb{R}$-Vektorraumes, welcher oft von praktischer Relevanz ist, ist der $\mathbb{R}$-Vektorraum $(\mathcal{L}^2(\Omega, \mu), \oplus, \odot)$. Dabei ist $\Omega \subseteq \mathbb{R}^d$, $\mu$ ist das sogenannte Lebesgue-Maß auf $\mathbb{R}^d$ \cite{cohn2013measure} und 
        \begin{align}
            \mathcal{L}^2(\Omega, \mu) := \{ f : \Omega \longrightarrow \mathbb{R} \:\: \textit{messbar} \:\: | \:\: \int_{\mathbb{R}^d} |f(x)|^2 \,d \mu (x) < \infty \}.
        \end{align}
        Das Integral in der Definition von $\mathcal{L}^2(\Omega, \mu)$ verstehen wir also als Lebesgue-Integral \cite{cohn2013measure}. Weiter verstehen wir die Messbarkeit von $f \in \mathcal{L}^2(\Omega, \mu)$ in Bezug auf eine auf $\Omega$ beliebig gewählten $\sigma$-Algebra $\mathbb{A}$ und der Borel-$\sigma$-Algebra auf $\mathbb{R}$. Die Vektoraddition $\oplus$ und die Skalarmultiplikation $\odot$ sind dabei folgendermaßen erklärt:
        \begin{itemize}
            \item Sei $f, g \in \mathcal{L}^2(\Omega, \mu)$. Dann ist $f \oplus g$ erklärt durch 
            \begin{align}
                (f \oplus g) (x) := f(x) + g(x) \in \mathbb{R}.
            \end{align}
            Dabei ist $+$ wieder die gewöhnliche reelle Addition. Man bezeichnet eine so erklärte Vektoraddition auch als punktweise Vektoraddition.
            \item Sei $\alpha \in \mathbb{R}$ und $f \in \mathcal{L}^2(\Omega, \mu)$. Dann ist $\alpha \odot f$ erklärt durch 
            \begin{align}
                (\alpha \odot f) (x) := \alpha \cdot f(x) \in \mathbb{R}.
            \end{align}
            Dabei ist $\cdot$ wieder die gewöhnliche reelle Multiplikation. Man bezeichnet eine so erklärte Skalarmultiplikation auch als punktweise Skalarmultiplikation. 
        \end{itemize}
    \end{itemize}
\end{example}

Der Einfachkeitshalber wollen wir nun in Zukunft statt $(V, \oplus, \odot)$ oft einfach nur $V$ schreiben, falls klar sein sollte, wie $\oplus$ und $\odot$ erklärt sind. Darüberhinaus wollen wir nun oft statt $\oplus$ einfach nur $+$ und statt $\odot$ einfach nur $\cdot$ schreiben. Aus dem Kontext sollte sich dann immer ergeben, ob es sich beispielsweise bei $+$ um die sogenannte Vektoraddition $\oplus$ in $V$ oder um die Addition $+$ im Körper $\mathbb{K}$ handelt. Selbiges gilt für die Skalarmultiplikation $\odot$ und der Multiplikation $\cdot$ im Körper $\mathbb{K}$. Desweiteren wollen wir das zu $v \in V$ additiv inverse Element $v^{-1}$ einfach mit $\ominus v$ oder $-v$ bezeichnen. Diese Setzung ergibt sich aus den Eigenschaften der Skalarmultiplikation, denn für $v \in V$ folgt, dass
\begin{align}
    \mathbf{0} = (1 - 1) \odot v = 1 \odot v \oplus (-1 \odot v),
\end{align}
d.h. wegen der Eindeutigkeit der Inversen in einer Gruppe folgt 
\begin{align}
    v^{-1} = (-1 \odot v) =: \ominus v =: -v
\end{align}

Ein Begriff, welcher für uns später im Kapitel über Mannigfaltigkeiten von großer Bedeutung sein wird, ist der Begriff der Dimension eines $\mathbb{K}$-Vektorraumes $V$. Im Folgenden werden wir diesen Begriff entwickeln. Wir beginnen mit dem Begriff der linearen Hülle einer Menge $E \subseteq V$.

\begin{definition}
    Sei $V$ ein $\mathbb{K}$-Vektorraum und $E \subseteq V$ eine Teilmenge von $V$. Wir erklären die lineare Hülle von $E$ in $V$ als die Menge 
    \begin{align}
        \{ v \in V \:\: | \:\: \exists N \in \mathbb{N} \:\; \alpha_1, ..., \alpha_N \in \mathbb{R}, \:\: v_1, ..., v_N \in E : v = \sum_{i = 1}^{N} \alpha_i \cdot v_i\}.
    \end{align}
    Wir bezeichnen diese Menge kurz als $\textit{span}_{\mathbb{K}}(E)$.
\end{definition}

Die lineare Hülle von $E$, $\textit{span}_{\mathbb{K}}(E)$, ist also die Menge aller sogenannter endlicher Linearkombinationen, die aus den Elementen in $E$ gebildet werden können. 

Als nächstes definieren wir den Begriff des sogenannten Erzeugendensystems eines $\mathbb{K}$-Vektorraumes $V$. 

\begin{definition}
    Sei $V$ ein $\mathbb{K}$-Vektorraum und $E \subseteq V$ eine Teilmenge von $V$. $E$ heißt Erzeugendensystem von $V$, falls gilt, dass 
    \begin{align}
        \textit{span}_{\mathbb{K}}(E) = V 
    \end{align}
    ist.
\end{definition}

Mittels dieser Begrifflichkeiten können wir nun den Begriff eines minimalen Erzeugendensystems einführen.

\begin{definition}
    Sei $V$ ein $\mathbb{K}$-Vektorraum und $E$ ein Erzeugendensystem. Dann heißt $E$ minimales Erzeugendensystem, falls für beliebiges $v \in E$ gilt, dass $E \setminus \{v\}$ bereits kein Erzeugendensystem mehr von $V$ ist. Wir nennen im Folgenden ein minimales Erzeugendensystem von $V$ eine Basis von $V$. Wir bezeichnen $V$ als endlichdimensionalen $\mathbb{K}$-Vektorraum, falls eine Basis $E$ in $V$ existiert, sodass $|E| < \infty$ ist. Andernfalls nennen wir $V$ unendlichdimensional.
\end{definition}

Man kann sich nun die Frage stellen, ob eine Basis $E$ von $V$ eindeutig bestimmt ist. Die Antwort auf diese Frage ist Nein! Anschaulich sehen wir das bereits im Falle $V = \mathbb{R}^2$, ausgestattet mit der Vektoraddition und Skalarmultiplikation aus Beispiel $4.2.$. Ein minimales Erzeugendensystem von $\mathbb{R}^2$ ist bespielsweise 
\begin{align}
    E_1 = \{ (1,0), (0,1) \},
\end{align}
denn jeder Vektor in $\mathbb{R}^2$ kann geschrieben als 
\begin{align}
    (\alpha, \beta) = \alpha \cdot (1,0) + \beta \cdot (0,1).
\end{align}
Andererseits gilt das aber auch für 
\begin{align}
    E_2 = \{ (-1,0), (0,-1) \}.
\end{align}
Bei $E_1$ und $E_2$ handelt es sich also um Basen, denn beide erfüllen die Minimalitätseigenschaft und beide erzeugen ganz $\mathbb{R}^2$.

Wir wollen nun einen schnellen Weg finden, um zu überprüfen, ob eine gegebene Menge $E \subseteq V$ ein minimales Erzeugendensystem ist oder nicht. Dazu betrachten wir die folgende Definition:

\begin{definition}
    Sei $V$ ein $\mathbb{K}$-Vektorraum und $\{v_1, ..., v_m\} \subseteq V$, mit $d \in \mathbb{N}$, eine Teilmenge von $V$. Dann heißen die Vektoren $v_1, ..., v_m \in V$ linear unabhängig, falls aus 
    \begin{align}
        \sum_{i=1}^m \alpha_i \cdot v_i = \mathbf{0}, \label{f}
    \end{align}
    mit $\alpha_i \in \mathbb{K}$ für alle $i \in \{1,...,m\}$, bereits $\alpha_1 = ... = \alpha_m = 0$ folgt. Dabei bezeichnet $\mathbf{0} \in V$ das neutrale Element von $(V, +)$ und $0 \in \mathbb{K}$ das neutrale Element von $(\mathbb{K}, +)$. Gibt es hingegen für $i \in \{1,...,m\}$ Zahlen $\alpha_i \in \mathbb{K}$ mit $\alpha_i \neq 0$, sodass \eqref{f} gilt, so nennen wir die Menge $\{v_1, ..., v_m\}$ linear abhängig. 
\end{definition}

Mittels dieser Definition können wir nun den folgenden Satz beweisen, der es uns in Zukunft erleichtern wird zu überprüfen, ob eine gegebene Teilmenge $E \subseteq V$ eines $\mathbb{K}$-Vektorraumes $V$ eine Basis ist oder nicht.

\begin{proposition}
    Sei $V$ ein $\mathbb{K}$-Vektorraum und $E \subseteq V$ eine Teilmenge von $V$ mit $|E| = d < \infty$. Dann sind folgende Aussagen äquivalent:
    \begin{itemize}
        \item[\textit{i)}] $E$ ist eine Basis von $V,$ d.h. ein minimal erzeugendes System von $V$.
        \item[\textit{ii)}] $E$ ist ein maximal linear unabhängiges System in $V$, d.h. alle Vektoren in $E$ sind zueinander linear unabhängig und $\textit{span}_{\mathbb{K}}(E) = V$. 
    \end{itemize}
\end{proposition}

\begin{proof}
    \textit{i)} $\implies$ \textit{ii)}: Sei $E$ eine Basis von $V$. Angenommen die Menge $E$ sei nicht linear unabhängig, d.h. es existieren Zahlen $\alpha_i \in \mathbb{K}$ mit $i \in \{1, ..., d\}$, welche nicht alle gleich $0$, sodass für $\{v_1, ..., v_d\} = E$ gilt, dass 
    \begin{align}
        \sum_{i = 1}^d \alpha_i \cdot v_i = \mathbf{0} \iff \sum_{i = 2}^d \alpha_i \cdot v_i = - \alpha_1 \cdot v_1.
    \end{align}
    Ohne Beschränkung der Allgemeinheit nehmen wir an, dass $\alpha_1 \neq 0$ ist. Dann folgt 
    \begin{align}
        \sum_{i = 2}^d \lambda_i \cdot v_i = v_1
    \end{align}
    mit $\lambda_i = -\frac{\alpha_i}{\alpha_1}$ für alle $i \in \{1,...,d\}$. Sei nun $w \in V$ beliebig. Da $E$ ein Erzeugendensystem ist gilt, dass Zahlen $\beta_i \in \mathbb{K}$, $i \in \{ 1, ..., d \}$, existieren, sodass 
    \begin{align}
        w = \sum_{i = 1}^d \beta_i \cdot v_i
    \end{align}
    gilt. Dann gilt 
    \begin{align}
        w = \sum_{i = 2}^d \beta_i \cdot v_i + \beta_1 \cdot v_1 = \sum_{i = 2}^d \beta_i \cdot v_i + \beta_1 \cdot \big(\ \sum_{i = 2}^d \lambda_i \cdot v_i \big)
    \end{align}
    und damit 
    \begin{align}
        w = \sum_{i = 2}^d (\beta_i + \beta_1 \cdot \lambda_i) \cdot v_i \in \textit{span}_{\mathbb{K}}(E \setminus \{v_1\}).
    \end{align}
    Da $w \in V$ beliebig war, folgt, dass $\textit{span}_{\mathbb{K}}(E \setminus \{v_1\}) = V$ ist, im Widerspruch zur Minimalität von $E$.

    \textit{ii)} $\implies$ \textit{i)}: Angenommen $E$ sei eine Menge linear unabhängiger Vektoren mit $\textit{span}_{\mathbb{K}}(E) = V$, aber $E$ sei nicht minimal. Ohne Beschränkung der Allgemeinheit wäre dann die Menge $E \setminus \{v_1\}$, mit $v_1 \in E$ passend gewählt, noch immer ein Erzeugendensystem. Daraus folgt aber die Existenz von Zahlen $\alpha_i$, $i \in \{ 1, ..., d \}$, sodass 
    \begin{align}
        v_1 = \sum_{i = 2}^d \alpha_i \cdot v_i.
    \end{align}
    Da $v_i \neq \mathbf{0}$ folgt, dass nicht alle $\alpha_i$ gleich Null sein können. Daraus folgt nun aber, das die Menge $E$ kein linear unabhängiges System sein kann. Folglich muss $E$ ein minimales Erzeugendensystem sein.
\end{proof}

Es lässt sich darüberhinaus zeigen, dass alle Basen in einem $\mathbb{K}$-Vektorraum $V$ die selbe Länge bzw. Kardinalität besitzen \cite{fischer2003lineare}. Das führt auf die folgende Definition:

\begin{definition}
    Sei $V$ ein  endlichdimensionaler $\mathbb{K}$-Vektorraum und $E \subseteq V$ eine Basis von $V$. Dann nennen wir die Kardinalität $|E| = d$ von $E$ die Dimension des $\mathbb{K}$-Vektorraumes $V$. Wir bezeichnen $V$ in diesem Fall als $d$-dimensionalen $\mathbb{K}$-Vektorraum.
\end{definition}

Der Begriff der Dimension zählt also anschaulich die Zahl unabhängiger Vektoren, die nötig sind, um aus diesen den ganzen Vektorraum $V$ zu erzeugen. Wird jeder Vektor in dem minimalen Erzeugensystem $E$, welches wir kurz als Basis bezeichnen wollen, mit einem Freiheitsgrad assoziert, so ergibt sich das anschauliche Bild der Dimension, da die Dimension eines Vektorraumes $V$ nun einfach die Anzahl der Freiheitsgrade von $V$ meint.


\subsection{Normierte Räume und Prähilberträume}

Nachdem wir nun also den Begriff der Dimension eines $\mathbb{K}$-Vektorraumes $V$ eingeführt haben, wollen wir nun als Nächstes zeigen, dass sich auf einem $\mathbb{K}$-Vektorraum $V$ natürliche und nützliche geometrische Begriffe erklären lassen. Unser Vorgehen wird dabei sein, dass wir zuerst abstrakt die Begriffe einer sogenannten Norm und eines sogenannten Skalarproduktes auf $V$ einführen werden, um dann zu zeigen, dass im Falle $V = \mathbb{R}^d$ viele der uns bekannten und wichtigen geometrischen Begriffe sich über bestimmte Normen und Skalarprodukte erklären lassen. 

Wir wollen dabei im Folgenden, mit Blick auf die noch kommenden Kapitel, $\mathbb{K}$ auf $\mathbb{R}$ setzen. Der Grund dafür ist, dass z.B. im Falle $\mathbb{K} = \mathbb{C}$ sich die Definition eines Skalarprodukt auf einem $\mathbb{C}$-Vektorraum von der Definition eines Skalarproduktes auf einem $\mathbb{R}$-Vektorraum unterscheidet. Da wir in den folgenden Kapitel nur reelle Vektorräume betrachten werden, ist die Definition komplexer Skalarprodukte folglich für uns nicht relevant. 

In diesem und den folgenden Kapiteln wird außerdem $| \cdot |$ immer für die Betragsfunktion auf $\mathbb{R}$ stehen.

Wir beginnen zuerst mit dem Begriff der sogenannten Norm. Anschaulich wollen wir mit der Norm im folgenden in der Lage sein, den Vektoren im $\mathbb{R}$-Vektorraum $V$ eine (abstrakte) \textit{Größe} oder \textit{Länge} zuzuordnen. In Hinblick auf die noch kommenden Abschnitte über die sogenannten geodätischen Kurven einer Mannigfaltigkeit, wird die Möglichkeit, Vektoren mittels der Norm eine Größe zuzuordnen, noch von großer Bedeutung sein. 

\begin{definition}
    Sei $V$ ein $\mathbb{R}$-Vektorraum und $\| \cdot \| : V \longrightarrow [0, \infty )$ eine Funktion auf $V$. Wir nennen die Funktion $\| \cdot \|$ eine Norm, falls die folgenden Eigenschaften gelten.
    \begin{itemize}
        \item[\textit{i)}] Positive Definitheit:  $\| v \| \geq 0$ für alle $v \in V$ und $\| x \| = 0$ genau dann wenn $x = \mathbf{0}$ ist.
        \item[\textit{ii)}] Positive Homogenität: $\| \alpha \cdot v \| = |\alpha| \cdot \| v \|$ für alle $v \in V$ und für alle $\alpha \in \mathbb{R}$.
        \item[\textit{iii)}] Dreiecksungleichung: $\| v + w \| \leq \| v \| + \| w \|$ für alle $v, w \in V$. 
    \end{itemize}
    Ist $\| \cdot \|$ eine Norm, so nennen wir das Tupel $(V, \| \cdot \|)$ einen normierten Raum.
\end{definition}

Aufgrund der Eigenschaften einer Normfunktion $\| \cdot \|$ kann man tatsächlich davon sprechen, dass $\| v \|$ die Länge des Vektors $v \in V$ ist. Genauer gesagt können wir $\| v \|$ als den Abstand von $v$ zum Nullvektor $\mathbf{0} \in V$ interpretieren, da $\| \cdot \|$ nämlich immer auch eine Metrik $d_{\| \cdot \|}$ über $d_{\| \cdot \|}(v,w) := \| v - w \|$ induziert. Wir überprüfen kurz, dass es sich bei $d_{\| \cdot \|}$ tatsächlich um eine Metrik handelt:

\begin{proposition}
    Sei $(V, \| \cdot \|)$ ein reeller normierter Vektorraum und 
    \begin{align}
        d_{\| \cdot \|} : V \times V \longrightarrow [0, \infty )
    \end{align}
    eine Funktion, definiert durch 
    \begin{align}
        d_{\| \cdot \|}(v,w) := \| v - w \|.
    \end{align}
    Dann ist $d_{\| \cdot \|}$ eine Metrik. Wir nennen $d_{\| \cdot \|}$ die durch $\| \cdot \|$ induzierte Metrik.
\end{proposition}

\begin{proof}
    Wir überprüfen die Metrikaxiome:
    \begin{itemize}
        \item[\textit{i)}] Zuerst bemerken wir, dass wegen $\| \cdot \| : V \longrightarrow [0, \infty)$ per Definition auch $d_{\| \cdot \|}$ nur nach $[0, \infty)$ abbilden kann. Wir wissen weiter wegen der Normaxiome, dass $\|v\|$ genau dann Null ist, wenn $v = \mathbf{0}$ ist. Das heißt, dass $d_{\| \cdot \|}(v,w) = \|v-w\| = 0$ genau dann, wenn $v-w = \mathbf{0}$ ist oder äquivalent $v = w$. Das verfifiziert das erste Metrikaxiom.
        \item[\textit{ii)}] Für alle $v,w \in V$ gilt 
        \begin{align}
            d_{\| \cdot \|}(v,w) =& \: \| v - w \| \nonumber \\ =& \: \|(-1)\cdot (w - v)\| \nonumber \\ =& \: |-1| \cdot \|w-v\| = d_{\| \cdot \|}(w,v).
        \end{align}
        Das verfifiziert das zweite Metrikaxiom.
        
        \item[\textit{iii)}] Für $v,w,r \in V$ gilt 
        \begin{align}
            d_{\| \cdot \|}(v,r) =& \: \|(v - w) + (w - r)\| \nonumber \\ \leq& \:  \|v - w\| +  \|w - r\| \nonumber \\ =& \: d_{\| \cdot \|}(v,w) + d_{\| \cdot \|}(w,r).
        \end{align}
        Das verfiziert das dritte Metrikaxiom.
    \end{itemize}
    Damit handelt es sich bei $d_{\| \cdot \|}$ tatsächlich um eine Metrik.
\end{proof}

Damit ist also jeder normierte Vektorraum auch immer gleich ein metrischer Raum und nach dem vorangegangen Kapitel $3$ dann auch immer gleich ein topologischer Raum. 

Da die Norm $\| \cdot \|$ über ihre Eigenschaften immer eine Metrik $d_{\| \cdot \|}$ induziert, lässt sich mittels der Interpretation, dass $\| v \| = d_{\| \cdot \|}(\mathbf{0}, v)$ der Abstand von $v \in V$ zum Nullvektor $\mathbf{0} \in V$ ist, leicht nachvollziehen, warum wir $\| v \|$ als die Größe von $v$ (gemessen vom Nullvektor) verstehen können. 

So sollte unter obiger Interpretation für eine Funktion, die jedem Vektor seine Größe zuordnet, natürlich gelten, dass nur die Länge des Nullvektors $\mathbf{0}$ Null und die Länge aller übrigen Vektoren größer als Null ist. Weiter sollte für eine derartige Funktion die Länge eines um $\alpha \in \mathbb{R}$ skalierten Vektors gerade der Länge des unskalierten Vektors multipliziert mit dem Skalierungsfaktor $|\alpha|$ sein. 

Schließlich lässt sich unter der obigen Interpretation die Dreiecksungleichung bezüglich $\| \cdot \|$ so interpretieren, dass, wenn wir $\|v + w\|$ als den Abstand von $v+w \in V$ vom Nullvektor $\mathbf{0}$ verstehen, der Abstand von $v + w$ zum Nullvektor immer kleiner gleich dem Abstand von $v$ zum Nullvektor und dem Abstand von $w$ zum Nullvektor zusammengenommen ist, denn 
\begin{align}
    \| v + w\| = d_{\| \cdot \|} (\mathbf{0},v+w) \leq d_{\| \cdot \|} (\mathbf{0},v) + d_{\| \cdot \|} (v,v+w) = \| v \| + \| w \|.
\end{align}
Betrachten wir noch kurz ein wichtiges Beispiel einer Norm:

\begin{example}
    Sei $V = \mathbb{R}^d$. Wir definieren die Funktion 
    \begin{align}
        \| \cdot \|_{\mathbb{R}^d} : \mathbb{R}^d \longrightarrow \mathbb{R}
    \end{align} 
    durch
    \begin{align}
        \| v \|_{\mathbb{R}^d} := \sqrt{ \sum_{i = 1}^d (v^i)^2 } 
    \end{align}
    mit $v = (v^1, ..., v^d) \in \mathbb{R}^d$. Es ist leicht einzusehen, dass es sich bei $\| \cdot \|_{\mathbb{R}^d}$ tatsächlich um eine Norm handelt. 
    
    Betrachten wir den Spezialfall $d = 2$, so erkennen wir, dass für $v = (v^1,v^2) \in \mathbb{R}^2$ der Ausdruck $\|v\|$ nach dem \textit{Satz des Pythagoras} gerade der Hypothenuse eines rechtwinkeligen Dreiecks entspricht, deren Katheten die Längen $|v^1|$ bzw. $|v^2|$ haben. 
    
    Wir können also den Vektorraum $\mathbb{R}^2$, basierend auf diesem Längenbegriff, als zweidimensionales kartesisches Koordinatensystem visualisieren, sodass der Ursprung des Koordinatensystems durch den Nullvektor repräsentiert wird und der Vektor $v$ in diesem Koordinatensystem als Pfeil dargestellt werden kann, welcher auf das Koordinatenpaar $[v^1, v^2]$ zeigt. 
    
    In dieser anschaulichen Visualisierung des $\mathbb{R}^2$ hätte der Vektor $v$ nämlich auf natürliche Weise die Länge $\|v\|_{\mathbb{R}^2}$, d.h. die Funktion $\| \cdot \|_{\mathbb{R}^d}$ wäre keine abstrakte Längenfunktion, sondern würde mit der auf der Visualisierung basierenden geometrischen Länge von $v$ übereinstimmen. 
    
    Ausgehend von dieser intuitiven geometrischen Visualisierung des $\mathbb{R}^2$ entspricht $\| \cdot \|_{\mathbb{R}^2}$ also der anschaulichen geometrischen Länge von $v$, was abermals zeigt, dass wir $\| \cdot \|_{\mathbb{R}^d}$ als Längenfunktion auf $\mathbb{R}^2$ verstehen können, wenn wir die kartesische Visualisierung des $\mathbb{R}^2$ zugrunde legen.  

    In der eben angesprochenen Visualisierung des $\mathbb{R}^2$ haben wir bereits das Konzept des Winkels verwendet, denn immerhin stehen die Koordinatenachsen in einem kartesischen Koordinatensystem rechtwinkelig aufeinander. Im folgenden werden wir sehen, dass wir mittels eines Skalarproduktes auf $\mathbb{R}^2$ den auf obiger Visualisierung basierenden geometrischen Winkel zwischen zwei Vektoren erklären können. 
\end{example}

\begin{remark}
    Eine nützliche Eigenschaft der Normfunktion $\| \cdot \|_{\mathbb{R}^d}$, die wir unter anderem noch im Kapitel über die Analytischen Grundlagen benötigen werden, ist die Folgende: Sei der $\mathbb{R}^d$ ausgestattet mit der von der Norm $\| \cdot \|_{\mathbb{R}^d}$ induzierten Topologie. Sei weiter $(x_n)_{n \in \mathbb{N}} \subseteq \mathbb{R}^d$ eine Folge mit $x_n = (x^1_n, ..., x^d_n)$ und $x_n \xrightarrow{n \rightarrow \infty} x = (x^1, ..., x^d) \in \mathbb{R}^d$. Das heißt für alle $\delta > 0$ existiert ein $N \in \mathbb{N}$, sodass
    \begin{align}
        \|x_n - x\|_{\mathbb{R}^d} = \sqrt{( x^1_n - x^1 )^2 + ... + ( x^d_n - x^d )^2} < \delta \:\:\:\: \forall n \geq N
    \end{align}
    gilt. Damit folgt aber für die Folge $(x^{i}_n)_{n \in \mathbb{N}} \subseteq \mathbb{R}$ mit $i \in \mathbb{N}$, dass für das gleiche beliebige $\delta > 0$ gilt, dass
    \begin{align}
        |x^{i}_n - x_i| =& \: \sqrt{(x^{i}_n - x_i)^2} \nonumber \\ \leq& \: \sqrt{( x^1_n - x^1 )^2 + ... + (x^{i}_n - x_i)^2 + ... + ( x^d_n - x^d )^2} \nonumber \\ =& \: \|x_n - x\|_{\mathbb{R}^d} < \delta \:\:\:\: \forall n \geq N.
    \end{align}
    D.h. aus der Konvergenz der der Folge $(x_n)_{n \in \mathbb{N}}$ folgt $x^{i}_n \xrightarrow{n \rightarrow \infty} x^{i} \in \mathbb{R}$ für alle $i \in \{1,...,d\}$. Umgedreht folgt natürlich auch aus der Konvergenz $x^{i}_n \xrightarrow{n \rightarrow \infty} x^{i} \in \mathbb{R}$ für alle $i \in \{1,...,d\}$ die Konvergenz $x_n \xrightarrow{n \rightarrow \infty} x = (x^1, ..., x^d) \in \mathbb{R}^d$. 
\end{remark}

Wir sind nun also mit der Normfunktion in der Lage den für die Geometrie wichtigen Begriff der Länge zu erklären. Als Nächstes definieren wir den Begriff des reellen Skalarproduktes und des reellen Prähilbertraumes. Wie wir sehen werden, lässt sich mittels eines Skalarproduktes auf $V$ ein Winkelbegriff einführen, was unter anderem zeigt, warum Skalarprodukte wichtige mathematische Größen sind. 

\begin{definition}
    Sei $V$ ein endlichdimensionaler $\mathbb{R}$-Vektorraum. Sei weiter 
    \begin{align}
        \langle \cdot , \cdot \rangle : V \times V \longrightarrow \mathbb{R}
    \end{align}
    eine bilineare Abbildung, d.h.
    \begin{align}
        \langle \alpha_1 \cdot v_1 + \alpha_2 \cdot v_2 , w \rangle = \alpha_1 \cdot \langle v_1, w \rangle + \alpha_2 \cdot \langle v_2, w \rangle
    \end{align}
    und
    \begin{align}
        \langle v , \beta_1 \cdot w_1 + \beta_2 \cdot w_2 \rangle = \beta_1 \cdot \langle v, w_1 \rangle + \beta_2 \cdot \langle v, w_2 \rangle
    \end{align}
    für alle $\alpha_1, \alpha_2, \beta_1, \beta_2 \in \mathbb{R}$ und alle $v, v_1, v_2, w, w_1, w_2 \in V$. Wir wollen $\langle \cdot , \cdot \rangle$ als Skalarprodukt auf $V$ bezeichnen, falls $\langle \cdot, \cdot \rangle$ den folgenden Eigenschaften genügt:
    \begin{itemize}
        \item[\textit{i)}] Positive Definitheit: $\langle v , v \rangle \geq 0$ für alle $v \in V$ und $\langle v, v \rangle = 0$ genau dann wenn $v = \mathbf{0}$ ist.
        \item[\textit{ii)}] Symmetrie: $\langle v , w \rangle = \langle w, v \rangle \:\:\:\:\:\: \forall v,w \in V$
    \end{itemize}
    Ist $\langle \cdot , \cdot \rangle$ ein Skalarprodukt, so nennen wir das Tupel $(V, \langle \cdot , \cdot \rangle)$ Prähilbertraum.
\end{definition}

Betrachten wir ein wichtiges Beispiel eines Skalarproduktes:

\begin{example}
    Sei $V = \mathbb{R}^d$. Weiter erklären wir die Funktion 
    \begin{align}
        \langle \cdot, \cdot \rangle_{\mathbb{R}^d} : \mathbb{R}^d \times \mathbb{R}^d \longrightarrow \mathbb{R}
    \end{align}
    durch
    \begin{align}
        \langle v, w \rangle_{\mathbb{R}^d} := \sum_{i,j = 1}^d v^{i} \cdot w^j
    \end{align}
    für $v = (v^1, ..., v^d), \: w = (w^1, ..., w^d) \in \mathbb{R}^d.$ 
    
    Es ist leicht einzusehen, dass es sich bei $\langle \cdot, \cdot \rangle_{\mathbb{R}^d}$ tatsächlich um eine symmetrische, positiv definite Bilinearform handelt. Wir nennen $\langle \cdot, \cdot \rangle_{\mathbb{R}^d}$ im folgenden Standardskalarprodukt des $\mathbb{R}^d$. Eine erste wichtige Beobachtung ist, dass wegen der Definition von $\| \cdot \|_{\mathbb{R}^d}$ aus Beispiel $4.3.$ und der Definition von $\langle \cdot, \cdot \rangle_{\mathbb{R}^d}$ folgt, dass $\| \cdot \|_{\mathbb{R}^d} = \sqrt{\langle \cdot, \cdot \rangle_{\mathbb{R}^d}}$, d.h. $\langle \cdot, \cdot \rangle_{\mathbb{R}^d}$ induziert die Norm $\| \cdot \|_{\mathbb{R}^d}$. Weiter unten werden wir sehen, dass jedes reelle Skalarprodukt auf diese Weise eine zugehörige Norm induziert.

    Betrachten wir abermals den Spezialfall $d=2$ und die aus Beispiel $4.3.$ bekannte Standardvisualisierung des $\mathbb{R}^2$ mittels eines zweidimensionalen kartesischen Koordinatensystems, in welcher wir Punkte des $\mathbb{R}^2$ mit den Koordinatenpaaren des kartesischen Koordinatensystems identifizieren, so lässt sich nun leicht zeigen, dass der gewöhnliche Winkelbegriff zwischen zwei Vektoren in dieser Standardvisualisierung mittels des obigen Standardskalarproduktes ausgedrückt werden kann:

    Wir betrachten das folgende Dreieck, welches von den Vektoren $v = (v^1,v^2)$ und $u = (a,0)$ erzeugt wird:
    
    
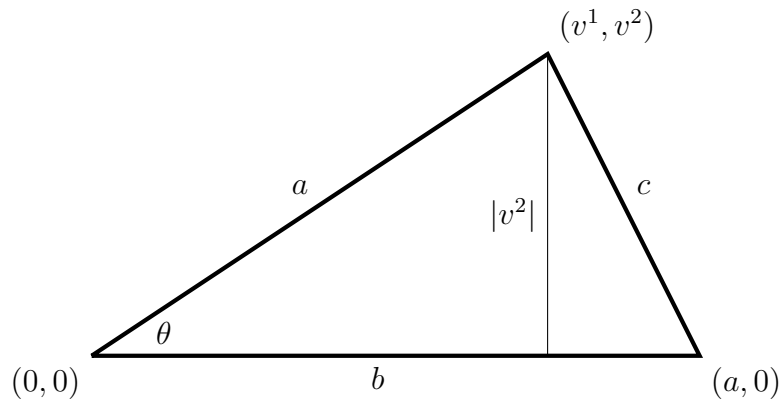
\begin{center}
        \begin{tikzpicture}
            \draw [ultra thick, black] (-4,0)--(2,4)--(4,0)--(-4,0);
            \draw [thin, black] (2,4)--(2,0);
            \node at (-1,2) [above left]{$a$};
            \node at (2,1.5) [above left]{$|v^2|$};
            \node at (3.5,2) [above left]{$c$};
            \node at (0,0) [below left]{$b$};
            \node at (-4,0) [below left]{$(0,0)$};
            \node at (2,4) [above right]{$(v^1,v^2)$};
            \node at (4,0) [below right]{$(a,0)$};
            \node at (-3.3,0.6) [below right]{$\theta$};
        \end{tikzpicture}
    \end{center}
    Aus der Grafik lässt sich ablesen, dass $a = \|v\|_{\mathbb{R}^2}$, $b = \|u\|_{\mathbb{R}^2}$ und $c = \|v - u\|_{\mathbb{R}^2}.$ Wie wir weiter der Grafik entnehmen, schließen in dieser Standardvisualisierung des $\mathbb{R}^2$ die Vektoren $v$ und $u$ einen Winkel $\theta$ ein. 

    Gemäß der euklidischen Geometrie gilt $(v^1,v^2) = (a \cdot \textit{cos}(\theta), a \cdot \textit{sin}(\theta))$ und mittels des Satzes des Pythagoras gilt 
    \begin{align}
        c^2 =& \: (b - a \cdot \textit{cos}(\theta))^2 + (a \cdot \textit{sin}(\theta))^2 \nonumber \\ =& \: a^2 + b^2 - 2ab \cdot \textit{cos}(\theta)
    \end{align}
    Im letzten Schritt haben wir dabei ausgenutzt, dass $\textit{sin}^2(\theta) + \textit{cos}^2(\theta) = 1$ ist. Wir haben damit gezeigt, dass in einem allgemeinen Dreiecken obiger Bauart gilt, dass 
    \begin{align}
        c^2 = a^2 + b^2 + 2ab \cdot \textit{cos}(\theta)
    \end{align}
    ist. Dieses Resultat nennt man auch den Kosinussatz.

    Betrachten wir nun zwei allgemeine Vektoren $w,h \in \mathbb{R}^2$ und das zugehörige Dreieck, welches diese in der Standardvisualisierung aufspannen. Wir nehmen dabei an, dass in dieser Visualsierung die Vektoren $w$ und $h$ den Winkel $\theta$ einschließen.

    
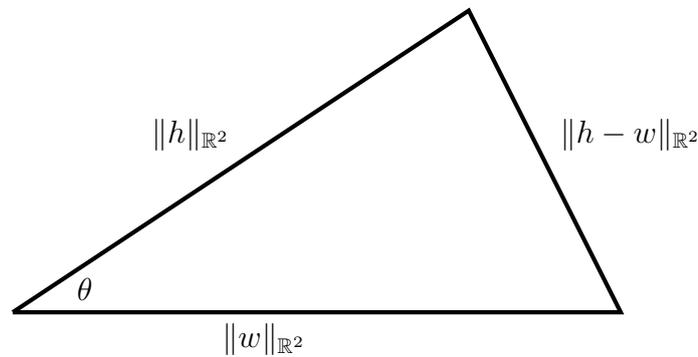
\begin{center}
        \begin{tikzpicture}
            \draw [ultra thick, black] (-4,0)--(2,4)--(4,0)--(-4,0);
            \node at (-1,2) [above left]{$\|h\|_{\mathbb{R}^2}$};
            \node at (5.2,2) [above left]{$\|h-w\|_{\mathbb{R}^2}$};
            \node at (0,0) [below left]{$\|w\|_{\mathbb{R}^2}$};
            \node at (-3.3,0.6) [below right]{$\theta$};
        \end{tikzpicture}
    \end{center}
    Mittels des obigen Kosinussatzes gilt  in diesem Dreieck 
    \begin{align}
        \|h-w\|_{\mathbb{R}^2}^2 =& \: \|h\|_{\mathbb{R}^2}^2 + \|w\|_{\mathbb{R}^2}^2 + 2 \: \|h\|_{\mathbb{R}^2} \|w\|_{\mathbb{R}^2} \cdot \textit{cos}(\theta) \label{hh}
    \end{align}
    Andererseits folgt aus der Definition von $\| \cdot \|_{\mathbb{R}^2}$, dass
    \begin{align}
        \|h-w\|_{\mathbb{R}^2}^2 =& \: \langle h-w , h-w \rangle_{\mathbb{R}^2} \nonumber \\ =& \: \|h\|_{\mathbb{R}^2}^2 + \|w\|_{\mathbb{R}^2}^2 + 2 \: \langle h , w \rangle_{\mathbb{R}^2} \label{hhh}
    \end{align}
    Mit \eqref{hh} und \eqref{hhh} folgt damit 
    \begin{align}
        \langle h ,w \rangle_{\mathbb{R}^2} = \|h\|_{\mathbb{R}^2}  \|w\|_{\mathbb{R}^2} \textit{cos}(\theta).
    \end{align}
    Es ist anzumerken, dass sich dieses Resultat leicht in den $\mathbb{R}^d$ verallgemeinern lässt.
\end{example}

Eine wichtige Eigenschaft reeller Skalarprodukte ist die Korrektheit der sogenannten schwarzschen Ungleichung. Mit dieser werden wir gleich in der Lage sein, zu zeigen, dass jedes Skalarprodukt auf dem $\mathbb{R}$-Vektorraum $V$ immer auch gleich eine Norm auf $V$ induziert, sodass jeder Prähilbertraum auch immer gleich ein normierter, und damit auch ein metrischer und topologischer Raum ist.

\begin{proposition}
    Sei $(V, \langle \cdot , \cdot \rangle)$ ein reeller, endlichdimensionaler Prähilbertraum. Dann gilt die Cauchy-Schwarzsche Ungleichung 
    \begin{align}
        |\langle v , w \rangle |^2 \leq \langle v,v \rangle \cdot \langle w,w \rangle
    \end{align}
    für alle $v,w \in V$.
\end{proposition}

\begin{proof}
    Sei $v,w \in V$ beliebig gewählt und ohne Beschränkung der Allgemeinheit nehmen wir an, dass $v \neq \mathbf{0}$ und $w \neq \mathbf{0}$. Andernfalls wäre die Ungleichung tivial erfüllt, da im Falle $v = \mathbf{0}$ oder $w = \mathbf{0}$ oder $v = w = \mathbf{0}$ aus der Bilinearität des Skalarproduktes sofort folgen würde, dass beide Seiten der Ungleichung gleich Null wären. Aus der positiven Definitheit, der Symmetrie und der Bilinearität des Skalarproduktes folgt nun für beliebiges $0 \neq \alpha \in \mathbb{R}$ 
    \begin{align}
        0 \leq \langle v - \alpha \cdot w , v - \alpha \cdot w \rangle = \langle v,v \rangle - 2 \alpha \cdot \langle v,w \rangle + \alpha^2 \cdot \langle w,w \rangle.
    \end{align}
    Wir wählen nun $\alpha := \frac{\langle v,w \rangle}{\langle w,w \rangle}$. Diese Wahl von $\alpha$ ist dabei wohldefiniert, da aus $w \neq \mathbf{0}$ folgt, dass $\langle w,w \rangle > 0$ ist. Es folgt nun mit dieser Setzung, dass 
    \begin{align}
        0 \leq \langle v,v \rangle - 2 \cdot \frac{\langle v,w \rangle^2}{\langle w,w \rangle} + \frac{\langle v,w \rangle^2}{\langle w,w \rangle} = \langle v,v \rangle - \frac{\langle v,w \rangle^2}{\langle w,w \rangle}
    \end{align}
    ist. Umstellen liefert die Cauchy-Schwarzsche Ungleichung 
    \begin{align}
        |\langle v , w \rangle |^2 \leq \langle v,v \rangle \cdot \langle w,w \rangle.
    \end{align}
\end{proof}

Für einen reellen Prähilbertraum $V$ folgt nun aus der Schwarzschen Ungleichung für beliebige $v,w \in V$ mit $v \neq \mathbf{0}$ und $w \neq \mathbf{0}$, dass 
\begin{align}
    \frac{|\langle v, w \rangle|}{\sqrt{\langle v,v \rangle} \cdot \sqrt{\langle w,w \rangle}} \leq 1
\end{align}
oder äquivalent 
\begin{align}
    -1 \leq \frac{\langle v, w \rangle}{\sqrt{\langle v,v \rangle} \cdot \sqrt{\langle w,w \rangle}} \leq 1.
\end{align}
Wie im Beispiel $4.4.$ können wir daher ein eindeutiges $\theta \in [0, \pi)$ finden, sodass
\begin{align}
    \textit{cos}(\theta) = \frac{\langle v, w \rangle}{\sqrt{\langle v,v \rangle} \cdot \sqrt{\langle w,w \rangle}} \label{33.}
\end{align}
bzw.
\begin{align}
    \langle v,w \rangle = \sqrt{\langle v,v \rangle} \cdot \sqrt{\langle w,w \rangle} \cdot \textit{cos}(\theta) \label{34.}
\end{align}
ist. Im Falle des $\mathbb{R}^d$, ausgestattet mit dem Standardskalarprodukt $\langle \cdot, \cdot \rangle_{\mathbb{R}^d}$, haben wir im Beispiel $4.4.$ gesehen, dass die Zahl $\theta$ in \eqref{34.} mit dem geometrischen Winkel der kartesischen Visualisierung des $\mathbb{R}^d$ übereinstimmt. In einem beliebigen endlichdimensionalen Prähilbertraum $V$ ist \eqref{33.} hingegen die abstrakte Definition des Winkels. 

Wir haben nun also gesehen, dass wir mittels eines Skalarproduktes auf einem Vektorraum $V$ einen Winkelbegriff einführen können und dieser in Spezialfällen mit dem gewöhnlichen Winkelbegriff zusammenfällt. Mittels dieses Winkelbegriffs können wir nun noch eine nützliche Definition angeben:

\begin{definition}
    Sei $(V, \langle \cdot, \cdot \rangle)$ ein endlichdimensionaler Prähilbertraum. Dann sagen wir, dass zwei Vektoren $v,w \in V$ senkrecht oder orthogonal aufeinander stehen, wenn gilt, dass 
    \begin{align}
        \langle v , w \rangle = 0
    \end{align}
    ist.
\end{definition}

Wegen \eqref{34.} ist dabei sofort klar, warum wir im Falle $\langle v, w \rangle = 0$ davon sprechen, dass die Vektoren orthogonal zueinander sind, da in diesem Fall $\textit{cos}(\theta) = 0$ und damit $\theta = \frac{\pi}{2}$ ist.

Als Nächstes wollen wir nun noch zeigen, dass jedes Skalarprodukt $\langle \cdot, \cdot \rangle$ eine Norm $\| \cdot \|_{\langle \cdot, \cdot \rangle}$ über $\| v \|_{\langle \cdot, \cdot \rangle} = \sqrt{\langle v, v \rangle}$, mit $v \in V$, induziert.

\begin{proposition}
    Sei $(V, \langle \cdot, \cdot \rangle)$ ein Prähilbertraum. Dann ist die Funktion 
    \begin{align}
        \| \cdot \|_{\langle \cdot, \cdot \rangle} : V \longrightarrow [0,\infty),
    \end{align}
    definiert durch
    \begin{align}
        \| v \|_{\langle \cdot, \cdot \rangle} := \sqrt{\langle v, v \rangle},
    \end{align}
    mit $v \in V$, eine Norm. Wir nennen $\| \cdot \|_{\langle \cdot, \cdot \rangle}$ die durch $\langle \cdot , \cdot \rangle$ induzierte Norm. 
\end{proposition}

\begin{proof}
    Wir überprüfen die Normaxiome:
    \begin{itemize}
    
        \item[\textit{i)}] Zuerst bemerken wir, dass aufgrund der positiven Definitheit des Skalarproduktes die Funktion $\| \cdot \|_{\langle \cdot, \cdot \rangle}$ nur auf $[0, \infty)$ abbilden kann und $\| v \|_{\langle \cdot, \cdot \rangle}$ genau dann Null ist, wenn $v = \mathbf{0}$ ist. Das zeigt das erste Normaxiom.
        
        \item[\textit{ii)}] Es gilt für alle $\alpha \in \mathbb{R}$ und $v \in V$, dass 
        \begin{align}
            \| \alpha \cdot v \|_{\langle \cdot, \cdot \rangle} \sqrt{\langle \alpha \cdot v, \alpha \cdot v \rangle} = \sqrt{\alpha^2 \langle v , v \rangle} = |\alpha| \sqrt{\langle v,v \rangle} = |\alpha| \| v \|_{\langle \cdot, \cdot \rangle}
        \end{align}
        ist. Dies zeigt das zweite Normaxiom.
        \item[\textit{iii)}] Seien $v,w \in V$. Es gilt mittels der schwarzschen Ungleichung 
        \begin{align}
            \| v + w \|_{\langle \cdot, \cdot \rangle} =& \: \sqrt{\langle v+w,v+w\rangle} = \sqrt{\langle v,v \rangle + 2 \cdot \langle v,w \rangle + \langle w,w \rangle} \nonumber \\ \leq& \: \sqrt{\langle v,v \rangle + 2 \cdot |\langle v, w \rangle| + \langle w,w \rangle} \nonumber \\ \leq& \: \sqrt{\langle v,v \rangle + 2 \cdot \sqrt{\langle v,v \rangle} \cdot \sqrt{\langle w,w \rangle} + \langle w,w \rangle} \nonumber \\ =& \: \sqrt{\big( \sqrt{\langle v,v \rangle} + \sqrt{\langle w,w \rangle} \big)^2} \nonumber \\ =& \: \sqrt{\langle v,v \rangle} + \sqrt{\langle w,w \rangle} \nonumber \\ =& \: \| v \|_{\langle \cdot, \cdot \rangle} + \| w \|_{\langle \cdot, \cdot \rangle}.
        \end{align}
        Dies zeigt das dritte Normaxiom.
    \end{itemize}
    Damit ist gezeigt, dass es sich bei $\| \cdot \|_{\langle \cdot, \cdot \rangle}$ tatsächlich um eine Norm handelt.
\end{proof}

Wir haben damit also gezeigt, dass jeder endlichdimensionale Prähilbertraum $(V, \langle \cdot, \cdot \rangle)$ automatisch auch ein topologischer Hausdorffraum ist und alle Begrifflichkeiten, die wir in den Kapiteln über topologische und metrischen Räumen diskutiert haben, auf diese Räume mittels der durch das Skalarprodukt induzierten Topologie übertragbar sind. Die durch das Skalarprodukt induzierte Topologie ist dabei die von der Metrik $d_{\|\cdot\|_{\langle \cdot , \cdot \rangle}}$ induzierte Topologie. 

Die Übertragbarkeit der Begriffe der Konvergenz und der Stetigkeit, sowie die Eindeutigkeit der Grenzwerte konvergenter Folgen in Prähilberträume und die geometrischen Informationen, die im Skalarprodukt codiert sind, werden später in den Kapiteln über Mannigfaltigkeiten noch von großer Bedeutung sein, um mit diesen mathematischen Objekte richtig arbeiten zu können, sowie um sich ein hilfreiches geometrisches Verständnis dieser zu erarbeiten. 

Zusammengefasst haben wir in diesem Abschnitt also gelernt, was Prähilberträume sind, das diese auf natürliche Weise Hausdorff-Topologien tragen, was es uns ermöglicht die Resultate und Definitionen aus dem Kapitel über topologische Räume auf diese anzuwenden, und wir haben gelernt, dass der $\mathbb{R}^d$, ausgestattet mit dem Standardskalarprodukt $\langle \cdot , \cdot \rangle_{\mathbb{R}^d}$, die gewöhnliche euklidische Geometrie trägt. 


\subsection{Produkttopologien und topologische Vektorräume}

In diesem Abschnitt wollen wir uns noch mit dem wichtigen Begriff der Produkttopologie beschäftigen und darüberhinaus ein wichtiges algebraisches Anwendungsbeispiel für diesen Begriff vorstellen. Dieses Anwendungsbeispiel soll dabei einerseits zeigen, für was man eine Produkttopologie üblicherweise benötigt, und andererseits auch als Vorbereitung auf das Kapitel $9$ dienen, in welchem wir den Begriff einer unendlichdimensionalen Mannigfaltigkeit umreissen wollen.

Die primäre Aufgabe dieses Abschnittes ist es also, die bisher entwickelten Begriffe aus dem Bereich der Topologie und der (linearen) Algebra sinnvoll zusammenzubringen, um die Resultate aus diesem Abschnitt später in dieser Arbeit nutzen zu können. 

Als Vorbereitung auf die Definition einer Produkttopologie betrachten wir zuvor noch den folgenden Satz, der sich später noch als sehr hilfreich erweisen wird, wenn es darum geht die Definition einer Produkttopologie kompakt hinzuschreiben.

\begin{proposition}
    Sei $X$ eine Menge, $\mathcal{I}$ eine beliebige Indexmenge und $\tau_i$ für alle $i \in \mathcal{I}$ eine beliebige Topologie auf $X$. Dann gilt, dass $\tau := \bigcap_{i \in \mathcal{I}} \tau_i$ wieder eine Topologie auf $X$ ist.
\end{proposition}

\begin{proof}
    Wir weisen die Topologie-Axiome nach:
    \begin{itemize}
        \item [\textit{i)}] Da $\emptyset \in \tau_i$ für alle $i \in \mathcal{I}$ gilt folgt, dass $\emptyset \in \tau$ ist. Analog folgt auch, dass $X \in \tau$ ist.
        \item[\textit{ii)}] Sei $\mathcal{J}$ eine beliebige Indexmenge und $\mathcal{U}_j \in \tau$ für alle $j \in \mathcal{J}$. Wir betrachten die Menge $\mathcal{U} := \bigcup_{j \in \mathcal{J}} \mathcal{U}_j$. Da per Definition von $\tau$ gilt, dass $\mathcal{U}_j \in \tau_i$ für alle $j \in \mathcal{J}$ und alle $i \in \mathcal{I}$, folgt, dass $\mathcal{U} \in \tau_i$ für alle $i \in \mathcal{I}$ ist. Damit folgt $\mathcal{U} \in \tau$.
        \item[\textit{iii)}] Sei $\mathcal{J} := \{ 1, ..., m\}$ mit $m \in \mathbb{N}$. Seien weiter beliebige $\mathcal{U}_j \in \tau$ für alle $j \in \mathcal{J}$ gegeben. Wir definieren die Menge $\mathcal{U} = \bigcap_{j \in \mathcal{J}} \mathcal{U}_j$. Da per Definition von $\tau$ gilt, dass $\mathcal{U}_j \in \tau_i$ für alle $j \in \mathcal{J}$ und alle $i \in \mathcal{I}$, folgt, dass $\mathcal{U} \in \tau_i$ für alle $i \in \mathcal{I}$ ist. Damit folgt $\mathcal{U} \in \tau$.
    \end{itemize}
\end{proof}

Im folgenden betrachten wir nun eine abzählbare Familie von topologischen Räumen, d.h. $(X_j, \tau_j)$ seien für alle $j \in \mathbb{N}$ topologische Räume. Wir betrachten nun den Produktraum 
\begin{align}
    \prod_{j \in \mathbb{N}} X_j := X_1 \times X_2 \times ... :=  \{ (x_1, x_2, ...) \: \: | \: \; x_1 \in X_1, x_2 \in X_2, ...  \} .
\end{align}
Weiter wollen wir nun auf $X := \prod_{j \in \mathbb{N}}X_j$ eine Topologie $\tau \subseteq \mathcal{P}(X)$ erklären, die aus den gegebenen Topologien $\tau_j$, mit $j \in \mathbb{N}$, gebaut sein soll. Wie konstruieren wir eine derartige Topologie und wie stellen wir sicher. dass es sich dabei um eine intuitive Konstruktion handelt? 

Die Grundidee ist recht einfach: Da $\tau$ aus allen $\tau_j$ konstruiert sein soll, müssen die $\tau_j$ anschaulich Bestandteile oder Komponenten von $\tau$ sein. D.h. wir fassen $\tau$ ersteinmal anschaulich als ein zusammengesetztes Objekt auf, dessen Bestandteile die $\tau_j$ sind. Unter dieser Auffassung können wir natürlich zurecht behaupten, dass $\tau$ aus den $\tau_j$ \textit{gemacht} ist. Wie codieren wir diese Forderung mathematisch? Wir definieren uns dazu zuerst die sogenannten Projektionsabbildungen 
\begin{align}
    \pi_i : X = \prod_{j \in \mathbb{N}} X_j \longrightarrow X_i
\end{align}
durch 
\begin{align}
    \pi_i (x_1, x_2, ..., x_i, ...) = x_i \in X_i. \label{ggg}
\end{align}
D.h. die Abbildung $\pi_i$ projeziert die $i$-te Komponente aus den Elementen von $X$ heraus. Wenn $\tau$ aus all den $\tau_j$ gemacht sein soll, dann sollte heuristisch gelten, dass
\begin{align}
    \pi_i (\tau) = \tau_i \:\:\:\: \forall i \in \mathbb{N} \label{r}
\end{align}
oder etwas rigoroser: 
\begin{align}
    \forall \mathcal{V} \in \tau \:\: \textit{gilt, dass} \:\: \pi_i(\mathcal{V}) \in \tau_i \:\:\:\: \forall i \in \mathbb{N} \label{rr}
\end{align}
und 
\begin{align}
    \forall \mathcal{V}_i \in \tau_i \:\: \exists \mathcal{V} \in \tau : \pi_i(\mathcal{V}) = \mathcal{V}_i \label{rrr} 
\end{align}
Dabei verstehen wir $\pi_i(\mathcal{V})$ in \eqref{rr} und \eqref{rrr} als die Menge 
\begin{align}
    \{ x_i \in X_i \:\: | \:\: \exists x \in \mathcal{V} \:\: \textit{mit} \:\: \pi_i(x) = x_i \}.
\end{align}
Mit der Gleichung \eqref{r} wird dabei ausgedrückt, dass die Bausteine oder Komponenten, aus denen $\tau$ anschaulich zusammengesetzt ist, gerade die Topologien $\tau_i$ sind, da die Projektionen $\pi_i$ angewendet auf $\tau$ gerade den $i$-ten Betstandteil von $\tau$ liefert, so wie in Gleichung \eqref{ggg} die Projektion $\pi_i$ die $i$-te Komponente des Elementes $(x_1, x_2, ..., x_i, ...) \in X$ liefert, wobei wir uns $(x_1, x_2, ..., x_i, ...)$ als aus den einzelnen $x_1 \in X_1, x_2 \in X_2, ... , x_i \in X_i ,...$ zusammengesetzt vorstellen.

Anschaulich fordern wir also, dass es gemäß \eqref{r} unter anderem zu jedem $\mathcal{V}_i \in \tau_i$ ein $\mathcal{V} \in \tau$ gibt, sodass $\pi(\mathcal{V}) = \mathcal{V}_i$ gilt, da nur so \eqref{r} erfüllt sein kann. Eine natürliche und auch naheliegende Wahl für dieses $\mathcal{V}$ ist das Urbild von $\mathcal{V}_i$ unter $\pi_i$, d.h. 
\begin{align}
    \mathcal{V} = \pi^{-1}_i(\mathcal{V}_i) =& \: \nonumber \{ x \in X \:\: | \:\: \pi_i(x) \in \mathcal{V}_i \} \\ =& \: X_1 \times X_2 \times ... \times X_{i-1} \times \mathcal{V}_i \times X_{i+1} \times ... \label{gggg}
\end{align}
Damit fordern wir nun also, dass die Urbilder aller Elemente aus $\tau_i$ unter $\pi_i$ für alle $i \in \mathbb{N}$ in $\tau$ liegen. 

Man kann sich nun fragen, ob eine Menge $\tau$, die nur solche Mengen enthält, bereits eine Topologie ist. Die Antwort ist \textit{Nein}, denn in einer solchen Menge haben alle Elemente die Gestalt \eqref{gggg}. 

Nach den Topologie-Axiomen gilt aber, dass eine Topologie beispielsweise abgeschlossen unter endlichen Mengenschnitten sein muss. Seien also für $i < j$ und $\mathcal{V}_i \in \tau_i$ und $\mathcal{V}_j \in \tau_j$ die Mengen $X_1 \times ... \times X_{i-1} \times \mathcal{V}_i \times X_{i+1} \times ...$ und $X_1 \times ... \times X_{j-1} \times \mathcal{V}_j \times X_{j+1} \times ...$ in $\tau$, dann müsste folgen, dass auch $X_1 \times ... \times \mathcal{V}_i \times ... \times \mathcal{V}_j \times ...$ in $\tau$ liegt, was aber nicht sein kann, wenn $\tau$ nur aus den Urbildern der offenen Mengen aus $\tau_i$ unter den $\pi_i$ für alle $i \in \mathbb{N}$ bestünde. 

Ein Ansatz, der dieses Problem behebt, wäre folglich die Hinzunahme aller möglichen endlichen Mengenschnitte der obigen Urbilder zur Menge $\tau$. Aus analogen Gründen müssen wir dann natürlich auch alle möglichen Vereinigungen der beliebigen endlichen Mengenschnitte der Urbilder offener Mengen aus den $\tau_i$ unter den $\pi_i$ der Menge $\tau$ hinzufügen. 

Auf diese Weise wird $\tau$ auf natürliche Weise eine Topologie, die trivialerweise den Topologie-Axiomen genügt. Es ist dabei wichtig anzumerken, dass wir zu den der Menge $\tau$ bereits hinzugefügten Mengen nun keine weiteren Mengen mehr hinzufügen sollten. Der Grund ist, dass solche weiteren Mengen, die keine beliebige Vereinigung endlicher Mengenschnitte der $\pi_i$-Urbildern offener Mengen aus den $\tau_i$ darstellen, nicht aus den $\tau_i$ \textit{gemacht} sind, und dementsprechend dem Mengensystem $\tau$ mehr Struktur geben würden, als es gemäß den $\tau_i$ haben sollte, wenn $\tau$ anschaulich aus den $\tau_i$ bestünde. 

Man sagt daher auch, da der Schnitt beliebiger Topologien wieder eine Topologie darstellt, dass die so konstruierte Topologie $\tau$ die kleinste Topologie ist, unter der die Projektionen $\pi_i$ per Definition alle stetig sind. Das liefert die folgende Definition:

\begin{definition}
    Seien $(X_i, \tau_i)$ für alle $i \in \mathbb{N}$ topologische Räume. Wir definieren die Menge $X := \prod_{j \in \mathbb{N}} X_j$ und die $i$-ten Projektionsabbildungen $\pi_i : X \longrightarrow X_i$ für $i \in \mathbb{N}$. Dann ist die sogenannte Produkttopologie $\tau$ auf $X$ erklärt als die kleinste Topologie, unter der alle Projektionsaabildungen $\pi_i$ stetig sind.
\end{definition}

Diese Produkttopologie $\tau$ soll nun also anschaulich eine Topologie auf dem Produktraum $X = \prod_{j \in \mathbb{N}} X_j$ darstellen, die aus den einzelnen $\tau_j$ gebaut ist. Der Ausgangspunkt oder die Motivation dieser Konstruktion war dabei \eqref{r} bzw. \eqref{rr} und \eqref{rrr}. Wir können uns nun abschließend noch fragen, ob diese Produkttopologie am Ende wirklich das macht, was sie tun soll. 

D.h. wir fragen uns, ob die nun so konstruierte Topologie $\tau$ tatsächlich \eqref{r} erfüllt, denn immerhin diente uns \eqref{r} nur als anschauliche Motivation und wir haben bisher keinen konkreten Grund zu glauben, dass das so konstruierte $\tau$ nach der Hinzunahme aller beliebiger Mengenvereinigungen endlicher Mengenschnitte der $\pi_i$-Urbilder offener Mengen aus den $\tau_i$ tatsächlich noch \eqref{r} genügt. Es ist daher wichtig zu überprüfen, ob das noch stimmt, denn nur wenn \eqref{r} für die Produkttopologie $\tau$ gilt, war unsere Konstruktion am Ende wirklich sinnvoll oder intuitiv.

Um \eqref{r} für die Produkttopologie $\tau$ nachzuweisen, bemerken wir zuerst, dass $\eqref{r}$ äquivalent zu \eqref{rr} und \eqref{rrr} ist. Per Konstruktion genügt die Produkttopologie definitiv \eqref{rrr}, da die Urbilder aller offener Mengen $\mathcal{V}_i \in \tau_i$ unter den $\pi_i$ in $\tau$ enthalten sind. Damit müssen wir für die Produkttopologie $\tau$ nur \eqref{rr} nachweisen. Wir führen dazu den folgenden nützlichen Begriff ein:

\begin{definition}
    Seien $(X, \tau)$ und $(Y, \sigma)$ zwei topologische Räume und sei $f : X \longrightarrow Y$ eine Abbildung. Wir nennen die Abbildung $f$ offen, falls gilt, dass 
    \begin{align}
        \forall \mathcal{U} \in \tau : f(\mathcal{U}) \in \sigma
    \end{align}
    ist. Dabei ist $f(\mathcal{U})$ wieder erklärt als 
    \begin{align}
        f(\mathcal{U}) := \{ y \in Y \:\: | \:\: \exists x \in \mathcal{U} \:\: \textit{mit} \:\: f(x) = y \}.
    \end{align}
\end{definition}

Mittels dieser Definition sind wir nun in der Lage \eqref{rr} kompakter zu formulieren, da \eqref{rr} äquivalent zur Aussage ist, dass alle Projektionen $\pi_i$ offene Abbildungen sind. Wir erhalten den folgenden Satz:

\begin{proposition}
    Seien $(X_i, \tau_i)$ für alle $i \in \mathbb{N}$ topologische Räume und sei $X = \prod_{j \in \mathbb{N}} X_j$ ausgestattet mit der bezüglich der $\tau_i$ konstruierten Produkttopologie $\tau$. Dann gilt, dass für alle $i \in \mathbb{N}$ die Projektionen $\pi_i : X \longrightarrow X_i$ offen sind.
\end{proposition}

\begin{proof}
    Zuerst bemerken wir, dass per Konstruktion der Produkttopologie jedes Element $\mathcal{U} \in \tau$ geschrieben werden kann als 
    \begin{align}
        \mathcal{U} = \bigcup_{j \in \mathcal{J}} \bigcap_{k = 1}^{n_j} \pi^{-1}_{i_{k,j}} (\mathcal{U}_{k,j}).
    \end{align}
    Dabei ist $\mathcal{J}$ eine nicht notwendigerweise endliche Indexmenge, $n_j \in \mathbb{N}$, $i_{k,j} \in \mathbb{N}$ und $\mathcal{U}_{k,j} \in \tau_{i_{k,j}}$. Wir definieren für alle $i \in \mathbb{N}$ die Menge $\mathcal{V}_{i,k,j}$ durch
    \begin{align}
        \mathcal{V}_{i,k,j} \left\{\begin{array}{ll} \mathcal{U}_{k,j}, & i = i_{k,j} \\
         X_i, & i \neq i_{k,j} \end{array}\right. .
    \end{align}
    Per Definition gilt damit, dass $\mathcal{V}_{i,k,j} \in \tau_i$ und 
    \begin{align}
        \pi^{-1}_{i_{k,j}} (\mathcal{U}_{k,j}) = \prod_{i \in \mathbb{N}} \mathcal{V}_{i,k,j}
    \end{align}
    ist. Damit folgt für eine beliebige Projektionsabbildung $\pi_l$ mit $l \in \mathbb{N}$ und einem beliebigen Element $\mathcal{U} \in \tau$, dass
    \begin{align}
        \pi_l (\mathcal{U}) =& \: \pi_l \big( \bigcup_{j \in \mathcal{J}} \bigcap_{k = 1}^{n_j} \pi^{-1}_{i_{k,j}} (\mathcal{U}_{k,j}) \big) \nonumber \\ =& \: \bigcup_{j \in \mathcal{J}} \pi_l \big( \bigcap_{k = 1}^{n_j} \pi^{-1}_{i_{k,j}} (\mathcal{U}_{k,j}) \big) \nonumber \\ =& \: \bigcup_{j \in \mathcal{J}} \pi_l \big( \bigcap_{k = 1}^{n_j} \prod_{i \in \mathbb{N}} \mathcal{V}_{i,k,j} \big) \nonumber \\ =& \: \bigcup_{j \in \mathcal{J}} \pi_l \big( \prod_{i \in \mathbb{N}} \bigcap_{k = 1}^{n_j} \mathcal{V}_{i,k,j} \big) \nonumber \\ =& \: \bigcup_{j \in \mathcal{J}} \bigcap_{k = 1}^{n_j} \mathcal{V}_{l,k,j}. \label{jjj}
    \end{align}
    Da $\mathcal{V}_{l,k,j}$ für alle $j \in \mathcal{J}$ und alle $k \in \{ 1, ..., n_j \}$ eine offene Menge aus $\tau_l$ ist, und damit die rechte Seite von \eqref{jjj} eine beliebige Vereinigung endlicher Mengenschnitte von Elementen aus $\tau_l$ darstellt, folgt, dass $\pi_l (\mathcal{U}) \in \tau_l$ ist. Damit ist $\pi_l$ eine offene Abbildung. Da $l \in \mathbb{N}$ beliebig war folgt, dass alle Projektionen offen sind.
\end{proof}

Wir haben damit also eine Möglichkeit gefunden, eine Topologie $\tau$ auf $X = \prod_{j \in \mathbb{N}} X_j$ einzuführen, die anschaulich aus den $\tau_j$ konstruiert ist. Anders ausgedrückt wissen wir damit, wie wir topologische Räume miteinander verkleben können, ohne die topologische Struktur der Komponenten des resultierenden Produktraumes zu zerstören. In diesem Sinne handelt es sich bei der Produkttopologie um eine natürliche Topologie auf dem Produktraum $X$. 

Mittels der Produkttopologie können wir nun den Begriff der topologischen Gruppe einführen:

\begin{definition}
    Sei $(G, \circ)$ eine Gruppe und $\tau \subseteq \mathcal{P}(G)$ eine Topologie auf $G$. Wir nennen das Tripel $(G, \tau, \circ)$ eine topologische Gruppe, falls die Abbildungen 
    \begin{align}
        \circ : G \times G \longrightarrow& \: \: G \nonumber\\
        (g,h) \longmapsto& \: \: g \circ h
    \end{align}
    und 
    \begin{align}
        \cdot^{-1} : G \longrightarrow& \: \: G \nonumber\\
        g \longmapsto& \: \: g^{-1}
    \end{align}
    stetige Abbildungen sind. Dabei wurde auf $G \times G$ diejenige Produkttopologie gewählt, die wir erhalten, wenn wir den topologischen Raum $(G, \tau)$ mit sich selbst verkleben.
\end{definition}

Damit ist eine topologische Gruppe also eine mathematische Gruppe die gleichzeitig ein topologischer Raum ist, sodass die Gruppenverknüpfung und die Invertierungsabbildung, die jedem Element der Gruppe ihr Inverses zuordnet, stetig ist. Zusätzlich merken wir an, dass die Invertierungsabbildung in der Definition der topologischen Gruppe wohldefiniert ist, da gemäß Satz $4.1.$ das zu einem Element $g \in G$ inverse Element eindeutig bestimmt ist. 

Der Begriff einer topologischen Gruppe besitzt in der Mathematik zahlreiche Anwendungen. Insbesondere ist dieser Begriff entscheidend für die Definition sogenannter Lie-Gruppen, die wiederum ihre Anwendungen in der Physik, aber auch in der Theorie des \textit{Geodesic shooting}-Algorithmus und damit in der Bildverarbeitung finden. 

Darüberhinaus lässt sich mittels des Begriffes einer topologischen Gruppe nun auch der Begriff eines topologische Vektorraumes definieren, welchen wir insbesondere im Kapitel $9$ nutzten werden, um zu skizzieren, wie man den Begriff der endlichdimensionalen Mannigfaltigkeit, mit denen wir uns im Kapitel $6$ dieser Arbeit auseinandersetzen wollen, in das Unendlichdimensionale verallgemeinert. 

\begin{definition}
    Sei $(V, \oplus, \odot)$ ein $\mathbb{R}$-Vektorraum und $\tau \subseteq \mathcal{P}(V)$ eine Topologie auf $V$. Wir nennen das Quadrupel $(V, \tau, \oplus, \odot)$ einen topologischen $\mathbb{R}$-Vektorraum, falls $(V, \tau, \oplus)$ eine topologische Gruppe bildet und die Skalarmultiplikation $\odot : \mathbb{R} \times V \longrightarrow V$ stetig ist. Dabei statten wir $\mathbb{R}$ mit der Standardtopologie $\tau_d$ aus, welche von der Betragsfunktion-Metrik $d$ induziert wird, und $\mathbb{R} \times V$ statten wir mit derjenigen Produkttopologie aus, die entsteht, wenn wir $(\mathbb{R}, \tau_d)$ und $(V, \tau)$ als topologische Räume miteinander verkleben.
\end{definition}

Wichtige Beispiele von topologischen Räumen sind normierte Räume und Prähilberträume, da deren Vektoraddition und Skalarmultiplikation bezüglich der von der Norm oder dem Skalarprodukt induzierten Topologie stetig sind. Um das einzusehen, müssen wir lediglich zeigen, dass normierte Räume topologische Vektorräume sind, da jeder Prähilbertraum, wie wir oben gesehen haben, auf natürliche Weise ein normierter Raum ist. Wir erhalten den folgenden Satz:

\begin{proposition}
    Sei $(V, \| \cdot \|)$ ein normierter $\mathbb{R}$-Vektorraum, den wir mit der Normtopologie $\tau_V$ ausstatten, d.h. eben jener Topologie, die von der Metrik $d$, erklärt durch $d(v,w) = \|v - w\|$, induziert wird. Weiter sei $\mathbb{R}$ ausgesattet mit der durch die Betragsfunktion induzierten Topologie $\tau_{\mathbb{R}}$. Statten wir nun $V \times V$ und $\mathbb{R} \times V$ mit den Produkttopologien aus, so folgt, dass die Vektoraddition
    \begin{align}
        A := \oplus : V \times V \longrightarrow& \: \: V \nonumber\\
        (v,w) \longmapsto& \: \: A((v,w)) =: v \oplus w
    \end{align}
    und die Skalarmultiplikation 
    \begin{align}
        S := \odot : \mathbb{R} \times V \longrightarrow& \: \: V \nonumber\\
        (\lambda,v) \longmapsto& \: \: S((\lambda,v)) =: \lambda \odot v
    \end{align}
    stetig sind.
\end{proposition}

\begin{proof}
    Im Folgenden sei $\mathcal{U}^{\tau_V}_\delta (v) = \{ w \in V \:\: | \:\: \| v - w \| < \delta \}$ und $\mathcal{U}^{\tau_{\mathbb{R}}}_\delta (\lambda) = \{ \alpha \in \mathbb{R} \:\: | \:\: | \lambda - \alpha | < \delta \}$. An dieser Stelle sei daran erinnert, dass $\mathcal{U} \subseteq V$ offen ist, d.h. $\mathcal{U} \in \tau_V$, falls für alle $v \in \mathcal{U}$ ein $\delta > 0$ existiert, sodass $\mathcal{U}^{\tau_V}_\delta (v) \subseteq \mathcal{U}$ gilt. Analoges gilt für die offenen Mengen in $\tau_{\mathbb{R}}$. Wir überprüfen nun die Stetigkeit der Abbildungen $A$ und $S$:
    \begin{itemize}
        \item[\textit{i)}] Wir beginnen mit der Abbildung $A$. Sei $\mathcal{U}$ eine beliebige offene Menge. Wir wollen zeigen, dass $A^{-1}(\mathcal{U}) \subseteq V \times V$ offen bezüglich der Produkttopologie auf $V \times V$ ist. Sei dazu $(v,w) \in A^{-1}(\mathcal{U})$, d.h. $v \oplus w \in \mathcal{U}$. Da $\mathcal{U}$ offen ist, existiert ein $\delta > 0$, sodass $\mathcal{U}^{\tau_V}_\delta (v \oplus w) \subseteq \mathcal{U}$ ist. Sei weiter $\delta' > 0$ gegeben durch $\delta' = \frac{\delta}{2}$. Wir betrachten die Menge $\mathcal{U}^{\tau_V}_{\delta'} (v) \times \mathcal{U}^{\tau_V}_{\delta'} (w) \subseteq V \times V$. Da für beliebiges $z \in V$ und $\epsilon > 0$ die Menge $\mathcal{U}^{\tau_V}_\epsilon (z)$ offen ist, folgt aus der Konstruktion der Produkttopologie, dass die Menge $\mathcal{U}^{\tau_V}_{\delta'} (v) \times \mathcal{U}^{\tau_V}_{\delta'} (w)$ offen bezüglich der Produkttopologie ist. Sei nun $x \in \mathcal{U}^{\tau_V}_{\delta'} (v)$ und $y \in \mathcal{U}^{\tau_V}_{\delta'} (w)$ beliebig gewählt. Es gilt nun
        \begin{align}
            \| (v \oplus w) \ominus (x \oplus y) \| =& \: \| (v \ominus x) \oplus (w \ominus y) \| \nonumber \\ \leq& \: \| v \ominus x \| + \| w \ominus y \| \nonumber \\ <& \: \delta' + \delta' = \delta,
        \end{align}
        d.h. $(x \oplus y) \in \mathcal{U}^{\tau_V}_\delta (v \oplus w)$. Daraus folgt nun aber, dass 
        \begin{align}
            A(\mathcal{U}^{\tau_V}_{\delta'} (v) \times \mathcal{U}^{\tau_V}_{\delta'} (w)) \subseteq \mathcal{U}^{\tau_V}_\delta (v \oplus w) \subseteq \mathcal{U}
        \end{align}
        ist. Daraus folgt nun weiter, dass $A^{-1}(\mathcal{U})$ bezüglich der Produkttopologie offen ist, denn für beliebiges $(v,w) \in A^{-1}(\mathcal{U})$ haben wir eine bezüglich der Produkttopologie offene Menge $\mathcal{U}^{\tau_V}_{\delta'} (v) \times \mathcal{U}^{\tau_V}_{\delta'} (w)$ finden können, für die gilt, dass $(v,w) \in \mathcal{U}^{\tau_V}_{\delta'} (v) \times \mathcal{U}^{\tau_V}_{\delta'} (w) \subseteq A^{-1}(\mathcal{U})$ ist. Es ist leicht einzusehen, dass mit $\delta' = \delta'_{(v,w)}$
        \begin{align}
            A^{-1}(\mathcal{U}) = \bigcup_{(v,w) \in A^{-1}(\mathcal{U})} \mathcal{U}^{\tau_V}_{\delta'_{(v,w)}} (v) \times \mathcal{U}^{\tau_V}_{\delta'_{(v,w)}} (w) \label{yyy}
        \end{align}
        gilt. Da die rechte Seite von \eqref{yyy} eine Vereinigung offener Menge aus der Produkttopologie ist, ist gemäß den Topologie-Axiomen $A^{-1}(\mathcal{U})$ ein Element der Produkttopologie. Da $\mathcal{U} \in \tau_V$ beliebig war, folgt die Stetigkeit von $A$.   
        
        \item[\textit{ii)}] Wir zeigen die Stetigkeit der Abbildung $S$. Sei dazu wieder $\mathcal{U} \subseteq V$ eine beliebige offene Menge. Wir wollen zeigen, dass $S^{-1}(\mathcal{U}) \subseteq \mathbb{R} \times V$ offen bezüglich der Produkttopologie ist. Sei dazu $(\lambda, v) \in S^{-1}(\mathcal{U})$, d.h. $\lambda \odot v \in \mathcal{U}$. Da $\mathcal{U}$ offen ist, existiert ein $\delta > 0$, sodass $\mathcal{U}^{\tau_V}_\delta (\lambda \odot v) \subseteq \mathcal{U}$. Sei weiter $\delta' > 0$ hinreichend klein gewählt, sodass $\delta'(\| v \| + |\lambda|) + \delta'^2 < \delta$ ist. Wir betrachten die Menge $\mathcal{U}^{\tau_V}_{\delta'} (\lambda) \times \mathcal{U}^{\tau_V}_{\delta'} (v) \subseteq \mathbb{R} \times V$. Diese Menge ist bezüglich der Produkttopologie auf $\mathbb{R} \times V$ offen. Seien nun $\alpha \in \mathcal{U}^{\tau_{\mathbb{R}}}_{\delta'} (\lambda)$ und $w \in \mathcal{U}^{\tau_V}_{\delta'} (v)$ beliebig gewählt. Dann gilt 
        \begin{align}
            \| (\lambda \odot v) \ominus (\alpha \odot w) \| =& \: \| (\lambda \odot v) \ominus (\alpha \odot v) \oplus (\alpha \odot v) \ominus (\alpha \odot w) \| \nonumber \\ \leq& \: \| (\lambda \odot v) \ominus (\alpha \odot v) \| + \| (\alpha \odot v) \ominus (\alpha \odot w) \| \nonumber \\ =& \: |\lambda - \alpha| \|  v \| + |\alpha| \| v \ominus w \| \nonumber \\ <& \: \delta' \| v \| + |\alpha| \delta' \nonumber \\ =& \: \delta' (\| v \| + |\alpha - \lambda + \lambda|) \nonumber \\ \leq& \: \delta'(\| v \| + |\lambda|) + \delta'^2 < \delta,
        \end{align}
        d.h. es gilt $(\alpha \odot w) \in \mathcal{U}^{\tau_V}_\delta (\lambda \odot v)$. Daraus folgt nun aber, dass 
        \begin{align}
            S(\mathcal{U}^{\tau_V}_{\delta'} (\lambda) \times \mathcal{U}^{\tau_V}_{\delta'} (v)) \subseteq \mathcal{U}^{\tau_V}_\delta (\lambda \odot v) \subseteq \mathcal{U}
        \end{align}
        ist. Aus zu \textit{i)} analogen Argumenten folgt damit, dass $S^{-1}(\mathcal{U})$ offen ist. Da $\mathcal{U}$ eine beliebige offene Menge aus $V$ war ist $S$ damit stetig.
    \end{itemize}
\end{proof}


\subsection{Tensoren}

In Vorbereitung auf das Kapitel über riemannsche Mannigfaltigkeiten, wollen wir uns nun noch kurz mit dem Tensorbegriff auseinandersetzen. 

Ziel dieses kurzen Abschnittes soll es sein, aufzuzeigen, dass die oben besprochenen Skalarprodukte spezielle Tensoren darstellen, um so später die aus der Literatur bekannte Definition eines metrischen Tensorfeldes anschaulicher zu gestalten.

Zuerst definieren wir den Begriff einer linearen Abbildungen:

\begin{definition}
    Seien $V$ und $W$ jeweils $\mathbb{R}$-Vektorräume. Wir nennen eine Abbildung $f : V \longrightarrow W$ linear, falls 
    \begin{align}
        f ( \alpha \cdot v + \beta \cdot w  ) = \alpha \cdot f(v) + \beta \cdot f(w) \:\:\:\:\:\:\:\: \forall \alpha, \beta \in \mathbb{R}, \: v,w \in V   
    \end{align}
\end{definition}

Die Eigenschaft der Linearität einer Abbildung ist uns bereits im Abschnitt über Prähilberträume begegnet. So haben wir Skalarprodukte unter anderem als bilineare Abbildungen definiert, was bedeutet, das Skalarprodukte linear in ihren beiden Argumenten sind. 

\begin{remark}
    Eine wichtige und nützliche Eigenschaft linearer Abbildungen ist, dass die sogenannte Verknüpfung linearer Abbildungen stets wieder linear ist. Das bedeutet, dass für die linearen Abbildungen $f : V \longrightarrow W$ und $g : Z \longrightarrow X$, wobei $V,W,Z$ jeweils $\mathbb{R}$-Vektorräume sind, gilt, dass die Abbildung $h : Z \longrightarrow W$, definiert durch 
    \begin{align}
        h(z) := f(g(z)) \:\:\:\: \forall z \in Z,
    \end{align}
    wieder eine lineare Abbildung ist. Das lässt sich wie folgt einsehen: Seien $\alpha, \beta \in \mathbb{R}$ und $x,y \in V$ beliebig. Dann gilt 
    \begin{align}
        h(\alpha \cdot x + \beta \cdot y) =& \: f(g(\alpha \cdot x + \beta \cdot y)) \nonumber \\ =& \: f(\alpha \cdot g(x) + \beta \cdot g(y)) \nonumber \\ =& \: \alpha \cdot f(g(x)) + \beta \cdot f(g(y)) \nonumber \\ =& \: \alpha \cdot h(x) + \beta \cdot h(y),
    \end{align}
    was die Linearität von $h$ zeigt.
\end{remark}

Mittels der Beobachtung, dass $\mathbb{R}$ bezüglich der gewöhnlichen Addition selbst ein $\mathbb{R}$-Vektorraum ist, wobei die Skalarmultiplikation in diesem Fall durch die gewöhnliche reelle Multiplikation gegeben ist, können wir den Begriff des sogenannten Dualraumes erklären:

\begin{definition}
    Sei $V$ ein $\mathbb{R}$-Vektorraum. Dann erklären wir den (algebraischen) Dualraum $V^*$ von $V$ als 
    \begin{align}
        V^* := \{ \eta : V \longrightarrow \mathbb{R} \:\: | \:\: \textit{$\eta$ ist linear} \}.
    \end{align}
\end{definition}

Es ist dabei leicht einzusehen, dass $V^*$, ausgestattet mit der punktweisen Addition und Skalarmultiplikation, auch wieder ein $\mathbb{R}$-Vektorraum ist. Ist $V$ sogar ein endlichdimensionaler $\mathbb{R}$-Vektorraum der Dimension $n$, so folgt, dass der Dualraum $V^*$ ebenfalls ein $n$-dimensionaler $\mathbb{R}$-Vektorraum ist:

\begin{proposition}
    Ist $V$ ein endlichdimensionaler $\mathbb{R}$-Vektorraum der Dimension $n$, dann folgt, dass der zu $V$ gehörige Dualraum $V^*$ ebenfalls ein endlichdimensionaler $\mathbb{R}$-Vektorraum der Dimension $n$ ist.
\end{proposition}

\begin{proof}
    Sei $\mathcal{A} := \{ e_1, ..., e_n \} \subseteq V$ eine Basis von $V$, d.h. $\mathcal{A}$ ist eine Menge linear unabhängiger Vektoren, die ganz $V$ aufspannt. Wir definieren für $j \in \{ 1, ..., n \}$ die Abbildungen $\eta^j : V \longrightarrow \mathbb{R}$ durch 
    \begin{align}
         \eta^j (e_i) = \left\{\begin{array}{ll} 1, & i = j \\
         0, & i \neq j \end{array}\right. =: \delta^j_i. \label{dual}
    \end{align}
    Setzen wir die $\eta^j$ linear auf $V$ fort, so folgt, dass die $\eta^j$ in $V^*$ liegen. Wir zeigen zuerst, dass $\textit{span}_\mathbb{R} ( \{ \eta^1, ..., \eta^n \} ) = V^*$ ist: Sei dazu $\phi \in V^*$ beliebig. Da $\phi$ linear ist, ist $\phi$ vollständig über seine Wirkung auf die Basisvektoren $e_i$ bestimmt, denn sei $v = \sum_{k = 1}^n v^{i} e_i$, dann gilt $\phi (v) = \sum_{k = 1}^n v^{i} \phi(e_i)$. Wir setzen $\phi(e_i) =: \theta_i$. Weiter definieren wir die Abbildung $\sigma = \sum_{k = 1}^n \theta_k \eta^k$. Da $\sigma \in V^*$ ist folgt, dass $\sigma$ vollständig über seine Wirkung auf die Basis $\mathcal{A}$ bestimmt ist. Es gilt für $i \in \{ 1, ..., n \}$
    \begin{align}
        \sigma (e_i) =& \: \big( \sum_{k = 1}^n  \theta_k \eta^k \big) (e_i) \nonumber \\ =& \: \sum_{k = 1}^n \theta_k \eta^k(e_i) \nonumber \\ =& \: \sum_{k = 1}^n \theta_k \delta^k_i = \theta_i,
    \end{align}
    d.h. $\sigma$ und $\phi$ stimmen auf $\mathcal{A}$ überein. Da $\sigma$ und $\phi$ per Definition beide linear sind, folgt, dass $\sigma$ und $\phi$ auf ganz $V$ übereinstimmen, d.h. $\sigma = \phi$. Damit lässt sich aber $\phi$ wegen der Definition von $\sigma$ als Linearkombination der $\eta^j$ schreiben. Da $\phi$ beliebig war folgt, dass die Menge $\{ \eta^1, ..., \eta^n \}$ $V^*$ aufspannt. Es bleibt zu zeigen, dass die Menge $\{ \eta^1, ..., \eta^n \}$ linear unabhängig ist. Seien dazu für $k \in \{ 1, ..., n \}$ $\lambda_k \in \mathbb{R}$ so gegeben, dass 
    \begin{align}
        \sum_{k = 1}^n \lambda_k \eta^k = \mathbf{0} \label{iii}
    \end{align}
    ist. Dabei bezeichnet $\mathbf{0} \in V^*$ das additive Nullelement von $V^*$, welches jedes Element von $V$ auf $0 \in \mathbb{R}$ abbildet. Wir wollen zeigen, dass notwendigerweise folgt, dass alle $\lambda_k$ gleich $0$ sind. Da $\sum_{k = 1}^n \lambda_k \eta^k$ ein Element aus $V^*$ definiert, können wir mit diesem auf Elemente $v \in V$ wirken. Für $v = e_i$ gilt, dass 
    \begin{align}
        \big ( \sum_{k = 1}^n \lambda_k \eta^k \big) (e_i) = \lambda_i.
    \end{align}
    Mit \eqref{iii} folgt damit $\lambda_i = 0$ für alle $i \in \{  1, ..., n \}$. Das zeigt die Behauptung.
\end{proof}

Aus dem vorherigen Beweis erhalten wir noch die folgende nützliche Definition:

\begin{definition}
    Sei $V$ ein $n$-dimensionaler $\mathbb{R}$-Vektorraum mit Basis $\mathcal{A} := \{ e_1, ..., e_n \}$. Dann ist die zu $\mathcal{A}$ zugehörige Dualbasis in $V^*$ die Menge $\mathcal{B} := \{ \eta^1 , ..., \eta^n \}$, die erklärt ist durch \eqref{dual}. Nach dem Beweis von Satz $4.10.$ ist die so definierte Dualbasis von $V^*$ tatsächlich eine Basis von $V^*$.  
\end{definition}

Aufgrund der linearen Struktur auf $V^*$ lassen sich lineare Abbildungen auf $V^*$ definieren. Es folgt die Definition eines allgemeinen $(r,s)$-Tensors:

\begin{definition}
    Sei $V$ wieder ein endlichdimensionaler $\mathbb{R}$-Vektorraum der Dimension $n$. Dann bezeichnen wir jede multilineare Abbildung $T$ der Form
    \begin{align}
        T: \underbrace{V \times ... \times V}_{s} \times \underbrace{V^* \times ... \times V^*}_{r} \longrightarrow \mathbb{R}
    \end{align}
    als $(r,s)$-Tensor. Multilinear meint dabei, dass $T$ in jedem seiner $(r+s)$ vielen Argumenten linear ist. Weiter wollen wir die Menge aller derartiger $(r,s)$-Tensoren als $T^r_s(V)$ bezeichnen.
\end{definition}

\begin{example}
    Wir betrachten zwei einfache Beispiele für Tensoren.
    \begin{itemize}
        \item[\textit{i)}] Ein einfaches Beispiel eines $(r,s)$-Tensors mit $r = 0$ und $s=2$ ist das oben diskutierte Skalarprodukt $\langle \cdot , \cdot \rangle$. Per Definition handelt es sich dabei nämlich um eine bilineare Abbildung der Form 
        \begin{align}
            \langle \cdot , \cdot \rangle : V \times V \longrightarrow \mathbb{R}.
        \end{align}
        \item[\textit{ii)}] Ein weiteres wichtiges Beispiel für Tensoren, welches insbesondere später von Beudeutung sein wird, wenn wir den Begriff des Tensorbündels motivieren wollen, sind die Vektoren eines Vektorraums. So können wir den Vektor $v \in V$ als einen $(1,0)$-Tensor auffassen. Um das zu sehen, bemerken wir zuerst, dass gemäß Satz $4.10.$ der Dualraum des Dualraums von $V$, der sogenannte Bidualraum $V^{**}$, ebenso von der Dimension $n$ ist. Weiter wissen wir, dass alle endlichdimensionalen Vektorräume der selben Dimension isomorph zueinander sind \cite{fischer2003lineare}.
        
        Das bedeutet, dass zwischen Vektorräumen derselben Dimension eine bijektive und lineare Abbildung existiert, was anschaulich bedeutet, dass es sich bei solchen Vektorräumen um die selben mathematischen Gebilde handelt, in denen die Elemente nur anders bezeichnet werden. Strukturell besteht zwischen diesen aber kein Unterschied. 

        Es gilt also, dass $V$ und $V^{**}$ isomorph zueinander sind, was wir in Zeichen durch $V \cong V^{**}$ ausdrücken. D.h. wir können u.a. jedes Element $v \in V$ eineindeutig mit einem Element $\sigma \in V^{**}$ identifizieren. Sei nun
        \begin{align}
            \mathbb{i} : V \longrightarrow& \: \: V^{**} \nonumber\\
        v \longmapsto& \: \: \sigma = \mathbb{i}(v)
        \end{align}
        eine Abbildung, definiert durch 
        \begin{align}
            \sigma (\eta) = (\mathbb{i}(v))(\eta) := \eta (v) \:\:\:\: \forall \eta \in V^*.
        \end{align}
        Diese Abbildung ist, wir wir gleich sehen werden, linear und bijektiv, was bedeutet, dass es sich bei $\mathbb{i}$ um eine sogenannten Isomorphismus handelt, der die Isomorphie $V \cong V^{**}$ realisiert. Bijektiv heißt dabei, dass es sich um eine injektive, wie auch um eine surjektive Abbildung handelt. 
        
        Injektiv heißt eine Funktion dabei genau dann, wenn keine zwei Elemente auf das selbe Element abgebildet werden, während eine Funktion als surjektiv bezeichnet wird, wenn anschaulich jedes Element im Wertebereich der Funktion von der Funktion getroffen wird, d.h. sei $f : V \longrightarrow W$ eine Funktion zwischen zwei Mengen $V$ und $W$, dann heißt $f$ surjektiv, falls
        \begin{align}
            \forall w \in W \:\: \exists v \in V \:\: \textit{mit} \:\: f(v) = w.
        \end{align}
        Anschaulich ist eine bijektive Abbildung also eine $1$ zu $1$ Abbildung, die jedem Element aus der einen Mengen genau einem Element aus der anderen Menge und umgedreht zuordnet.
        
        Wie sehen wir nun ein, dass es sich bei der Abbildung $\mathbb{i}$ tatsächlich um einen Isomorphismus handelt? 

        Wir betrachten dazu im Zusammmenhang mit linearen Abbildungen zuerst die folgenden zwei wichtigen Mengen. Seien dazu $V, W$ zwei $\mathbb{K}$-Vektorräume und $f : V \longrightarrow W$ eine lineare Abbildung. Zuerst definieren wir den sogenannten Kern der Abbildung $f$ als die Menge
        \begin{align}
            \textit{ker}(V) := \{ v \in V \:\: | \:\: f(v) = \mathbf{0}_W \} \subseteq V.
        \end{align}
        Dabei bezeichnet $\mathbf{0}_W$ das neutrale Element bzw. den Nullvektor von $W$. Weiter definieren wir das sogenannte Bild der Abbildung $f$ als die Menge 
        \begin{align}
            \textit{im}(f) := \{ w \in W \:\: | \:\: \exists v \in V \:\: \textit{mit} \:\: f(v) = w \} \subseteq W.
        \end{align}
        Beide Mengen können jeweils als $\mathbb{K}$-Untervektorräume von $V$ bzw. von $W$ verstanden werden \cite{fischer2003lineare}, was bedeutet, dass wir diese mit der Vektorraumstruktur von $V$ bzw. von $W$ ausstatten können, sodass diese selber zu Vektorräumen werden. 

        Mittels dieser beiden $\mathbb{K}$-Vektorräume können wir nun zwei wichtige Eigenschaften von linearen Abbildungen betrachten: 

        \begin{lemma}
            Seien $V, W$ zwei $\mathbb{K}$-Vektorräume und sei $f : V \longrightarrow W$ eine lineare Abbildung. Dann sind die folgenden Aussagen äquivalent.
            \begin{itemize}
                \item[\textit{(1)}] $f$ ist injektiv, d.h. für $v,w \in V$ gilt
                \begin{align}
                    f(v) = f(w) \:\: \implies \:\: v = w.
                \end{align}
                \item[\textit{(2)}] Der Kern der Abbildung $f$ ist trivial, d.h. 
                \begin{align}
                    \textit{ker}(f) = \{ \mathbf{0}_V \},
                \end{align}
                wobei $\mathbf{0}_V$ das neutrale Element bzw. der Nullvektor von $V$ ist.
            \end{itemize}
        \end{lemma}

        \begin{proof}
            \textit{(1)} $\implies$ \textit{(2)}: Angenommen $f$ ist injektiv. Sei $v \in V$ ein Vektor aus $V$ mit der Eigenschaft, dass $f(v) = \mathbf{0}_W$ ist. Wir wissen, dass für $f(\mathbf{0}_V) \in V$ gilt, dass 
            \begin{align}
                f(\mathbf{0}_V) = f(0 \cdot \mathbf{0}_V) = 0 \cdot f(\mathbf{0}_V) = \mathbf{0}_W.
            \end{align}
            Letztere Gleichheit folgt dabei aus $0 \cdot w = (0 + 0) \cdot w = 0 \cdot w + 0 \cdot w$. Daraus folgt nun aber $f(v) = f(\mathbf{0}_V)$ und nach der Injektivität damit $v = \mathbf{0}_V$. Dies zeigt die erste Richtung.

            \textit{(2)} $\implies$ \textit{(1)}: Angenommen der Kern von $f$ ist trivial und es gelte für $v,w \in V$, dass $f(v) = f(w)$ ist. Mittels der Linearität von $f$ folgt
            \begin{align}
                f(v) = f(w) \iff& \: f(v) - f(w) = \mathbf{0}_W \nonumber \\ \iff& \: f(v - w) = \mathbf{0}_W \nonumber \\ \iff& \: v - w \in \textit{ker}(f) = \{ \mathbf{0}_V \} \nonumber \\ \iff& \: v = w.
            \end{align}
            Dies zeigt die andere Richtung.
        \end{proof}

        \begin{lemma}
            Seien $V, W$ zwei endlichdimensionale $\mathbb{K}$-Vektorräume und $f : V \longrightarrow W$ eine lineare Abbildung. Dann gilt die folgende Formel:
            \begin{align}
                \textit{dim}(V) = \textit{dim}(\textit{ker}(f)) + \textit{dim}(\textit{im}(f)).
            \end{align}
            Dabei ordnet $\textit{dim}(\cdot)$ den jeweiligen Vektorräumen ihre Dimension zu und es gilt die Konvention, dass die Dimension eines Nullvektorraums, d.h. eines Vektorraumes der nur seinen Nullvektor enthält, gerade $0$ ist.
        \end{lemma}

        \begin{proof}
            Siehe \cite{fischer2003lineare}.
        \end{proof}

        Mittels dieser beiden Lemmata können wir nun zeigen, dass es sich bei $\mathbb{i}$ tatsächlich um einen Isomorphismus handelt. Zuerst überprüfen wir die Linearität von $\mathbb{i}$: Seien dazu $\alpha, \beta \in \mathbb{R}$ und $v,w \in V$ gegeben. Dann gilt für alle $f \in V^{*}$
        \begin{align}
            (\mathbb{i}(\alpha \cdot v + \beta \cdot w)) (f) =& \: f (\alpha \cdot v + \beta \cdot w) \nonumber \\ =& \: \alpha \cdot f(v) + \beta \cdot f(w) \nonumber \\ =& \: \alpha (\mathbb{i}(v)) (f) + \beta \cdot (\mathbb{i}(w)) (f) \nonumber \\ =& \: (\alpha \cdot \mathbb{i}(v) + \beta \cdot \mathbb{i}(w)) (f).
        \end{align}
        Da das für alle $f \in V^*$ gilt, folgt die Linearität von $\mathbb{i}$. 

        Wir stellen darüberhinaus fest, dass $\mathbb{i}(v) \in V^{**}$ ist, denn per Definition gilt für $\eta \in V^{*}$, dass
        \begin{align}
            (\mathbb{i}(v))(\eta) = \eta(v).
        \end{align}
        Daraus folgt, dass für $\beta_1, \beta_2 \in \mathbb{R}$ und $\eta^1, \eta^2 \in V^*$ 
        \begin{align}
            (\mathbb{i}(v))(\beta_1 \cdot \eta^1 + \beta_2 \cdot \eta^2) =& \: (\beta_1 \cdot \eta^1 + \beta_2 \cdot \eta^2)(v) \nonumber \\ =& \: \beta_1 \cdot \eta^1(v) + \beta_2 \cdot \eta^2(v) \nonumber \\ =& \: \beta_1 \cdot (\mathbb{i}(v))(\eta^1) + \beta_2 \cdot (\mathbb{i}(v))(\eta^2)
        \end{align}
        gilt, was zeigt, dass es sich bei $\mathbb{i}(v)$ um eine lineare Abbildung der Form $V^* \longrightarrow \mathbb{R}$ handelt.

        Als Nächstes zeigen wir die Injektivität von $\mathbb{i}$. Nach Lemma $4.1.$ müssen wir dazu nur zeigen, dass der Kern von $\mathbb{i}$ trivial ist. Sei also $v \in V$ so gewählt, dass $\mathbb{i}(v) = \mathbf{0}_{V^{**}}$ ist. Das impliziert aber 
        \begin{align}
            (\mathbb{i}(v)) (f) = f(v) = 0 \:\:\:\: \forall f \in V^*,
        \end{align}
        per Definition von $\mathbf{0}_{V^{**}} \in V^{**}$. Angenommen $v \neq \mathbf{0}_V$. Dann können wir im $\textit{dim}(V)$-dimensionalen $\mathbb{R}$-Vektorraum $V$ eine Basis der Form $\{ v, v_2, ..., v_{\textit{dim}(V)} \} \subseteq V$ finden. Mittels dieser Basis können wir die folgende lineare Abbildung $g \in V^*$  durch
        \begin{align}
            g(v) = 1, \:\: g(v_2) = g(v_3) = ... = g(v_{\textit{dim}(V)}) = 0
        \end{align}
        definieren (Beachte, dass eine lineare Abbildung eindeutig über ihre Wirkung auf eine Basis definiert ist!). Damit folgt
        \begin{align}
            (\mathbb{i}(v)) (g) = g(v) = 1 \neq 0,
        \end{align}
        woraus folgt, dass $v = \mathbf{0}_V$ sein muss. Da $v \in V$ ein beliebiger Vektor mit $\mathbb{i}(v) = \mathbf{0}_{V^{**}}$ ist, folgt das der Kern von $\mathbb{i}$ trivial ist und damit, dass $\mathbb{i}$ injektiv. 
        
        Da der Kern von $\mathbb{i}$ trivial ist folgt darüber hinaus mit Lemma $4.2.$, dass $\mathbb{i}$ surjektiv ist, da $\textit{dim}(\textit{ker}(\mathbb{i})) = 0$ ist und daher folglich $\textit{dim}(\textit{im}(\mathbb{i})) = \textit{dim}(V) = \textit{dim}(V^{**})$ sein muss. Daraus folgt die Bijektivität von $\mathbb{i}$, denn angenommen $\textit{im}(\mathbb{i}) =: \mathbb{i}(V) \neq V^{**}$, d.h. $\mathbb{i}(V) \subset V^{**}$. Wegen $\textit{dim}(\textit{im}(\mathbb{i})) = \textit{dim}(V^{**})$ folgt, dass es eine Basis $\{ \epsilon_1, ..., \epsilon_{\textit{dim}(V^{**})} \} \subset V^{**}$ der Länge $\textit{dim}(V^{**})$ in $\mathbb{i}(V)$ gibt. Da wir $\mathbb{i}(V) \subset V^{**}$ vorraussetzen folgt, dass es ein $\delta \in V^{**}$ geben muss, sodass
        \begin{align}
            \delta \neq \sum_{i = 1}^{\textit{dim}(V^{**})} \alpha^{i} \cdot \epsilon_i
        \end{align}
        für alle $\alpha_1, ..., \alpha_{\textit{dim}(V^{**})} \in \mathbb{R}$. Daraus folgt aber sofort die lineare Unabhängigkeit der Menge $\{ \delta, \epsilon_1, ..., \epsilon_{\textit{dim}(V^{**})} \} \subset V^{**}$, womit wir in $V^{**}$ eine Basis der Länge $\textit{dim}(V^{**}) + 1$ gefunden hätten, was einen offensichtlichen Widerspruch darstellt. Damit folgt die Surjektivität von $\mathbb{i}$.
        
        Mittels der Beobachtung, dass $\mathbb{i}$ einen Isomorphismus darstellt, können wir nun $v \in V$ folgendermaßen als Abbildung der Form $V^* \longrightarrow \mathbb{R}$ und damit als Element von $V^{**} = T^1_0(V)$ verstehen:
        \begin{align}
            \eta(v) = (\mathbb{i}(v))(\eta) =: v(\eta) 
        \end{align}
         Auf diese Weise lassen sich die Vektoren aus $V$ selbst als Tensoren auffassen.
    \end{itemize}
\end{example}

Wir wollen diesen Abschnitt nun mit einer letzten nützlichen Definition abschließen:

\begin{definition}
    Sei $V$ ein $n$-dimensionaler $\mathbb{R}$-Vektorraum und $\mathcal{A} := \{ e_1, ..., e_n \}$ eine Basis von $V$. Weiter sei $\mathcal{B} := \{ \eta^1, ..., \eta^n \}$ die zu $\mathcal{A}$ gehörige Dualbasis in $V^*$. Wir betrachten den $(r,s)$-Tensor $T$. Wir erklären die von $\mathcal{A}$ und $\mathcal{B}$ abhängigen Tensorkomponenten von T durch 
    \begin{align}
        T_{i_1, ..., i_s}^{j_1, ..., j_r} := T(e_{i_1}, ..., e_{i_s}, \eta^{j_1}, ..., \eta^{j_r}),
    \end{align}
    wobei $i_1, ..., i_s, j_1, ..., j_r \in \{ 1, ..., n \}$ sind.
\end{definition}


\newpage

\quad 
\newpage

\section{Analytische Grundlagen}

Wie in der Einleitung erwähnt, soll eines der in dieser Arbeit angestrebten Ziele sein, die Länge einer Kurven auf einer endlichdimensionalen Mannigfaltigkeit einzuführen und grob zu zeigen, wie sich der Längenbegriff auf einer unendlichdimensionalen Mannigfaltigkeit erklären lässt. 

Der Grund dafür ist, dass wir am Ende dieser Arbeit, wie schon oft erwähnt, über sogenannte stationäre Kurven des sogenannten Längenfunktionals sprechen wollen, wobei das Längenfunktional, wie der Name es schon andeutet, salopp gesprochen jeder Kurve eine Länge zuordnet.

Um den Begriff der Länge einer Kurve rigoros einzuführen, benötigen wir aber, wie wir noch sehen werden, analytische Methoden. Wir wollen in diesem Kapitel die nötigen analytischen Werkzeuge benennen und einige nützliche Eigenschaften dieser Werkzeuge zeigen. 

Wir werden dabei ab jetzt auf die wohlbekannten Rechenregeln bezüglich der Folgenkonvergenz in normierten Räumen zurückgreifen. Auch wenn viele dieser Rechenregeln unter Rückgriff der im normierten Raum erklärten Norm elementar aus der Definition der Folgenkonvergenz beweisbar sind, verweisen wir den Leser an dieser Stelle an \cite{trench2013introduction}. 

Wir merken weiter noch an, dass die analytischen Werkzeuge, die wir in diesem Kapitel diskutieren möchten, vornehmlich im Sinne der Funktionalanalysis definiert werden. Das hat zum einen den Vorteil, dass wir die zu besprechenden Begrifflichkeiten sehr allgemein und rigoros erklären können, und zum anderen hat dieser abstraktere Zugang den Vorteil, dass er sich in Teilen im Kapitel $9$ nutzen lässt, in welchem wir uns unter anderem grob mit unendlichdimensionalen Vektorräumen und Mannigfaltigkeiten konfrontiert sehen werden.

In diesem Kapitel sei der $\mathbb{R}^d$ für $d \in \mathbb{N}$ durchweg mit der durch das Standardskalarprodukt induzierten Topologie, der Standardtopologie des $\mathbb{R}^d$, versehen (Beachte: Die durch das Standardskalarprodukt induzierte Norm auf $\mathbb{R}$ ist gerade der Betrag.). 

Wir wollen nun unter anderem klären, was wir unter der Differenzierbarkeit einer Funktion verstehen, welche beispielsweise auf dem $\mathbb{R}^d$ erklärt ist. Um den Differenzierbarkeitsbegriff einer Funktion zu motivieren, betrachten wir zuerst den anschaulicheren Begriff der Richtungsableitung einer Funktion $f : X \longrightarrow Y$, welche auf einem normierten Raum $(X, \| \cdot \|_X)$ definiert ist und Werte in einem normierten Raum $(Y, \| \cdot \|_Y)$ annimmt. Der Definitions- und Wertebereich der Funktion $f$ muss dabei notwendigerweise ein Vektorraum sein, da wir ohne die Vektoraddition in $X$ und $Y$ den anschaulichen Differenzenquotienten der Richtungsableitung nicht bilden können.

Im Folgenden sind alle normierten Räume reell und enthalten mehr als nur den Nullvektor $\mathbf{0}$. Die zweite Forderung ergibt sich daraus, dass zum einen dieser triviale Fall für uns vollkommen uninteressant ist, und zum anderen sind wir mit dieser Forderung in der Lage eine besonders einfache und handhabbare Definition der später noch diskutierten Operatornorm zu nutzen. 

\begin{definition}
    Seien $(X, \| \cdot \|_X)$ und $(Y, \| \cdot \|_Y)$ zwei normierte Räume und sei $f : X \longrightarrow Y$ eine Funktion. $f$ heißt richtungsdifferenzierbar am Punkt $x \in X$ in Richtung $h \in X$, falls der Grenzwert 
    \begin{align}
        f'(x; h) := \lim_{t \longrightarrow 0} \frac{f(x + t \cdot h) - f(x)}{t} \label{iiii}
    \end{align}
    in $Y$ existiert. In diesem Fall nennen wir $f'(x; h)$ Richtungsableitung von $f$ am Punkt $x$ in Richtung $h$. $f$ heißt richtungsableitbar im Punkt $x$, falls $f'(x;h)$ für alle $h \in X$ existiert. Schließlich nennen wir $f$ richtungsableitbar auf der offenen Menge $\mathcal{U} \subseteq X$, falls $f$ für alle $x \in \mathcal{U}$ richtungsableitbar ist.
\end{definition}

In obiger Definition bedeutet dabei $t \longrightarrow 0$, dass in der rechten Seite von \eqref{iiii} statt dem $t$ ein Folgenglied $t_n$ einer beliebigen Folge $(t_n)_{n \in \mathbb{N}} \subseteq \mathbb{R}$ mit $t_n \xrightarrow{n \rightarrow \infty} 0$ und $t_n \neq 0$ für alle $n \in \mathbb{N}$ eingesetzt werden soll, und der Grenzübergang dann für $n \longrightarrow \infty$ gemacht wird. 

Man beachte weiter, dass der obige Grenzwert, falls er existiert, eindeutig bestimmt ist, da die obige implizite Folge bezüglich der von der Norm induzierten Topologie konvergiert und der normierte Raum $(Y, \| \cdot \|_Y)$ bezüglich dieser Topologie ein, wie wir aus den letzten Abschnitten wissen, topologischer Hausdorffraum ist.

Die Richtungsableitung von $f$ im Punkt $x \in X$ in Richtung $h \in X$ gibt also an, wie sehr sich die Funktion am Punkt $x$ in Richtung $h$ verändert. Für $X = Y = \mathbb{R}$ und $h = 1 \in \mathbb{R}$ erhalten wir den gewöhnlichen Differenzierbarkeitsbegriff für Funktionen der Form $f : \mathbb{R} \longrightarrow \mathbb{R}$. 

Nachdem wir nun geklärt haben, was wir unter der Richtungsableitung einer Funktion der Form $f : X \longrightarrow Y$ verstehen, können wir uns noch die Frage stellen, was wir unter der Richtungsableitung einer Funktion verstehen, die nur auf einer offenen Menge $\mathcal{W} \subseteq X$ definiert ist. Beachte, dass unsere derzeitige Definition der Richtungsableitung die Definition einer Funktion auf dem ganzen normierten Raum $X$ erfordert.

Sei also $f : \mathcal{W} \longrightarrow Y$ eine Funktion, die lediglich auf einer offenen Menge $\mathcal{W} \subseteq X$ definiert ist. Diese wollen wir genau dann als richtungsdifferenzierbar im Punkt $x \in \mathcal{W}$ bezeichnen, falls die Funktion 
\begin{align}
    \hat{f}(x) = \left\{\begin{array}{ll} f(x), & x \in \mathcal{W} \\
         0, & x \notin \mathcal{W} \end{array}\right.\label{diffbar}
\end{align}
in $x \in \mathcal{W}$ richtungsdifferenzierbar ist. D.h. wir setzen die Funktion $f : \mathcal{W} \longrightarrow Y$ einfach auf den ganzen normierten Raum $X$ fort, indem wir diese außerhalb von $\mathcal{W}$ auf Null setzen. Beachte, dass die Definition der Richtungsdifferenzierbarkeit tatsächlich so naiv auf Funktionen der Form $\mathcal{W} \longrightarrow Y$, $\mathcal{W} \subseteq X$ offen, erweiterbar ist, da $\mathcal{W}$ zum einen offen ist, und zum anderen der Grenzwert $t \longrightarrow 0$ in der Definition der Richtungsdifferenzierbarkeit verwendet wird.

Analog erklären wir die Richtungsdifferenzierbarkeit einer Funktion $f : \mathcal{W} \longrightarrow Y$ auf einer offenen Menge $\mathcal{U} \subseteq \mathcal{W}$: Sei $\mathcal{W} \subseteq X$ offen und eine Funktion $f : \mathcal{W} \longrightarrow Y$ gegeben. Dann nennen wir $f$ auf der offenen Menge $\mathcal{U} \subset \mathcal{W}$ richtungsdifferenzierbar, falls die Funktion 
\begin{align}
    \hat{f}(x) = \left\{\begin{array}{ll} f(x), & x \in \mathcal{W} \\
         0, & x \notin \mathcal{W} \end{array}\right.\label{dif}
\end{align}
auf $\mathcal{U}$ richtungsdifferenzierbar ist.

Als Nächstes wollen wir die Menge aller stetigen linearen Operatoren definieren. Wir merken dabei an, dass wir unter einer stetigen Funktionen $f: X \longrightarrow Y$, welche auf dem normierten Raum $(X, \| \cdot \|_X)$ erklärt ist und in den normierten Raum $(Y, \| \cdot \|_Y)$ abbildet, eine Funktion verstehen, die bezüglich der von den Normen $\| \cdot \|_X$ und $\| \cdot \|_Y$ induzierten Topologien stetig ist.

\begin{definition}
    Seien $(X, \| \cdot \|_X)$ und $(Y, \| \cdot \|_Y)$ zwei normierte Räume. Wir bezeichnen mit $\mathcal{L}(X, Y)$ die Menge aller stetiger linearen Abbildungen der Form 
    \begin{align}
        A : X \longrightarrow Y.
    \end{align}
    Es lässt sich dabei leicht nachprüfen, dass mittels der punktweisen Addition und Skalarmultiplikation die Menge $\mathcal{L}(X,Y)$ einen $\mathbb{R}$-Vektorraum bildet.
\end{definition}

Mittels $\mathcal{L}(X, Y)$ können wir nun für die Funktion $f : X \longrightarrow Y$ den Begriff der Gâteaux-differenzierbarkeit einführen:

\begin{definition}
    Seien $(X, \| \cdot \|_X)$ und $(Y, \| \cdot \|_Y)$ zwei normierte Räume und sei $f : X \longrightarrow Y$ eine Funktion, welche im Punkt $x \in X$ richtungsableitbar ist. Existiert ein $A \in \mathcal{L}(X, Y)$, sodass
    \begin{align}
        f'(x;h) = A(h) \:\:\:\:\:\:\:\:\: \forall h \in X
    \end{align}
    gilt, so nennen wir $f$ Gâteaux-differenzierbar im Punkt $x$ und wir bezeichnen $A =: f'(x)$ als die Gâteaux-Ableitung oder als das Gâteaux-Differential von $f$ im Punkt $x$. Wir nennen $f$ Gâteaux-differenzierbar auf der offenen Menge $\mathcal{U} \subseteq X$, falls $f$ Gâteaux-differenzierbar für alle $x \in \mathcal{U}$ ist.
\end{definition}

Beachte, dass aus obiger Definition folgt, dass für eine auf $\mathcal{U}$ Gâteaux-differenzierbaren Funktion die Ableitungfunktion $f'$, die wir kurz als die Ableitung von $f$ bezeichnen, eine Abbildung der Form
\begin{align}
    f' : \mathcal{U} \longrightarrow \mathcal{L}(X, Y)
\end{align}
ist. Beachte weiter, dass wir hier immer fordern, dass die Menge $\mathcal{U} \subseteq X$ eine bezüglich der von $\| \cdot \|_X$ induzierten Topologie offene Menge ist. Diese Forderung gewährleistet, dass für jedes $x \in \mathcal{U}$ immer ein $\delta > 0$ existiert, sodass $\mathcal{U}_\delta (x) = \{ y \in X \:\: | \:\: \|x - y\|_X < \delta \} \subseteq \mathcal{U}$ ist. Diese Beobachtungen werden insbesondere dann wichtig werden, wenn wir höhere Ableitungen einer Funktion $f : X \longrightarrow Y$ bilden möchten.

Schließlich merken wir abermals an, dass für eine auf einer offenen Menge $\mathcal{W} \subseteq X$ definierten Funktion $f : \mathcal{W} \longrightarrow Y$ die Gâteaux-differenzierbarkeit in einem Punkt $x \in \mathcal{W}$, sowie die Gâteaux-differenzierbarkeit auf einer offener Menge $\mathcal{U} \subseteq \mathcal{W}$ wie in \eqref{diffbar} und \eqref{dif} über triviale Fortsetzungen erklärt sind.

\begin{example}
    Sei $f : \mathbb{R} \longrightarrow \mathbb{R}$ eine am Punkt $x \in \mathbb{R}$ Gâteaux-differenzierbare Funktion. Es gilt für $h \in \mathbb{R}$
    \begin{align}
        f'(x; h) =& \: \lim_{t \longrightarrow 0} \frac{f(x + t \cdot h) - f(x)}{t} \nonumber \\ =& \:  \lim_{t \longrightarrow 0, \: \lambda = t \cdot h} \frac{f(x + \lambda) - f(x)}{t} \nonumber \\ =& \:  \bigg( \lim_{t \longrightarrow 0, \: \lambda = t \cdot h} \frac{f(x + \lambda) - f(x)}{\lambda} \bigg) \cdot h \nonumber \\ =& \: \bigg(\lim_{\lambda \longrightarrow 0} \frac{f(x + \lambda) - f(x)}{\lambda} \bigg) \cdot h \nonumber \\ =& \: f^{(1)}(x) \cdot h. \label{fffffff}
    \end{align}
    Da \eqref{fffffff} linear und stetig in $h \in \mathbb{R}$ ist, gilt, dass das Gâteaux-Differential von $f$ an der Stelle $x$ existiert und gegeben ist durch 
    \begin{align}
        f'(x; h) = f'(x)(h) := f^{(1)}(x) \cdot h.
    \end{align}
    D.h. das Gâteaux-Differential der Funktion $f$ an der Stelle $x$ kann mit dem aus der reellen Analysis bekannten gewöhnlichen Ableitungbegriff $f^{(1)}(x) := \lim_{\lambda \longrightarrow 0} \frac{f(x + \lambda) - f(x)}{\lambda}$ identifiziert werden. Damit ist eine derartige Funktion $f$ insbesondere eine Funktion, die anschaulich, wie aus der reellen Analysis bekannt, keine Knicke oder gar Unstetigkeitsstellen in ihrem Funktionsgraphen $\{ (x, f(x)) \:\: | \:\: x \in \mathbb{R} \} \subseteq \mathbb{R}^2$ aufweist. 

    Beachte außerdem, dass gilt, dass
    \begin{align}
        f'(x;1) = f'(x)(1) = f^{(1)}(x)
    \end{align}
    ist.
\end{example}

Als Nächstes erklären wir den Begriff der Fréchet-differenzierbarkeit:

\begin{definition}
    Seien $(X, \| \cdot \|_X)$ und $(Y, \| \cdot \|_Y)$ zwei normierte Räume und sei $f : X \longrightarrow Y$ eine Funktion. Die Funktion $f$ heißt im Punkt $x \in X$  Fréchet-differenzierbar, falls ein $A \in \mathcal{L}(X, Y)$ und eine Abbildung $r_x : X \longrightarrow Y$ existiert, sodass für alle $h \in X$ gilt, dass
    \begin{align}
        f(x + h) = f(x) + A(h) + r_x(h)  \label{mmmm}
    \end{align}
    und 
    \begin{align}
        \frac{\| r_x(h) \|_Y}{\|h\|_X} \longrightarrow 0 \:\:\:\:\:\:\:\:\: \textit{für} \:\:\: \|h\|_X \longrightarrow 0
    \end{align}
    ist. Wir nennen $A$ in diesem Fall Fréchet-Ableitung oder totales Differential von $f$ im Punkt $x \in X$. Weiter nennen wir $f$ total differenzierbar auf $\mathcal{U} \subseteq X$, falls $f$ in allen $x \in \mathcal{U}$ Fréchet-differenzierbar ist.
\end{definition}

Wieder bezeichnen wir eine nur auf der offenen Menge $\mathcal{W} \subseteq X$ erklärte Funktion $f : \mathcal{W} \longrightarrow Y$ als Fréchet-differenzierbar im Punkt $x \in \mathcal{W}$ oder auf der offenen Menge $\mathcal{U} \subseteq \mathcal{W}$, falls die zugehörige trivile Fortsetzung \eqref{diffbar} von $f$ in $x$ oder auf $\mathcal{U}$ Fréchet-differenzierbar ist.

Es ist nun leicht einzusehen, dass jede Fréchet-differenzierbare Funktion auch Gâteaux-differenzierbar ist und die Gâteaux-Ableitung mit der Fréchet-Ableitung übereinstimmt:

\begin{proposition}
    Seien $(X, \| \cdot \|_X)$ und $(Y, \| \cdot \|_Y)$ zwei normierte Räume und sei $f : X \longrightarrow Y$ eine Fréchet-differenzierbare Funktion. Dann ist $f$ auch Gâteaux-differenzierbar und die Fréchet-Ableitung stimmt mit dem Gâteaux-Differential überein.
\end{proposition}

\begin{proof}
    Da $f : X \longrightarrow Y$ Fréchet-differenzierbar ist, folgt, dass es ein $A \in \mathcal{L}(X,Y)$ und ein $r_x : X \longrightarrow Y$ gibt, sodass für beliebiges $h \in X$
    \begin{align}
        f(x+h) - f(x) - A(h) = r_x (h) 
    \end{align}
    mit $\frac{\| r_x(h) \|_Y}{\|h\|_X} \longrightarrow 0$ für $\|h\|_X \longrightarrow 0$ ist. Wir setzen $h = t \cdot v$, wobei $t \in \mathbb{R}$ und $v \in X$ ist. Es folgt mit der Linearität von $A$, dass
    \begin{align}
        \frac{f(x+ t \cdot v) - f(x)}{t} - A(v) = \frac{r_x (t \cdot v)}{t}
    \end{align}
    gilt. Dabei gilt $\frac{\| r_x(t \cdot v) \|_Y}{|t| \cdot \|v\|_X} \longrightarrow 0$ für $\|t \cdot v\|_X = |t| \cdot \| v \|_X \longrightarrow 0$. Das impliziert 
    \begin{align}
        \frac{\| r_x(t \cdot v) \|_Y}{|t|} = \bigg\| \frac{r_x(t \cdot v)}{t} \bigg\|_Y \longrightarrow 0
    \end{align}
    für $t \longrightarrow 0$. Mittels der Definition des Grenzwertes folgt damit aber 
    \begin{align}
        \lim_{t \longrightarrow 0} \frac{f(x + t \cdot v) - f(x)}{t} = A(v).
    \end{align}
    Da $A \in \mathcal{L}(X,Y)$ und $h \in X$ beliebig war, womit auch $v \in X$ beliebig ist, folgt, dass $f$ in $x$ Gâteaux-differenzierbar mit dem Gâteaux-Differential $A$ ist.
\end{proof}

Wegen dieser Beobachtung schreiben wir für die Fréchet-Ableitung auch einfach nur $f'(x)$. Weiter haben wir damit auch eine anschauliche Interpretation der Fréchet-Ableitung erarbeitet, denn gemäß \eqref{mmmm} ist $f'(x)$ zum einen eine lineare Approximation von $f$, und zum anderen ist in $f'(x)$ nach der eben gemachten Beobachtung auch codiert, wie sehr sich $f$ im Punkt $x \in X$ ändert. 

Wir nennen im folgenden eine Funktion $f : X \longrightarrow Y$ auf $\mathcal{U} \subseteq X$ stetig differenzierbar, falls $f$ für alle Punkte $x \in \mathcal{U}$ Fréchet-differenzierbar und 
\begin{align}
    f' : \mathcal{U} \longrightarrow \mathcal{L}(X, Y)
\end{align}
stetig ist. Um dabei über die Stetigkeit der Funktion $f'$ sprechen zu können, wird $\mathcal{U}$ mit der Teilraumtopologie der durch $\| \cdot \|_X$ induzierten Topologie auf $X$ ausgestattet, während $\mathcal{L}(X, Y)$ mit der durch die von der sogenannten Operatornorm $\| \cdot \|_\mathcal{O}$ induzierten Topologie versehen wird. Die Operatornorm $\| \cdot \|_\mathcal{O}$ ist dabei erklärt durch
\begin{align}
    \| A \|_\mathcal{O} :=& \: \sup_{x \in X, \:\: x \neq \mathbf{0}} \frac{ \| A(x) \|_Y }{\| x \|_X} \nonumber \\ :=& \: \sup \bigg\{ \frac{\|A(x) \|_Y }{\| x \|_X} \in \mathbb{R} \:\: \bigg| \:\: x \in X \:\: \textit{mit} \:\: x \neq \mathbf{0} \bigg\},
\end{align}
und diese macht den Vektorraum $\mathcal{L}(X, Y)$ selbst zu einem normierten Raum. Dabei verstehen wir unter dem Mengenoperator $\sup : \mathcal{P}(\mathbb{R}) \longrightarrow \mathbb{R}$, den wir Supremumsoperator nennen, einen Operator, der einer Teilmenge $\mathcal{A} \subseteq \mathbb{R}$ seine kleinste obere Schranke zuordnet, d.h. das kleinste $a \in \mathbb{R}$, für das $b \leq a$ für alle $b \in \mathcal{A}$ gilt. 

Weiter nennen wir nun die Funktion $f$ zweifach stetig differenzierbar auf $\mathcal{U}$, falls $f$ und die Funktion $f'$, betrachtet als Funktion von $\mathcal{U}$ in den normierten Raum $\mathcal{L}(X,Y)$, stetig differenzierbar auf $\mathcal{U}$ sind. Die Ableitung von $f'$ bezeichnen wir dabei mit $f''$. Induktiv definiert man sich so für $k \in \mathbb{N}$ den Begriff einer auf $\mathcal{U}$ $k$-fach stetig differenzierbaren Funktion $f$. Die Menge aller auf $\mathcal{U}$ $k$-fach stetig differenzierbaren Funktionen $f : X \longrightarrow Y$ bezeichnen wir mit $\mathcal{C}^k (\mathcal{U}, Y)$. Schließlich wollen wir unter einer auf $\mathcal{U}$ unendlich oft differenzierbaren Funktion eine Funktion $f$ verstehen, für die gilt, dass 
\begin{align}
    f \in \bigcap_{k \in \mathbb{N}} \mathcal{C}^k (\mathcal{U}, Y) =: \mathcal{C}^{\infty} (\mathcal{U},Y)
\end{align}
ist. Wir werden im Folgenden die Elemente von $\mathcal{C}^{\infty} (\mathcal{U},Y)$ auch kurz als glatte Funktionen bezeichnen.

Der so definierte Begriff einer auf $\mathcal{U}$ $k$-fach stetig differenzierbaren Funktion mag anfangs recht abstrakt wirken. Tatsächlich stimmt dieser Begriff aber im Falle von Funktionen der Form $f: \mathbb{R} \longrightarrow \mathbb{R}$ in gewissem Sinne mit dem aus der reellen Analysis bekannten und anschaulichen Definition einer auf $\mathcal{U}$ $k$-fach stetig differenzierbaren Funktion überein. In der reellen Analysis ist nämlich eine auf der offenen Menge $\mathcal{U} \subseteq \mathbb{R}$ $k$-fach stetig differenzierbare Funktion $f$ induktiv erklärt als eine $(k-1)$-fach stetig differenzierbare Funktion, sodass für die $(k-1)$te Ableitung $f^{(k-1)}$ der Funktion $f$ gilt, dass der Ausdruck
\begin{align}
    \lim_{t \longrightarrow 0} \frac{f^{(k-1)}(x + t) - f^{(k-1)}(x)}{t} =: f^k(x)
\end{align}
für alle $x \in \mathcal{U}$ existiert und die Funktion $f^k : \mathcal{U} \longrightarrow \mathbb{R}$, die jedem $x \in \mathcal{U}$ den Funktionswert $f^k(x)$ zuordnet, stetig ist. In welchem Sinne die funktionalanalytische Definition und die aus der reellen Analysis bekannte Definition einer auf $\mathcal{U}$ $k$-fach stetig differenzierbaren Funktion $f : \mathbb{R} \longrightarrow \mathbb{R}$ im Falle $k = 1$ übereinstimmen beantwortet der folgende Satz:

\begin{proposition}
    Sei $f : \mathbb{R} \longrightarrow \mathbb{R}$ eine stetige Funktion. Dann sind die folgenden Aussagen äquivalent:
    \begin{itemize}
        \item [\textit{i}] $f$ ist auf $\mathcal{U}$ im Sinne der Funktionalanalysis stetig differenzierbar, d.h. $f'$ existiert und ist stetig.
        \item [\textit{ii)}] $f$ ist auf $\mathcal{U}$ im Sinne der reellen Analysis stetig differenzierbar, d.h. $f^{(1)}$ existiert und ist stetig. 
    \end{itemize}
    Weiter lässt sich die Funktion $f'$ mit der Funktion $f^{(1)}$ identifizieren. 
\end{proposition}

\begin{proof}
    Wir gehen zuerst davon aus, dass die Funktionen $f'$ und $f^{(1)}$ existieren und zeigen, dass wir in diesem Fall beide miteinander identifizieren können. Dazu behaupten wir zuerst, dass wir die Menge $\mathcal{L}(\mathbb{R}, \mathbb{R})$ mit der Menge $\mathbb{R}$ identifizieren können. Damit meinen wir, dass jedes Element in $\mathcal{L}(\mathbb{R}, \mathbb{R})$ eineindeutig einem Element in $\mathbb{R}$ entspricht und umgekehrt. Wir suchen also eine Abbildung
    \begin{align}
        \mathcal{J} : \mathcal{L}(\mathbb{R}, \mathbb{R}) \longrightarrow \mathbb{R}
    \end{align}
    mit den Eigenschaften
    \begin{itemize}
        \item[\textit{a)}] $\mathcal{J}(A) = \mathcal{J}(B) \implies A = B$ für alle $A,B \in \mathcal{L}(\mathbb{R}, \mathbb{R})$,
        \item[\textit{b)}] $ \forall a \in \mathbb{R} \:\: \exists A \in \mathcal{L}(\mathbb{R}, \mathbb{R})$ mit  $\mathcal{J}(A) = a $,
    \end{itemize}
    da über diese Eigenschaften gerade die Forderung codiert wird, dass die Mengen $\mathcal{L}(\mathbb{R}, \mathbb{R})$ und $\mathbb{R}$ in einer $1:1$-Beziehung stehen. Wie könnte eine derartige Abbildung aussehen? Eine natürlich Wahl für $\mathcal{J}$ liefert die folgende Beobachtung: Wir wissen, dass eine Abbildung $A \in \mathcal{L}(\mathbb{R}, \mathbb{R})$ aufgrund ihrer Linearität eindeutig über ihre Wirkung auf das Element $1 \in \mathbb{R}$ erklärt ist. Es gilt $A(1) = a \cdot 1 = a$, wobei $a \in \mathbb{R}$ ist. Weiter gilt für beliebiges $\lambda \in \mathbb{R}$
    \begin{align}
        A(\lambda) = A(\lambda \cdot 1) = \lambda \cdot A(1) = a \cdot \lambda,
    \end{align}
    d.h. wir können $A$ mit $a$ identifizieren. Damit ist eine naheliegende und natürliche Wahl der Abbildung $\mathcal{J}$ gegeben durch 
    \begin{align}
        \mathcal{J} : \mathcal{L}(\mathbb{R}, \mathbb{R}) \longrightarrow& \: \: \mathbb{R} \nonumber\\
        A \longmapsto& \: \: \mathcal{J}(A) = a \:\: \textit{mit} \:\: A(1) = a.
    \end{align}
    Um zu sehen, dass die so definierte Abbildung die beiden obigen Eigenschaften erfüllt bemerken wir zuerst, dass jedes $a \in \mathbb{R}$ über 
    \begin{align}
        \lambda \longmapsto a \cdot \lambda =: A(\lambda) \:\:\:\:\:\: \forall \lambda
    \end{align}
    eine lineare und stetige Abbildung definiert, für die per Konstruktion $A(1) = a$ gilt. Die Stetigket sieht man explizit darüber, dass gilt, dass wir für alle $\epsilon > 0$ ein $\delta > 0$ mit $\delta < \frac{\epsilon}{|a|}$ finden können, sodass für alle $\lambda, \rho \in \mathbb{R}$ gilt, dass
    \begin{align}
        |A(\lambda) - A(\rho)| = |a \cdot \lambda - a \cdot \rho| = |a| \cdot |\lambda - \rho| < |a| \cdot \delta < \epsilon
    \end{align}
    ist. Damit können wir also zu jedem $a \in \mathbb{R}$ einen Operator $\mathcal{L}(\mathbb{R},\mathbb{R})$ finden, sodass $\mathcal{J}(A) = a$ gilt, was die obige Eigenschaft \textit{b)} zeigt. Weiter folgt unmittelbar aus der Definition der Abbildung $\mathcal{J}$, dass die obige Eigenschaft $\textit{a)}$ erfüllt ist, woraus folgt, dass $\mathcal{J}$ tatsächlich eine $1:1$-Abbildung darstellt. 
    
    Damit können wir nun also auf natürliche Weise die Mengen $\mathcal{L}(\mathbb{R}, \mathbb{R})$ und $\mathbb{R}$ miteinander identifizieren. Mit dieser Beobachtung, dem Satz $5.1.$ und dem Beispiel $5.1.$ folgt nun, dass wir die auf $\mathcal{U}$ definierte Funktion $f'$ mit der auf $\mathcal{U}$ Funktion $f^{(1)}$ identifizieren können und umgedreht, denn sei $f'(x)$ die Fréchet-Ableitung von $f$ am Punkt $x \in \mathcal{U}$ und $f^{(1)}(x)$ bezeichne die $1.$ Ableitung der Funktion $f$ am Punkt $x \in \mathcal{U}$ im Sinne der Definition aus der reellen Analysis, so wissen wir nach Satz $5.1.$, dass $f'(x)$ das Gâteaux-Differential von $f$ ist. Weiter wissen wir nach Beispiel $5.1.$, dass das Gâteaux-Differential von $f$ die Eigenschaft
    \begin{align}
        f'(x,h) = f'(x)(h) = f^{(1)}(x) \cdot h \:\:\:\: \forall h \in \mathbb{R}
    \end{align}
    erfüllt, d.h. es gilt 
    \begin{align}
        \mathcal{J}(f'(x)) = f^{(1)}(x).
    \end{align}
    Damit können wir $f' : \mathcal{U} \longrightarrow \mathcal{L}(\mathbb{R}, \mathbb{R})$ über $\mathcal{J}$ mit der Funktion $f^{(1)}: \mathcal{U} \longrightarrow \mathbb{R}$ identifizieren, indem wir $f'(x)$ mit $f^{(1)}(x)$ für alle $x \in \mathcal{U}$ ersetzen.
    
    Als nächstes stellen wir fest, dass gemäß dem Beispiel $5.1.$ die Existenz von $f'$ die Existenz von $f^{(1)}$ impliziert. Die Umkehrung sieht man folgendermaßen: Sei $x \in \mathcal{U}$ gegeben. Zuerst definiert man sich mit $f^{(1)}(x) \in \mathbb{R}$ den linearen und stetigen Operator $f'(x) \in \mathcal{L}(\mathbb{R}, \mathbb{R})$ durch $f'(x)(\lambda) := f^{(1)}(x) \cdot \lambda$. Gemäß Beispiel $2.7$ ist dieser das Gâteaux-Differential von $f$. Weiter gilt für beliebiges $\lambda \in \mathbb{R}$, dass per Definition von $f^{(1)}(x)$ 
    \begin{align}
        \frac{|f(x+\lambda) - f(x) - f'(x)(\lambda)|}{|\lambda|} =& \: \bigg| \frac{f(x + \lambda) - f(x) - f^{(1)}(x) \cdot \lambda}{\lambda} \bigg| \nonumber \\ =& \bigg| \frac{f(x + \lambda) - f(x)}{\lambda} - f^{(1)}(x) \bigg| \longrightarrow 0 
    \end{align}
    für $|\lambda| \longrightarrow 0$ ist. Das zeigt die Fréchet-differenzierbarkeit von $f$ im Punkt $x \in \mathcal{U}$.
    
    Es bleibt also zu zeigen, dass $f'$ genau dann stetig ist, wenn $f^{(1)}$ stetig ist. Um das einzusehen betrachten wir die Operatornorm $\| \cdot \|_\mathcal{O}$ auf $\mathcal{L}(\mathbb{R}, \mathbb{R})$. Es gilt für beliebiges $A \in \mathcal{L}(\mathbb{R}, \mathbb{R})$ mit $A(1) = a$, dass
    \begin{align}
        \|A\|_\mathcal{O} =& \: \sup \bigg\{ \frac{|A(\lambda)|}{|\lambda|} \:\: \bigg| \:\: \lambda \in \mathbb{R}, \:\: \lambda \neq 0 \bigg\} \nonumber \\ =& \: \sup \bigg\{ \frac{|a| \cdot |\lambda|}{|\lambda|} \:\: \bigg| \:\: \lambda \in \mathbb{R}, \:\: \lambda \neq 0 \bigg\} \nonumber \\ =& \: |a| = |\mathcal{J}(A)|,
    \end{align}
    d.h. $|\mathcal{J}(A)| = \|A\|_\mathcal{O}$. Weiter ist leicht einzusehen, dass aus der Definition von $\mathcal{J}$ folgt, das $\mathcal{J}$ eine lineare Abbildung ist. Daraus folgt nun aber 
    \begin{align}
        &(\forall \epsilon > 0 \:\: \exists \delta >0 : |\lambda - \rho| < \delta \implies |f^{(1)}(\lambda) - f^{(1)}(\rho)| < \epsilon) \nonumber \\ \iff& \: (\forall \epsilon > 0 \:\: \exists \delta >0 : |\lambda - \rho| < \delta \implies |\mathcal{J}(f'(\lambda)) - \mathcal{J}(f'(\rho))| < \epsilon) \nonumber \\ \iff& \: (\forall \epsilon > 0 \:\: \exists \delta >0 : |\lambda - \rho| < \delta \implies |\mathcal{J}(f'(\lambda) - f'(\rho))| < \epsilon) \nonumber \\ \iff& (\forall \epsilon > 0 \:\: \exists \delta >0 : |\lambda - \rho| < \delta \implies \| f'(\lambda) - f'(\rho) \|_\mathcal{O} < \epsilon),
    \end{align}
    was schließlich die Äquivalenz der beider Definitionen im Falle $k = 1$ zeigt. Man beachte, dass man dieses Resultat auch dazu nutzen kann, um die Verwendung der abstrakt anmutenden Operatornorm in diesem Kontext zu begründen. Im folgenden werden wir oft auf diese Identifikation zurückgreifen.
\end{proof}

Nachdem wir nun geklärt haben, was wir unter der Differenzierbarkeit einer Funktion $f : X \longrightarrow Y$ zu verstehen haben, wollen wir noch kurz den für uns später sehr nützlichen Begriff der partiellen Ableitung einer Funktion $f$ der Form $f : \mathbb{R}^m \longrightarrow \mathbb{R}^n$ einführen. Im Folgenden werden wir außerdem alle Funktionen der Bauart $\mathbb{R}^m \longrightarrow \mathbb{R}^n$ kurz als euklidische Funktionen bezeichnen.

\begin{definition}
    Sei $f : \mathbb{R}^m \longrightarrow \mathbb{R}^n$ eine richtungsdifferenzierbare Funktion und $x = (x^1, ..., x^m) \in \mathbb{R}^m$. Wir schreiben $f(x) =: f(x^1, ..., x^m)$. Dann ist für $i \in \{1, ..., m\}$ die $i$-te partielle Ableitung von $f$ an der Stelle $x$ erklärt durch
    \begin{align}
        \partial_i f (x) := f'(x, e_i), 
    \end{align}
    wobei die $e_i \in \mathbb{R}^m$, definiert durch $e_i = (\delta^1_i, ..., \delta^m_i)$ mit 
    \begin{align}
        \delta^j_i = \left\{\begin{array}{ll} 1, & i = j \\
         0, & i \neq j \end{array}\right. 
    \end{align}
    die sogenannten kanonischen Einheitsvektoren sind. Die Menge $\{e_1, ..., e_m\}$ heißt Standardbasis des $\mathbb{R}^m$ und bildet eines Basis des $\mathbb{R}^m$.
\end{definition}

Sei $x = (x^1, ..., x^m) \in \mathbb{R}^m$. Für differenzierbare Funktionen der Form $f : \mathbb{R}^m \longrightarrow \mathbb{R}^n$ mit 
\begin{align}
    f(x) = f(x^1, ..., x^m) = (f^1(x^1, ..., x^m), ..., f^n(x^1, ..., x^m)) = (f^1(x), ..., f^n(x)) \label{komp}
\end{align}
können wir die Funktion $f'(x)$ mittels der eben definierten partiellen Ableitungen weiter studieren. Dazu berechnen wir zuerst die sogenannten (basisabhängigen) Komponenten von $f'(x)$. Sei dazu $\mathcal{A} = \{v_1, ..., v_m\} \subseteq \mathbb{R}^m$ eine Basis von $\mathbb{R}^m$ und $\mathcal{B} = \{w_1, ..., w_n\} \subseteq \mathbb{R}^n$ eine Basis im $\mathbb{R}^n$. Es gilt für $i \in \{1,...,m\}$
\begin{align}
    f'(x)(v_i) = \sum_{j = 1}^n a^j_i(x) \cdot w_j = f'(x, v_i).
\end{align}
Wählen wir für $\mathcal{A}$ und $\mathcal{B}$ die Standardbasis des $\mathbb{R}^m$ bzw. die Standardbasis des $\mathbb{R}^n$, so gilt
\begin{align}
    f'(x)(v_i) = f'(x, v_i) = \partial_i f(x) = (a^1_i(x), ..., a^n_i(x)).
\end{align}
Mittels der Definition von $f'(x, v_i)$, der Konvergenz im $\mathbb{R}^n$ bezüglich der Standardtopologie (Bemerkung $4.1.$), und der Gleichung \eqref{komp} folgt, dass dann in diesem Fall
\begin{align}
    a^j_i(x) = \partial_i f^j(x)
\end{align}
ist. Für ein allgemeines $y = \sum_{i = 1}^m y^{i} \cdot v_i = (y^1, ..., y^m)$ gilt dann
\begin{align}
    f'(x)(y) =& \: \sum_{i = 1}^m y^{i} \cdot f'(x) (v_i) \nonumber \\ =& \: \sum_{j = 1}^n \bigg( \sum_{i = 1}^m \partial_i f^j(x) \cdot y^{i} \bigg) \cdot w_j \nonumber \\ =& \: \bigg( \sum_{i = 1}^m \partial_i f^1(x) \cdot y^{i}, ..., \sum_{i = 1}^m \partial_i f^n(x) \cdot y^{i}  \bigg) \nonumber \\ =& \: y \cdot (J_f (x))^T. \label{Jacobi}
\end{align}
Dabei ist in der letzten Gleichung $J_f (x) \in \mathbb{R}^{n \times m}$ eine reelle $(n \times m)$-Matrix mit den Komponenten $J_f(x)_i^j = \partial_i f^j (x)$ für alle $i \in \{ 1,...,m \}$ und $j \in \{ 1,...,n \}$, die Operation $\cdot$ entspricht der gewöhnlichen Matrixmultiplikation und $\cdot^T$ meint die Transpositionsoperation, die bei einer Matrix die Zeilen mit den Spalten vertauscht. 

Die Matrix $J_f(x)$ wird oft auch als Jacobimatrix von $f$ im Punkt $x$ bezeichnet und mit $f'(x)$ identifiziert. Ist $f$ eine am Punkt $x \in \mathbb{R}^m$ differenzierbare euklidische Funktion, so sind wir mittels der Jacobimatrix $J_f(x)$ offensichtlichweise in der Lage, den Ausdruck $f'(x)(y)$ für alle $y \in \mathbb{R}^m$ einfach zu berechnen.

Abschließend wollen wir uns noch mit einer wichtigen Eigenschaft des Differenzierbarkeitsbegriffes bezüglich komponierter euklidischer Funktionen auseinandersetzen. Dazu definieren wir uns zuerst noch den Begriff der allgemeinen Funktionenkomposition.

\begin{definition}
    Seien $X,Y,Z$ beliebige Mengen und $f : X \longrightarrow Y$ sowie $g: Z \longrightarrow X$ zwei Funktionen. Dann verstehen wir unter der Funktionenkomposition $f \circ g$ eine Funktion der Form $Z \longrightarrow Y$, definiert durch 
    \begin{align}
        (f \circ g)(t) := (f(g(t)) \:\:\:\: \forall t \in Z. 
    \end{align}
\end{definition}

Das folgende Theorem liefert nun eine wichtige Rechenregel für das totale Differential bezüglich komponierter euklidischer Funktionen. Der auf das Theorem folgende Beweis orientiert sich dabei in Teilen an \cite{forster2017analysis}.

\begin{theorem}
    Seien $f = (f^1, ..., f^n) : \mathbb{R}^m \longrightarrow \mathbb{R}^n$ und $g = (x^1, ..., x^m) : \mathbb{R}^l \longrightarrow \mathbb{R}^m$ zwei stetige Funktionen. Wir nehmen an, dass $g$ am Punkt $y = (y^1, ..., y^l) \in \mathbb{R}^l$ und $f$ am Punkt $g(y) \in \mathbb{R}^m$ stetig differenzierbar ist. Wir betrachten die Funktionenkomposition 
    \begin{align}
        z := f \circ g : \mathbb{R}^l \longrightarrow& \: \: \mathbb{R}^n \nonumber\\
        y \longmapsto& \: \: (f \circ g)(y) = f(g(y)) = f(x^1(y), ..., x^n(y)).
    \end{align}
    Dann ist $z$ am Punkt $y \in \mathbb{R}^l$ differenzierbar und es gilt für die Jacobimatrix von $z$ im Punkt $y$
    \begin{align}
        J_z(y) = J_f(g(y)) \cdot J_g(y). \label{multi}
    \end{align}
    Dabei ist $\cdot$ in \eqref{multi} abermals die Matrixmultiplikation.   
\end{theorem}

\begin{proof}
    Da $f : \mathbb{R}^m \longrightarrow \mathbb{R}^n$ an der Stelle $x =g(y) \in \mathbb{R}^m$ mit $y \in \mathbb{R}^l$ stetig differenzierbar ist, folgt, dass ein $f'(x) \in \mathcal{L}(\mathbb{R}^m,\mathbb{R}^n)$ und eine Abbildung $r_x : \mathbb{R}^m \longrightarrow \mathbb{R}^n$ existiert, sodass
    \begin{align}
        f(x+h) = f(x) + f'(x)(h) + r_x(h) \:\:\:\: \forall h \in \mathbb{R}^m
    \end{align}
    mit
    \begin{align}
        \frac{\|r_x(h)\|_{\mathbb{R}^n}}{\|h\|_{\mathbb{R}^m}} \longrightarrow 0 \:\:\:\:\:\:\:\: \textit{für} \:\: \|h\|_{\mathbb{R}^m} \longrightarrow 0
    \end{align}
    gilt. Analog folgt aus der stetigen Differenzierbarkeit von $g: \mathbb{R}^l \longrightarrow \mathbb{R}^m$ an $y \in \mathbb{R}^l$, dass ein $g'(y) \in \mathcal{L}(\mathbb{R}^l, \mathbb{R}^m)$ und ein $r_y : \mathbb{R}^l \longrightarrow \mathbb{R}^m$ existiert, sodass 
    \begin{align}
         g(y+k) = g(y) + g'(y)(k) + r_y(k) \:\:\:\: \forall k \in \mathbb{R}^l
    \end{align}
    mit 
    \begin{align}
        \frac{\|r_y(k)\|_{\mathbb{R}^m}}{\|k\|_{\mathbb{R}^l}} \longrightarrow 0 \:\:\:\:\:\:\:\: \textit{für} \:\: \|k\|_\mathbb{R}^l \longrightarrow 0
    \end{align}
    gilt. Wir zeigen die Differenzierbarkeit von $z = f \circ g$ am Punkt $y \in \mathbb{R}^l$. Dazu betrachten wir für beliebiges $k \in \mathbb{R}^l$
    \begin{align}
        z(y + k) =& \: f(g(y + k)) \nonumber \\ =& \: f(g(y) + u),   \:\:\:\:\:\:\:\: u := g(y+k) - g(y) \nonumber \\ =& \: f(x + u) \nonumber \\ =& \: f(x) + f'(x)(u) + r_x(u) \nonumber \\ =& \: (f \circ g)(y) + f'(g(y)) (g(y+k) - g(y)) + r_x(u) \nonumber \\ =& \: z(y) + f'(g(y)) (g(y+k)) - f'(g(y)) (g(y)) + r_x(u) \nonumber \\ =& \: z(y) + f'(g(y)) (g(y) + g'(y)(k) + r_y(k)) \nonumber \\ &- \: f'(g(y)) (g(y)) + r_x(u) \nonumber \\ =& \: z(y) + f'(g(x)) (g(y)) - f'(g(y)) (g(y)) + f'(g(y)) (g'(y)(k)) \nonumber \\ &+ \: f'(g(y)) (r_y(k)) + r_x(u) \nonumber \\ =& z(y) + (f'(g(y)) \circ g'(y)) (k) + \big( f'(g(y)) (r_y(k)) + r_x(u) \big).
    \end{align}
    Beachte, dass aus Lemma $3.1.$ und Bemerkung $4.2.$ folgt, dass $(f'(g(y)) \circ g'(y)) \in \mathcal{L}(\mathbb{R}^l, \mathbb{R}^n)$ ist. Es bleibt zu zeigen, dass für die Abbildung $\hat{r}_y : \mathbb{R}^l \longrightarrow \mathbb{R}^n$, erklärt durch
    \begin{align}
        \hat{r}_y(k) := f'(g(y))(r_y(k)) + r_{g(y)} (g(y+k) - g(y)) \:\:\:\: \forall k \in \mathbb{R}^l
    \end{align}
    gilt, dass 
    \begin{align}
        \frac{\| \hat{r}_y (k) \|_{\mathbb{R}^n}}{\|k\|_{\mathbb{R}^l}} \longrightarrow 0 \label{jjjjjjj}
    \end{align}
    für $\|k\|_{\mathbb{R}^l} \longrightarrow 0$ ist. Man beachte, dass aus der Definition der Operatornorm $\|\cdot\|_\mathcal{O}$ auf $\mathcal{L}(\mathbb{R}^m, \mathbb{R}^n)$ direkt folgt, dass für $A \in \mathcal{L}(\mathbb{R}^m, \mathbb{R}^n)$
    \begin{align}
        \|A(h)\|_{\mathbb{R}^n} \leq \|A\|_\mathcal{O} \cdot \| h \|_{\mathbb{R}^m} \:\:\:\: \forall h \in \mathbb{R}^m  
    \end{align}
    ist. Damit gilt für $k \neq \mathbf{0}$
    \begin{align}
        0 \leq \frac{\| \hat{r}(k) \|_{\mathbb{R}^n}}{\|k\|_{\mathbb{R}^l}} \leq \frac{\| f'(g(y)) \|_\mathcal{O} \cdot \| r_y(k) \|_{\mathbb{R}^m}}{\| k \|_{\mathbb{R}^l}} + \frac{\|r_x(u)\|_{\mathbb{R}^n}}{\|k\|_{\mathbb{R}^l}}. \label{ooooooo}
    \end{align}
    Der erste Term der rechten Seite konvergiert für $\|k\|_{\mathbb{R}^l} \longrightarrow 0$ gegen $0$. Können wir zeigen, dass der zweite Term für $\|k\|_{\mathbb{R}^l} \longrightarrow 0$ auch gegen $0$ konvergiert, so konvergiert die gesamte rechte Seite gegen $0$. Wir bemerken, dass $u = g(y+k) - g(y) = g'(y)(k) + r_y(k)$ ist. Weiter folgt aus 
    \begin{align}
        \frac{\|r_y(k)\|_{\mathbb{R}^m}}{\|k\|_{\mathbb{R}^l}} \longrightarrow 0 \:\:\:\: \textit{für} \:\: \|k\|_{\mathbb{R}^l} \longrightarrow 0,
    \end{align}
    dass ein $\delta > 0$ existiert, sodass 
    \begin{align}
        \|r_y(k)\|_{\mathbb{R}^m} \leq \|k\|_{\mathbb{R}^l} \:\:\:\: \forall k \in \mathbb{R}^l \:\:\:\: \textit{mit} \:\: \|k\|_{\mathbb{R}^l} < \delta
    \end{align}
    gilt. Aus der Definition von $r_x$ folgt nun, dass 
    \begin{align}
        \|r_x(u)\|_{\mathbb{R}^n} = \|u\|_{\mathbb{R}^m} \cdot \|\psi (u)\|_{\mathbb{R}^n}
    \end{align}
    mit $\psi : \mathbb{R}^m \longrightarrow \mathbb{R}^n$ und $\|\psi(u)\|_{\mathbb{R}^m} \longrightarrow 0$ für $\|u\|_{\mathbb{R}^m} \longrightarrow 0$. Daraus folgt nun aber für $\|k\|_{\mathbb{R}^l} < \delta$ 
    \begin{align}
        \|r_x(u)\|_{\mathbb{R}^n} =& \: \|r_x(g'(y)(k + r_y(k)))\|_{\mathbb{R}^n} \nonumber \\ =& \: \| g'(y)(k) + r_y(k) \|_{\mathbb{R}^m} \cdot \| \psi ( g'(y)(k) + r_y(k) ) \|_{\mathbb{R}^n} \nonumber \\ \leq& \: \bigg( \|g'(y) \|_\mathcal{O} \cdot \|k\|_{\mathcal{R}^l} + \|r_y(k)\|_{\mathbb{R}^m} \bigg) \cdot \|\psi (g(y + k) - g(y))\|_{\mathbb{R}^n} \nonumber \\ \leq& \: \|k\|_{\mathbb{R}^l} \bigg( \|g'(y)\|_\mathcal{O} + 1 \bigg) \cdot \| \psi (g(y+k) - g(y)) \|_{\mathbb{R}^n}.
    \end{align}
    Beachte, dass aus der Stetigkeit von $g$ und der Definition der Norm-Konvergenz sofort folgt, dass für $\|k\|_{\mathbb{R}^l} \longrightarrow 0$ auch $\|g(y+k) - g(y)\|_{\mathbb{R}^m} \longrightarrow 0$ ist. Damit ist also $\eta(k) := ( \|g'(y)\|_\mathcal{O} + 1 ) \cdot \| \psi (g(y+k) - g(y)) \|_{\mathbb{R}^n}$ eine positive Funktion und es gilt, dass $\eta(k) \longrightarrow 0$ für $\|k\|_{\mathbb{R}^l} \longrightarrow 0$ gilt. Somit folgt, dass auch der letzte Term in \eqref{ooooooo} für $\|k\|_{\mathbb{R}^l} \longrightarrow 0$ gegen $0$ konvergiert, was letztlich zeigt, dass $z$ in $y$ differenzierbar ist.
    
    Es bleibt zu zeigen, dass die bezüglich der kanonischen Standardbasen darstellende Matrix $J_z(y)$ von $z'(y)$ tatsächlich die gewünschte Form $J_f(g(y)) \cdot J_g(y)$ besitzt. Seien dazu jeweils $\{e_1, ..., e_l\} \subseteq \mathbb{R}^l$, $\{v_1, ..., v_m\} \subseteq \mathbb{R}^m$ und $\{w_1, ..., w_n\} \subseteq \mathbb{R}^n$ die kanonischen Standardbasen im $\mathbb{R}^l$, $\mathbb{R}^m$ und im $\mathbb{R}^n$. Die Aussage folgt aus der obigen Berechnung $z'(y) = f'(g(y)) \circ g'(y)$ und \eqref{Jacobi}, denn sei $h = \sum_{i = 1}^l h^{i} e_i \in \mathbb{R}^l$ beliebig, so gilt 
    
    \begin{align}
        z'(y)(h) =& \: \sum_{i = 1}^l h^{i} \cdot z'(y) (e_i) \nonumber \\ =& \: \sum_{i = 1}^l h^{i} \cdot f'(g(y))(g'(y) (e_i)) \nonumber \\ =& \: \sum_{i = 1}^l h^{i} \cdot f'(g(y))\bigg( \sum_{j = 1}^m \partial_i g^j(y) \cdot v_j \bigg) \nonumber \\ =& \: \sum_{k = 1}^m \sum_{i = 1}^l h^{i} \cdot \partial_i g^j(y) \cdot f'(g(y))(v_j) \nonumber \\ =& \sum_{k = 1}^m \sum_{i = 1}^l h^{i} \cdot \partial_i g^j(y) \cdot \bigg( \sum_{k = 1}^n \partial_j f^k(g(y)) \cdot w_k \bigg) \nonumber \\ =& \: \sum_{k=1}^n \bigg( \sum_{i = 1}^l \bigg( \sum_{j = 1}^m \partial_i g^j(y) \cdot \partial_j f^k(g(y)) \bigg) \cdot h^{i} \bigg) \cdot w_k \nonumber \\ =& \: \bigg(\sum_{i = 1}^l \bigg( \sum_{j = 1}^m \partial_i g^j(y) \cdot \partial_j f^1(g(y)) \bigg) \cdot h^{i}, ..., \sum_{i = 1}^l \bigg( \sum_{j = 1}^m \partial_i g^j(y) \cdot \partial_j f^n(g(y)) \bigg) \cdot h^{i} \bigg) \nonumber \\ =& \: h \cdot (J_z(y))^T. \label{multidimensional}
    \end{align}
    Daraus folgt nun aber, dass 
    \begin{align}
        J_z(y)_i^k =& \: \bigg( \sum_{j = 1}^m \partial_i g^j(y) \cdot \partial_j f^k(g(y)) \bigg) \nonumber \\ =& \: \bigg( \sum_{j = 1}^m J_g(y)^j_i \cdot J_f(g(y))_j^k \bigg),
    \end{align}
    woraus, per Definition der Matrixmultiplikation die Behauptung folgt.
\end{proof}

\begin{remark}
    Beachte, dass im ersten Teil des eben diskutierten Theorems die Räume $\mathbb{R}^m$, $\mathbb{R}^n$ und $\mathbb{R}^l$ auch durch allgemeinere normierte Räume $(X, \|\cdot\|_X)$, $(Y, \|\cdot \|_Y)$ und $(Z, \| \cdot \|_Z)$ ersetzt werden können. Der Beweis für diese allgemeinere Formulierung ist analog zum ersten Teil des Beweises von Theorem $5.1.$. 
\end{remark}

Das vorangegangene Theorem liefert den folgenden wichtigen Spezialfall für $l = n = 1$:

\begin{proposition}
    Seien $f : \mathbb{R}^m \longrightarrow \mathbb{R}$ und $g = (x^1, ..., x^m) : \mathbb{R} \longrightarrow \mathbb{R}^m$ zwei stetige Funktionen. Wir nehmen an, dass $g$ am Punkt $t \in \mathbb{R}$ und $f$ am Punkt $g(t) \in \mathbb{R}^m$ stetig differenzierbar sind. Wir betrachten die Funktionenkomposition
    \begin{align}
        z := f \circ g : \mathbb{R} \longrightarrow& \: \: \mathbb{R} \nonumber\\
        t \longmapsto& \: \: (f \circ g)(t) = f(g(t)) = f(x^1(t), ..., x^m(t)).
    \end{align} 
    Dann ist $z$ am Punkt $t \in \mathbb{R}$ differenzierbar und es gilt
    \begin{align}
        z'(t) = \sum_{i = 1}^m \partial_i f (x^1(t), ..., x^n(t)) \cdot (x^{i})' (t). \label{Spezialfall Kettenregel}
    \end{align}
    Beachte, dass für alle $i \in \{1, ..., m\}$ die Funktionen $x^{i}$ Funktionen der Form $\mathbb{R} \longrightarrow \mathbb{R}$ und gemäß Satz $5.2.$ $z'$ und $z^{(1)}$, sowie $(x^{(i)})'$ und $(x^{i})^{(1)}$ Synonyme für uns sind. Beachte weiter, dass die Differenzierbarkeit von $g$ gemäß Bemerkung $4.1.$ die Differenzierbarkeit der $x^{i}$ für alle $i \in \{1,...,m\}$ impliziert. 
\end{proposition}

\begin{proof}
    Folgt direkt aus dem vorangegangenem Theorem (siehe unter anderem \eqref{multidimensional}).
\end{proof}

\begin{corollary}
    Zusammen mit der Bemerkung $5.1.$ lässt sich mittels der multidimensionalen Kettenregel aus Theorem $5.1.$ zeigen, dass für $k, l \in \mathbb{N} \cup \{ \infty \}$ die Verkettung zweier Funktionen vom Typ $\mathcal{C}^k$ und $\mathcal{C}^l$ wieder eine Funktion vom Typ $\mathcal{C}^{\textit{min}\{k,l\}}$ liefert, oder etwas formaler ausgedrückt: Seinen $X,Y,Z$ normierte Räume und seien $f : X \longrightarrow Y$ und $g : Z \longrightarrow X$ zwei Funktionen. Zusätzlich gelte $f \in \mathcal{C}^k(X,Y)$ und $g \in \mathcal{C}^l(Z,X)$. Dann folgt, dass $(f \circ g) \in \mathcal{C}^{\textit{min}\{k,l\}}(Z,Y)$ ist. 
\end{corollary}

\begin{proof}
    Nutze Theorem $5.1,$, Bemerkung $5.1.$ und \cite{cartan1967calcul}.
\end{proof}


\newpage

\quad
\newpage

\section{Endlichdimensionale Mannigfaltigkeiten und Tensorbündel}

Nach dem wir in den oberen Kapiteln geklärt haben, was wir unter topologische Räume und Vektorräume zu verstehen haben, können wir uns nun endlich dem für uns zentralen Begriff der Mannigfaltigkeit nähern. 


\subsection{Topologische Mannigfaltigkeiten}

Um die Definition einer (topologischen) Mannigfaltigkeit möglichst kompakt zu halten, werden wir zunächst noch den Begriff des Homöomorphismus definieren.

\begin{definition}
    Seien $(X, \tau)$ und $(Y, \sigma)$ zwei topologische Räume. Sei weiter $f : X \longrightarrow Y$ eine Abbildung. Wir nennen $f$ einen Homöomorphismus, falls gilt: 
    \begin{itemize}
        \item[\textit{i)}] $f$ ist injektiv, d.h. für $x,y \in X$ mit $f(x) = f(y)$ folgt, dass $x = y$ ist.
        \item[\textit{ii)}] $f$ ist surjektiv, d.h. für alle $y \in Y$ existiert ein $x \in X$, sodass $f(x) = y$ ist.
        \item[\textit{iii)}] $f$ und die zu $f$ gehörige Umkehrabbildung $f^{-1}$ ist stetig bezüglich der Topologien $\tau$ und $\sigma$.
    \end{itemize}
    Existiert zwischen zwei topologischen Räumen ein Homöomorphismus, so heißen diese Räume homöomorph zueinander.
\end{definition}

Die ersten zwei Bedingungen an einen Homöomorphismus bedeuten wieder, dass $f$ eine $1:1$-Abbildung ist, d.h. jedes Element $x \in X$ wird über $f$ eineindeutig einem Element $y$ zugeordnet und umgedreht. Die Bedingung \textit{iii)} sichert dabei, dass diese Zuordnung stetig oder kontinuierlich verläuft. 

Wegen Satz $3.3.$ können wir dabei die Stetigkeitsforderung von $f$ und $f^{-1}$ als Forderung interpretieren, nach welcher Punkte in $X$, die \textit{nah} beieinander liegen, auch in $Y$ \textit{nah} beinander liegen müssen und umgedreht. Anzumerken ist dabei noch, dass die Bedingungen \textit{i)} und \textit{ii)} die Existenz der Umkehrabbildung $f^{-1}$ garantieren.

Ein Homöomorphismus $f : X \longrightarrow Y$ stellt also anschaulich gesprochen eine stetige Deformation des topologischen Raumes $X$ hin zum topologischen Raum $Y$ dar und zwei topologische Räume $X$ und $Y$ sind salopp gesprochen genau dann homöomorph, wenn man diese über stetige Defomationen ineinander überführen kann. 

Mittels des Begriffes des Homöomorphismus können wir nun den Begriff der Mannigfaltigkeit erklären.

\begin{definition}
    Sei $(\mathcal{M}, \tau)$ ein topologischer Hausdorffraum. Sei weiter der $\mathbb{R}$-Vektorraum $\mathbb{R}^d$ mit $d \in \mathbb{N}$ gegeben, ausgestattet mit der durch das Standardskalarprodukt $\langle \cdot , \cdot \rangle_{\mathbb{R}^d}$ induzierten Topologie $\tau_{\mathbb{R}^d}$. Wir bezeichnen $(\mathcal{M}, \tau)$ als $d$-dimensionale topologische Mannigfaltigkeit, falls es zu jedem $p \in \mathcal{M}$ ein $\mathcal{U} \in \tau$ mit $p \in \mathcal{U}$ gibt, sodass $\mathcal{U}$, ausgestattet mit der durch $\tau$ induzierten Teilraumtopologie, homöomorph zu einer offenen Menge $\mathcal{V} \subseteq \mathbb{R}^d$ ist, wobei $\mathcal{V}$ analog mit der durch $\tau_{\mathbb{R}^d}$ induzierten Teilraumtopologie versehen wurde. D.h. zu jedem $p \in \mathcal{M}$ existiert ein Homöomorphismus der Form
    \begin{align}
        \phi = (x^1, ..., x^d) : \mathcal{U} \longrightarrow \mathcal{V} \subseteq \mathbb{R}^d.
    \end{align}
    Man sagt auch, dass eine $d$-dimensionale Mannigfaltigkeit lokal homöomorph zum mit der Standardgeometrie versehenen $\mathbb{R}^d$ ist. Das Tupel $(\mathcal{U}, \phi)$ heißt Karte von $\mathcal{M}$, die den Punkt $p \in \mathcal{M}$ enthält. Wir bezeichnen weiter für $q \in \mathcal{U}$ und $i \in \{1,...,d\}$ die Zahl $x^{i}(q)$ als die $i$-te Koordinate von $q$ bezüglich der Karte $\mathcal{U}$.
\end{definition}

Eine $d$-dimensionale topologische Mannigfaltigkeit ist also ein spezieller topologischer Raum, der lokal so \textit{'aussieht'} wie der mit der Standardgeometrie versehene euklidische Raum $\mathbb{R}^d$. Die in der Definition geforderte Hausdorffeigenschaft ist dabei eine technische Forderung, die für Mannigfaltigkeiten oft gestellt wird, damit diese ordentlich klassifiziert werden können.

Weiter können wir gemäß der obigen Definition eine Karte $(\mathcal{U}, \phi)$ einer $d$-dimensionalen topologischen Mannigfaltigkeit in gewisser Weise als eine Art \textit{Landkarte} der offenen Menge $\mathcal{U} \subseteq \mathcal{M}$ im $\mathbb{R}^d$ auffassen. Unter dieser Interpretation führen wir den anschaulichen Begriff eines Atlanten einer topologischen Mannigfaltigkeit ein:

\begin{definition}
    Sei $(\mathcal{M}, \tau)$ eine $d$-dimensionale topologische Mannigfaltigkeit und sei $\mathcal{I}$ eine Indexmenge. Wir bezeichnen die Menge 
    \begin{align}
        \mathcal{A} := \{ (\mathcal{U}_i, \phi_i) \:\: | \:\: \forall i \in \mathcal{I} : \mathcal{U}_i \in \tau, \:\: \phi_i : \mathcal{U}_i \longrightarrow \phi(\mathcal{U}_i) \subseteq \mathbb{R}^d \: \textit{Homöomorphismus}\}
    \end{align}
    als einen Atlas der Mannigfaltigkeit $\mathcal{M}$, falls 
    \begin{align}
        \mathcal{M} = \bigcup_{i \in \mathcal{I}} \mathcal{U}_i
    \end{align}
    gilt.
\end{definition}

Unter einem Atlas einer $d$-dimensionalen topologischen Mannigfaltigkeit verstehen wir anschaulich also eine Sammlung an Karten, die zusammen die ganze Mannigfaltigkeit überdecken.

Betrachten wir ein einfaches Beispiel einer topologischen Mannigfaltigkeit:

\begin{example}
    Ein anschauliches und auch einfaches Beispiel einer $2$-dimensionalen Mannigfaltigkeit ist gegeben durch die Kugeloberfläche der Einheitskugel im $\mathbb{R}^3$. Diese ist modelliert durch den topologischen Raum $(\mathbb{S}^2, \tau_{\mathbb{R}^3, \mathbb{S}^2})$. Dabei ist die Menge $\mathbb{S}^2$ gegeben durch 
    \begin{align}
        \mathbb{S}^2 := \{ p = (p^1, p^2, p^3) \in \mathbb{R}^3 \:\: | \:\: (p^1)^2 + (p^2)^2 + (p^3)^2 = 1 \}
    \end{align}
    und die Topologie $\tau_{\mathbb{R}^3, \mathbb{S}^2}$ bezeichne die Teilraumtopologie bezüglich der Standardtopologie $\tau_{\mathbb{R}^3}$ im $\mathbb{R}^3$, d.h. 
    \begin{align}
        \tau_{\mathbb{R}^3, \mathbb{S}^2} := \{ \mathcal{U} \cap \mathbb{S}^2 \subseteq \mathbb{S}^2 \:\: | \:\: \mathcal{U} \in \tau_{\mathbb{R}^3} \}.
    \end{align}
    
    Wir zeigen nun zuerst, dass es sich beim topologischen Raum $(\mathbb{S}^2, \tau_{\mathbb{R}^3, \mathbb{S}^2})$ um einen Hausdorffraum handelt. Seien dazu $p$ und $q$ zwei Punkte aus $\mathbb{S}^2$ mit $p \neq q$. Fassen wir $\mathbb{S}^2$ als Teilmenge des $\mathbb{R}^3$ auf, so können wir die Punkte $p$ und $q$ jeweils als Elemente im $\mathbb{R}^3$ auffassen. 
    
    Wir wissen, dass der $\mathbb{R}^3$, ausgestattet mit der Standardtopologie $\tau_{\mathbb{R}^3}$, ein Hausdorffraum ist. Das wissen wir deshalb, da es sich bei der Standardtopologie $\tau_{\mathbb{R}^3}$ um diejenige Topologie handelt, die von der Metrik
    \begin{align}
        d(p,q) := \|p - q\|_{\mathbb{R}^3}
    \end{align}
    induziert wird. Gemäß dem Satz $3.1.$ folgt daraus unmittelbar, dass es sich beim Tupel $(\mathbb{R}^3, \tau_{\mathbb{R}^3})$ um einen Hausdorffraum handeln muss.

    Diese Beobachtung können wir uns nun zunutze machen: Da für $p$ und $q$, nun aufgefasst als Elemente im $\mathbb{R}^3$, gilt, dass $p \neq q$ ist, folgt, dass wir offene Menge $\mathcal{U}$ und $\mathcal{V}$ aus $\tau_{\mathbb{R}^3}$ finden können, sodass $p \in \mathcal{U}$, $q \in \mathcal{V}$ und $\mathcal{U} \cap \mathcal{V} = \emptyset$ ist. 
    
    Per Definition der Topologie $\tau_{\mathbb{R}^3, \mathbb{S}^2}$ gilt nun aber, dass $\mathcal{U} \cap \mathbb{S}^2$ und $\mathcal{V} \cap \mathbb{S}^2$ in $\tau_{\mathbb{R}^3, \mathbb{S}^2}$ liegen. Es gilt weiter nach der Wahl von $p$ und $q$, dass $p \in \mathcal{U} \cap \mathbb{S}^2$, $q \in \mathcal{V} \cap \mathbb{S}^2$ und $(\mathcal{U} \cap \mathbb{S}^2) \cap (\mathcal{V} \cap \mathbb{S}^2) = (\mathcal{U} \cap \mathcal{V}) \cap \mathbb{S}^2 = \emptyset$, da $\mathcal{U} \cap \mathcal{V} = \emptyset$ ist.

    Damit haben wir zwei offene Mengen in $\tau_{\mathbb{R}^3, \mathbb{S}^2}$ gefunden, mittels der sich die beiden Punkte $p$ und $q$ trennen lassen und es folgt, dass $(\mathbb{S}^2, \tau_{\mathbb{R}^3, \mathbb{S}^2})$ ein Hausdorffraum ist.

    Als Nächsten müssen wir einen Atlas $\mathcal{A}_{\mathbb{S}^2}$ für $(\mathbb{S}^2, \tau_{\mathbb{R}^3, \mathbb{S}^2})$ konstruieren: Dazu stellen wir zuerst die folgende Beobachtung an: 
    
    Da $(\mathbb{S}^2, \tau_{\mathbb{R}^3, \mathbb{S}^2})$ ein Hausdorffraum ist, folgt, dass $\mathbb{S}^2 \setminus \{ p \}$ mit $p \in \mathbb{S}^2$ eine offene Menge bzgl. $\tau_{\mathbb{R}^3, \mathbb{S}^2}$ ist, denn wegen der Hausdorffeigenschaft existiert für jedes $q \in \mathbb{S}^2$ mit $q \neq p$ eine Menge $\mathcal{W}_q \in \tau_{\mathbb{R}^3, \mathbb{S}^2}$ mit $q \in \mathcal{W}_q$ und $p \notin \mathcal{W}_q$. Es gilt
    \begin{align}
        \mathbb{S}^2 \setminus \{ p \} \subseteq \bigcup_{q \in \mathbb{S}^2 \setminus \{ p \}} \mathcal{W}_q
    \end{align}
    und wegen $\mathcal{W}_q \subseteq \mathbb{S}^2$ für alle $q \in \mathbb{S}^2 \setminus \{ p \}$ auch
    \begin{align}
        \bigcup_{q \in \mathbb{S}^2 \setminus \{ p \}} \mathcal{W}_q \subseteq \mathbb{S}^2 \setminus \{ p \}, 
    \end{align}
    womit insgesamt die Gleichheit beider Mengen folgt. Aus der Definition einer Topologie ergibt sich demnach, dass es sich bei der Menge $\mathbb{S}^2 \setminus \{ p \}$ um eine offene Menge bzgl. $\tau_{\mathbb{R}^3, \mathbb{S}^2}$ handeln muss.

    Mittels dieses Wissens können wir nun leicht zwei offene Mengen finden, welche zum einen die Menge $\mathcal{S}^2$ überdecken sollen und zum anderen auch gleichzeitig unsere Kartengebiete darstellen sollen. Sei also 
    \begin{align}
        \mathcal{W} := \mathbb{S}^2 \setminus \{ (0,0,1) \} \in \tau_{\mathbb{R}^3, \mathbb{S}^2}
    \end{align}
    und 
    \begin{align}
        \mathcal{Z} := \mathcal{S}^2 \setminus \{ (0,0,-1) \} \in \tau_{\mathbb{R}^3, \mathbb{S}^2}.
    \end{align}
    Trivialerweise überdecken $\mathcal{W}$ und $\mathcal{Z}$ zusammen ganz $\mathbb{S}^2$. Wir definieren nun auf $\mathcal{W}$ die Kartenabbildung
    \begin{align}
         \phi : \mathcal{W} \longrightarrow& \: \: \mathbb{R}^2 \nonumber\\
            (p^1,p^2,p^3) \longmapsto& \: \: \bigg( \frac{p^1}{1 - p^3}, \frac{p^2}{1 - p^3} \bigg)
    \end{align}
    und auf $\mathcal{Z}$ die Kartenabbildung
    \begin{align}
        \psi : \mathcal{Z} \longrightarrow& \: \: \mathbb{R}^2 \nonumber\\
            (p^1,p^2,p^3) \longmapsto& \: \: \bigg( \frac{p^1}{1 + p^3}, \frac{p^2}{1 + p^3} \bigg).
    \end{align}
    Bei den Abbildungen $\phi$ und $\psi$ handelt es sich jeweils um eine sogenannte stereografische Projektion der Einheitssphäre auf den $\mathbb{R}^2$. In der untenstehenden Grafik ist dabei visualisiert, was die stereografische Projektion geometrisch bedeutet.

    
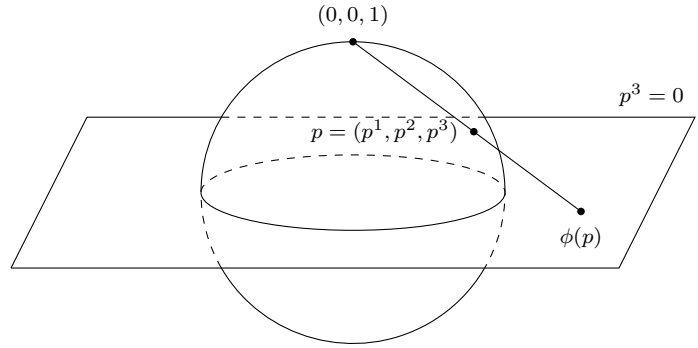
\begin{center}
    \begin{tikzpicture}
            \coordinate (A) at (3,-0.25);
            \coordinate (P) at (0,2);

            \draw (0:2cm)   arc[radius=2cm,start angle=0,end angle=180]
                  (210:2cm) arc[radius=2cm,start angle=210,end angle=330];
            \draw (180:2cm) arc[x radius=2cm, y radius=0.5cm, start angle=180,end angle=360];

            \draw [dashed] (210:2cm) 
                  arc[start angle=210,delta angle=-30,radius=2cm]
                  arc[start angle=180,delta angle=-180,x radius=2cm,y radius=0.5cm]
                  arc[start angle=0,delta angle=-30,radius=2cm];

            \draw [dashed] (150:2cm) coordinate(ul) -- (30:2cm) coordinate(ur);

            \draw (-4.5,-1) -- (3.5,-1) -- (4.5,1) node[anchor=south east] {\scriptsize$ p^3=0 $} -- (ur) (ul) -- (-3.5,1) -- (-4.5,-1);

            \draw (A) -- (P) coordinate[pos=0.47](B);
            \path (A) node[circle, fill, inner sep=1pt, label=below:{\scriptsize$ \phi(p) $}]{};
            \path (B) node[circle, fill, inner sep=1pt, label=left:{\scriptsize$ p = (p^1,p^2,p^3) $}]{};
            \path (P) node[circle, fill, inner sep=1pt, label=above:{\scriptsize$ (0,0,1) $}]{};
            
    \end{tikzpicture}
    \end{center}
    
    Es lässt sich nun zeigen, dass es sich bei den beiden stereografischen Projektionen $\phi$ und $\psi$ um Homöomorphismen handelt \cite{SP}. 
    
\end{example}

Der Grund für die Nützlichkeit des Mannigfaltigkeitenbegriffes liegt darin, dass wir, aufgrund der obigen Ähnlichkeitsbedingung zum $\mathbb{R}^d$, Konzepte aus der Analysis, die für gewöhnlich im $\mathbb{R}^d$ entwickelt werden, auf die Mannigfaltigkeiten, die einen allgemeineren Rahmen darstellen, übertragen können. Wie diese Übertragbarkeit im Detail aussieht, werden wir in den kommenden Unterabchnitt sehen. 

Die Möglichkeit dieser Übertragbarkeit, insbesondere der Übertragbarkeit der aus der Analysis bekannten Differenzierbarkeitsbegriffe auf spezielle Mannigfaltigkeiten, wird später noch eine wichtige Rolle spielen, wenn es darum geht den Begriff der Geodätischen einzuführen. Um nämlich auf einer Mannigfaltigkeit einen Längenbegriff einer Kurven einzuführen, benötigen wir, wie wir noch sehen werden, einen Differenzierbarkeitsbegriff.


\subsection{Differenzierbare Mannigfaltigkeiten}

Wir wollen nun die Differenzierbarkeitskonzepte aus dem $\mathbb{R}^d$ auf eine topologische Mannigfaltigkeit $(\mathcal{M}, \tau)$ übertragen. Dazu betrachten wir zuerst den Begriff der sogenannten Kartenübergangsabbildung. 

\begin{definition}
    Sei $(\mathcal{M}, \tau)$ eine $d$-dimensionale topologische Mannigfaltigkeit. Weiter seien $(\mathcal{U}, \phi)$ und $(\mathcal{V}, \psi)$ zwei Karten, d.h. $\mathcal{U}, \mathcal{V} \in \tau$ und $\phi : \mathcal{V} \longrightarrow \phi(\mathcal{U})$, $\psi : \mathcal{V} \longrightarrow \psi(\mathcal{V})$ sind Homöomorphismen. Wir nehmen weiter an, dass $\mathcal{U} \cap \mathcal{V} \neq \emptyset$ ist. Dann ist die Kartenübergangsabbidlung von der Karte $(\mathcal{U}, \phi)$ zur Karte $(\mathcal{V}, \psi)$ definiert durch die Abbildung
    \begin{align}
        \psi \circ \phi^{-1} : \phi(\mathcal{U} \cap \mathcal{V}) \longrightarrow \psi(\mathcal{U} \cap \mathcal{V}).
    \end{align}
    Man beachte, dass die Abbildung $\psi \circ \phi^{-1}$ eine stetige Abbildung (siehe Lemma $3.1$) von einer offenen Teilmenge des $\mathbb{R}^d$ in eine offene Teilmenge des $\mathbb{R}^d$ darstellt. Damit können wir die Kartenübergangsabbildung $\psi \circ \phi^{-1}$ insbesondere auf Differenzierbarkeit untersuchen. Sollte die Abbildung $\psi \circ \phi^{-1}$ für $k \in \mathbb{N} \cup \{ \infty \}$ eine auf $\phi(\mathcal{U}\cap\mathcal{V})$ $k$-mal differenzierbare Funktion im Sinne des vorangegangenen Kapitels sein, so nennen wir $\psi \circ \phi^{-1}$ einen differenzierbaren Kartenwechsel vom Grad $k$. Im Falle $\psi \circ \phi^{-1} \in \mathcal{C}^\infty (\phi(\mathcal{U} \cap \mathcal{V}), \psi(\mathcal{U} \cap \mathcal{V}))$ bezeichnen wir $\psi \circ \phi^{-1}$ auch kurz als glatten Kartenwechsel.
\end{definition}

Was ist nun eine Kartenwechselabbildung anschaulich. Um eine Intuition für diese Abbildung zu gewinnen betrachten wir die folgende Grafik, die das Zusammenspiel zweier sich überlappender Kartengebiete zeigt:

    
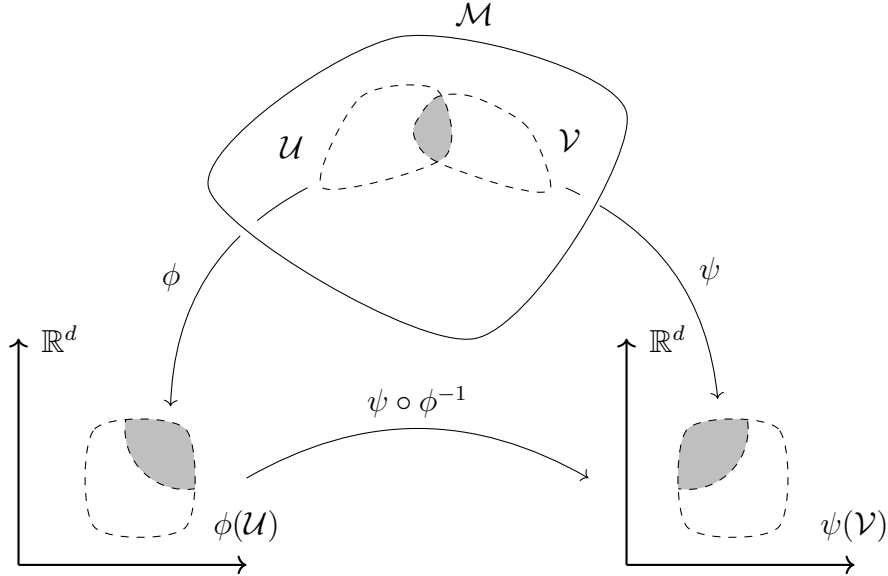
\begin{tikzpicture}
    
    \path[->] (0.8, 0) edge [bend right] node[left, xshift=-2mm] {$\phi$} (-1, -2.9);
    \draw[white,fill=white] (0.06,-0.57) circle (.15cm);
    
    \path[->] (4.2, 0) edge [bend left] node[right, xshift=2mm] {$\psi$} (6.2, -2.8);
    \draw[white, fill=white] (4.54,-0.12) circle (.15cm);
    
    \draw[smooth cycle] plot coordinates{(2,2) (-0.5,0) (3,-2) (5,1)} node at (3,2.3) {$\mathcal{M}$};
    
    \begin{scope}
        \clip[smooth cycle] plot coordinates {(1,0) (1.5, 1.2) (2.5,1.3) (2.6, 0.4)};
        \fill[gray!50, smooth cycle] plot coordinates {(4, 0) (3.7, 0.8) (3.0, 1.2) (2.5, 1.2) (2.2, 0.8) (2.3, 0.5) (2.6, 0.3) (3.5, 0.0)};
    \end{scope}
    \draw[dashed, smooth cycle] plot coordinates {(1,0) (1.5, 1.2) (2.5,1.3) (2.6, 0.4)}
    node [label={[label distance=-0.3cm, xshift=-2cm, fill=white]:$\mathcal{U}$}] {};
    \draw[dashed, smooth cycle] plot coordinates {(4, 0) (3.7, 0.8) (3.0, 1.2) (2.5, 1.2) (2.2, 0.8) (2.3, 0.5) (2.6, 0.3) (3.5, 0.0)} node [label={[label distance=-0.8cm, xshift=.75cm, yshift=1cm, fill=white]:$\mathcal{V}$}] {};
        
    \draw[thick, ->] (-3,-5) -- (0, -5) node [label=above:$\phi(\mathcal{U})$] {};
    \draw[thick, ->] (-3,-5) -- (-3, -2) node [label=right:$\mathbb{R}^d$] {};
    
    \path[->] (0, -3.85) edge [bend left]  node[midway, above]{$\psi \circ \phi^{-1} $} (4.5, -3.85);
    
    \draw[thick, ->] (5, -5) -- (8, -5) node [label=above:$\psi(\mathcal{V})$] {};
    \draw[thick, ->] (5, -5) -- (5, -2) node [label=right:$\mathbb{R}^d$] {};
    
\begin{scope}
    \clip [smooth cycle] plot coordinates{(-2, -4.5) (-2, -3.2) (-0.8, -3.2) (-0.8, -4.5)};
    \draw[fill=gray!50, dashed] (-0.8, -3.2) circle (0.8);
\end{scope}
\draw [smooth cycle, dashed] plot coordinates{(-2, -4.5) (-2, -3.2) (-0.8, -3.2) (-0.8, -4.5)};
    
\begin{scope}
    \clip [smooth cycle] plot coordinates{(7, -4.5) (7, -3.2) (5.8, -3.2) (5.8, -4.5)};
    \draw[fill=gray!50, dashed] (5.8, -3.2) circle (0.8);
\end{scope}
\draw[smooth cycle,dashed] plot coordinates{(7, -4.5) (7, -3.2) (5.8, -3.2) (5.8, -4.5)};
\end{tikzpicture}

Interpretieren wir $\phi(\mathcal{U})$ und $\psi(\mathcal{V})$ als \textit{Landkarten} der sich überlappenden offenen Mengen $\mathcal{U} \subseteq \mathcal{M}$ und $\mathcal{V} \subseteq \mathcal{M}$, so können wir mit der Kartenwechselabbildung $\psi \circ \phi^{-1}$ die Landkarte bezüglich der offenen Mengen $\mathcal{U} \cap \mathcal{V}$ von $\phi(\mathcal{U} \cap \mathcal{V})$ zu $\psi(\mathcal{U} \cap \mathcal{V})$ wechseln. Durch diesen Kartenwechsel wechseln wir auch die Koordinaten, die wir benutzen, um z.B. einen Punkt $q \in \mathcal{U} \cap \mathcal{V}$ in einer Karte zu beschreiben. Ist beispielweise $\phi = (x^1, ..., x^d)$ und $\psi = (y^1, ..., y^d)$, so wechseln wir beim Kartenwechsel $\psi \circ \phi^{-1}$ von den Koordinaten $x^{i}$ zu den Koordinaten $y^j$. 

Der Begriff der differenzierbaren Kartenübergangsabbildung motiviert die folgende nützliche Definition:

\begin{definition}
    Sei $(\mathcal{M}, \tau)$ eine $d$-dimensionale topologische Mannigfaltigkeit und seien $(\mathcal{U}, \phi)$ und $(\mathcal{V}, \psi)$ zwei Karten auf $\mathcal{M}$. Für $k \in \mathbb{N} \cup \{ \infty \}$ nennen wir diese beiden Karten $\mathcal{C}^k$-kompatibel, falls ihre Kartenübergangsabbildungen 
    \begin{align}
        \psi \circ \phi^{-1} : \phi(\mathcal{U} \cap \mathcal{V}) \longrightarrow \psi(\mathcal{U} \cap \mathcal{U}) \subseteq \mathbb{R}^d
    \end{align}
    und 
    \begin{align}
        \phi \circ \psi^{-1} : \psi(\mathcal{U} \cap \mathcal{V}) \longrightarrow \phi (\mathcal{U \cap \mathcal{V}}) \subseteq \mathbb{R}^d
    \end{align}
    $k$-mal stetig differenzierbar sind, d.h. $(\psi \circ \phi^{-1}) \in \mathcal{C}^k(\phi(\mathcal{U \cap \mathcal{V}}), \psi(\mathcal{U} \cap \mathcal{V}))$ und $(\phi \circ \psi^{-1}) \in \mathcal{C}^k(\psi(\mathcal{U} \cap \mathcal{V}), \phi(\mathcal{U} \cap \mathcal{V}))$. Wir legen weiter fest, dass falls $\mathcal{U} \cap \mathcal{V} = \emptyset$ ist, wir sagen, dass $(\mathcal{U} , \phi)$ und $(\mathcal{V}, \psi)$ automatisch $\mathcal{C}^{\infty}$-kompatibel sind, da in diesem Fall keine Kartenübergangsabbildungen existieren, die die Differenzierbarkeitsbedingungen verletzen könnten.
\end{definition}

Mittels der $\mathcal{C}^k$-Kompatibilität lässt sich nun der Begriff des differenzierbaren Atlas vom Grad $k$ einführen:

\begin{definition}
    Sei $(\mathcal{M}, \tau)$ eine $d$-dimensionale topologische Mannigfaltigkeit mit einem Atlas $\mathcal{A}$. Sei weiter  $k \in \mathbb{N} \cup \{ \infty \}$ gegeben. Wir bezeichnen $\mathcal{A}$ als differenzierbaren Atlas vom Grad $k$ oder kurz als $k$-fach differenzierbaren Atlas, falls alle Karten in $\mathcal{A}$ paarweise $\mathcal{C}^k$-kompatibel sind. Weiter bezeichnen wir in diesem Fall das Tripel $(\mathcal{M}, \tau, \mathcal{A})$ als $d$-dimensionale $k$-fach differenzierbare Mannigfaltigkeit. Für den Fall $k = \infty$ bezeichnen wir $\mathcal{A}$ auch als glatten Atlas und $(\mathcal{M}, \tau, \mathcal{A})$ nennen wir dann kurz $d$-dimensionale glatte Mannigfaltigkeit.  
\end{definition}

Betrachten wir ein Beispiel einer differenzierbaren Mannigfaltigkeit:

\begin{example}
    Wir betrachten die aus Beispiel $6.1.$ bekannte topologische Mannigfaltigkeit $(\mathbb{S}^2, \tau_{\mathbb{R}^3, \mathbb{S}^2}, \mathcal{A}_{\mathbb{S}^2})$. Bei dieser handelt es sich bereits um eine glatte Mannigfaltigkeit. Um das einzusehen, betrachten wir die aus \cite{SP} bekannte Umkehrabbildungen der stereographischen Projektionen $\phi$ und $\psi$, definiert durch
    \begin{align}
        \phi^{-1}((x^1, x^2)) = \bigg( \frac{2 \cdot x^1}{(x^1)^2 + (x^2)^2 + 1} , \frac{2 \cdot x^2}{(x^1)^2 + (x^2)^2 + 1} , \frac{(x^1)^2 + (x^2)^2 - 1}{(x^1)^2 + (x^2)^2 + 1} \bigg)
    \end{align}
    und 
    \begin{align}
        \psi^{-1}((y^1, y^2)) = \bigg( \frac{2 \cdot y^1}{(y^1)^2 + (y^2)^2 + 1} , \frac{2 \cdot y^2}{(y^1)^2 + (y^2)^2 + 1} , \frac{1 - (y^1)^2 - (y^2)^2}{(y^1)^2 + (y^2)^2 + 1} \bigg).
    \end{align}
    Wir müssen nun zeigen, dass es sich bei den Kartenübergangsabbildungen
    \begin{align}
        \psi \circ \phi^{-1} : \phi(\mathcal{W} \cap \mathcal{Z}) \longrightarrow \psi(\mathcal{W} \cap \mathcal{Z})
    \end{align}
    und 
    \begin{align}
        \phi \circ \psi^{-1} : \psi(\mathcal{W} \cap \mathcal{Z}) \longrightarrow \phi(\mathcal{W} \cap \mathcal{Z})
    \end{align}
    um glatte Abbildungen handelt. Zuerst bemerken wir, dass aus der in Beispiel $6.1.$ angegebenen Definition von $\mathcal{W}$ und $\mathcal{Z}$ folgt, dass 
    \begin{align}
        \mathcal{W} \cap \mathcal{Z} = \mathbb{S}^2 \setminus \{ (0,0,1), (0,0,-1) \}
    \end{align}
    ist. Daraus folgt wegen der Definition der Kartenabbildungen $\phi$ und $\psi$, dass
    \begin{align}
        \phi (\mathcal{W} \cap \mathcal{Z}) = \psi (\mathcal{W} \cap \mathcal{Z}) = \mathbb{R}^2 \setminus \{ (0, 0) \}
    \end{align}
    ist. Wir berechnen nun die beiden obigen Kartenübergangsabbildungen. Für $(\psi \circ \phi^{-1})$ gilt
    \begin{align}
        (\psi \circ \phi^{-1})& ((x^1, x^2)) = \psi(\phi^{-1}((x^1,x^2))) \nonumber \\ =& \: \psi \bigg( \bigg( \frac{2 \cdot x^1}{(x^1)^2 + (x^2)^2 + 1} , \frac{2 \cdot x^2}{(x^1)^2 + (x^2)^2 + 1} , \frac{(x^1)^2 + (x^2)^2 - 1}{(x^1)^2 + (x^2)^2 + 1} \bigg) \bigg) \nonumber \\ =& \: \bigg( \frac{x^1}{(x^1)^2 + (x^2)^2} , \frac{x^2}{(x^1)^2 + (x^2)^2} \bigg)
    \end{align}
    und für $\phi \circ \psi^{-1}$ gilt
    \begin{align}
        (\phi \circ \psi^{-1})& ((y^1, y^2)) = \phi(\psi^{-1}((y^1, y^2))) \nonumber \\ =& \: \phi \bigg( \bigg( \frac{2 \cdot y^1}{(y^1)^2 + (y^2)^2 + 1} , \frac{2 \cdot y^2}{(y^1)^2 + (y^2)^2 + 1} , \frac{1 - (y^1)^2 - (y^2)^2}{(y^1)^2 + (y^2)^2 + 1} \bigg) \bigg) \nonumber \\ =& \: \bigg( \frac{y^1}{(y^1)^2 + (y^2)^2} , \frac{y^2}{(y^1)^2 + (y^2)^2} \bigg).
    \end{align}
    Mittels der aus der multivariaten Analysis bekannten Differentitationsregeln \cite{heuser1992lehrbuch} folgt, dass es sich auf dem Gebiet $\mathbb{R}^2 \setminus \{ (0, 0) \}$ bei den beiden Kartenübergangsabbildungen tatsächlich um glatte Abbildungen handelt. Damit ist gezeigt, dass \\ $(\mathbb{S}^2, \tau_{\mathbb{R}^3, \mathbb{S}^2}, \mathcal{A}_{\mathbb{S}^2})$ eine glatte Mannigfaltigkeit ist. 
\end{example}

Für die nun eben definierten differenzierbaren Mannigfaltigkeiten lassen sich nun Konzepte aus der Analysis bzw. der Funktionalanalysis einführen. 

Sei also $k \in \mathbb{N} \cup \{ \infty \}$ und $(\mathcal{M}, \tau, \mathcal{A})$ eine $d$-dimensionale $k$-fach differenzierbare Mannigfaltigkeit. Wir wollen nun zuerst erklären, wann eine stetige Funktion der Form 
\begin{align}
    f : \mathcal{M} \longrightarrow \mathbb{R}
\end{align}
als differenzierbar bezeichnet werden kann. Man beachte, dass es nicht vollkommen offensichtlich ist, wann wir eine stetige Funktion $f : \mathcal{M} \longrightarrow \mathbb{R}$ differenzierbar nennen können, da wir für die bisher betrachteten differenzierbaren Funktionen immer eine Vektorraumstruktur im Definitionsbereich, wie auch im Wertebereich benötigt hatten, damit wir überhaupt einen Differenzenquotienten hinschreiben konnten. 

Für die nun betrachtete Funktion $f$ können wir aber keinen naiven Differenzenquotienten erklären, da $\mathcal{M}$ im Allgemeinen keine Vektorraumstruktur trägt. D.h. Ausdrücke der Form $\frac{f(p + t \cdot q) - f(x)}{t}$ machen erst einmal keinen Sinn, da nicht klar ist, was $p + t \cdot q$ für $p, q \in \mathcal{M}$ und $t \in \mathbb{R}$ sein soll, da im Allgemeinen weder eine Addition $+$, noch eine Skalarmultiplikation $\cdot$ auf $\mathcal{M}$ erklärt ist. 

Wie können wir also die Differenzierbarkeit einer stetigen Funktion $f : \mathcal{M} \longrightarrow \mathbb{R}$ erklären? Die Idee ist, dass wir die Differenzierbarkeit einer solchen Funktion $f$ über die Karten aus dem $k$-fach differenzierbaren Atlas $\mathcal{A}$ erklären. Sei dazu $(\mathcal{U}, \phi) \in \mathcal{A}$ eine beliebige Karte aus dem $k$-fach differenzierbaren Atlas $\mathcal{A}$. Wir betrachten die Funktion 
\begin{align}
    f \circ \phi^{-1} : \phi(\mathcal{U}) \longrightarrow \mathbb{R}. 
\end{align}
Da $\phi(\mathcal{U}) \subset \mathbb{R}^d$ eine offene Menge aus dem $\mathbb{R}^d$ ist (beachte das $\phi$ eine Kartenabbildung ist), folgt, dass wir die Funktion $f \circ \phi^{-1}$ im Sinne von \eqref{diffbar} als euklidische Funktion auffassen können und wir diese damit ganz normal auf Differenzierbarkeit untersuchen können. 

Der Gedanke ist nun, dass wir $f : \mathcal{M} \longrightarrow \mathbb{R}$ als $l$-mal stetig differenzierbar mit $l \in \{1,...,k\}$ bezeichnen wollen, wenn $f \circ \phi^{-1} : \phi(\mathcal{U}) \longrightarrow \mathbb{R}$ $l$-mal stetig differenzierbar ist. Die Frage ist nun aber, ob dieser so eingeführte Differenzierbarkeitsbegriff für die Funktion $f : \mathcal{M} \longrightarrow \mathbb{R}$ überhaupt wohldefiniert ist. D.h. wir müssen uns die Frage stellen, ob dieser Differenzierbarkeitsbegriff kartenabhängig ist bzw. von der Wahl der Karte $(\mathcal{U}, \phi) \in \mathcal{A}$ abhängt.

Sei dazu $(\mathcal{V}, \psi) \in \mathcal{A}$ eine weitere Karte mit $\mathcal{U} \cap \mathcal{V} \neq \emptyset$. Bezüglich der Karte $(\mathcal{U}, \phi)$ ist $f$ per Annahme $l$-fach stetig differenzierbar auf der offenen Menge $\mathcal{U} \cap \mathcal{V}$. Wir müssen zeigen, dass auch bezüglich der Karte $(\mathcal{V}, \psi)$ die Funktion $f$ auf $\mathcal{U} \cap \mathcal{V}$ $l$-fach stetig differenzierbar ist. Wir betrachten das folgende Diagramm:

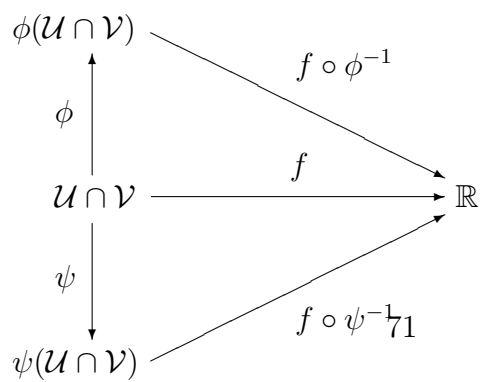
\begin{center}
    \setlength{\unitlength}{2pt}
    \begin{picture}(150,50)
        \put(10,0){$\mathcal{U} \cap \mathcal{V}$} 
        \put(85,0){$\mathbb{R}$}
        \put(2,32){$\phi(\mathcal{U} \cap \mathcal{V})$}
        \put(2,-31){$\psi(\mathcal{U} \cap \mathcal{V})$}
        \put(28,2){\vector(1,0){55}} \put(54,6){$f$}
        \put(28,33){\vector(2,-1){55}} \put(55,25){$f \circ \phi^{-1}$}
        \put(28,-29){\vector(2,1){55}} \put(55,-23){$f \circ \psi^{-1}$}
        \put(17,6){\vector(0,2){23}} \put(10,16){$\phi$}
        \put(17,-3){\vector(0,-2){22}} \put(10,-15){$\psi$}
    \end{picture}
\end{center}
.\\[2cm]
Gemäß dem oben stehenden Diagramm lässt sich $f \circ \psi^{-1}$ auch schreiben als 
\begin{align}
    f \circ \psi^{-1} = (f \circ \phi^{-1}) \circ (\phi \circ \psi^{-1}).
\end{align}
Wegen $(\mathcal{U}, \phi)$, $(\mathcal{V}, \psi) \in \mathcal{A}$ und $\mathcal{A}$ $k$-fach differenzierbarer Atlas, folgt, dass  $(\phi \circ \psi^{-1}) \in \mathcal{C}^k(\psi(\mathcal{U} \cap \mathcal{V}), \phi(\mathcal{U} \cap \mathcal{V}))$ ist. Da weiter $f \circ \phi^{-1}$ $l$-fach stetig differenzierbar mit $l \leq k$ ist, folgt, dass $f \circ \psi^{-1}$ ebenfalls $l$-fach stetig differenzierbar ist.

Man beachte, dass der $k$-fach stetig differenzierbare Atlas der Mannigfaltigkeit sicherstellt, dass der Begriff der $l$-fach stetig differenzierbaren Funktion für $l \leq k$ wohldefiniert ist. Wir fassen unsere Erkenntnisse in der fogenden Definition zusammen:

\begin{definition}
    Sei $(\mathcal{M}, \tau, \mathcal{A})$ eine $d$-dimensionale $k$-fach differenzierbare Mannigfaltigkeit mit $k \in \mathbb{N} \cup \{ \infty \}$ und $f : \mathcal{M} \longrightarrow \mathbb{R}$ eine stetige Funktion. Wir nennen $f$ $l$-fach stetig differenzierbar auf einer offenen Menge $\mathcal{W} \subseteq \mathcal{M}$ mit $l \leq k$, wenn die Kartendarstellungen von $f$ bezüglich $\mathcal{W}$ $l$-fach stetig differenzierbar sind, d.h. sei $(\mathcal{U}, \phi) \in \mathcal{A}$ eine beliebige Karte mit $\mathcal{U} \cap \mathcal{W} \neq \emptyset$, dann ist $f \circ \phi^{-1} : \phi(\mathcal{U} \cap \mathcal{W}) \longrightarrow \mathbb{R}$ eine $k$-fach stetig differenzierbare Abbildung. Wir bezeichnen die Menge der auf $\mathcal{W}$ $l$-fach stetig differenzierbaren Funktionen $f : \mathcal{W} \longrightarrow \mathbb{R}$ mit $\mathcal{C}^l(\mathcal{W})$. Gilt $l = k = \infty$, so nennen wir eine Funktion $f \in \mathcal{C}^\infty (\mathcal{W})$ auch kurz glatte (reellwertige) Funktion auf $\mathcal{W}$. Für $\mathcal{W} = \mathcal{M}$ nennen wir eine solche Funktion auch einfach nur als glatt.  
\end{definition}

Nahezu analog lässt sich nun dieser Begriff auf Funktion $f : \mathcal{M} \longrightarrow \mathcal{N}$ zwischen $k$-fach differenzierbarer Mannigfaltigkeiten $\mathcal{M}$ und $\mathcal{N}$ verallgemeinern. 

\begin{definition}
    Sei $(\mathcal{M}, \tau, \mathcal{A})$ eine $d$-dimensionale $k$-fach stetig differenzierbare Mannigfaltigkeit und sei $(\mathcal{N}, \sigma, \mathcal{B})$ eine $d'$-dimensionale $k$-fach stetig differenzierbare Mannigfaltigkeit. Sei weiter $f : \mathcal{M} \longrightarrow \mathcal{N}$ eine stetige Funktion. Dann heißt $f$ $l$-fach stetig differenzierbar auf $\mathcal{M}$ mit $l \leq k$, falls die Kartendarstellungen der Form
    \begin{align}
        \psi \circ f \circ \phi^{-1} : \phi(\mathcal{U} \cap f^{-1}(\mathcal{V})) \longrightarrow \psi(\mathcal{V})
    \end{align}
    $l$-fach stetig differenzierbar sind, wobei $(\mathcal{U}, \phi) \in \mathcal{A}$, $(\mathcal{V} , \psi) \in \mathcal{B}$ und $f^{-1}(\mathcal{V}) \cap \mathcal{U} \neq \emptyset$ ist. Beachte, dass wegen der Stetigkeit von $f$ die Menge $f^{-1}(\mathcal{V})$ offen ist und damit auch $\mathcal{U} \cap f^{-1}(\mathcal{V})$. Da $\phi$ ein Homöomorphismus ist folgt auch die Offenheit der Menge $\phi(\mathcal{U} \cap f^{-1}(\mathcal{V}))$. Die Menge der $l$-fach stetig differenzierbaren Funktionen $f : \mathcal{M} \longrightarrow \mathcal{N}$ bezeichnen wir als $\mathcal{C}^l(\mathcal{M}, \mathcal{N})$. Im Falle $l = k = \infty$ nennen wir $f \in \mathcal{C}^\infty (\mathcal{M}, \mathcal{N})$ auch glatte Funktion von $\mathcal{M}$ nach $\mathcal{N}$.
\end{definition}

Wie oben lässt sich mittels des folgenden Diagramms leicht einsehen, dass diese Definition der Differenzierbarkeit nicht von der Wahl der Karten in $\mathcal{M}$ und $\mathcal{N}$ abhängig ist:

\begin{center}
    \setlength{\unitlength}{2pt}
    \begin{picture}(150,50)
        \put(-4,0){$\mathcal{U} \cap \mathcal{U}' \cap \mathcal{W}$} 
        \put(85,0){$\mathcal{V} \cap \mathcal{V}'$}
        \put(-12,32){$\phi(\mathcal{U} \cap \mathcal{U}' \cap \mathcal{W})$}
        \put(-13,-31){$\phi'(\mathcal{U} \cap \mathcal{U}' \cap \mathcal{W})$}
        \put(28,2){\vector(1,0){55}} \put(54,6){$f$}
        \put(28,33){\vector(1,0){55}} \put(45,37){$\psi \circ f \circ \phi^{-1}$} \put(85,32){$\psi(\mathcal{V} \cap \mathcal{V}')$}
        \put(28,-29){\vector(1,0){55}} \put(45,-25){$\psi' \circ f \circ \phi'^{-1}$} \put(85,-31){$\psi'(\mathcal{V} \cap \mathcal{V}')$}
        \put(9,6){\vector(0,2){23}} \put(3,16){$\phi$}
        \put(9,-3){\vector(0,-2){22}} \put(2,-15){$\phi'$}
        \put(93,6){\vector(0,2){23}} \put(95,16){$\psi$}
        \put(93,-3){\vector(0,-2){22}} \put(95, -15){$\psi'$}
    \end{picture}
\end{center}
.\\[2cm]
In dem obigem Diagramm ist dabei $(\mathcal{U}, \phi), (\mathcal{U}', \phi') \in \mathcal{A}$, $(\mathcal{V}, \psi), (\mathcal{V}', \psi') \in \mathcal{B}$ mit $\mathcal{U} \cap \mathcal{U}' \cap \mathcal{W} \neq \emptyset$, $\mathcal{V} \cap \mathcal{V}' \neq \emptyset$, wobei $\mathcal{W} = f^{-1}(\mathcal{V} \cap \mathcal{V}')$ ist. Wie sich sofort erkennen lässt, lässt sich die Funktion $\psi' \circ f \circ \phi'^{-1}$ auch schreiben als 
\begin{align}
    \psi' \circ f \circ \phi'^{-1} = (\psi' \circ \psi^{-1}) \circ (\psi \circ f \circ \phi^{-1}) \circ (\phi \circ \phi'^{-1}).
\end{align}
Da sowohl $\psi' \circ \psi^{-1}$, als auch $\phi \circ \phi'^{-1}$ von der Klasse $\mathcal{C}^k$ sind und wir annehmen, dass $\psi \circ f \circ \phi^{-1}$ von der Klasse $\mathcal{C}^l$ ist, folgt nach Folgerung $5.1.$, dass auch $\psi' \circ f \circ \phi'^{-1}$ auf $\phi'(\mathcal{U} \cap \mathcal{U}' \cap \mathcal{W})$ von der Klasse $\mathcal{C}^l$ ist. Daraus folgt, dass auch der eben definierte Begriff der stetigen Differenzierbarkeit von Funktionen zwischen differenzierbaren Mannigfaltigkeiten kartenunabhängig ist.

\begin{remark}
    Wir betrachten noch eine spezielle Klasse stetig differenzierbare Funktionen, die in der Einleitung dieser Arbeit Erwähnung fand. 
    \begin{itemize}
   
        \item[\textit{i)}] Seien $(\mathcal{M}, \tau, \mathcal{A})$ und $(\mathcal{N}, \sigma, \mathcal{B})$ zwei $k$-fach stetig differenzierbare Mannigfaltigkeiten endlicher Dimension mit $k \in \mathbb{N} \cup {\infty}$ und $f : \mathcal{M} \longrightarrow \mathcal{N}$ eine $l$-fach stetig differenzierbare Funktion auf $\mathcal{M}$ mit $l \leq k$. Ist $f$ darüberhinaus auch ein Homöomorphismus, so bezeichnen wir $f$ als einen $\mathcal{C}^l$-Diffeomorphismus. Im Falle $l = k = \infty$ bezeichnen wir $f$ auch als glatten Diffeomorphismus. 

        \item[\textit{ii)}]  Wir können den $\mathbb{R}^d$, ausgestattet mit der Standardtopologie, als glatte Mannigfaltigkeit mit dem Atlas 
        \begin{align}
            \mathcal{B} := \{ (\mathcal{V}, \textit{id}_{\mathcal{V}}) \:\: | \:\: \mathcal{V} \:\:\:\: \textit{offen bzgl. der Standardtopologie} \}
        \end{align}
        verstehen. Dabei sind die Abbildungen $\textit{id}_{\mathcal{V}}$ für alle offenen $\mathcal{V} \subseteq \mathbb{R}^d$ erklärt durch 
            \begin{align}
                \textit{id}_{\mathcal{V}}(y) = y \:\:\:\:\:\:\:\: \forall y \in \mathcal{V}.
            \end{align}
        Dann sind bezüglich der obigen Definition $6.8$ alle Kartenabbildung $\phi$ einer Karte $(\mathcal{U}, \phi) \in \mathcal{A}$ glatten Abbildungen. Das lässt sich über das folgende Diagramm leicht einsehen:

        \begin{center}
            \setlength{\unitlength}{2pt}
            \begin{picture}(150,50)
                \put(15,0){$\mathcal{U}$} 
                \put(95,0){$\phi(\mathcal{U})$}
                \put(11,32){$\phi(\mathcal{U})$}
                \put(22,2){\vector(1,0){71}} \put(56,6){$\phi$}
                \put(25,33){\vector(1,0){58}} \put(37,37){$\textit{id}_\mathcal{\phi(\mathcal{U})} \circ \phi \circ \phi^{-1}$} \put(85,32){$\textit{id}_{\phi(\mathcal{U})}(\phi(\mathcal{U})) = \phi(\mathcal{U})$}
                \put(17,6){\vector(0,2){23}} \put(10,16){$\phi$}
                \put(100,6){\vector(0,2){23}} \put(102,16){$\textit{id}_{\phi(\mathcal{U})}$}
            \end{picture}
        \end{center}
        Da $\textit{id}_\mathcal{\phi(\mathcal{U})} \circ \phi \circ \phi^{-1} = \textit{id}_\mathcal{\phi(\mathcal{U})}$ offensichtlich eine glatte Abbildung ist, folgt auch die Glattheit der Abbildung $\phi$. Auf analoge Weise lässt sich zeigen, dass die Abbildung $\phi^{-1}$ ebenfalls eine glatte Abbildung ist. Da $\phi$ und $\phi^{-1}$ beide jeweils auch Homöomorphismen sind folgt, dass es sich bei den Kartenabbildungen einer glatten Mannigfaltigkeit um glatte Diffeomorphismen handelt. 
    \end{itemize}
\end{remark}

Als Nächstes wollen wir uns mit dem Begriff einer differenzierbaren Kurve $\gamma : (a,b) \longrightarrow \mathcal{M}$ beschäftigen, wobei $(a,b) \subseteq \mathbb{R}$ mit $a < b$ und $(\mathcal{M}, \tau, \mathcal{A})$ wieder eine $d$-dimensionale $k$-fach differenzierbare Mannigfaltigkeit ist.

\begin{definition}
    Sei $(\mathcal{M}, \tau, \mathcal{A})$ eine $d$-dimensionale $k$-fach differenzierbare Mannigfaltigkeit und sei $\gamma : (a,b) \longrightarrow \mathcal{M}$ eine stetige Kurve auf $\mathcal{M}$, wobei das offene Intervall $(a,b)$, mit $a,b \in \mathbb{R}$ und $a < b$, wieder mit der Teilraumtopologie versehen ist, welche von der Standardtopologie auf $\mathbb{R}$ induziert wird. Wir nennen die Kurve $\gamma$ $l$-fach stetig differenzierbar mit $l \leq k$, falls die Kartendarstellungen
    \begin{align}
        \phi \circ \gamma : \gamma^{-1}(\mathcal{U}) \longrightarrow \phi(\mathcal{U})
    \end{align}
    $l$-fach stetig differenzierbar ist. Dabei ist $(\mathcal{U}, \phi) \in \mathcal{A}$ eine Karte mit $\gamma((a,b)) \cap \mathcal{U} \neq \emptyset$. Beachte weiter, dass wegen der Stetigkeit von $\gamma$ die Menge $\gamma^{-1}(\mathcal{U})$ eine offene Menge bzgl. der auf $(a,b)$ erklärten Teilraumtopologie ist und per Definition der Teilraumtopologie $\gamma^{-1}(\mathcal{U})$ damit auch eine offene Menge auf dem normierten Raum $(\mathbb{R}, |\cdot|)$ ist. Im Falle $l = k = \infty$ nennen wir eine $l$-fach stetig differenzierbare Kurve auch glatte Kurve. Wir bezeichnen die Menge derartiger $l$-fach differenzierbarer Kurven mit $\mathcal{C}^l ((a,b), \mathcal{M})$.
\end{definition}

Beachte, dass der so definierte Differenzierbarkeitsbegriff für Kurven abermals kartenunabhängig ist, wie das folgende Diagramm zeigt:

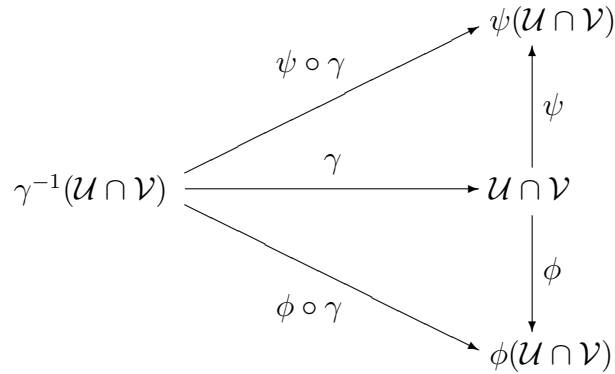
\begin{center}
    \setlength{\unitlength}{2pt}
    \begin{picture}(150,50)
        \put(-4,0){$\gamma^{-1}(\mathcal{U} \cap \mathcal{V})$} 
        \put(85,0){$\mathcal{U} \cap \mathcal{V}$}
        \put(28,2){\vector(1,0){55}} \put(54,6){$\gamma$}
        \put(28,5){\vector(2,1){55}} \put(45,24){$\psi \circ \gamma$} \put(85,32){$\psi(\mathcal{U} \cap \mathcal{V})$}
        \put(28,-1){\vector(2,-1){55}} \put(45,-22){$\phi \circ \gamma$} \put(85,-31){$\phi(\mathcal{U} \cap \mathcal{V})$}
        \put(93,6){\vector(0,2){23}} \put(95,16){$\psi$}
        \put(93,-3){\vector(0,-2){22}} \put(95, -15){$\phi$}
    \end{picture}
\end{center}
.\\[2cm]
Im obigen Diagramm ist dabei $(\mathcal{U}, \phi), (\mathcal{V}, \psi) \in \mathcal{A}$ mit $\gamma((a,b)) \cap \mathcal{U} \cap \mathcal{V} \neq \emptyset$. Anhand des Diagramms lässt sich sehen, dass die Funktion $\psi \circ \gamma$ auch geschreiben werden kann als 
\begin{align}
    \psi \circ \gamma = (\psi \circ \phi^{-1}) \circ (\phi \circ \gamma).
\end{align}
Da $\psi \circ \phi^{-1}$ nach Vorraussetzung von der Klasse $\mathcal{C}^k$ ist und wir annehmen, dass $\phi \circ \gamma$ von der Klasse $\mathcal{C}^l$ ist, folgt mit Folgerung $5.1.$, dass die Funktion $\psi \circ \gamma$ auf $\gamma^{-1}(\mathcal{U} \cap \mathcal{V})$ von der Klasse $\mathcal{C}^l$ sein muss. Daraus folgt die Wohldefiniertheit des Differenzierbarkeitsbegriffes für Kurven auf einer hinreichend oft differenzierbaren Mannigfaltigkeit.


\subsection{Tangentialvektoren und der Tangentialbündel}

Nachdem wir nun gesehen haben, wie wir die aus der Analysis und Funktionalanalysis bekannten Differenzierbarkeitsbegriffe für Funktionen zwischen differenzierbaren Mannigfaltigkeiten und Kurven auf differenzierbaren Mannigfaltigkeiten übertragen können, wollen wir uns nun mit dem Begriff des Tangentialsvektors beschäftigen, der eine wichtige Rolle beim Vermessen der Länge einer Kurve spielen wird. 

Wir werden dabei, wie wir gleich sehen werden, die eben diskutierten Begrifflichkeiten benötigen. Im Folgenden wird angenommen, dass es sich bei der differenzierbaren Mannigfaltigkeit $(\mathcal{M}, \tau, \mathcal{A})$ um eine glatte Mannigfaltigkeit handelt. Darüberhinaus nehmen wir an, dass, falls nichts anderes behauptet wird, alle in diesem Abschnitt betrachteten Kurven entweder injektiv oder konstant sind. 

Um den Begriff des Tangentialsvektors einzuführen, betrachten wir zuerst eine glatte injektive Kurve $\gamma : (a,b) \longrightarrow \mathcal{M}$, mit $a,b \in \mathbb{R}$ mit $a < b$ und eine glatte Funktion $f : \mathcal{M} \longrightarrow \mathbb{R}$, sowie einen Punkt $p \in \gamma((a,b)) \subseteq \mathcal{M}$. Wir sind nun an der Frage interessiert, wie sich die Funktion $f$ am Punkt $p$ in Richtung $\gamma$ verändert, d.h. wir wollen anschaulich die Richtungsableitung der Funktion $f$ im Punkt $p$ in Richtung der Kurve $\gamma$ berechnen. Wie machen wir das? Die Idee ist, dass wir zuerst die stetige Funktion 
\begin{align}
    f \circ \gamma : (a,b) \longrightarrow \mathbb{R}
\end{align}
betrachten, bei welcher es sich um eine reellwertige Funktion auf dem offenen Intervall $(a,b)$ handelt. Anschaulich sagt die Funktion $f \circ \gamma$ gerade, dass wir die Funktion $f$ nur über der Kurve $\gamma((a,b))$ auswerten sollen. Sei $(\mathcal{U}, \phi) \in \mathcal{A}$ eine Karte mit $p \in \mathcal{U}$. Wegen 
\begin{align}
    f \circ \gamma = (f \circ \phi^{-1}) \circ (\phi \circ \gamma) 
\end{align}
ist $f \circ \gamma$ eine an $p$ differenzierbare Funktion. Sei $ \lambda_0 \in (a,b)$ diejenige reelle Zahl aus $(a,b)$ mit $\gamma(\lambda_0) = p$. Die Richtungsableitung der Funktion $f$ am Punkt $p$ in Richtung $\gamma$ ist dann erklärt als 
\begin{align}
    X_{\gamma, p} f := (f \circ \gamma)' (\lambda_0). \label{richtungsab.}
\end{align}
Wir bezeichnen dabei den Operator $X_{\gamma, p} : \mathcal{C}^\infty (\mathcal{M}) \longrightarrow \mathbb{R}$ als Tangentialvektor im Punkt $p$ in Richtung $\gamma$. Beachte dabei, dass wegen der Annahme, dass alle Kurven in diesem Abschnitt injektiv oder konstant seien, die Schreibweise $X_{\gamma,p}$ unproblematisch ist, da im Falle, wenn $\gamma$ nicht konstant ist, es nur ein $\lambda_0 \in (a,b)$ gibt, sodass $\gamma(\lambda_0) = p$ ist, womit wir den Ausdruck $X_{\gamma,p} f$ eindeutig in den Ausdruck $(f \circ \gamma)' (\lambda_0)$ übersetzen können. Für eine konstante Kurve hingegen wäre die Wahl von $\lambda \in (a,b)$ mit $\gamma(\lambda) = p$ irrelevant, da in diesem Falle $f \circ \gamma$ eine konstante Funktion wäre und damit $(f \circ \gamma)'(\lambda) = 0$ für alle $\lambda \in (a,b)$ gilt.

Prinzipell sind wir natürlich dennoch in der Lage, den obigen Begriff des Tangentialvektors auf nicht injektive Kurven zu übertragen. Sei beispielsweise $\eta : (a,b) \longrightarrow \mathcal{M}$ eine Kurve, sodass zwei verschiedene Zahlen $\lambda_1, \lambda_2 \in (a,b)$ existieren mit $\eta(\lambda_1) = \eta(\lambda_2) = p \in \mathcal{M}$. Wollten wir in diesem Fall den Tangentialvektor von $\eta$ am Punkt $p$ bilden, so müssten wir zusätzlich den Kurvenparameter in die Bezeichnung des Tangentialvektors aufnehmen, damit der Begriff des Tangentialvektors wohldefiniert bleibt. Sei nämlich $X_{\eta, \lambda_i, p} : \mathcal{C}^\infty (\mathcal{M}) \longrightarrow \mathbb{R}$ mit $i \in \{1,2\}$ definiert durch 
\begin{align}
    X_{\eta, \lambda_i, p} f = (f \circ \eta )' (\lambda_i),
\end{align}
so lässt sich leicht einsehen, dass im Allgemeinen $X_{\eta, \lambda_1, p} \neq X_{\eta, \lambda_2, p}$ gilt, was die zusätzliche Referenz auf den Kurvenparameter $\lambda$ in der Bezeichnung des Tangentialvektors rechtfertigt. Da wir im Folgenden aber vornehmlich injektive Kurven betrachten wollen, da dies für diese Arbeit völlig ausreicht, wollen wir den Kurvenparameter in der Bezeichnung eines Tangentialvektors einer Kurve weglassen. 

Betrachten wir nun abermals den Ausdruck \eqref{richtungsab.}, so dürfte anschaulich klar sein, warum wir die rechte Seite von \eqref{richtungsab.} als Richtungsableitung von $f$ im Punkt $p$ in Richtung $\gamma$ bezeichnen, da der Ausdruck $(f \circ \gamma)'$ (den wir wieder mit der Funktion $(f \circ \gamma)^{(1)}$ identifizieren) die Änderung der Funktionswerte von $f$ über $\gamma((a,b))$ beschreibt. Was aber eventuell nicht sofort einleuchtend wirkt, ist die Bezeichnung des Ausdruckes $X_{\gamma, p}$ als Tangentialvektor.

Interpretieren wir die Mannigfaltigkeit $\mathcal{M}$ für einen Moment als zweidimensionale glatte Fläche, eingebettet im $\mathbb{R}^3$ (Diese Interpretation ist nicht vollkommen abwegig, wie das Beispiel $6.1.$ zeigt.), so würden wir anschaulich natürlich erwarten, dass ein Tangentialvektor $\mathcal{X}_p$ am Punkt $p \in \mathcal{M}$ ein Pfeil am Punkt $p$ ist, der dort tangential zu $\mathcal{M}$ ist. D.h., dass unter anderem die Gesamtheit aller Tangentialvektoren am Punkt $p$ die Mannigfaltigkeit $\mathcal{M}$ im Punkt $p$ linear approximieren würde. Wie hängt dieses anschaulich geometrische Bild eines Tangentialvektors $\mathcal{X}_p$ mit dem eben eingeführten Begriff des Tangentialvektors $X_{\gamma, p}$ zusammen? Immerhin handelt es sich bei $X_{\gamma, p}$ um einen Differentialoperator, statt um einen geometrischen Pfeil.    

Zuerst bemerken wir, dass $X_{\gamma, p}$ ähnlich einem geometrischen Pfeil eine Orientierung oder Richtung besitzt. So ordnet $X_{\gamma, p}$ der Funktion $f \in \mathcal{C}^\infty(\mathcal{M})$ die Richtungsableitung in Richtung $\gamma$ am Punkt $p$ zu. Diese Richtung können wir mit der Richtung von $X_{\gamma, p}$ identifizieren. Darüberhinaus ist diese Richtung wegen $\gamma((a,b)) \subseteq \mathcal{M}$ am Punkt $p$ auch tangential zur Mannigfaltigkeit $\mathcal{M}$, weshalb wir damit auch den Differentialopertor $X_{\gamma, p}$ im obigem Sinne als Tangentialvektor verstehen können. 

Gleichzeitig können wir $X_{\gamma, p}$ aber auch als die Geschwindigkeit der Kurve $\gamma$ am Punkt $p$ verstehen: Wir interpretieren dazu das Argument der Kurve $\gamma$ als Zeit, d.h. $\gamma(\lambda)$ stellt den Punkt der Kurve $\gamma$ dar, an welchem wir uns zur Zeit $\lambda$ befinden. Weiter interpretieren wir die Funktion $f \in \mathcal{C}^\infty (\mathcal{M})$ als Landschaft auf $\mathcal{M}$, indem wir $f$ als Höhenkarte auf $\mathcal{M}$ verstehen. Dann ist $X_{\gamma, p} f = (f \circ \gamma)'(\lambda_0)$ die Änderung der Landschaft entlang der Kurve $\gamma$ am Punkt $p$, was anschaulich einer Geschwindigkeit entspricht, da wir über die Art der Änderung der umgebenden Landschaft feststellen können, wie wir uns auf der Kurve $\gamma$ bewegen. Die Rate der Änderung dieser Landschaft wird dabei durch die Kurve $\gamma$ vorgegeben. Das lässt sich wie folgt einsehen: Wir definieren dazu die Kurve 
\begin{align}
    \delta(\lambda) := \gamma(2 \cdot \lambda).
\end{align}
Anschaulich beschreibt $\delta$ die selbe geometrische Kurve wie $\gamma$, nur wird diese geometrische Kurve bzgl. $\delta$ doppelt so schnell durchlaufen wie $\gamma$, sofern wir den Bahnparameter $\lambda$ wieder als Zeit interpretieren. Dann gilt mittels der Kettenregel, indem wir $\lambda \longmapsto 2 \cdot \lambda$ als Funktion betrachten, dass
\begin{align}
    X_{\delta, p} f = (f \circ \delta)' \bigg(\frac{\lambda_0}{2}\bigg) = 2 \cdot (f \circ \gamma)'(\lambda_0) = 2 \cdot X_{\gamma, p} f
\end{align}
ist, d.h. bezüglich der Kurve $\delta$ durchlaufen wir den Punkt $p$ doppelt so schnell, was sich daran zeigt, dass die Änderungsrate der Landschaft $f$ verglichen mit $X_{\gamma, p} f$ doppelt so hoch ist. Aus dieser Rechnung folgt nun auch, dass die Durchlaufgeschwindigkeit in der Kurve codiert ist, wie oben behauptet. 

Damit können wir also $X_{\gamma, p}$ auch physikalisch als Geschwindikeit der Kurve $\gamma$ am Punkt $p$ auffassen. Da die Geschwindikeit eines Teilchens, welches sich auf der Kurve $\gamma$ bewegt, anschaulich immer tangential an $\mathcal{M}$ anliegt, da wegen $\gamma((a,b)) \subseteq \mathcal{M}$ die Bewegung des Teilchens auf $\mathcal{M}$ eingeschränkt ist, können wir $X_{\gamma, p}$ abermals als tangential auffassen, was zeigt, dass beide Interpretationen miteinander kompatibel sind. 

Mittels der beiden obigen Interpretationen sind wir damit in der Lage zum einen $X_{\gamma, p}$ als geometrischen Tangentialvektor zu interpretieren, und zum anderen sind wir in der Lage ihn mit einer konkrete phsikalische Interpretation zu versehen. Zusammengefasst haben wir damit die Bezeichnung von $X_{\gamma, p}$ gerechtfertigt.

Wir definieren nun den Begriff des Tangentialraumes, der die Gesamtheit aller derartigen Tangentialvektoren am Punkt $p \in \mathcal{M}$ darstellt und anschaulich, wie oben, die Mannigfaltigkeit $\mathcal{M}$ am Punkt $p$ linear approximiert.

\begin{definition}
    Sei $(\mathcal{M}, \tau, \mathcal{A})$ eine $d$-dimensionale glatte Mannigfaltigkeit und $p \in \mathcal{M}$ ein Punkt der Mannigfaltigkeit. Wir definieren den sogenannten Tangentialraum am Punkt $p$ durch
    \begin{align}
        T_p \mathcal{M} := \{ X_{\gamma, p} \:\: | \:\: &\exists a,b \in \mathbb{R} \:\: \textit{mit} \:\: a < b \:\: \textit{und} \:\: \gamma : (a,b) \longrightarrow \mathcal{M} \:\: \textit{glatt,} \nonumber \\ &\textit{sodass} \:\: \exists \lambda_0 \in (a,b) : \gamma(\lambda_0) = p  \},
    \end{align}
    wobei für $X_{\gamma, p} \in T_p \mathcal{M}$ gilt, dass
    \begin{align}
        X_{\gamma, p} f = (f \circ \gamma)'(\lambda_0) \:\:\:\:\:\:\:\: \forall f \in \mathcal{C}^\infty (\mathcal{M})
    \end{align}
    ist. Die in der Definition von $T_p \mathcal{M}$ auftauchenden glatten Kurven $\gamma : (a,b) \longrightarrow \mathcal{M}$, für die $p \in \gamma((a,b))$ gilt, sollen dabei nicht notwendigerweise injektiv sein. Um die Notation jedoch simpel zu gestalten treffen wir folgende Vereinbarung: Falls die Kurve $\gamma$ weder konstant, noch injektiv ist, so wird vorrausgesetzt, dass sich aus dem Kontext ergibt, welches $\lambda_0 \in (a,b)$ mit $\gamma(\lambda_0) = p$ genutzt werden soll, um den Ausdruck $X_{\gamma,p} f$ mit $f \in \mathcal{C}^\infty (\mathcal{M})$ zu berechnen.
\end{definition}

Die Interpretation von $T_p \mathcal{M}$ als lineare Approximation von $\mathcal{M}$ am Punkt $p \in \mathcal{M}$ deutet bereits an, dass sich $T_p \mathcal{M}$ mit einer linearer Struktur versehen lässt, bzw. $T_p \mathcal{M}$ sich zu einem $\mathbb{R}$-Vektorraum machen lässt. In der Tat ist das machbar, wie der folgende Satz zeigt:

\begin{proposition}
    Sei $(\mathcal{M}, \tau, \mathcal{A})$ eine $d$-dimensionale glatte Mannigfaltigkeit und $p \in \mathcal{M}$ ein Punkt in $\mathcal{M}$. Wir statten $T_p \mathcal{M}$ mit den folgenden Operationen aus, mit welchen $T_p \mathcal{M}$ zu einem $\mathbb{R}$-Vektorraum wird:
    \begin{itemize}
        \item[\textit{i)}] Skalarmultiplikation:
        \begin{align}
            \odot : \mathbb{R} \times T_p \mathcal{M} \longrightarrow& \: \: T_p \mathcal{M} \nonumber\\
            (\alpha, X_{\gamma, p}) \longmapsto& \: \: \alpha \odot X_{\gamma,p},
        \end{align}
        wobei $\alpha \odot X_{\gamma,p}$ definiert ist als 
        \begin{align}
            (\alpha \odot X_{\gamma,p}) f := \alpha \cdot (X_{\gamma, p} f) \:\:\:\:\:\:\:\: \forall f \in \mathcal{C}^\infty (\mathcal{M}).
        \end{align}
        \item[\textit{ii)}] Vektoraddition:
        \begin{align}
            \oplus : T_p \mathcal{M} \times T_p \mathcal{M} \longrightarrow& \: \: T_p \mathcal{M} \nonumber\\
            (X_{\gamma,p}, X_{\sigma,p}) \longmapsto& \: \: X_{\gamma,p} \oplus X_{\sigma,p},
        \end{align}
        wobei $X_{\gamma,p} \oplus X_{\gamma,p}$ definiert ist als
        \begin{align}
            (X_{\gamma,p} \oplus X_{\gamma,p}) f := X_{\gamma,p} f + X_{\sigma, p} f \:\:\:\:\:\:\:\: \forall f \in \mathcal{C}^\infty (\mathcal{M}).
        \end{align}
    \end{itemize} 
\end{proposition}

\begin{proof}
    Zuerst müssen wir zeigen, dass die oben erklärte Skalarmultiplikation und Vektoraddition abgeschlossen sind, d.h. wir müssen zeigen, dass $\alpha \odot X_{\gamma,p}$ und $X_{\gamma, p} \oplus X_{\sigma,p}$ für alle $\alpha \in \mathbb{R}$ und alle $X_{\gamma,p}, X_{\sigma,p} \in T_p \mathcal{M}$ mit $\gamma : (a,b) \longrightarrow \mathcal{M}$ und $\sigma: (c,d) \longrightarrow \mathcal{M}$ tatsächlich wieder in $T_p \mathcal{M}$ liegen. Wir wollen dazu glatte Kurven $\delta : (e,f) \longrightarrow \mathcal{M}$ und $\eta : (g,h) \longrightarrow \mathcal{M}$ konstruieren, sodass 
    \begin{align}
        \alpha \odot X_{\gamma,p} = X_{\delta,p} \:\:\: \textit{und} \:\:\: X_{\gamma, p} \oplus X_{\sigma,p} = X_{\eta,p}
    \end{align}
    gilt. Sei dazu zuerst $\alpha \in \mathbb{R} \setminus \{ 0 \}$ und $\gamma(\lambda_0) = p$. Dann ist $\delta$ gegeben durch die glatte Kurve $\delta : (\frac{a}{\alpha}, \frac{b}{\alpha}) \longrightarrow \mathcal{M}$ mit $\delta(\lambda) = \gamma(\alpha \cdot \lambda)$. Dann gilt $\delta(\frac{\lambda_0}{\alpha}) = \gamma(\lambda_0) = p$ und mittels der multidimensionalen Kettenregel
    \begin{align}
        X_{\delta,p} f = (f \circ \delta)'\bigg(\frac{\lambda_0}{\alpha}\bigg) = \alpha \cdot (f \circ \gamma)' (\lambda_0) = \alpha \cdot X_{\gamma,p} f \:\:\:\: \forall f \in \mathcal{C}^\infty (\mathcal{M})
    \end{align}
    und damit $\alpha \odot X_{\gamma,p} = X_{\delta,p}$.

    Als Nächstes betrachten wir den Fall, dass $\alpha = 0$ ist. Dann ist $\delta$ gegeben durch die glatte Kurve $\delta : (a,b) \longrightarrow \mathcal{M}$ mit $\delta(\lambda) = \gamma(\lambda_0) = p$. Dann gilt, da $f \circ \delta$ eine konstante Funktion ist, dass 
    \begin{align}
        X_{\delta,p} f = (f \circ \delta)'(\lambda_0) = 0 = 0 \cdot X_{\gamma,p} f \:\:\:\: \forall f \in \mathcal{C}^\infty (\mathcal{M})
    \end{align}
    und damit $0 \odot X_{\gamma,p} = X_{\delta,p}$ ist.

    Betrachten wir nun die Vektoraddition. Wir konstruieren die Kurve $\eta$ folgendermaßen: Zuerst nehmen wir ohne Beschränkung der Allgemeinheit an, dass die Definitionsbereiche der Kurven $\gamma : (a,b) \longrightarrow \mathcal{M}$ und $\sigma : (c,d) \longrightarrow \mathcal{M}$ einen nichtleeren Schnitt besitzen, d.h. $(a,b) \cap (c,d) = (g,h) \neq \emptyset$. Insbesondere soll dabei gelten, dass es ein $\lambda_0 \in (g,h)$ gibt, sodass $\gamma(\lambda_0) = \sigma(\lambda_0) = p$ ist. Andernfalls verschieben wir eines der beiden Intervalle, z.B. das Intervall $(a,b)$, um eine Konstante $\alpha \in \mathbb{R}$ und definieren dann die Kurve $\Tilde{\gamma} : (a + \alpha, b + \alpha) \longrightarrow \mathcal{M}$ mit $\Tilde{\gamma}(\lambda) = \gamma(\lambda - \alpha)$, sodass $(a + \alpha, b + \alpha) \cap (c,d) \neq \emptyset$. 
    
    Beachte, dass in diesem Fall $\gamma((a,b)) = \Tilde{\gamma}((a+\alpha, b+\alpha))$ gilt und die Geschwindigkeit von $\Tilde{\gamma}$ am Punkt $p$ identisch ist mit der Geschwindigkeit von $\gamma$ am Punkt $p$. Gilt außerdem $\sigma(\lambda_0) = \gamma(\lambda_1) = p$ mit $\lambda_0 \in (c,d)$ und $\lambda_1 \in (a,b)$, so folgt mit der Wahl $\alpha = \lambda_0 - \lambda_1$, dass $\Tilde{\gamma}(\lambda_0) = \gamma(\lambda_0 - \lambda_0 + \lambda_1) = \gamma(\lambda_1) = p$, womit wir obiges Vorgehen gerechtfertigt haben.

    Sei nun weiter $(\mathcal{U}, \phi) \in \mathcal{A}$ eine Karte die $p \in \mathcal{M}$ enthält. Dann definieren wir die Kurve $\eta : (g,h) \longrightarrow \mathcal{M}$ durch 
    \begin{align}
        \eta(\lambda) := (\phi^{-1} (\phi (\gamma(\lambda)) + \phi(\sigma(\lambda)) - \phi(p)).
    \end{align}
    Beachte, dass wegen 
    \begin{align}
        \phi \circ \eta = \phi \circ \gamma + \phi \circ \sigma - \phi(p) \label{eta}
    \end{align}
    die Funktion $\eta$ selbst eine glatte Funktion mit $\eta(\lambda_0) = \phi^{-1}(\phi(p) + \phi(p) - \phi(p)) = p$ ist. Beachte dabei weiter, dass die Addition in \eqref{eta} punktweise definiert ist. Wir berechnen nun den Ausdruck $X_{\eta, p} f$ für ein beliebiges $f \in \mathcal{C}^\infty (\mathcal{M})$: Sei dazu $\phi = (x^1, ..., x^d)$, dann folgt mit der Identität \eqref{Spezialfall Kettenregel} aus Satz $5.3.$, dass
    \begin{align}
        X_{\eta,p} f =& \: X_{\eta, \lambda_0, p} f = (f \circ \eta)'(\lambda_0) \nonumber \\ =& \: \big( ( f \circ \phi^{-1} ) \circ ( \phi \circ \gamma + \phi \circ \sigma - \phi(p) ) \big)'(\lambda_0) \nonumber \\ =& \: \sum_{i = 1}^d \partial_i (f \circ \phi^{-1}) (\phi(p)) \cdot (x^{i} \circ \gamma + x^{i} \circ \sigma - x^{i}(p))' (\lambda_0) \nonumber \\ =& \: \sum_{i = 1}^d \partial_i (f \circ \phi^{-1}) (\phi(p)) \cdot (x^{i} \circ \gamma)'(\lambda_0) + \sum_{i = 1}^d \partial_i (f \circ \phi^{-1}) (\phi(p)) \cdot (x^{i} \circ \sigma)'(\lambda_0) \nonumber \\ =& \: ((f \circ \phi^{-1}) \circ (\phi \circ \gamma))'(\lambda_0) + ((f \circ \phi^{-1}) \circ (\phi \circ \sigma))'(\lambda_0) \nonumber \\ =& \: (f \circ \gamma)'(\lambda_0) + (f \circ \sigma)'(\lambda_0) \nonumber \\ =& \: X_{\gamma,p} f + X_{\sigma,p} f \nonumber \\ =& \: (X_{\gamma,p} \oplus X_{\sigma,p}) f.
    \end{align}
    Da $f \in \mathcal{C}^\infty (\mathcal{M})$ beliebig war folgt, dass $X_{\eta,p} = X_{\gamma,p} \oplus X_{\sigma,p}$ ist.

    Die anderen Vektorraum-Axiome folgen dabei direkt aus den Körpereigenschaften von $(\mathbb{R}, +, \cdot)$. So folgt beispielsweise die Kommutativität von $\oplus$ aus der Kommutativität der reellen Addition $+$, wie sich über 
    \begin{align}
        (X_{\gamma,p} \oplus X_{\sigma,p}) f =& \: X_{\gamma,p} f + X_{\sigma,p} f \nonumber \\ =& \: X_{\sigma,p} f + X_{\gamma,p} f \nonumber \\ =& \: (X_{\sigma,p} \oplus X_{\gamma,p}) f \:\:\:\: \forall f \in \mathcal{C}^\infty (\mathcal{M})
    \end{align}
    einsehen lässt. Auf diese Weise wird sofort ersichtlich, dass $(T_p \mathcal{M}, \oplus)$ eine abelsche Gruppe ist. Auf analoge Weise lassen sich die Eigenschaften der Skalarmultiplikation $\odot$ und deren Verträglichkeit mit der Vektoraddition $\oplus$ nachprüfen.
\end{proof}

Im Nachfolgenden werden wir wieder statt $\odot$ die Schreibweise $\cdot$ und statt $\oplus$ die Schreibweise $+$ nutzen.

Wir können nun darüberhinus noch zeigen, dass $T_p (\mathcal{M})$ als $\mathbb{R}$-Vektorraum die Dimension $d$ besitzt, falls $(\mathcal{M}, \tau, \mathcal{A})$ eine glatte $d$-dimensionale Mannigfaltigkeit sein sollte, was auch unserer Intuition entspricht, wenn $T_p \mathcal{M}$ anschaulich die Mannigfaltigkeit $\mathcal{M}$ am Punkt $p \in \mathcal{M}$ linear approximieren soll.

\begin{proposition}
    Sei $(\mathcal{M}, \tau, \mathcal{A})$ eine $d$-dimensionale glatte Mannigfaltigkeit und $p \in \mathcal{M}$ ein Punkt in $\mathcal{M}$. Dann besitzt der Vektorraum $T_p \mathcal{M}$ die Dimension $d$.
\end{proposition}

\begin{proof}
    Wir betrachten eine beliebige Karte $(\mathcal{U}, \phi) \in \mathcal{A}$, die den Punkt $p \in \mathcal{M}$ enthält, wobei $\phi = (x^1, ..., x^d)$ ist. Weiter definieren wir den Differentialoperator $\big(\frac{\partial}{\partial x^{i}}\big)_p : \mathcal{C}^\infty (\mathcal{M}) \longrightarrow \mathbb{R}$ mit $i \in \{ 1, ..., d \}$ durch
     \begin{align}
        \bigg( \frac{\partial}{\partial x^{i}} \bigg)_p f = \partial_i (f \circ \phi^{-1}) (\phi(p)) \:\:\:\:\:\:\:\: \forall f \in \mathcal{C}^\infty (\mathcal{M}). \label{basisinduzierter Tangentialvektor}
    \end{align}
    Wir zeigen zuerst, dass für alle $i \in \{ 1, ..., d \}$ gilt, dass $\big(\frac{\partial}{\partial x^{i}}\big)_p \in T_p \mathcal{M}$ ist.  Sei dazu $\phi(p) = (y^1,...,y^d) \in \mathbb{R}^d$ und 
    \begin{align}
        \eta_i : (a_i, b_i) \longrightarrow \mathcal{M},
    \end{align}
    für $i \in \{ 1, ..., d \}$ eine Kurve, definiert durch
    \begin{align}
        (a_i, b_i) \ni \lambda \longmapsto \phi^{-1}(y^1, ..., y^{i-1}, \underbrace{\lambda}_{\textit{$i$-te Stelle}}, y^{i+1}, ..., y^d).
    \end{align}
    Dabei sind die Intervalle $(a_i, b_i) \subseteq \mathbb{R}$ jeweils so gewählt, dass
    \begin{align}
        \{ (y^1, ..., y^{i-1}. \lambda, y^{d+1}, ..., y^d) \:\: | \:\: \lambda \in (a_i, b_i) \} \subseteq \phi(\mathcal{U})
    \end{align}
    und $y^{i} \in (a_i, b_i)$ gilt. Beachte, dass damit $\eta_i((a_i,b_i)) \subseteq \mathcal{U}$ ist und mit 
    \begin{align}
        (\phi \circ \eta_i)(\lambda) =& \: (\phi \circ \phi^{-1}) (y^1, ..., y^{i-1}, \lambda, y^{i+1}, ..., y^d) \nonumber \\ =& \: (y^1, ..., y^{i-1}, \lambda, y^{i+1}, ..., y^d)
    \end{align}
    gemäß der Definition $6.9.$ folgt, dass die Kurve $\eta_i$ für alle $i \in \{ 1, ..., d \}$  eine glatte Kurve ist. Darüberhinaus gilt weiter für alle $i \in \{ 1, ..., d \}$, dass $\eta_i$ injektiv ist, da $\phi^{-1}$ injektiv ist, und $\eta_i(y^{i}) = p$ ist. Es gilt nun mit $\phi = (x^1, ..., x^d)$ und der Abbildung
    \begin{align}
        \textit{proj}^k : \mathbb{R}^d \longrightarrow \mathbb{R}, 
    \end{align}
    definiert durch $\mathbb{R}^d \ni v = (v^1, ..., v^{k-1}, v^k, v^{k+1}, ..., v^d) \longmapsto \textit{proj}^k(v) = v^k$, wobei $k \in \{ 1, ..., d \}$ ist, dass
    \begin{align}
        (x^k \circ \eta_i)(\lambda) =& \: x^k(\eta_i(\lambda)) \nonumber \\ =& \: x^k(\phi^{-1} (y^1, ..., \lambda, ..., y^d)) \nonumber \\ =& \: \textit{proj}^k(\phi(\phi^{-1}(y^1, ..., \lambda, ..., y^d))) \nonumber \\ =& \: \textit{proj}^k(y^1, ..., \lambda, ..., y^d) \nonumber \\ =& \: \left\{\begin{array}{ll} y^k, & k \neq i \\ \lambda, & k = i \end{array}\right. 
    \end{align}
    ist. Dann gilt weiter für ein beliebiges $f \in \mathcal{C}^\infty (\mathcal{M})$, dass
    \begin{align}
        X_{\eta_i, p} f =& \: (f \circ \eta_i)'(y^{i}) \nonumber \\ =& \: ((f \circ \phi^{-1}) \circ (\phi \circ \eta_i))'(y^{i}) \nonumber \\ =& \: \sum_{k=1}^d \partial_k (f \circ \phi^{-1})(\phi(p)) \cdot (x^k \circ \eta_i)'(y^{i}) \nonumber \\ =& \: \partial_i (f \circ \phi^{-1}) (\phi(p)) \nonumber \\ =& \bigg( \frac{\partial}{\partial x^{i}}  \bigg)_p f
    \end{align}
    und damit 
    \begin{align}
        X_{\eta_i, p} = \bigg( \frac{\partial}{\partial x^{i}} \bigg)_p
    \end{align}
    ist. Damit ist bewiesen, dass für alle $i \in \{ 1, ..., d \}$ gilt, dass $\big( \frac{\partial}{\partial x^{i}} \big)_p \in T_p \mathcal{M}$ ist. Sei nun $X_{\gamma,p} \in T_p \mathcal{M}$ ein beliebiger Tangentialvektor in $T_p \mathcal{M}$. Dann gilt mit der multidimensionalen Kettenregel und im speziellen mit der Identität \eqref{Spezialfall Kettenregel}, dass
    \begin{align}
        X_{\gamma,p} f =& \: (f \circ \gamma)'(\lambda_0) \nonumber \\ =& \: ((f \circ \phi^{-1}) \circ (\phi \circ \gamma))'(\lambda_0) \nonumber \\ =& \: \sum_{i = 1}^d \partial_i (f \circ \phi^{-1}) (\phi(p)) \cdot (x^{i} \circ \gamma)'(\lambda_0)  
    \end{align}
    ist. Wir beachten, dass für alle $i \in \{ 1, ..., d \}$ der Ausdruck $(x^{i} \circ \gamma)'(\lambda_0)$ eine feste reelle Zahl ist. Wir setzen
    \begin{align}
        X^{i} := (x^{i} \circ \gamma)'(\lambda_0) \in \mathbb{R} \:\:\:\: \forall i \in \{ 1, ..., d \}. \label{Vektorkomponente}
    \end{align}
    Mit dieser Setzung folgt
    \begin{align}
        X_{\gamma, p} f = \sum_{i = 1}^d X^{i} \cdot \partial_i (f \circ \phi^{-1}) (\phi(p)). \label{Basisentwicklung}
    \end{align}
    Mittels \eqref{basisinduzierter Tangentialvektor} folgt für die Gleichung \eqref{Basisentwicklung} die Identität
    \begin{align}
        X_{\gamma,p} f = \sum_{i = 1}^d X^{i} \cdot \bigg(  \frac{\partial}{\partial x^{i}} \bigg)_p f \:\:\:\:\:\:\:\: \forall f \in \mathcal{C}^\infty (\mathcal{M})
    \end{align}
    und damit 
    \begin{align}
        X_{\gamma, p} = \sum_{i = 1}^d X^{i} \cdot \bigg(  \frac{\partial}{\partial x^{i}} \bigg)_p.\label{Linearkombinat.}
    \end{align}
    Die rechte Seite von \eqref{Linearkombinat.} entspricht dabei einer Linearkombination in $T_p \mathcal{M}$ und zeigt, da $X_{\gamma,p} \in T_p \mathcal{M}$ beliebig war, dass sich jedes Element in $T_p \mathcal{M}$ über eine derartige Linearkombination darstellen lässt. Das impliziert, dass die Menge $\{ \big( \frac{\partial}{\partial x^1} \big)_p, ..., \big( \frac{\partial}{\partial x^d} \big)_p \} \subseteq T_p \mathcal{M}$ den ganzen Tangentialraum $T_p \mathcal{M}$ aufspannt, in Zeichen 
    \begin{align}
        \textit{span}_\mathbb{R}\bigg( \bigg\{  \bigg( \frac{\partial}{\partial x^1} \bigg)_p, ..., \bigg( \frac{\partial}{\partial x^d} \bigg)_p  \bigg\} \bigg) = T_p \mathcal{M}.
    \end{align}
    Können wir nun noch zeigen, dass die Menge  $\{ \big( \frac{\partial}{\partial x^1} \big)_p, ..., \big( \frac{\partial}{\partial x^d} \big)_p \}$ darüber hinaus linear unabhängig ist, so würde folgen, dass es sich bei dieser Menge um eine Basis vom $T_p \mathcal{M}$ handelt, was wiederum implizieren würde, dass der Tangentialraum $T_p \mathcal{M}$ die Dimension $d$ besäße.

     Wir zeigen also noch, dass die Menge $\{ \big( \frac{\partial}{\partial x^1} \big)_p, ..., \big( \frac{\partial}{\partial x^d} \big)_p \}$ linear unabhängig ist: Zuerst bemerken wir, dass, wenn wir wie in Bemerkung $6.1.$ den $\mathbb{R}^d$ wieder als glatte Mannigfaltigkeit verstehen, es sich bei den Abbildungen 
    \begin{align}
        x^{i} : \mathcal{U} \longrightarrow x^{i}(\mathcal{U}) \subseteq \mathbb{R},
    \end{align}
    mit $i \in \{ 1, ..., d \}$, um glatte Abbildungen handelt, was das folgende Diagramm zeigt:

    \begin{center}
        \setlength{\unitlength}{2pt}
        \begin{picture}(150,50)
            \put(15,0){$\mathcal{U}$} 
            \put(95,0){$x^{i}(\mathcal{U})$}
            \put(11,32){$\phi(\mathcal{U})$}
            \put(22,2){\vector(1,0){71}} \put(58,6){$x^{i}$}
            \put(25,33){\vector(1,0){68}} \put(39,37){$\textit{id}_{x^{i}(\mathcal{U})} \circ x^{i} \circ \phi^{-1}$} \put(95,32){$x^{i}(\mathcal{U})$}
            \put(17,6){\vector(0,2){23}} \put(10,16){$\phi$}
            \put(100,6){\vector(0,2){23}} \put(102,16){$\textit{id}_{x^{i}(\mathcal{U})}$}
        \end{picture}
    \end{center}
    Wegen 
    \begin{align}
        id_{x^{i}(\mathcal{U})} \circ x^{i} \circ \phi^{-1} =& \: x^{i} \circ \phi^{-1} \nonumber \\ =& \: \textit{proj}^{i} \circ \phi \circ \phi^{-1} \nonumber \\ =& \: \textit{proj}^{i}
    \end{align}
    folgt für alle $i \in \{ 1, ..., d \}$ gemäß der Definition $6.8.$ explizit die Glattheit der Abbildungen $x^{i}$. Damit können wir nun aber die Abbildung $x^{i} \in \mathcal{C}^\infty (\mathcal{M})$ in die Differentialoperatoren $\big(  \frac{\partial}{\partial x^j} \big)_p$ als Input geben. Sei nun 
    \begin{align}
        \mathbf{0}_{T_p \mathcal{M}} = \sum_{i = 1}^d \alpha^{i} \cdot \bigg( \frac{\partial}{\partial x^{i}} \bigg)_p
    \end{align}
    mit $\alpha^1, ..., \alpha^d \in \mathbb{R}$ gegeben, wobei $\mathbf{0}_{T_p \mathcal{M}}$ das neutrale Element bzw. der Nullvektor von $T_p \mathcal{M}$ ist. Dann gilt für $j \in \{ 1, ..., d\}$
    \begin{align}
        0 =& \: \sum_{i = 1}^d \alpha^{i} \cdot \bigg( \frac{\partial}{\partial x^{i}} \bigg)_p x^j \nonumber \\ =& \: \sum_{i = 1}^d \alpha^{i} \cdot \partial_i( \underbrace{x^j \circ \phi^{-1}}_{\textit{proj}^j})(\phi(p)) \nonumber \\ =& \: \sum_{i = 1}^d \alpha^{i} \cdot \delta^j_i = \alpha^j.
    \end{align}
    und damit $\alpha^j = 0$ für alle $j \in \{ 1, ..., d \}$. Daraus folgt die lineare Unabhängigkeit der Menge $\{ \big( \frac{\partial}{\partial x^1} \big)_p, ..., \big( \frac{\partial}{\partial x^d} \big)_p \}$, womit diese eine Basis von $T_p \mathcal{M}$ darstellt. Da die Länge oder Kardinalität dieser Menge $d$ beträgt, folgt, dass $T_p \mathcal{M}$ ein $d$-dimensionaler $\mathbb{R}$-Vektorraum ist.
\end{proof}

\begin{remark}
    Oft wird die Menge $\{ \big( \frac{\partial}{\partial x^1} \big)_p, ..., \big( \frac{\partial}{\partial x^d} \big)_p \} \subseteq T_p \mathcal{M}$ als die durch die Karte $(\mathcal{U}, \phi) \in \mathcal{A}$ induzierte Basis bezeichnet. Die Elemente der Basis bezeichnen wir als die von der Karte $(\mathcal{U}, \phi)$ induzierten Tangentialvektoren.
\end{remark}

Im Folgenden werden wir uns die Tangentialräume $T_p \mathcal{M}$ stets mit dieser Vektorraumstruktur versehen denken. Die Tatsache, dass wir die Tangentialräume mit einer derartigen linearen Struktur ausstatten können, hat dabei nicht nur Konsequenzen für unsere geometrische Interpretation, sondern auch technische Vorteile, da wir so beispeilsweise auch in der Lage sind, den Dualraum des Tangentialraumes $T_p \mathcal{M}$ zu erklären. Dieser wird auch Kotangentialraum am Punkt $p \in \mathcal{M}$ genannt und kurz mit $T_p^* \mathcal{M}$ bezeichnet. Diese Kotangentialräume werden später wiederrum eine entscheindende Rolle in der Definition der sogenannten Tensorfelder spielen, die wir unter anderem benötigen, um einen anschaulichen Längenbegriff für Kurve auf einer Mannigfaltigkeit zu erklären.

\begin{remark}
    Nach dem wir nun mit den beiden vorrangegangen Sätzen die Plausibilität unserer geometrischen Interpretation der Tangentialvektoren $X_{\gamma,p} \in T_p \mathcal{M}$ begründet hatten, können wir uns nun noch kurz die Frage stellen, warum wir unter einem Tangentialvektor $X_{\gamma,p}$ einen Differentialoperator verstehen, der auf $\mathcal{C}^\infty (\mathcal{M})$ erklärt ist. 
    
    Tatsächlich können wir $X_{\gamma,p}$ nämlich auch als einen Operator der Form $\mathcal{C}^1(\mathcal{M}) \longrightarrow \mathbb{R}$ interpretieren, da nach obiger Konstruktion und den beiden vorrangegangenen Beweisen keine höheren Regularitäten als $\mathcal{C}^1$ an die Funktionen $f$ gestellt werden brauchen. Selbiges gilt im Prinzip auch für die Kurve $\gamma$, die in der Definition von $X_{\gamma,p}$ verbaut ist. 
    
    Das wir dennoch Tangentialvektoren nur für glatte Kurven als Differentialoperatoren erklären, die auf glatte Funktionen $f \in \mathcal{C}^\infty (\mathcal{M})$ wirken, hat zwei Gründe: Zum einen wollen wir uns bei konkreten Berechnungen keine Gedanken um die Regularitäten der beteiligten Kurven und Funktionen machen. Zum anderen werden aus praktischen Gründen in der modernen Differentialgeometrie Tangentialvektoren am Punkt $p \in \mathcal{M}$ einer differenzierbaren Mannigfaltigkeit $(\mathcal{M}, \tau, \mathcal{A})$ oft auch als sogenannte am Punkt $p$ sitzende, lineare Derivationen
    \begin{align}
        \mathcal{D}_p : (\mathcal{C}^\infty (\mathcal{M}), \odot) \longrightarrow \mathbb{R}
    \end{align}
    definiert, wobei $\mathcal{C}^\infty (\mathcal{M})$ mit der punktweisen Multiplikation $\odot$ versehen wurde, bzgl. derer die Menge $\mathcal{C}^\infty (\mathcal{M})$ abgeschlossen ist \cite{lee2012smooth}. Eine Derivation am Punkt $p$ der Form $(\mathcal{C}^\infty (\mathcal{M}), \odot) \longrightarrow \mathbb{R}$ ist dabei eine Abbildung $\mathcal{D}_p : \mathcal{C}^\infty (\mathcal{M}) \longrightarrow \mathbb{R}$, die der folgenden Leibnitzregel genügt, die die aus der reellen Analysis bekannte Produktregel \cite{heuser2013lehrbuch} nachahmt:
    \begin{align}
        \mathcal{D}_p (f \odot g) = \mathcal{D}_p f \cdot g(p) + f(p) \cdot \mathcal{D}_p(g) \:\:\:\:\:\:\:\: \forall f,g \in \mathcal{C}^\infty (\mathcal{M}). 
    \end{align}
    Im Gegensatz zum Tangentialvektor $X_{\gamma,p}$ ist dabei die konkrete Gestalt einer allgemeinen, am Punkt $p$ sitzenden, linearen Derivation $\mathcal{D}_p$ nicht gegeben, da es sich dabei salopp gesprochen nur um eine abstrakte Abbildung auf der Menge der differenzierbaren Funktionen handelt, die der Leibnitzregel genügt. Es lässt sich darüberhinaus zeigen, dass die Menge aller linearen Derivationen $\mathcal{D}_p$ am Punkt $p$ im Allgemeinen nur dann einen endlichdimensionalen Vektorraum bilden, dessen Dimension mit der Dimension der Mannigfaltigkeit $(\mathcal{M}, \tau, \mathcal{A})$ zusammenfällt, falls die Derivationen an $p$ auf der Menge $\mathcal{C}^\infty (\mathcal{M})$ erklärt sind. 
    
    Da man, wie oben, die Menge aller linearen Derivationen an $p$ anschaulich als lineare Approximation der Mannigfaltigkeit $(\mathcal{M}, \tau, \mathcal{A})$ am Punkt $p$ auffassen möchte, da man die einzelnen linearen Derivationen am Punkt $p$ als am Punkt $p$ angeheftete Tangentialvektoren interpretiert, benötigen wir dazu die Menge $\mathcal{C}^\infty (\mathcal{M})$ als Definitionsbereich, um so die geometrische Interpretation der Menge aller an $p$ sitzenden linearen Derivationen aufrechtzuerhalten. 

    Wollen wir unseren obigen Zugang zum Tangentialraum $T_p \mathcal{M}$, in welchem die einzelnen Tangentialvektoren Richtungsableitungen am Punkt $p$ sind, mit der Menge aller an $p$ sitzenden linearen Derivationen auf intuitive Weise zusammenbringen, so müssen wir die Tangentialvektoren $X_{\gamma,p} \in T_p \mathcal{M}$ als lineare Derivation an $p$ auffassen, womit wir implizit vorraussetzen, dass $X_{\gamma,p}$ als linearer Differentialoperator auf $\mathcal{C}^\infty (\mathcal{M})$ definiert ist. 

    Das wir einen Tangentialvektor $X_{\gamma,p}$ überhaupt als lineare Derivation auffassen können, sieht man wie folgt: Zuerst bemerken wir, dass die punktweise Multiplikation zweier Funktionen $f,g \in \mathcal{C}^\infty (\mathcal{M})$ erklärt ist als 
    \begin{align}
        (f \odot g)(p) := f(p) \cdot g(p) \:\:\:\:\:\:\:\: \forall p \in \mathcal{M}.
    \end{align}
    Daraus folgt für eine Kurve $\gamma : [a,b] \longrightarrow \mathcal{M}$, dass 
    \begin{align}
        ((f \odot g) \circ \gamma)(p) =& \: (f \odot g)(\gamma(p)) \nonumber \\ =& \: f(\gamma(p)) \cdot g(\gamma(p)) \nonumber \\ =& \: (f \circ \gamma)(p) \cdot (g \circ \gamma)(p) \nonumber \\ =& \: ((f \circ \gamma) \odot (g \circ \gamma))(p) \:\:\:\:\:\:\:\: \forall p \in \mathcal{M},
    \end{align}
    womit 
    \begin{align}
        X_{\gamma,p} (f \odot g) =& \: ((f \odot g) \circ \gamma)'(0) \nonumber \\ =& \: ((f \circ \gamma) \odot (g \circ \gamma))'(0) \nonumber \\ =& \: (f \circ \gamma)'(0) \cdot g(\gamma(0)) + f(\gamma(0)) \cdot (g \circ \gamma)'(0) \nonumber \\ =& \: X_{\gamma,p} f \cdot g(p) + f(p) \cdot X_{\gamma,p} g
    \end{align}
    folgt, d.h. $X_{\gamma,p}$ kann als spezielle Derivation aufgefasst werden. Dabei haben wir im vorletzten Schritt die gewöhnliche, aus der reellen Analysis bekannte Produktregel genutzt \cite{heuser2013lehrbuch}.
\end{remark}

Nachdem wir nun also die Tangentialräume $T_p \mathcal{M}$ kennengelernt und etwas studiert haben, um unter anderem eine Intuition für diese zu entwickeln bzw. unsere Vorstellungen bezüglicher dieser zu rechtfertigen, können wir nun den sogenannten Tangentialbündel erklären. Etwas informal ausgedrückt ist der Tangentialbündel nichts weiter als eine Zusammenfassung aller Tangentialräume der glatten Mannigfaltigkeit $(\mathcal{M}, \tau, \mathcal{A})$, wobei die Tangentialräume über den Punkt $p \in \mathcal{M}$, an dem diese definiert sind, parametrisiert sind. Formal erhalten wir die folgende Definition.

\begin{definition}
    Sei $(\mathcal{M}, \tau, \mathcal{A})$ eine $d$-dimensionale glatte Mannigfaltigkeit. Dann ist der Tangentialbündel von $\mathcal{M}$ als Menge, bezeichnet mit $T \mathcal{M}$, definiert durch
    \begin{align}
        T \mathcal{M} := \bigsqcup_{p \in \mathcal{M}} T_p \mathcal{M} := \bigcup_{p \in \mathcal{M}} \{ p \} \times T_p \mathcal{M}.
    \end{align}
\end{definition}

Anschaulich heften wir also an jedem Punkt $p$ der Mannigfaltigkeit $\mathcal{M}$ den zugehörigen Tangentialraum $T_p$ an. Diese Anschauung wird im nächsten Abschnitt wichtig werden, wenn es darum geht, glatte Vektorfelder auf einer glatten Mannigfaltigkeit zu erklären. Die Nutzung der disjunkten Vereinigung $\sqcup$ in der Konstruktion von $T \mathcal{M}$ rechtfertigt sich dabei darüber, dass per Definition der Nullvektor in jedem einzelnen Tangentialraum gegeben ist durch die Abbildung 
\begin{align}
    \mathbf{0} :\mathcal{C}^\infty (\mathcal{M}) \longrightarrow& \: \: \mathbb{R} \nonumber\\
            f \longmapsto& \: \: 0.
\end{align}
D.h. formal sind die Tangentialräume nicht disjunkt. Um also das anschauliche Bild aufrechtzuerhalten, nach welchem wir an jedem Punkt $p$ der glatten Mannigfaltigkeit den zugehörigen Tangentialraum $T_p \mathcal{M}$ anheften, sodass jeder Tangentialvektor aus $T_p \mathcal{M}$ den eindeutigen Fußpunkt $p$ besitzt, müssen wir auf die obige Konstruktion mittels der disjunkten Vereinigung zurückgreifen. 

Zum Tangentialbündel $T \mathcal{M}$ lässt sich nun weiter eine sogenannte Bündelprojektion $\pi$ erklären.

\begin{definition}
    Sei $(\mathcal{M}, \tau, \mathcal{A})$ eine $d$-dimensionale glatte Mannigfaltigkeit und $T \mathcal{M}$ der zugehörige Tangentialbündel. Dann ist die Bündelprojektion des Tangentialbündels erklärt als die Abbildung
    \begin{align}
        \pi : T \mathcal{M} \longrightarrow& \: \: \mathcal{M} \nonumber\\
            (p, X_{\gamma,p}) \longmapsto& \: \: p.
    \end{align}
    Im Folgenden werden wir statt $(p, X_{\gamma.p}) \in T \mathcal{M}$ oft auch nur kurz $X_{\gamma, p}$ schreiben. Die Bedeutung der Abbildung $\pi$ ist dabei, dass diese dem Tangentialvektor $X_{\gamma,p}$ dem zugehörigen Fußpunkt $p \in \mathcal{M}$ zuordnet, an dem der Tangentialvektor $X_{\gamma,p}$ bzw. der zu $X_{\gamma, p}$ zugehörige Tangentialraum $T_p \mathcal{M}$ anschaulich angeheftet ist. Oft wird auch erst das Tripel $(T \mathcal{M}, \mathcal{M}, \pi)$, welches manchmal auch mit $T \mathcal{M} \xrightarrow{ \:\:\: \pi \:\:\:} \mathcal{M}$ bezeichnet wird, als Tangentialbündel bezeichnet.
\end{definition}

Der Tangentialbündel $T \mathcal{M}$ kann nun noch zusätzlich mit einer kanonischen und natürlichen Topologie ausgestattet werden, welche mehr oder weniger durch die Topologie $\tau$ der glatten Mannigfaltigkeit $(\mathcal{M}, \tau, \mathcal{A})$ induziert wird, sodass $T \mathcal{M}$ selbst zu einer topologischen Mannigfaltigkeit wird. Genauer können wir den Tangentialbündel $T \mathcal{M}$ mittels dieser Topologie als $2d$-dimensionale topologische Mannigfaltigkeit verstehen, wenn es sich bei $(\mathcal{M}, \tau, \mathcal{A})$ um eine $d$-dimensionale Mannigfaltigkeit handelt.

Darüberhinaus lässt sich der Tangentialbündel auch mit einer glatten Struktur ausstatten, sodass wir $T \mathcal{M}$ als $2d$-dimensionale glatte Mannigfaltigkeit auffassen können. An dieser Stelle sollte man sich nun die Frage stellen, warum wir überhaupt den Tangentialbündel als glatte Mannigfaltigkeit auffassen wollen. 

Wie bereits angesprochen, wollen wir uns im kommenden Abschnitt unter anderem mit glatten Vektorfeldern beschäftigen, um mittels dieser auf intuitive Weise den Begriff der sogenannten $(r,s)$-Tensorfelder einführen zu können, die später noch eine große Rolle spielen werden, wenn wir das sogenannte Längenfunktional definieren wollen. 

Wie der Name es schon andeutet, sollte es sich bei einem Vektorfeld dabei anschaulich gesprochen um eine Funktion handeln, die jedem Punkt der Mannigfaltigkeit $(\mathcal{M}, \tau, \mathcal{A})$ einen Vektor zuordnet. Da intuitiv gesprochen an jedem Punkt $p$ der Mannigfaltigkeit ein Tangentialraum $T_p \mathcal{M}$ angeheftet ist, der die Freiheitsgrade der Mannigfaltigkeit ausschöpft, macht es Sinn davon auszugehen, dass ein Vektorfeld auf $(\mathcal{M}, \tau, \mathcal{A})$ jedem Punkt $p \in \mathcal{M}$ einen Tangentialvektor $X \in T_p \mathcal{M}$ zuordnet. D.h., dass wir ein Vektorfeld als eine Funktion der Form $\mathcal{M} \longrightarrow T \mathcal{M}$ auffassen können. 

Nun sollten anschaulicherweise viele aus der Physik bekannten Phänomene, wie das Geschwindikeitsfeld einer Flüssigkeit, über ein Vektorfeld im obigen Sinne beschreibbar sein, da man bei Phänomenen dieser Art abermals jedem Raumpunkt einen (Geschwindigkeits-) Vektor zuordnen kann, wobei diese Zuordnung intuitiv von Punkt zu Punkt \textit{stetig} oder gar \textit{glatt} variieren sollte (was auch immer das zu diesem Zeitpunkt heißen mag). Wenn wir also differentialgeometrisch einen Vektorfeldbegriff entwickeln wollen, der darüberhinaus in der Lage ist, die eben angesprochenen physikalischen Phänomene sinnvoll zu modellieren, so sollten wir sicherstellen, dass wir auch in der Lage sind, über stetige oder gar glatte Vektorfelder zu sprechen.

Damit wir rein theoretisch eine Chance haben, sogenannte glatte Vektorfelder definieren zu können, müssen wir das Vektorfeld selbst als Funktion zwischen zwei glatten Mannigfaltigkeiten auffassen. Da wir oben bereits erläutert hatten, dass ein Vektorfeld als Funktion von der Form $\mathcal{M} \longrightarrow T \mathcal{M}$ sein sollte, folgt daraus unmittelbar, dass wir versuchen sollten, den Tangentialbündel $T \mathcal{M}$ als glatte Mannigfaltigkeit aufzufassen.

Im Folgenden umreissen wir daher kurz, wie sich $T \mathcal{M}$ mit einer Topologie und einer glatten Struktur ausstatten lässt. Dabei beziehen wir uns im wesentlichen auf \cite{warner1983foundations}.

Zuerst müssen wir uns der Tatsache vergegenwärtigen, dass wir einen natürlichen und auch nützlichen Nähebegriff auf $T \mathcal{M}$ erklären wollen, damit am Ende stetige oder gar glatte Vektorfelder genau das sind, was wir anschaulich auch meinen, wenn wir von stetigen oder glatten Vektorfeldern sprechen. Wie könnte also ein intuitiver Nähebegriff auf $T \mathcal{M}$ erklärt werden?

Wir betrachten dazu die als $d$-dimensional angenommene glatte Mannigfaltigkeit $(\mathcal{M}, \tau, \mathcal{A})$  und den zu dieser zugehörigen Tangentialbündel $T \mathcal{M}$. Weiter betrachten wir eine Karte $(\mathcal{U}, \phi) \in \mathcal{A}$ von $(\mathcal{M}, \tau, \mathcal{A})$ mit $\phi = (x^1, ..., x^d)$. Wir erinnern uns, dass die Topologie $\tau$ einen Nähebegriff auf der Menge $\mathcal{M}$ einführt und das in der Karte $(\mathcal{U}, \tau)$ dieser Nähebegriff indirekt codiert ist, da $\mathcal{U} \subseteq \tau$ und $\phi$ $(\tau_\mathcal{U}, \tau_{\mathbb{R}^d, \phi(\mathcal{U})})$-stetig ist. 

Da $\mathcal{M}$ im Tangentialbündel $T \mathcal{M}$ verbaut ist, und wir anschaulich die topologische Struktur von $\mathcal{M}$ in $T \mathcal{M}$ erhalten wollen, macht es Sinn, dass wir zuerst fordern, dass die Menge $T \mathcal{U} \subseteq T \mathcal{M}$, definiert durch
\begin{align}
    T \mathcal{U} := \bigsqcup_{p \in \mathcal{U}} T_p \mathcal{M} := \bigcup_{p \in \mathcal{U}} \{ p \} \times T_p \mathcal{M}.
\end{align} 
in $T \mathcal{M}$ eine offen Menge und gleichzeitig auch ein Kartengebiet sein soll. Wir stellen uns nun die Frage, wie die Kartenabbildung $\xi$ auf $T \mathcal{U}$ auszusehen hat. Zum einen wollen wir nämlich, dass die topologische Struktur von $\mathcal{M}$, die in der Kartenabbildung $\phi$ codiert ist, über die Karte $\xi$ auf $T \mathcal{M}$ übertragen wird. Zum anderen wollen wir, dass die einzelnen Tangentialräume stetig vom Fußpunkt $p$ abhängen. Wir wollen also auf der einen Seite topologische Eigenschaften von $\mathcal{M}$ auf $T \mathcal{M}$ übertragen, und zum anderen codieren, wie $T_p \mathcal{M}$ von $p \in \mathcal{M}$ abhängt. 

Wir betrachten daher ein $X \in T_p \mathcal{M}$ mit $p \in \mathcal{U}$. Bezüglich der Karte $(\mathcal{U}, \phi)$ können wir $X$ in der von der Karte induzierten Basis $\{ \big( \frac{\partial}{\partial x^1} \big)_p, ..., \big( \frac{\partial}{\partial x^d} \big)_p \}$ als 
\begin{align}
    X = \sum_{i=1}^d X^{i} \cdot \bigg( \frac{\partial}{\partial x^{i}} \bigg)_p.
\end{align}
entwickeln. Dabei sind die $X^{i}$ für alle $i \in \{ 1, ..., d \}$ induziert durch die Kartenabbildung $\phi$. Wir können nun also einerseits den Punkt $p \in \mathcal{U}$ über $\phi$ durch die Koordinaten $(x^1(p), ..., x^d(p))$ ausdrücken, und andererseits, ebenfalls durch $\phi$, den Tangentialvektor $X$ durch die Koordinaten $(X^1, ..., X^d)$ beschreiben. Mittels der Kartenabbildung $\phi$ können wir also beides, Fußpunkt und zugehörige Tangentialvektoren, einheitlich beschreiben. Anschaulich ist in dieser einheitlichen Beschreibung ein Zusammenhang zwischen $T_p \mathcal{M}$ und $p \in \mathcal{U}$ codiert. 

Das motiviert die folgende Kartenabbildung $\xi$ auf $T \mathcal{U}$, die wir wegen ihrer Abhängigkeit von der Kartenabbildung $\phi$ mit $\xi_\phi$ bezeichnen wollen: 
\begin{align}
    \xi_\phi : T \mathcal{U} \longrightarrow& \: \: \phi(\mathcal{U}) \times \mathbb{R}^d \nonumber\\
            (p, X) \longmapsto& \: \: ((x^1(p), ..., x^d(p)), (X^1, ..., X^d)). \label{Karte des Tangentialbündels}
\end{align}
Da es sich bei der Kartenabbildung $\phi$ um einen Diffeomorphismus handelt, womit diese insbesondere auch bijektiv ist, folgt, dass es sich auch bei der Abbildungum $\xi_\phi$ um eine Bijektion handelt, sodass wir $\xi_\phi$ auch tatsächlich als Kandidtat für eine Kartenabbildung auf $T \mathcal{U}$ in Betracht ziehen können.

Da es sich bei $T \mathcal{M}$ am Ende um eine Mannigfaltigkeit handeln soll, muss diese lokal homöomorph zu einem euklidischen Raum sein. Die Abbildung $\xi_\phi$ legt Nahe, da es sich bei diesem euklidischen Raum um den $\mathbb{R}^{2d}$ handeln dürfte. Um die lokale Ähnlichkeit zum $\mathbb{R}^{2d}$ zu gewährleisten, wollen wir anschaulich die Topologie des $\mathbb{R}^{2d}$ mittels $\xi_\phi$ auf $T \mathcal{U}$ übertragen.  

Dazu statten wir $\mathbb{R}^d \times \mathbb{R}^d =: \mathbb{R}^{2d}$ zuerst mit der Standardtopologie aus, sodass wir auf der Menge $\phi(\mathcal{U}) \times \mathbb{R}^d \subseteq \mathbb{R}^d \times \mathbb{R}^d$ die zugehörige Teilraumtopologie erklären können. Damit wird $\phi(\mathcal{U}) \times \mathbb{R}^d$ zu einem topologischen Raum. Mittels der Abbildung $\xi_\phi$ und dem Konzept der Pullback-Topologie aus Lemma $2.1.$ können wir nun eine Topologie auf $T \mathcal{U}$ induzieren, welche wir mit $\tau_{T \mathcal{U}}$ bezeichnen wollen. 

Machen wir das Obige für alle Karten $(\mathcal{U}, \phi)$ aus $\mathcal{A}$, so lässt sich auf dem ganzen Tangentialbündel $T \mathcal{M}$ eine Topologie erklären, indem wir die kleinste Topologie auf $T \mathcal{M}$ bilden, welche die $\tau_{T \mathcal{U}}$ für alle $(\mathcal{U}, \phi) \in \mathcal{A}$ enthält. Mit dieser Topologie wird dann $T \mathcal{M}$ zu einer $2d$-dimensionalen topologischen Mannigfaltigkeit, wobei die Karten von $T \mathcal{M}$ gerade durch die Tupel $(T \mathcal{U}, \xi_\phi)$ gegeben sind. Man kann darüberhinaus zeigen, dass bezüglich dieser Karten, da $(\mathcal{M}, \tau, \mathcal{A})$ eine glatte Mannigfaltigkeit ist, auch $T \mathcal{M}$ zu einer glatten Mannigfaltigkeit wird.

Im nächsten Abschnitt werden wir noch sehen, dass die glatten Vektorfelder auf $(\mathcal{M}, \tau. \mathcal{A})$, die bezüglich der eben erklärten Topologie auf $T \mathcal{M}$ definiert werden, die aus der Vektoranalysis bekannten anschaulichen Vektorfelder verallgemeinern, was eine weitere Rechtfertigung der Nutzung dieser Topologie liefert.   


\subsection{Glatte Vektorfelder und glatte Schnitte im Tensorbündel}

Durch die Tatsache, dass wir den Tangentialbündel $T \mathcal{M}$ selbst als Mannigfaltigkeit verstehen können, sind wir in der Lage, gemäß der Definition $6.8.$ auch glatte Funktionen der Form $s : \mathcal{M} \longrightarrow T \mathcal{M}$ zwischen $\mathcal{M}$ und $T \mathcal{M}$ zu betrachten. Eine wichtige Klasse derartiger glatter Funktionen sind die sogenannten glatten Schnitte, die man oft auch einfach als glatte Vektorfelder bezeichnet, und die besonders in der Physik Anwendung finden, um dort beispielsweise physikalische Felder zu modellieren.

\begin{definition}
    Sei $(\mathcal{M} , \tau, \mathcal{A})$ eine $d$-dimensionale glatte Mannigfaltigkeit und $T \mathcal{M}$ der zugehörige Tangentialbündel, gesehen als glatte Mannigfaltigkeit im obigem Sinne. Unter einem glatten Schnitt im Tangentialbündel verstehen wir eine glatte Funktion 
    \begin{align}
        s : \mathcal{M} \longrightarrow T \mathcal{M},
    \end{align}
    mit der zusätzlichen Eigenschaft, dass
    \begin{align}
        \pi \circ s = \textit{id}_\mathcal{M} \label{section}
    \end{align}
    mit $\textit{id}_{\mathcal{M}}(p) = p$ für alle $p \in \mathcal{M}$ ist. Die Gleichung \eqref{section} stellt dabei sicher, dass dem Punkt $p \in \mathcal{M}$ über die Abbildung $s$ nur ein Tangentialvektor zugeordnet werden kann, der auch im Tangentialraum $T_p \mathcal{M}$ liegt, welcher am Punkt $p$ definiert ist. Die Menge aller glatten Schnitte im Tangentialbündel bezeichnen wir dabei mit $\Gamma^\infty (\mathcal{M}, T \mathcal{M})$.
\end{definition}

Aus der Definition eines glatten Schnittes im Tangentialbündel wird sofort klar, warum dieser auch oft als glattes Vektorfeld auf $\mathcal{M}$ bezeichnet wird. Zum einen ist nämlich ein (glattes) Vektorfeld auf $\mathcal{M}$ eine Funktion, die jedem Punkt der Mannigfaltigkeit einen Vektor an diesem Punkt zuordnet, was durch \eqref{section} codiert ist. Zum anderen ist ein (glattes) Vektorfeld aber auch so etwas wie ein Geschwindigkeitsfeld einer Strömung auf $\mathcal{M}$, weshalb die Vektoren, die den einzelnen Punkten zugeordnet werden, tangential an der Mannigfaltigkeit an diesem Punkt anliegen sollten, da die Bewegung der Strömung auf die Mannigfaltigkeit eingeschränkt ist. Daher muss anschaulich der Geschwindikeitvektor der Strömung an jedem Punkt der Mannigfaltigkeit auch dort tangential an der Mannigfaltigkeit anliegen. 

Als nächstes wollen wir nun den Begriff des glatten Vektorfeldes auf $\mathcal{M}$ verallgemeinern. Dazu erinnern wir uns an das Beispiel $4.5.$, indem wir gesehen haben, dass wir die Elemente eines Vektorraumes als $(1,0)$-Tensoren verstehen können. Das heißt wir können den Vektorraum $T_p \mathcal{M}$ auch schreiben als $T^1_0 (T_p \mathcal{M})$ und den Tangentialbündel damit als 
\begin{align}
    T \mathcal{M} = \bigsqcup_{p \in \mathcal{M}} T_p \mathcal{M} = \bigsqcup_{p \in \mathcal{M}} T^1_0 (T_p \mathcal{M}) =: T^1_0 \mathcal{M}. \label{(1,0)-tensor bundle}
\end{align}
Aus \eqref{(1,0)-tensor bundle} wird klar, wie wir den anschaulichen Vektorfeldbegriff sinnvoll verallgemeinern können.

\begin{definition}
    Sei $(\mathcal{M}, \tau, \mathcal{A})$ eine $d$-dimensionale glatte Mannigfaltigkeit. Für die Zahlen $r,s \in \mathbb{N}_0$ definieren wir den sogenannten $(r,s)$-Tensorbündel $T^r_s \mathcal{M}$ als Menge durch
    \begin{align}
        T^r_s \mathcal{M} = \bigsqcup_{p \in \mathcal{M}} T^r_s (T_p \mathcal{M}) = \bigcup_{p \in \mathcal{M}} \{ p \} \times T^r_s (T_p \mathcal{M}).
    \end{align}
    Dabei ist $T^r_s (T_p \mathcal{M})$ wie in Abschnitt $4.5$ definiert als die Menge aller multilinearer Abbildungen der Form 
    \begin{align}
        T: \underbrace{T_p \mathcal{M} \times ... \times T_p \mathcal{M}}_{s} \times \underbrace{T_p^* \mathcal{M} \times ... \times T_p^* \mathcal{M}}_{r} \longrightarrow \mathbb{R}.
    \end{align}
\end{definition}

Auch $T^r_s \mathcal{M}$ lässt sich mit einer kanonischen und natürlichen Topologie ausstatten \cite{goldberg2011curvature}, sodass zum einen $T^r_s \mathcal{M}$ zu einer glatten Mannigfaltigkeit wird und zum anderen die Topologie von $T^r_s \mathcal{M}$ für $r = 1$ und $s = 0$ mit derjenigen Topologie übereinstimmt, die wir zuvor genutzt hatten, um $T \mathcal{M}$ zu einer glatten Mannigfaltigkeit zu machen. Diese kanonische Topologie auf $T^r_s \mathcal{M}$ wird dabei fast analog zur Topologie auf $T \mathcal{M}$ konstruiert.

Auch zum $(r,s)$-Tensorbündel $T^r_s \mathcal{M}$ lässt sich nun eine sogenannte Bündelprojektion $\pi^r_s$ definieren.

\begin{definition}
    Sei $(\mathcal{M}, \tau, \mathcal{A})$ eine $d$-dimensionale glatte Mannigfaltigkeit und $T^r_s \mathcal{M}$ das zugeförige $(r,s)$-Tensorbündel. Dann ist die sogenannte $(r,s)$-Bündelprojektion erklärt durch 
    \begin{align}
        \pi^r_s : T^r_s \mathcal{M} \longrightarrow& \: \: \mathcal{M} \nonumber\\
            (p, T) \longmapsto& \: \: p \:\:\:\: \textit{mit} \:\: T \in T^r_s(T_p \mathcal{M}).
    \end{align}
    Im Folgenden werden wir dabei wieder statt $(p, T) \in T^r_s \mathcal{M}$ kurz einfach nur $T_p$ schreiben.  
\end{definition}

Weiter definieren wir, analog zum glatten Schnitt im Tangentialbündel, den Begriff des glatten Schnittes im $(r,s)$-Tensorbündel:

\begin{definition}
    Sei $(\mathcal{M}, \tau, \mathcal{A})$ eine $d$-dimensionale glatte Mannigfaltigkeit und $T^r_s \mathcal{M}$ das zugehörige $(r,s)$-Tensorbündel, gesehen als glatte Mannigfaltigkeit im obigem Sinne. Dann ist ein glatter Schnitt im $(r,s)$-Tensorbündel eine glatte Funktion 
    \begin{align}
        \sigma : \mathcal{M} \longrightarrow T^r_s \mathcal{M}
    \end{align}
    mit der zusätzlichen Eigenschaft, dass
    \begin{align}
        \pi^r_s \circ \sigma = \textit{id}_\mathcal{M}
    \end{align}
    ist. Da $\sigma$ damit anschaulich jedem Punkt $p \in \mathcal{M}$ einen Tensor $T_p \in T^r_s (T_p \mathcal{M})$ zuordnet, nennt man $\sigma$ auch glattes $(r,s)$-Tensorfeld. Weiter bezeichnen wir die Menge aller glatten Schnitte im $(r,s)$-Tensorbündel mit $\Gamma^\infty (\mathcal{M}, T^r_s \mathcal{M})$.
\end{definition}

Einen weitere nützlicher Begriff im Zusammenhang mit glatten $(r,s)$-Tensorfeldern ist der Begriff der sogenannten Komponentenfunktionen eines glatten $(r,s)$-Tensorfeldes. Da wir im Folgenden nur einen Spezialfall dieses Begriffes benötigen, nämlich den Begriff der karteninduzierten Komponentenfunktionen eines glatten $(r,s)$-Tensorfeldes, geben wir hier nicht die allgemeine Definition an. 

Um den Begriff der karteninduzierten Komponentenfunktionen eines $(r,s)$-Tensorfeldes $\sigma$ einzuführen, betrachten wir auf unserer $d$-dimensionalen glatten Mannigfaltigkeit $(\mathcal{M}, \tau, \mathcal{A})$ zuerst eine Karte $(\mathcal{U}, \phi) \in \mathcal{A}$. 

An jedem Punkt $p \in \mathcal{U} \subseteq \mathcal{M}$ können wir den Tangentialraum $T_p \mathcal{M}$ bilden. Die Basis in $T_p \mathcal{M}$ sei dabei gegeben durch die von der Karte $(\mathcal{U}, \phi)$ induzierten Basis $\{ \big( \frac{\partial}{\partial x^1} \big)_p, ..., \big( \frac{\partial}{\partial x^d} \big)_p \} \subseteq T_p \mathcal{M}$. 

Zu jedem Tangentialraum $T_p \mathcal{M}$ mit $p \in \mathcal{U}$ können wir den zugehörigen Dualraum von $T_p \mathcal{M}$, den Kotangentialraum $T_p^* \mathcal{M}$, bilden. Gemäß dem Satz $4.10.$ aus dem Abschnitt $4.5$ können wir im Kotangentialraum $T_P^* \mathcal{M}$ immer eine zur Basis $\{ \big( \frac{\partial}{\partial x^1} \big)_p, ..., \big( \frac{\partial}{\partial x^d} \big)_p \}$ duale Basis finden, die wir mit $\{ d_p x^1, ..., d_p x^d \} \subseteq T_p^* \mathcal{M}$ bezeichnen. Es gilt also für alle $p \in \mathcal{U}$
\begin{align}
    d_p x^{i} \bigg( \bigg( \frac{\partial}{\partial x^j} \bigg)_p \bigg) = \delta^{i}_j \:\:\:\:\:\:\:\: \forall i,j \in \{1,...,d\}.
\end{align}
Wir betrachten nun das glatte $(r,s)$-Tensorfeld $\sigma$, für das gilt
\begin{align}
    \sigma : \mathcal{M} \longrightarrow& \: \: T^r_s \mathcal{M} \nonumber\\
            p \longmapsto& \: \: \sigma(p) := T_p \in T^r_s (T_p \mathcal{M}).
\end{align}
Dann können wir $T_p$ an der Basis $\{ \big( \frac{\partial}{\partial x^1} \big)_p, ..., \big( \frac{\partial}{\partial x^d} \big)_p \}$ und $\{ d_p x^1, ..., d_p x^d \}$ auswerten, was die karteninduzierten Komponentenfunktionen von $\sigma$ am Punkt $p \in \mathcal{U}$ definiert
\begin{align}
    (T^{i_1, ..., i_s}_{j_1, ..., j_r})(p) := T_p \bigg( \bigg( \frac{\partial}{\partial x^{i_1}} \bigg)_p, ..., \bigg( \frac{\partial}{\partial x^{i_s}} \bigg)_p, d_p x^{j_1}, ..., d_p x^{j_r} \bigg), \label{Tensorkomponenten bzgl einer Karte}
\end{align}
wobei die $i_1, ..., i_s, j_1, ..., j_r \in \{ 1, ..., d \}$ sind. Die karteninduzierten Komponentenfunktionen vom $(r,s)$-Tensorfeld $\sigma$ auf $\mathcal{U}$ sind dann für $$i_1, ..., i_s, j_1, ..., j_r \in \{ 1, ..., d \}$$ erklärt durch 
\begin{align}
    T^{i_1, ..., i_s}_{j_1, ..., j_r} : \mathcal{U} \longrightarrow& \: \: \mathbb{R} \nonumber\\
            p \longmapsto& \: \: (T^{i_1, ..., i_s}_{j_1, ..., j_r})(p).
\end{align}
Sobald wir später über sogenannte riemannsche Mannigfaltigkeiten sprechen werden, werden diese karteninduzierten Komponentenfunktionen von $(r,s)$-Tensorfeldern wieder von Relevanz sein.

Darüberhinaus lässt sich mit diesen karteninduzierten Komponentenfunktionen insbesondere eine weitere anschauliche Rechtfertigung der Verwendung der in \cite{goldberg2011curvature} angebenen Topologie auf dem $(r,s)$-Tensorbündel finden, die wir genutzt hatten, um glatte Schnitte im $(r,s)$-Tensorbündel zu erklären. So lässt sich nämlich zeigen, dass $\sigma \in \Gamma^\infty (\mathcal{M}, T^r_s \mathcal{M})$ genau dann gilt, falls für alle karteninduzierten Komponentenfunktionen gilt, dass 
\begin{align}
    T^{i_1, ..., i_s}_{j_1, ..., j_r} \in \mathcal{C}^\infty (\mathcal{M})
\end{align}
ist $\cite{tu2011manifolds}$. Warum ist das eine Rechtfertigung der Verwendung der in  \cite{goldberg2011curvature} angegebenen Topologie bzgl. $T^r_s \mathcal{M}$? 

Betrachten wir dazu den $\mathbb{R}^d$, den wir wieder mit der in Bemerkung $6.1.$ angegebenen glatten Struktur versehen, sodass wir den $\mathbb{R}^d$ als glatte Mannigfaltigkeit auffassen können. Betrachten wir die Karte $(\mathbb{R}^d, \textit{id}_{\mathbb{R}^d})$ mit $\textit{id}_{\mathbb{R}^d} = (x^1, ..., x^d)$ und die zur Karte zugehörige karteninduzierte Basis $\{ \big( \frac{\partial}{\partial x^1} \big)_p = e_{1,p}, ..., \big( \frac{\partial}{\partial x^d} \big)_p = e_{d,p} \} \in T_p \mathbb{R}^d$ mit $p \in \mathbb{R}^d$. Wir identifizieren nun für alle $p \in \mathbb{R}^d$ und alle $i \in \{ 1, ..., d \}$ den Tangentialvektor $\big( \frac{\partial}{\partial x^{i}} \big)_p = e_{i,p} \in T_p \mathbb{R}^d$ mit dem Vektor $e_i \in \mathbb{R}^d$, definiert durch
\begin{align}
    e_i := (0, ..., 0, \underbrace{1}_{\textit{$i$-te Stelle}}, 0, ..., 0) = \partial_i \textit{id}_{\mathbb{R}^d}^{-1} (p) \:\:\:\: \forall p \in \mathbb{R}^d. \label{oav}
\end{align}
Diese Identifikation dürfen wir vornehmen, da sowohl $T_p \mathbb{R}^d$ für alle $p \in \mathbb{R}^d$, als auch $\mathbb{R}^d$ die selbe Vektorraumdimension besitzen, womit diese zueinander isomorph sind \cite{fischer2003lineare}, d.h. $T_p \mathbb{R}^d \cong \mathbb{R}^d$. 

Beachte, dass dabei die geometrische Bedeutung von $e_{i,p}$ unter der Identifikation mit $e_i$ nicht verloren geht, da beide anschaulich geometrisch die selbe Orientierung besitzen. Da sich aus \eqref{oav} ergibt, dass sich sowohl $e_{i,p}$, als auch $e_i$ mittels der selben geometrischen Intuition aus der Karte $(\mathbb{R}^d, \textit{id}_{\mathbb{R}^d})$ konstruieren lassen, erscheint es plausibel, die obige Identifikation als natürlich zu bezeichnen, da scheinbar in beiden Ausdrücken die selben geometrischen Informationen codiert sind.

Die Natürlichkeit der obigen Identifikation wird sich später noch als wesentlich evidenter herausstellen, wenn wir den $\mathbb{R}^d$, betrachtet als glatte Mannigfaltigkeit, mit der noch zu diskutierenden riemannschen Metrik aus \eqref{spezielle riemannsche Metrik} versehen, wodurch $e_{i,p}$ die geometrische Länge von $1$ erhält, was der anschaulichen Länge von $e_i$ entspricht. 

Unter einem glatten Vektorfeld auf der glatten Mannigfaltigkeit $\mathbb{R}^d$ im Sinne der gewöhnlichen Vektoranalysis verstehen wir nun eine Funktion der Form
\begin{align}
    \mathbf{v} = (v^1,...,v^d) : \mathbb{R}^d \longrightarrow& \: \: \mathbb{R}^d \nonumber\\
            p \longmapsto& \: \: \mathbf{v}(p) = (v^1(p), ..., v^d(p)) = \sum_{i = 1}^d v^{i}(p) \cdot e_i,
\end{align}
wobei $v^1, ..., v^d \in \mathcal{C}^\infty (\mathbb{R}^d)$ ist. Andererseits verstehen wir in der Differentialgeometrie unter einem glatten Vektorfeld auf $\mathbb{R}^d$ einen glatten Schnitt $s$ im Tangentialbündel $T \mathbb{R}^d = T^1_0 \mathbb{R}^d$. D.h. unter anderem, dass die glatte Funktion $s$ jedem Punkt $p \in \mathbb{R}^d$ einen Tangentialvektor $X_{\gamma(p),p} \in T_p \mathbb{R}^d$ zuordnet, wobei wir den Tangentialvektor $X_{\gamma(p), p}$ auch in der Basis $\{e_{1,p}, ..., e_{d,p}\} \subseteq T_p \mathbb{R}^d$ entwickeln können, was die Darstellung
\begin{align}
    s(p) = X_{\gamma(p),p} = \sum_{i = 1}^d v_{\gamma(p)}^{i}(p) \cdot e_{i,p}
\end{align}
mit $v_{\gamma(p)}^{1}(p), ..., v_{\gamma(p)}^{d}(p) \in \mathbb{R}$ liefert. Dabei bedeutet $\gamma(p)$ lediglich, dass die Kurve $\gamma$ bzgl. derer der Tangentialvektor am Punkt $p \in \mathbb{R}^d$ konstruiert wird, vom Punkt $p$ selbst abhängig ist. Die bzgl. zur Karte $(\mathbb{R}^d, \textit{id}_{\mathbb{R}^d})$ zugehörigen Komponentenfunktionen von $s$ sind dann die $v_{\gamma}^{i}$, die nach obiger Bemerkung zur Klasse $\mathcal{C}^\infty (\mathbb{R}^d)$ gehören. Indem wir nun also für alle $i \in \{ 1, ..., d \}$ die geometrisch sinnvolle Idenfikation $e_{i,p} \equiv e_i$ für alle $p \in \mathbb{R}^d$ vornehmen, können wir das glatte Vektorfeld $s$ mit dem im Rahmen der Vektoranalysis definiertem glatten Vektorfeld 
\begin{align}
    \mathbf{v}_\gamma = (v_\gamma^1,...,v_\gamma^d) : \mathbb{R}^d \longrightarrow& \: \: \mathbb{R}^d \nonumber\\
            p \longmapsto& \: \: \mathbf{v}_{\gamma(p)}(p) = (v_{\gamma(p)}^1(p), ..., v_{\gamma(p)}^d(p)) = \sum_{i = 1}^d v_{\gamma(p)}^{i}(p) \cdot e_i,
\end{align}
identifizieren. Auf diese Weise überträgt sich der aus der Vektoranalysis bekannte Begriff eines glatten Vektorfeldes in die moderne Differentialgeometrie. Das diese Übertragbarkeit möglich ist, wird dabei unter anderem von der Wahl der Topologie auf $T \mathbb{R}^d$ oder allgemeiner auf $T^r_s \mathbb{R}^d$ garantiert, was eine Rechtfertigung für die Nutzung der in \cite{goldberg2011curvature} angegebenen Topologie für das $(r,s)$-Tensorbündel liefert.

In dem nun folgenden Kapitel werden wir eine besondere Klasse von Schnitten im $(0,2)$-Tensorbündel kennen lernen, mit deren Hilfe wir letztlich in der Lage sein werden, einen anschaulichen Längenbegriff auf einer glatten Mannigfaltigkeit $(\mathcal{M}, \tau, \mathcal{A})$ zu erklären.


\newpage

\quad
\newpage

\section{Riemannsche Mannigfaltigkeiten}

In diesem Kapitel wollen wir uns mit dem Begriff einer riemannschen Mannigfaltigkeit auseinandersetzen. Salopp gesprochen handelt es sich dabei um glatte Mannigfaltigkeiten, bei denen jeder einzelne Tangentialraum mit einem Skalarprodukt versehen wird. 

Aus dem Abschnitt $4.3$ wissen wir bereits, dass ein Prähilbertraum aufgrund seines Skalarproduktes eine vom Skalarprodukt induzierte Geometrie trägt. Die grobe Idee der sogenannten riemannschen Geometrie ist es nun, Geometrien auf jedem einzelnen Tangentialraum einer glatten Mannigfaltigkeit zu erklären, um so eine Geometrie auf der glatten Mannigfaltigkeit selbst zu induzieren. 

Mit einer derartigen Geometrie kann man sich dann bestimmten geometrischen Fragestellungen auf der glatten Mannigfaltigkeit widmen, wie beispielsweise der Frage, was die kürzeste Verbindungsstrecke zweier Punkte auf der glatten Mannigfaltigkeit ist. 

Die Beantwortung dieser Frage hängt dabei von der auf der Mannigfaltigkeit erklärten Geometrie ab. Eine anschauliche Erklärung dafür bietet abermals die Kugeloberfläche $\mathbb{S}^2$ (die wir uns wieder mit der gewöhnlichen Standardgeoemtrie versehen denken), auf der die kürzeste Verbindungsstrecke zweier Punkte anschaulich eine gekrümmte Kurve ist, statt einer einfachen Geraden, da die Oberfläche der Kugel selbst eine gekrümmte Fläche ist.   

Wir beginnen mit der Definition einer riemannschen Mannigfaltigkeit:

\begin{definition}
    Sei $(\mathcal{M}, \tau, \mathcal{A})$ eine $d$-dimensionale glatte Mannigfaltigkeit. Wir bezeichnen den glatten Schnitt $g \in \Gamma^\infty (\mathcal{M}, T^0_2 \mathcal{M})$ als riemannsches Tensorfeld oder als riemannsche Metrik, falls für alle $p \in \mathcal{M}$ gilt, dass $g_p := g(p) \in T^0_2 (T_p \mathcal{M})$ ein Skalarprodukt auf $T_p \mathcal{M}$ ist. Formal heißt das, dass für den glatten Schnitt $g$ die folgenden zusätzlichen Aussagen für alle Punkte $p \in \mathcal{M}$ gelten:
    \begin{itemize}
        \item[\textit{i)}] Die Abbildung
        \begin{align}
            g_p : T_p \mathcal{M} \times T_P \mathcal{M} \longrightarrow \mathbb{R}
        \end{align}
        ist eine bilineare Abbildung, d.h. linear in beiden Argumenten.
        \item[\textit{ii)}] Die Abbildung $g_p$ ist symmetrisch, d.h. 
        \begin{align}
            g_p (X_{\gamma,p}, X_{\sigma, p}) = g_p (X_{\sigma,p}, X_{\gamma,p}) \:\:\:\:\:\:\:\: \forall X_{\gamma,p}, X_{\sigma,p} \in T_p \mathcal{M}.
        \end{align} 
        \item[\textit{iii)}] Die Abbildung $g_p$ ist positiv definit, d.h. für alle $X_{\gamma, p} \in T_p \mathcal{M}$ gilt
        \begin{align}
            g_p (X_{\gamma,p}, X_{\gamma,p}) \geq 0 \:\:\:\: \textit{und} \:\:\:\: g_p(X_{\gamma,p}, X_{\gamma,p}) = 0 \iff X_{\gamma,p} = \mathbf{0}_{T_p \mathcal{M}},
        \end{align}
        wobei $\mathbf{0}_{T_p \mathcal{M}}$ wieder das neutrale Element oder der Nullvektor von $T_p \mathcal{M}$ ist.
    \end{itemize}
    Wir bezeichnen weiter das Quadrupel $(\mathcal{M}, \tau, \mathcal{A}, g)$ als riemannsche Mannigfaltigkeit. 
\end{definition}

Gemäß der Definition $7.1.$ fordern wir also, dass es sich bei einer riemannschen Metrik, die die Geometrie auf $(\mathcal{M}, \tau, \mathcal{A})$ induzieren soll, um einen speziellen glatten Schnitt $g$ im $(0,2)$-Tensorbündel handelt. Die Forderung nach der Glattheit liegt dabei anschaulich darin begründet, dass wir gemäß der auf $T^0_2 \mathcal{M}$ erklärten Topologie gewährleisten wollen, dass die einzelnen durch $g$ induzierten Geometrien in den Tangentialräumen $T_p \mathcal{M}$ mit $p \in \mathcal{M}$ auf glatte Weise miteinander zusammenhängen. Man könnte auch salopp sagen, das wir durch die Glattheitsforderung die einzelnen Geometrien in den Tangentialräumen auf glatte Weise miteinander verkleben.   

\begin{example}
    Zwei anschauliche Beispiele für riemannsche Mannigfaltigkeit sind gegeben durch:

    \begin{itemize}
    
        \item [\textit{i)}] Wir betrachten den $\mathbb{R}^3$, ausgestattet mit der aus Bemerkung $6.1.$ bekannten glatten Struktur. Weiter betrachten wir die Karte $(\mathbb{R}^3, \textit{id}_{\mathbb{R}^3}) \in \mathcal{B}$ mit $\textit{id}_{\mathbb{R}^3} = (x^1, ..., x^d)$. Darüberhinaus bezeichne $\langle \cdot , \cdot \rangle_{\mathbb{R}^3}$ das Standardskalarprodukt im $\mathbb{R}^3$ und 
        \begin{align}
            e_i := (0, ..., 0, \underbrace{1}_{\textit{$i$-te Stelle}}, 0, ..., 0) \in \mathbb{R}^3  \:\:\:\:\:\:\:\: \forall i \in \{ 1, 2, 3 \}.
        \end{align}
        Wir definieren auf dem $\mathbb{R}^3$ das glatte $(0,2)$-Tensorfeld 
        \begin{align}
            g : \mathbb{R}^3 \longrightarrow& \: \: T^0_2 \mathbb{R}^3 \nonumber\\
                p \longmapsto& \: \: g_p \:\:\:\: \textit{mit} \:\:\:\: g_p \bigg( \bigg( \frac{\partial}{\partial x^{i}} \bigg)_p, \bigg( \frac{\partial}{\partial x^{j}} \bigg)_p \bigg) := \langle  e_i , e_j \rangle_{\mathbb{R}^3}, \label{Standardmetrik im R3}
        \end{align}
        wobei $i, j \in \{ 1, 2, 3 \}$ ist. Bei dem $(0,2)$-Tensorfeld $g$ handelt es sich dabei tatsächlich um einen glatten Schnitt im $(0,2)$-Tensorbündel, da aus der Definition von $g$ folgt, dass die bezüglich der Karte $(\mathbb{R}^d, \textit{id}_{\mathbb{R}^d})$ karteninduzierten Komponentenfunktionen konstante Funktionen und damit $\mathcal{C}^\infty (\mathbb{R}^d)$-Funktionen sind. 

        Darüberhinaus handelt es sich bei $g$ wegen der Eigenschaften von $\langle \cdot , \cdot \rangle_{\mathbb{R}^3}$ auch um eine riemannsche Metrik, weshalb $(\mathbb{R}^3, \tau_{\mathbb{R}^3}, \mathcal{B}, g)$ eine riemannsche Mannigfaltigkeit darstellt. 

        \item [\textit{ii)}] Wir betrachten erneut die aus den Beispielen $6.1.$ und $6.2.$ bekannte glatte Mannigfaltigkeit $(\mathbb{S}^2, \tau_{\mathbb{R}^3, \mathbb{S}^2}, \mathcal{A}_{\mathbb{S}^2})$. Fassen wir diese als Teilmenge der aus \textit{i)} bekannten riemannschen Mannigfaltigkeit $(\mathbb{R}^3, \tau_{\mathbb{R}^3}, \mathcal{B}, g)$ auf, so können wir auf $(\mathbb{S}^2, \tau_{\mathbb{R}^3, \mathbb{S}^2}, \mathcal{A}_{\mathbb{S}^2})$ eine riemannsche Metrik $h$ induzieren, indem wir die riemannsche Metrik $g : \mathbb{R}^3 \longrightarrow T^0_2 \mathbb{R}^3$, definiert durch \eqref{Standardmetrik im R3}, auf $\mathbb{S}^2$ einschränken und festlegen, dass 
        \begin{align}
            h_p := g_p |_{T_p \mathbb{S}^2 \times T_p \mathbb{S}^2} \:\:\:\: \forall p \in \mathbb{S}^2
        \end{align}
        ist. Dabei ist der Ausdruck $g_p |_{T_p \mathbb{S}^2 \times T_p \mathbb{S}^2}$ die Einschränkung der Bilinearform $g_p : T_p \mathbb{R}^3 \times T_p \mathbb{R}^3 \longrightarrow \mathbb{R}$ auf $T_p \mathbb{S}^2 \times T_p \mathbb{S}^2$, wobei wir für alle $p \in \mathbb{S}^2$ die Menge $T_p \mathbb{S}^2$ als Teilmenge des $T_p \mathbb{R}^3$ verstehen können.

        Wie im Folgenden noch klar werden wird, handelt es sich per Konstruktion bei der riemannschen Mannigfaltigkeit $(\mathbb{S}^2, \tau_{\mathbb{R}^3, \mathbb{S}^2}, \mathcal{A}_{\mathbb{S}^2}, h)$ um ein Modell der Kugeloberfläche, welche mit der gewöhnlichen und anschaulichen Standardgeometrie versehen wurde.
        
    \end{itemize}
\end{example}

Wie induziert nun das glatte $(0,2)$-Tensorfeld $g$ eine Geometrie auf der glatten Mannigfaltigkeit? Wir fragen uns dazu zuerst, was es überhaupt bedeutet, eine Geometrie auf einer gegebenen Menge zu haben. Beispielsweise sollten wir intuitiv gesehen mittels einer Geometrie dazu in der Lage sein, über Kurven und insbesondere auch über so etwas wie Geraden sprechen zu können. 

Darüberhinaus sollten wir mittels einer Geometrie auch über einen Winkelbegriff verfügen, sodass wir die Winkel zweier sich schneidender Kurven angeben können. Da wir klassisch und auch anschaulich unter einer Verbindungsgeraden zweier Punkte die kürzeste Verbindungskurve dieser beiden Punkte verstehen, wollen wir auch auf einer glatten Mannigfaltigkeit eine Verbindungsgeraden zwischen zwei Punkten als die kürzeste Verbindungskurve dieser beiden Punkte definieren. 

Auf einer nackten glatten Mannigfaltigkeit $(\mathcal{M}, \tau, \mathcal{A})$ ist die Definition solcher Verbindungsgeraden nicht möglich, da wir noch kein Längenmaß auf der glatten Mannigfaltigkeit besitzen. D.h. das wir bisher noch nicht in der Lage sind, die Länge einer Kurve auf einer glatten Mannigfaltigkeit zu vermessen. 

Das Ziel dieses Kapitels wird es daher sein, zu zeigen, dass wir zum einen mit der riemannschen Metrik $g$ in der Lage sind, einen sinnvollen Winkelbegriff auf der glatten Mannigfaltigkeit $(\mathcal{M}, \tau, \mathcal{A})$ einzuführen, als auch zu zeigen, dass wir mittels der riemannschen Metrik $g$ auch in der Lage sind, ein sinnvolles Längenmaß zu definieren, über welches wir dann den Begriff einer Verbindungsgeraden ableiten können. 

Dabei werden wir ab jetzt statt von einer Verbindungsgeraden nur noch den Begriff einer (schwachen) Verbindungsgeodätischen bemühen, um den Begriff einer Verbindungsgeraden auf einer riemannschen Mannigfaltigkeit vom Begriff einer gewöhnlichen Verbindungsgeraden aus der klassischen Geometrie abzuheben. Damit wollen wir in diesem Kapitel also unter anderem eine Rechtfertigung für die Behauptung, dass $g$ eine Geometrie auf $(\mathcal{M}, \tau, \mathcal{A})$ induziert, finden.


\subsection{Der Winkelbegriff riemannscher Mannigfaltigkeiten}

Wir beginnen mit der Entwicklung des Winkelbegriffes auf Basis der riemannschen Metrik $g$. Dazu betrachten wir zwei injektive glatte Kurven $\gamma: (a,b) \longrightarrow \mathcal{M}$ und $\delta: (c,d) \longrightarrow \mathcal{M}$ mit $a,b,c,d \in \mathbb{R}$ und $a<b$ sowie $c<d$. Weiter fordern wir, dass sich die Kurven $\gamma$ und $\delta$ im Punkt $p \in \mathcal{M}$ schneiden, d.h. es existiert ein $\lambda_1 \in (a,b)$ und ein $\lambda_2 \in (c,d)$, sodass $\gamma(\lambda_1) = \delta(\lambda_2) = p$ ist.

Wie definieren wir nun einen Winkel zwischen den Kurven $\gamma$ und $\delta$ am Punkt $p$? Zuerst bemerken wir, dass wir demäß dem vorangegangenen Kapitel die Tangentialvektoren $X_{\gamma,p} \in T_p \mathcal{M}$ und $X_{\delta, p} \in T_p \mathcal{M}$ von $\gamma$ bzw. von $\delta$ am Punkt $p$ als am Punkt $p$ an $\gamma$ bzw. an $\delta$ tangential anliegend verstehen können. Wir nehmen nun darüberhinaus an, dass die beiden Kurven $\gamma$ und $\delta$ regulär am Punkt $p \in \mathcal{M}$ sind, d.h. $X_{\gamma,p}, X_{\delta, p} \in T_p \mathcal{M} \setminus \{ \mathbf{0}_{T_p \mathcal{M}} \}$.

Weiter wissen wir, dass uns das riemannsche Tensorfeld $g$ am Punkt $p$ ein Skalarprodukt $g_p$ auf $T_p \mathcal{M}$ liefert. Aus dem Abschnitt $4.3$ wissen wir, dass ein Skalarprodukt auf einem Vektorraum auf diesem Vektorraum eine Geometrie und damit auch einen Winkelbegriff erklärt. Wie in Gleichung \eqref{33.} erklären wir auf $T_p \mathcal{M}$ den Winkel $\theta$ zwischen $X_{\gamma,p} \in T_p \mathcal{M}$ und $X_{\delta,p} \in T_p \mathcal{M}$ als 
\begin{align}
    \textit{cos}(\theta) := \frac{g_p(X_{\gamma,p}, X_{\delta_p})}{\sqrt{ g_p (X_{\gamma,p}, X_{\gamma,p})} \cdot \sqrt{g_p(X_{\delta,p}, X_{\delta,p})} }. \label{angle...}
\end{align}
Im Abschnitt $4.3$ haben wir uns bereits mit der Plausibilität dieser Definition auseinandergesetzt und argumentiert, warum $\theta$ in \eqref{33.} bzw. in \eqref{angle...} als (abstrakter) Winkel bezeichnet werden kann. Damit liefert also \eqref{angle...} einen sinnvollen Winkelbegriff zwischen den Tangentialvektoren $X_{\gamma,p}$ und $X_{\delta,p}$. 

Da am Punkt $p$ diese beiden Tangentialvektoren als tangential an die Kurven $\gamma$ bzw. $\delta$ aufgefasst werden können, erscheint es naheliegend, den Winkel zwischen der Kurve $\gamma$ und $\delta$ am Punkt $p \in \mathcal{M}$ als den Winkel zwischen $X_{\gamma,p}$ und $X_{\delta,p}$ zu erklären. Das liefert die folgende Definition, die bereits eine partielle Rechtfertigung der Behauptung liefert, dass die riemannsche Metrik $g$ eine Geometrie auf der glatten Mannigfaltigkeit $(\mathcal{M}, \tau, \mathcal{A})$ induziert.

\begin{definition}
    Sei $(M, \tau, \mathcal{A}, g)$ eine riemannsche Mannigfaltigkeit. Weiter seien $\gamma : (a,b) \longrightarrow \mathcal{M}$ und $\delta : (c,d) \longrightarrow \mathcal{M}$ zwei injektive, glatte und reguläre Kurven mit $a,b,c,d \in \mathbb{R}$, $a<b$ und $c<d$, sodass ein $\lambda_1 \in (a,b)$ und $\lambda_2 \in (c,d)$ mit
    \begin{align}
        \gamma(\lambda_1) = \delta(\lambda_2) = p \in \mathcal{M}
    \end{align}
    existieren. Dann ist der Winkel $\theta \in [0, \pi)$ zwischen $\gamma$ und $\delta$ im Punkt $p$ erklärt durch
    \begin{align}
        \theta := \textit{cos}^{-1}\bigg(   \frac{g_p(X_{\gamma,p}, X_{\delta_p})}{\sqrt{ g_p (X_{\gamma,p}, X_{\gamma,p})} \cdot \sqrt{g_p(X_{\delta,p}, X_{\delta,p})} }  \bigg).
    \end{align}
\end{definition}


\subsection{Das Längenfunktional \texorpdfstring{$\mathbb{L}$}{} und (schwache) Geodätische}

Nachdem wir im vorangegangenen Abschnitt erfolgreich einen intuitiven Winkelbegriff mithilfe der riemannschen Mertrik $g$ erklären konnten, wollen wir nun mittels $g$ auch einen intuitiven Längenbegriff $\mathbb{L}$ definieren. Dazu wollen wir im Folgenden ein Längenfunktional $\mathbb{L}_{(a,b)}$ mit $a,b \in \mathbb{R}$ und $a<b$ definieren, dass salopp gesprochen glatte Kurven der Form $\gamma : (a,b) \longrightarrow \mathcal{M}$ entgegennimmt und diesen ihre zugehörige Länge zuordnet, die wir mit $\mathbb{L}_{[a,b]} (\gamma)$ bezeichnen wollen. Technisch gesehen ist das Längenfunktional definiert als Abbildung der Form 
\begin{align}
    \mathbb{L}_{[a,b]} : \mathcal{C}^\infty ([a,b], \mathcal{M}) \longrightarrow \mathbb{R},
\end{align}
wobei die Menge $\mathcal{C}^\infty ([a,b], \mathcal{M})$ definiert ist durch
\begin{align}
    \mathcal{C}^\infty ([a,b], \mathcal{M}) := := \{ \gamma \in \mathcal{C}^0 ([a,b], \mathcal{M}) \:\: | \:\: \gamma |_{(a,b)} \in \mathcal{C}^\infty ((a,b) \}.
\end{align}
Dabei bezeichnet die Menge $\mathcal{C}^0 ([a,b], \mathcal{M})$ die Menge aller stetigen Kurven $\gamma$ der Form $[a,b] \longrightarrow \mathcal{M}$ und $\gamma |_{(a,b)}$ meint, dass wir uns $\gamma$ auf das Intervall $(a,b)$ eingeschränkt denken. 

Man beachte, dass zum jetzigen Zeitpunkt das Längenfunktional $\mathbb{L}_{[a,b]}$ eine rein abstrakte Funktion ist, die im Prinzip beliebig definiert sein kann. Da wir aber einen intuitiven Längenbegriff anstreben, sollte am Ende die konkrete Form des Längenfunktionals einen Teil unserer intuitiven Vorstellung vom Begriff der Länge codieren, damit wir den Namen, dem wir dem Funktional $\mathbb{L}_{[a,b]}$ gegeben haben, rechtfertigen können. Ein weiterer Grund, warum $\mathbb{L}_{[a,b]}$ nicht beliebig gewählt sein sollte, ist, dass wir am Ende auch eine Theorie zum Längenfunktional ausarbeiten wollen. Damit wir das tun können, müssen wir aber dem Längenfunktional eine gewisse Struktur auferlegen, sprich eine allgemeine Form des Funktionals $\mathbb{L}_{[a,b]}$ fordern, da wir ohne diese zusätzliche Struktur nicht ordentlich arbeiten können.

Da diese zusätzliche Struktur bzw. die zu fordernde allgemeine Form des Längenfunktionals einen Teil unserer intuitiven Vorstellung vom Begriff der Länge codieren soll, wollen wir im Folgenden die allgemeine Form von $\mathbb{L}_{[a,b]}$ auf Basis eines uns vertrauten Rahmens herleiten bzw. motivieren.

Wir betrachten dazu den $\mathbb{R}^d$, versehen mit dem Standardskalarprodukt $\langle \cdot , \cdot \rangle_{\mathbb{R}^d}$, welches die Standardvisualisierung des $\mathbb{R}^d$ induziert. Sei nun eine stetige Kurve $\gamma : [a,b] \longrightarrow \mathbb{R}^d$ mit $\gamma = (\gamma^1, ..., \gamma^d)$ gegeben, sodass die Funktionen $\gamma^{i} : [a,b] \longrightarrow \mathbb{R}$ auf $(a,b)$ differenzierbar im Sinne der Funktionalanalysis sind, dann folgt, dass $\gamma$ auf $(a,b)$ Gâteaux-differenzierbar ist (Beachte, dass $(a,b) \subseteq \mathbb{R}$ eine offene Menge im normierten Vektorraum $(\mathbb{R}, |\cdot|)$ ist und es sich bei $(\mathbb{R}^d, \| \cdot \|_{\langle \cdot , \cdot \rangle_{\mathbb{R}^d}})$ ebenfalls um einen normierten Vektorraum handelt, weshalb wir die funktionalanalytische Definition der Differenzierbarkeit auf $\gamma$ anwenden können), denn es gilt für $\lambda_0 \in (a,b)$, $\alpha \in \mathbb{R}$ und der multidimensionalen Kettenregel
\begin{align}
    \gamma'(\lambda_0; \alpha) =& \: \lim_{t \longrightarrow 0} \frac{\gamma(\lambda_0 + t \cdot \alpha) - \gamma(\lambda_0)}{t} \nonumber \\ =& \: \lim_{t \longrightarrow 0} \frac{(\gamma^1(\lambda_0 + t \cdot \alpha), ..., \gamma^d(\lambda_0 + t \cdot \alpha)) - (\gamma^1(\lambda_0), ..., \gamma^d(\lambda_0))}{t} \nonumber \\ =& \: \lim_{t \longrightarrow 0} \bigg(  \frac{\gamma^1(\lambda_0 + t \cdot \alpha) - \gamma^1(\lambda_0)}{t}, ..., \frac{\gamma^d(\lambda_0 + t \cdot \alpha) - \gamma^d(\lambda_0)}{t} \bigg) \nonumber \\ =& \: \alpha \cdot \big(  (\gamma^1)^{(1)} (\lambda_0), ..., (\gamma^d)^{(1)} (\lambda_0) \big) \nonumber \\ =& \: \alpha \cdot \gamma' (\lambda_0),
\end{align}
wobei wir im letzten Schritt wieder für alle $i \in \{ 1, ..., d \}$ den Ausdruck $(\gamma^{i})'$ mit dem Ausdruck $(\gamma^{i})^{(1)}$ identifiziert haben. Damit gilt also  $\gamma' = ((\gamma^1)^{(1)}, ..., (\gamma^d)^{(1)})$ mit 
\begin{align}
    (\gamma^{i})^{(1)} (\lambda_0) = \lim_{t \longrightarrow 0} \frac{\gamma^{i}(\lambda_0 + t) - \gamma^{i}(\lambda_0)}{t} =: \frac{d \gamma^{i}}{d \lambda} (\lambda_0)
\end{align}
für alle $i \in \{ 1, ..., d \}$ und $\lambda_o \in (a,b)$. Beachte, dass wegen der Geometrie des $(\mathbb{R}^d, \langle \cdot , \cdot \rangle_{\mathbb{R}^d})$ der Vektor $\gamma' (\lambda_0) =: \frac{d \gamma}{d \lambda}(\lambda_0)$ tangential an der Kurve $\gamma$ am Punkt $\gamma(\lambda_0)$ anliegt. Wir bezeichnen weiter, wie oben, mit $\mathbb{L}_{[a,b]}(\gamma)$ die Länge der Kurve $\gamma$, wobei wir in diesem konkreten Fall unter $\mathbb{L}_{[a,b]}(\gamma)$ tatsächlich den intuitiven Längenbegriff einer Kurve aus der euklidischen Geometrie meinen. Sei nun $\mathcal{P} = \{ \lambda_0, ..., \lambda_{n_{\mathcal{P}}} \} \subseteq \mathbb{R}$ mit $n_\mathcal{P} \in \mathbb{N}$ eine äquidistante Zerlegung von $[a,b]$, d.h. 
\begin{align}
    a = \lambda_0 < \lambda_1 < ... < \lambda_{n_\mathcal{P}-1} < \lambda_{n_\mathcal{P}} = b
\end{align}
und 
\begin{align}
    \lambda_i - \lambda_{i-1} = \Delta \lambda \:\:\:\:\:\:\:\: \forall i \in \{ 1, ..., n_{\mathcal{P}} \}.
\end{align}
Mit $|\mathcal{P}| = \Delta \lambda$ bezeichnen wir die Feinheit der äquidistanten Zerlegung $\mathcal{P}$. Für ein $i \in \{ 1, ..., n_\mathcal{P} \}$ betrachten wir die Menge $\gamma([\lambda_{i-1}, \lambda_{i}]) \subseteq \mathbb{R}^d$, die für ein hinreichend klein gewähltes $\Delta \lambda$ ein hinreichend kleines Kurvenstück der Kurve $\gamma([a,b]) \subseteq \mathbb{R}^d$ darstellt, deren Länge wir mit $ds_{i, \mathcal{P}} := \| \gamma(\lambda_i) - \gamma(\lambda_{i-1}) \|_{\langle \cdot , \cdot \rangle_{\mathbb{R}^d}}$ approximieren können. Die intuitive Länge der Kurve $\gamma$ lässt sich dann approximativ ausdrücken durch den Ausdruck
\begin{align}
    \mathbb{L}_{[a,b], \mathcal{P}} (\gamma) := \sum_{i = 1}^{n_{\mathcal{P}}} ds_{i, \mathcal{P}}.
\end{align}
Intuitiv ist klar, dass für immer feinere äquidistante Zerlegungen $\mathcal{P}$ die Approximation $\mathbb{L}_{[a,b], \mathcal{P}} (\gamma)$ sich immer weiter der tatsächlichen intuitiven Länge $\mathbb{L}_{[a,b]}(\gamma)$ der Kurve $\gamma$ annähert, d.h.
\begin{align}
    \lim_{|\mathcal{P}| \longrightarrow 0} \mathbb{L}_{[a,b], \mathcal{P}} (\gamma) = \mathbb{L}_{[a,b]}(\gamma).
\end{align}
Wir merken an dieser Stelle an, dass wegen der Konstruktion der Approximationen $\mathbb{L}_{[a,b], \mathcal{P}} (\gamma)$ auch klar wird, warum wir die Forderung benötigen, dass die Kurve $\gamma$ auf dem Intervall $[a,b]$ und nicht nur auf dem offenen Intervall $(a,b)$ definiert ist. Weiter gilt, falls $\gamma |_{(a,b)}$ zusätzlich aus der Klasse $\mathcal{C}^1((a,b), \mathbb{R}^d)$ ist, dass
\begin{align}
    \lim_{|\mathcal{P}| \longrightarrow 0} \mathbb{L}_{[a,b], \mathcal{P}} (\gamma) =& \: \int_a^b \| \gamma'(\lambda) \|_{\langle \cdot, \cdot \rangle_{\mathbb{R}^d}} \,d \lambda \nonumber \\ =& \:  \int_a^b \sqrt{ \langle \gamma'(\lambda) , \gamma'(\lambda) \rangle_{\mathbb{R}^d} } \,d \lambda \nonumber \\ =& \: \int_a^b \sqrt{ \sum_{i = 1}^d ((\gamma^{i})^{(1)}(\lambda))^2 } \,d \lambda \nonumber \\ =& \: \int_a^b \sqrt{ \sum_{i = 1}^d \bigg(\frac{d\gamma^{i}}{d \lambda}(\lambda)\bigg)^2 } \,d \lambda
\end{align}
ist \cite{heuser1992lehrbuch}. Heuristisch kann man dieses Resultat bereits durch die Gleichung
\begin{align}
    \mathbb{L}_{[a,b], \mathcal{P}} (\gamma) =& \: \sum_{i = 1}^{n_{\mathcal{P}}} ds_{i, \mathcal{P}} \nonumber \\ =& \: \sum_{i = 1}^{n_{\mathcal{P}}} \| \gamma(\lambda_i) - \gamma(\lambda_{i-1}) \|_{\langle \cdot , \cdot \rangle_{\mathbb{R}^d}} \nonumber \\ =& \: \sum_{i = 1}^{n_{\mathcal{P}}} \bigg\| \frac{\gamma(\lambda_i) - \gamma(\lambda_{i-1})}{\Delta \lambda} \bigg\|_{\langle \cdot , \cdot \rangle_{\mathbb{R}^d}} \cdot \Delta \lambda 
\end{align}
und der Definition des gewöhnlichen Riemann-Integrals erahnen. Es folgt damit insgesamt
\begin{align}
    \mathbb{L}_{[a,b]}(\gamma) = \int_a^b \| \gamma'(\lambda) \|_{\langle \cdot, \cdot \rangle_{\mathbb{R}^d}} \,d \lambda = \int_a^b \sqrt{ \langle \gamma'(\lambda) , \gamma'(\lambda) \rangle_{\mathbb{R}^d} } \,d \lambda \label{euklidisches Längenfunktional}
\end{align}
Denken wir uns nun abermals den $\mathbb{R}^d$, wie in Bemerkung $6.1.$, als glatte Mannigfaltigkeit $(\mathbb{R}^d, \tau_{\mathbb{R}^d}, \mathcal{B})$, wobei $\tau_{\mathbb{R}^d}$ die Standardtopologie des $\mathbb{R}^d$ meint, so können wir diese glatte Mannigfaltigkeit wie in Beispiel $7.1.$ mit dem folgenden riemannschen $(0,2)$-Tensorfeld
\begin{align}
    g : \mathbb{R}^d \longrightarrow& \: \: T^0_2 \mathbb{R}^d \nonumber\\
            p \longmapsto& \: \: g_p \:\:\:\: \textit{mit} \:\:\:\: g_p \bigg( \bigg( \frac{\partial}{\partial x^{i}} \bigg)_p, \bigg( \frac{\partial}{\partial x^{j}} \bigg)_p \bigg) = \langle  e_i , e_j \rangle_{\mathbb{R}^d} \:\:\:\: i,j \in \{ 1, ..., d\} \label{spezielle riemannsche Metrik}
\end{align}
ausstatten, wobei wieder die Karte $(\mathbb{R}^d, \textit{id}_{\mathbb{R}^d} = (x^1, ..., x^d)) \in \mathcal{B}$ gewählt wurde und 
\begin{align}
    e_i := (0, ..., 0, \underbrace{1}_{\textit{$i$-te Stelle}}, 0, ..., 0) \in \mathbb{R}^d  \:\:\:\:\:\:\:\: \forall i \in \{ 1, ..., d \}
\end{align}
ist. Beachte abermals, dass es sich bei $g$ tatsächlich um einen glatten Schnitt im $(0,2)$-Tensorbündel handelt, da aus der Definition von $g$ folgt, dass die bzgl. der Karte $(\mathbb{R}^d, \textit{id}_{\mathbb{R}^d})$ karteninduzierten Komponentenfunktionen konstante Funktionen sind, womit diese zur Klasse $\mathcal{C}^\infty (\mathbb{R}^d)$ gehören. Beachte weiter, dass bzgl. des Ausdrucks \eqref{Vektorkomponente}, welcher erklärt, wie die Vektorkomponenten von $X_{\gamma, \gamma(\lambda)} \in T_p \mathbb{R}^d$ bzgl. der Karte $(\mathbb{R}^d, \textit{id}_{\mathbb{R}^d})$ zu berechnen sind, folgt, dass
\begin{align}
    X_{\gamma,\gamma(\lambda)} = \sum_{i=1}^d (\gamma^{i})^{(1)}(\lambda) \cdot \bigg(  \frac{\partial}{\partial x^{i}} \bigg)_{\gamma(\lambda)} \:\:\:\:\:\:\:\: \forall \lambda \in (a,b)
\end{align}
ist. Damit folgt sofort für alle $\lambda \in (a,b)$
\begin{align}
    g_{\gamma(\lambda)} ( X_{\gamma, \gamma(\lambda)}, X_{\gamma, \gamma(\lambda)} ) =& \: \sum_{i,j = 1}^d (\gamma^{i})^{(1)}(\lambda) \cdot (\gamma^{j})^{(1)}(\lambda) \cdot g_p \bigg( \bigg(  \frac{\partial}{\partial x^{i}} \bigg)_{\gamma(\lambda)}, \bigg(  \frac{\partial}{\partial x^{j}} \bigg)_{\gamma(\lambda)}  \bigg) \nonumber \\ =& \: \sum_{i,j = 1}^d (\gamma^{i})^{(1)}(\lambda) \cdot (\gamma^{j})^{(1)}(\lambda) \cdot \langle e_i, e_j \rangle_{\mathbb{R}^d} \nonumber \\ =& \langle \gamma'(\lambda), \gamma'(\lambda) \rangle_{\mathbb{R}^d}.
\end{align}
Damit können wir das euklidische Längenfunktional \eqref{euklidisches Längenfunktional} umschreiben als
\begin{align}
    \mathbb{L}_{(a,b)}(\gamma) = \int_a^b \sqrt{ g_{\gamma(\lambda)}( X_{\gamma, \gamma(\lambda)} , X_{\gamma, \gamma(\lambda)} ) } \,d \lambda.
\end{align} 
Dies motiviert die folgende Definition:

\begin{definition}
    Sei $(\mathcal{M}, \tau, \mathcal{A}, g)$ eine riemannsche Mannigfaltigkeit und sei weiter $\gamma \in \mathcal{C}^\infty ([a,b], \mathcal{M})$ mit $a,b \in \mathbb{R}$ und $a<b$ eine Kurve. Dann ist das riemannsche Längenfunktional $\mathbb{L}_{[a,b]} : \mathcal{C}^\infty ([a,b], \mathcal{M}) \longrightarrow \mathbb{R}$ definiert als 
    \begin{align}
        \mathbb{L}_{[a,b]}(\gamma) := \int_a^b \sqrt{ g_{\gamma(\lambda)}( X_{\gamma, \gamma(\lambda)} , X_{\gamma, \gamma(\lambda)} ) } \,d \lambda. \label{Längenfunktional}
    \end{align}
\end{definition}

Beachte, das auf einer allgemeinen Mannigfaltigkeit $(\mathcal{M}, \tau, \mathcal{A}, g)$ die Länge $\mathbb{L}_{[a,b]}(\gamma)$ der Kurve $\gamma \in \mathcal{C}^\infty ([a,b], \mathcal{M})$ oft eine abstrakte Länge meint, da zum einen die Menge $\mathcal{M}$ abstrakt gegeben sein kann und zum anderen selbst im Falle $\mathcal{M} \subseteq \mathbb{R}^d$ die riemannsche Metrik $g$ nicht mit dem riemannschen Tensorfeld aus \eqref{spezielle riemannsche Metrik} zusammenfallen muss, welches in den einzelnen Tangentialräumen $T_p \mathcal{M}$ die intuitive euklidische Geometrie induziert hatte. 

Nichtsdestotrotz ist ein Teil unserer intuitiven Vorstellung vom Längenbegriff im Funktional \eqref{Längenfunktional} codiert, da die strukturelle Form des Längenfunktionals \eqref{Längenfunktional} gewährleistet, dass in der vertrauten Umgebung $(\mathcal{M}, \tau, \mathcal{A}) = (\mathbb{R}^d, \tau_{\mathbb{R}^d}, \mathcal{B})$ und der Wahl \eqref{spezielle riemannsche Metrik} für die riemannsche Metrik, sich der Ausdruck \eqref{Längenfunktional} auf den gewöhnlichen und intuitiven Längenbegriff aus der euklidischen Geometrie reduziert. 

Bezüglich der obigen Definition $7.3.$ beachte man noch, dass es eigentlich formal richtiger wäre, wenn wir statt $X_{\gamma, \gamma(\lambda)}$, stattdessen $X_{\gamma, \lambda, \gamma(\lambda)}$ schreiben würden, da es sich bei $\gamma$ durchaus auch um eine nicht injektive Kurve handeln kann. Um aber die Notation simpel zu halten, verzichten wir darauf, da konzeptionell klar sein düfte, was gemeint ist.

\begin{remark}
    An dieser Stelle sei noch angemerkt, dass wir den eben definierte Begriff der Länge auf Kurven der Klasse $\mathcal{PC}^\infty ([a,b], \mathcal{M})$, der stückweise glatten Kurven, verallgemeinern können. Dabei ist die Menge $\mathcal{PC}^\infty ([a,b], \mathcal{M})$ definiert als die Menge aller Kurven $\gamma \in \mathcal{C}^0 ([a,b], \mathcal{M})$, für die eine Menge $\{ \lambda_1, \lambda_2, ..., \lambda_n, \lambda_{n+1}  \} \subseteq (a,b)$ mit $n \in \mathbb{N}$, $\lambda_1 = a$ und $\lambda_{n+1} = b$ existiert, sodass $\gamma |_{(\lambda_{i}, \lambda_{i+1})} \in \mathcal{C}^\infty ((\lambda_i, \lambda_{i+1}), \mathcal{M})$ für alle $i \in \{ 1, ..., n \}$ ist. Für eine derartige Kurve ist dann die Länge $\mathbb{L}_{[a,b]} (\gamma)$ definiert als 
    \begin{align}
        \mathbb{L}_{[a,b]} (\gamma) :=& \: \sum_{i=1}^n \mathbb{L}_{[\lambda_i, \lambda_{i+1}]} (\gamma |_{[\lambda_i, \lambda_{i+1}]}) \nonumber \\ =& \: \sum_{i=1}^n \int_{\lambda_i}^{\lambda_{i+1}} \sqrt{ g_{\gamma(\lambda)}( X_{\gamma, \gamma(\lambda)} , X_{\gamma, \gamma(\lambda)} ) } \,d \lambda. \label{Längenfunktional 2}
    \end{align}
    Beachte, dass wegen der Linearität des Integrals der obige Ausdruck wohldefiniert ist.
\end{remark}

Wir wollen nun auf der riemannschen Mannigfaltigkeit $(\mathcal{M}, \tau, \mathcal{A}, g)$ den für die Geometrie wichtigen, und für uns zentralen Begriff einer (schwachen) Geodätischen einführen, welcher den Begriff einer Geraden aus der euklidischen Geometrie verallgemeinern soll. 

Wir betrachten dazu zuerst zwei vorgegebene Punkte $p$ und $q$ aus $\mathcal{M}$ mit $p \neq q$. Weiter definieren wir für $a,b \in \mathbb{R}$ mit $a<b$ die Menge $\mathcal{C}^\infty_{p,q} ([a,b], \mathcal{M})$, welche alle Kurven $\gamma \in \mathcal{C}^\infty ([a,b], \mathcal{M})$ enthält, die im Punkt $p$ beginnen und im Punkt $q$ enden, d.h. für $\gamma \in \mathcal{C}^\infty_{p,q} ([a,b], \mathcal{M})$ gilt, dass $\gamma \in \mathcal{C}^\infty ([a,b], \mathcal{M})$ ist, sodass $\gamma(a) = p$ und $\gamma(b) = q$ gilt.

Als Nächstes definieren wir den Begriff der stetigen Deformation einer Kurve $\gamma_0 \in \mathcal{C}^\infty_{p,q} ([a,b], \mathcal{M})$:

\begin{definition}
    Sei $(\mathcal{M}, \tau, \mathcal{A})$ eine glatte Mannigfaltigkeit und $p$ und $q$ Punkte aus $\mathcal{M}$. Weiter sei $\gamma_0 \in \mathcal{C}^\infty_{p,q} ([a,b], \mathcal{M})$ mit $a,b \in \mathbb{R}$ und $a<b$ eine Kurve, die im Punkt $p$ beginnt und im Punkt $q$ endet. Eine stetige Deformation der Kurve $\gamma_0$ ist gegeben durch eine Familie $(\gamma_\alpha)_{\alpha \in (- \epsilon, \epsilon)} \subseteq \mathcal{C}^\infty_{p,q} ([a,b], \mathcal{M})$ mit $\epsilon > 0$ und $\gamma_{\alpha = 0} = \gamma_0$, sodass die Abbildung
    \begin{align}
            \gamma : (- \epsilon, \epsilon) \times [a, b] \longrightarrow& \: \: \mathcal{M} \nonumber\\
            (\alpha, \lambda) \longmapsto& \: \: \gamma_\alpha (\lambda)
    \end{align}
    stetig ist. Dabei sind auf $[a,b]$ und $(-\epsilon, \epsilon)$ jeweils die von der $\mathbb{R}$-Standardtopologie induzierten Teilraumtopologien erklärt, sodass $(- \epsilon, \epsilon) \times [a, b]$ mit der zugehörigen Produkttopologie ausgestattet ist. Auf analoge Weise lassen sich auch stetige Deformationen erklären, bei welcher die einzelnen Elemente der Familie aus der Menge $\mathcal{PC}^\infty_{p,q} ([a,b], \mathcal{M})$ kommen, wobei $\mathcal{PC}^\infty_{p,q} ([a,b], \mathcal{M})$ erklärt ist als die Menge aller Kurven aus $\mathcal{PC}^\infty ([a,b], \mathcal{M})$, mit der Eigenschaft, dass diese in $p$ beginnen und in $q$ enden. 
\end{definition}

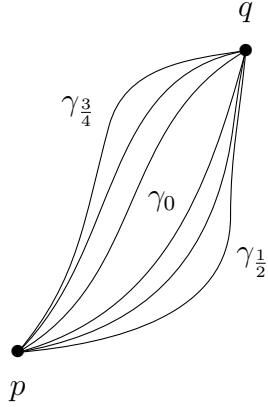
\begin{center}
\begin{tikzpicture}

    \node[label=below:$p$]  (x1) at (6,0)  {$\bullet$};
    \node[label=above:$q$]  (x0) at (9,4)  {$\bullet$};  
    \node  at (9.1,1.2)  {$\gamma_{\frac{1}{2}}$}; 
    \node  at (6.8,3.2)  {$\gamma_{\frac{3}{4}}$}; 
    \node  at (7.9,2)  {$\gamma_{0}$}; 

    \draw (x1.center) to [out=5,in=-90]++(2.8,1.8) to[out=90,in=-95](x0.center);
    \draw (x1.center) to [out=10,in=-110]++(2.6,2) to[out=70,in=-103](x0.center); 
    \draw (x1.center) to [out=15,in=-105](x0.center);
    \draw (x1.center) to [out=30,in=-150](x0.center);
    \draw (x1.center) to [out=45,in=-170](x0.center); 
    \draw (x1.center) to [out=50,in=-105]++(1.2,3)to [out=75,in=-172](x0.center); 

\end{tikzpicture}
\end{center}

Anschaulich ist dabei klar, warum die Familie $(\gamma_\alpha)_{\alpha \in (- \epsilon, \epsilon)}$ aus der obigen Definition als stetige Deformation der Kurve $\gamma \in \mathcal{C}^\infty_{p,q} ([a,b], \mathcal{M})$ bezeichnet wird, da zum einen für $\alpha = 0$ gilt, dass $\gamma_\alpha = \gamma_0$ ist, d.h. die Kurve $\gamma_0$ ist in der Familie $(\gamma_\alpha)_{\alpha \in (- \epsilon, \epsilon)}$ enthalten, und zum anderen ergeben sich die restlichen Kurven aus der Familie $(\gamma_\alpha)_{\alpha \in (- \epsilon, \epsilon)}$, die ebenfalls alle im Punkt $p$ beginnen und im Punkt $q$ enden, stetig aus der Kurve $\gamma_0$ durch die stetige Änderung des Parameters $\alpha \in (- \epsilon, \epsilon)$.

Für die stetige Deformation $(\gamma_\alpha)_{\alpha \in (- \epsilon, \epsilon)}$ der Kurve $\gamma_0$ betrachten wir nun für jedes einzelne $\alpha \in (-\epsilon, \epsilon)$ den Ausdruck $\mathbb{L}_{[a,b]}(\gamma_\alpha)$, d.h. die Länge der Kurve $\gamma_\alpha$. Diesen Ausdruck können wir nun selbst als eine Funktion $F_\gamma$ bzgl. des Parameters $\alpha \in (-\epsilon, \epsilon)$ auffassen, d.h. 
\begin{align}
    F_\gamma : (- \epsilon, \epsilon) \longrightarrow& \: \: \mathbb{R} \nonumber\\
            \alpha \longmapsto& \: \: F_\gamma (\alpha) := \mathbb{L}_{[a,b]}(\gamma_\alpha). \label{Variationsfunktion}
\end{align}
Die Abbildung $F_\gamma$ kann nun also als reellwertige Funktion interpretiert werden, welche auf der offenen Menge $(-\epsilon, \epsilon)$ des normierten Raumes $(\mathbb{R}, | \cdot |)$ definiert ist. Wir nehmen nun darüber hinaus an, dass die stetige Deformation $(\gamma_\alpha)_{\alpha \in (- \epsilon, \epsilon)}$ so gewählt wurde, sodass die zugehörige Funktion $F_\gamma$ am Punkt $\epsilon = 0$ differenzierbar ist (Wir wählen dabei im Wertebereich als Norm abermals die Betragsfunktion $|\cdot|$.). Mittels dieser Annahmen können wir nun erklären, was wir unter einer stationären Kurve des Funktionals $\mathbb{L}_{[a,b]} |_{\mathcal{C}^\infty_{p,q} ((a,b), \mathcal{M})}$ verstehen, wobei $\mathbb{L}_{[a,b]} |_{\mathcal{C}^\infty_{p,q} ((a,b), \mathcal{M})}$ das Längenfunktion $\mathbb{L}_{[a,b]}$ eingeschränkt auf die Menge $\mathcal{C}^\infty_{p,q} ((a,b), \mathcal{M})$ meint. 

\begin{definition}
    Sei $(\mathcal{M}, \tau, \mathcal{A}, g)$ eine riemannsche Mannigfaltigkeit, $p$ und $q$ zwei Punkte aus $\mathcal{M}$ mit $p \neq q$ und $a$ und $b$ zwei reelle Zahlen mit $a < b$. Weiter sei $\mathbb{L}_{[a,b]} |_{\mathcal{C}^\infty_{p,q} ((a,b), \mathcal{M})}$ das auf die Menge $\mathcal{C}^\infty_{p,q} ((a,b), \mathcal{M})$ eingeschränkte Längenfunktional. Dann nennen wir die Kurve $\gamma_0 \in \mathcal{C}^\infty_{p,q} ((a,b), \mathcal{M})$ eine bezüglich $\mathbb{L}_{[a,b]} |_{\mathcal{C}^\infty_{p,q} ((a,b), \mathcal{M})}$ stationäre Kurve, falls für alle stetigen Deformationen $\gamma = (\gamma_\alpha)_{\alpha \in (- \epsilon, \epsilon)}$ der Kurve $\gamma_0$, für die die in \eqref{Variationsfunktion} definierte zugehörige Variationsfunktion $F_\gamma$ differenzierbar im Punkt $0$ ist, die Bedingung
    \begin{align}
        \frac{dF_\gamma}{d\alpha} \bigg|_{\alpha = 0} := F_\gamma^{(1)}(0) = 0 \label{Variation}
    \end{align}
    erfüllt ist. Die bezüglich $\mathbb{L}_{[a,b]} |_{\mathcal{C}^\infty_{p,q} ((a,b), \mathcal{M})}$ stationären Kurven werden auch als schwache geodätische Kurven von $p$ nach $q$ oder kurz als schwache Geodäten oder als schwache Geodätische bzgl. $p$ und $q$ bezeichnet. Auf analoge Weise lässt sich der Begriff einer schwachen Geodätischen auch bzgl. der Menge $\mathcal{PC}^\infty_{p,q} ([a,b], \mathcal{M})$ erklären. Im weiteren Verlauf dieser Arbeit werden wir das Längenfunktional $\mathbb{L}_{[a,b]}$ immer als Funktional auf der Menge $\mathcal{PC}^\infty_{p,q} ([a,b], \mathcal{M})$ verstehen.
\end{definition}

\begin{remark}
    Der Ausdruck $\frac{dF_\gamma}{d\alpha} \big|_{\alpha = 0}$, welcher in der Gleichung \eqref{Variation} auftaucht, wird oft auch als die erste Variation des Längenfunktionals $\mathbb{L}_{[a,b]}$ bezeichnet. In der sogenannten Variationsrechnung, in welcher oft $\mathcal{M}$ als $\mathbb{R}^d$ gewählt wird, ausgestattet mit der gewöhnlichen glatten Struktur aus Bemerkung $6.1.$, nutzt man statt der hier diskutierten allgemeinen stetigen Deformationen $(\gamma_\alpha)_{\alpha \in (-\epsilon, \epsilon)}$ der Kurve $\gamma_0$ die folgenden Variationen von $\gamma_0$:
    \begin{align}
        \hat{\gamma}_\alpha : [a,b] \longrightarrow& \: \: \mathbb{R}^d \nonumber\\
            \lambda \longmapsto& \: \: \hat{\gamma}_\alpha (\lambda) := \gamma_0 (\lambda) + \alpha \cdot \eta (\lambda),
    \end{align}
    mit $\alpha \in (- \epsilon, \epsilon)$ und $\eta \in \mathcal{C}^\infty_0 ([a,b], \mathbb{R}^d)$. Dabei ist die Menge $\mathcal{C}^\infty_0 ([a,b], \mathbb{R}^d)$ definiert als die Menge aller Kurven $\eta \in \mathcal{C}^\infty ([a,b], \mathbb{R}^d)$ mit der Eigenschaft, dass $\eta(a) = \eta(b) = \mathbf{0}$ ist, wobei wir mit $\mathbf{0}$ den Nullvektor des $\mathbb{R}^d$ meinen. Die Menge $\mathcal{C}^\infty_0 ([a,b], \mathbb{R}^d)$ wird auch als die Menge aller Kurven $\eta \in \mathcal{C}^\infty ([a,b], \mathbb{R}^d)$ mit kompakten Träger bezeichnet und stellt ein Beispiel eines sogenannten Testfunktionenraums dar. Für jedes festes $\alpha \in (- \epsilon, \epsilon)$ stellt $\hat{\gamma}_\alpha$ dabei als Summe stetiger Kurven eine stetige Kurve dar.

    Beachte, dass für die Familie $\hat{\gamma} := (\hat{\gamma}_\alpha)_{\alpha \in (- \epsilon, \epsilon)}$ gilt, dass 
    \begin{align}
        \hat{\gamma}_\alpha (a) = \hat{\gamma}_0 (a) \:\:\:\: \textit{und} \:\:\:\: \hat{\gamma}_\alpha (b) = \hat{\gamma}_0 (b)
    \end{align}
    für alle $\alpha \in (- \epsilon, \epsilon)$ ist, d.h. jede Kurve aus der Familie $(\hat{\gamma}_\alpha)_{\alpha \in (- \epsilon, \epsilon)}$ beginnt im Punkt $\hat{\gamma}_0 (a)$ und endet im Punkt $\hat{\gamma}_0 (b)$. Weiter gilt
    \begin{align}
        \hat{\gamma}_{\alpha = 0} = \gamma_0,
    \end{align}
    d.h. dass die Kurve $\gamma_0$ in der Familie $(\hat{\gamma}_\alpha)_{\alpha \in (- \epsilon, \epsilon)}$ enthalten ist. Wegen $\mathcal{M} = \mathbb{R}^d$ können wir auf die lineare Struktur des $\mathbb{R}^d$ zurückgreifen, was es uns überhaupt erst ermöglicht, die Kurven $\hat{\gamma}_\alpha$ als punktweise Summe der beiden Funktionen $\gamma_0$ und $\alpha \cdot \eta$ zu definieren. Wegen
    \begin{align}
        \| \gamma_0(p) - (\gamma_0(p) + \alpha \cdot \eta (p)) \|_{\langle \cdot , \cdot \rangle_{\mathbb{R}^d}} = |\alpha| \cdot \|\eta(p)\|_{\langle \cdot , \cdot \rangle_{\mathbb{R}^d}} \longrightarrow 0 \label{Funktionennähe} 
    \end{align}
    für $\alpha \longrightarrow 0$ folgt, dass für hinreichend kleines $\alpha \in (-\epsilon, \epsilon)$ die Kurve $\hat{\gamma}_\alpha$, bzw. deren Bild, hinreichend nah am Bild der Kurve $\gamma_0$ liegt, weshalb zusammengenommen auch hier die Familie $(\hat{\gamma}_\alpha)_{\alpha \in (- \alpha, \alpha)}$ als eine kontinuierliche Deformation der Kurve $\gamma_0$ interpretiert werden kann. In diesem Rahmen ist der Ausdruck
    \begin{align}
        \frac{dF_{\hat{\gamma}}}{d\alpha} \bigg|_{\alpha = 0} = \frac{d}{d\alpha} \mathbb{L}_{[a,b]}(\hat{\gamma}_\alpha) \bigg|_{\alpha = 0} = \frac{d}{d \alpha} \mathbb{L}_{[a,b]} (\gamma_0 + \alpha \cdot \eta) \bigg|_{\alpha = 0} \label{Richtungsableitung auf Funktionenraum}
    \end{align}
    eine Richtungsableitung am 'Punkt' $\gamma_0$. Beachte dabei, dass im Falle $\mathcal{M} = \mathbb{R}^d$ die Menge $\mathcal{C}^\infty ([a,b], \mathbb{R}^d)$ einen Vektorraum bildet und dieser mit einer beliebigen Norm $\|\cdot\|$ ausgestattet werden kann, da bzgl. jeder Norm auf $\mathcal{C}^\infty ([a,b], \mathbb{R}^d)$ der für die Variationrechnung korrekte Nähebegriff induziert wird, welcher durch \eqref{Funktionennähe} vorgegeben wird. So gilt bzgl. einer beliebigen Norm $\| \cdot \|$ auf $\mathcal{C}^\infty ([a,b], \mathbb{R}^d)$
    \begin{align}
        \|\gamma_0 - \hat{\gamma}_\alpha\| = |\alpha| \cdot \|\eta\| \longrightarrow 0
    \end{align}
    für $\alpha \longrightarrow 0$. Damit können wir \eqref{Richtungsableitung auf Funktionenraum} tatsächlich als Richtungsableitung im Sinne der Funktionalanalysis auffassen. Man nutzt dabei auch oft die in der Variationsrechnung übliche Bezeichnung
    \begin{align}
        \delta \mathbb{L}_{[a,b]} (\gamma_0, \eta) := \frac{d}{d \alpha} \mathbb{L}_{[a,b]} (\gamma_0 + \alpha \cdot \eta) \bigg|_{\alpha = 0}.
    \end{align}
\end{remark}

Was bedeutet nun der Begriff einer schwachen Geodätischen anschaulich? Betrachten wir die beiden Punkte $p$ und $q$ aus $\mathcal{M}$ mit $p \neq q$, sowie die Menge aller glatter Kurven, die diese beiden Punkte miteinander verbinden. Dann ist beispielsweise eine schwache Geodätische von $p$ nach $q$ gegeben durch diejenige Kurve aus $\mathcal{PC}^\infty_{p,q} ([a,b], \mathcal{M})$ mit minimaler oder maximaler Länge (Vorrausgesetzt natürlich, dass es derartige Kurven auf einer riemanschen Mannigfaltigkeit überhaupt gibt.). Woran kann man sehen, dass diese Interpretation korrekt ist? 

Wir nehmen dazu beispielsweise an, dass es sich bei $\gamma_0 \in \mathcal{PC}^\infty_{p,q} ([a,b], \mathcal{M})$ um die Kurve minimaler Länge zwischen $p$ und $q$ handelt, d.h. $\gamma_0$ repräsentiert den kürzesten Weg zwischen $p$ und $q$. Wird $\epsilon$ hinreichend klein gewählt, so gilt für jede stetige Deformation $(\gamma_\alpha)_{\alpha \in (- \epsilon, \epsilon)}$ der Kurve $\gamma_0$, dass jedes $\gamma_\alpha \in (\gamma_\alpha)_{\alpha \in (- \epsilon, \epsilon)}$ nur eine minimale Variation der Kurve $\gamma_0$ darstellt. Da $\gamma_0$ der kürzeste Weg von $p$ nach $q$ ist, gilt
\begin{align}
    \mathbb{L}_{[a,b]}(\gamma_0) \leq \mathbb{L}_{[a,b]}(\gamma_\alpha) \:\:\:\:\:\:\:\: \forall \alpha \in (- \epsilon, \epsilon).
\end{align}
Da wir zusätzlich vorraussetzen, dass die stetige Deformationen so gewählt ist, sodass $F_\gamma$ im Punkt $\alpha = 0$ differenzierbar ist, folgt, dass der Funktiongraph von $F_\gamma$ in etwa die folgende Gestalt in der Nähe der $0$ aufweist:\\[0.1cm]
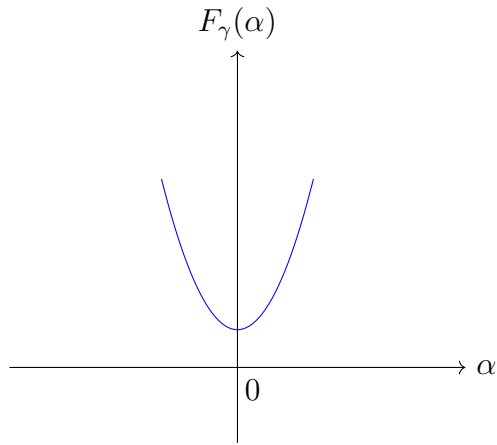
\begin{center}
\begin{tikzpicture}
\centering
    \node (a) at (0.2,-0.3) {$0$};
    \draw[->] (-3, 0) -- (3, 0) node[right] {$\alpha$};
    \draw[->] (0, -1) -- (0, 4.2) node[above] {$F_\gamma (\alpha)$};
    \draw[scale=0.5, domain=-2:2, smooth, variable=\x, blue] plot ({\x}, {\x*\x + 1} );
\end{tikzpicture}
\end{center}
Anhand der Grafik und dem Wissen, dass der Ausdruck $F_\gamma^{(1)}(0)$ den Anstieg des Funktionsgraphen am Punkt $0$ meint, ist klar, dass in diesem Fall $F_\gamma^{(1)}(0) = 0$ gelten muss. Aus der obigen Grafik folgt auch umgedreht, dass es sich bei $\gamma_0$ um ein (lokales) Minimum von $\mathbb{L}_{[a,b]} |_{\mathcal{C}^\infty_{p,q} ((a,b), \mathcal{M})}$ handeln muss, wenn diese Grafik für alle beliebigen stetigen Deformationen, für die die zugehörigen Variationsfunktionen $F_\gamma$ differenzierbar an der Stelle $0$ sind, in etwa von obiger Gestalt sind. Analog argumentiert man auch für Wege maximaler Länge von $p$ nach $q$.

Bevor wir nun diesen Abschnitt beenden, wollen wir noch den Begriff der schwachen Geodätischen verallgemeinern. Dazu müssen wir aber zuerst noch klären, was wir unter einer sogenannten glatten Deformation $\gamma : (-\epsilon, \epsilon) \times [a,b] \longrightarrow \mathcal{M}$ einer Kurve $\gamma_0 \in \mathcal{PC}^\infty ([a,b], \mathcal{M})$ mit $\epsilon > 0$ und $a,b \in \mathbb{R}$ mit $a<b$ zu verstehen haben.

Salopp gesprochen soll es sich dabei um eine stetige Deformation der Kurve $\gamma_0$ handeln, welche gleichzeitig glatt sein soll. Das wirft aber die Frage auf, wie eine glatte Funktion der Form $\gamma : (-\epsilon, \epsilon) \times [a,b] \longrightarrow \mathcal{M}$ definiert ist, wobei $(\mathcal{M}, \tau, \mathcal{A})$ wieder eine $d$-dimensionale glatte Mannigfaltigkeit sei.

Zuerst beachte man, dass, wenn $\gamma$ eine stetige Deformation sein soll, $\gamma$ stetig bezüglich der Produkttopologie auf $(-\epsilon, \epsilon) \times [a,b]$ und der Topologie $\tau$ auf $\mathcal{M}$ sein muss. Dabei ist die Produkttopologie auf $(-\epsilon, \epsilon) \times [a,b]$ induziert durch die Teilraumtopologien auf $(-\epsilon, \epsilon)$ und $[a,b]$, die wiederum induziert werden durch die Standardtopologie der reellen Zahlen.

Bezeichnen wir die Produkttopologie auf $(-\epsilon, \epsilon) \times [a,b]$ kurz mit $\tau_\mathcal{P}$, so würde der erste Ansatz zur Definition einer glatten Funktion der Form $(-\epsilon, \epsilon) \times [a,b] \longrightarrow \mathcal{M}$ sich analog zur Definition $6.9.$ gestalten, in welcher wir unter anderem gefordert haben, dass eine Kurve $\gamma_0 : (a,b) \longrightarrow \mathcal{M}$ glatt ist, falls für alle Karten $(\mathcal{U}, \phi) \in \mathcal{A}$ mit $\gamma_0^{-1}(\mathcal{U}) \neq \emptyset$ die Abbildung 
\begin{align}
    \phi \circ \gamma_0 : \gamma_0^{-1} (\mathcal{U}) \longrightarrow \phi(\mathcal{U})
\end{align}
glatt ist. Die Glattheit von $\phi \circ \gamma_0 |_{\gamma_0^{-1}(\mathcal{U})}$ war dabei im herkömmlichen funktionalanalytischen Sinne zu verstehen. Das wir überhaupt im funktionalanalytischen Sinne über die mögliche Glattheit der Abbildung $\phi \circ \gamma_0 |_{\gamma_0^{-1}(\mathcal{U})}$ sprechen konnten lag an der Konstruktion der Definition, die gewährleistete, dass $\gamma_0^{-1}(\mathcal{U})$ als offene Menge in einem normierten Raum aufgefasst werden konnte, damit wir überhaupt auf die Definition der Differenzierbarkeit einer Abbildung zurückgreifen konnten.

Um also die Definition der Glattheit für Kurven nachzuahmen, um so den Begriff einer glatten Abbildung $\gamma$ der Form $(-\epsilon, \epsilon) \times [a,b] \longrightarrow \mathcal{M}$ zu etablieren, müssen wir gewährleisten, dass für alle Karten $(\mathcal{U}, \phi) \in \mathcal{A}$ die Menge $\gamma^{-1} (\mathcal{U})$ als offene Menge in einem normierten Raum verstanden werden kann.  

Wir machen zuerst zwei Beobachtungen: Zuerst stellen wir fest, dass wir die Menge $(-\epsilon, \epsilon) \times [a,b]$ als Teilmenge des $\mathbb{R}^2$ verstehen können. Den $\mathbb{R}^2$ können wir wiederum als normierten Raum verstehen, wobei die Norm gegeben ist durch die Standardnorm, welche wir kurz mit $\| \cdot \|_2$ bezeichnen wollen, und welche definiert ist durch
\begin{align}
    \|(x,y)\|_2 := \sqrt{x^2 + y^2}.
\end{align}
Als nächstes stellen wir fest, dass die Teilraumtopologie auf $(-\epsilon, \epsilon) \times (a,b)$, induziert durch die Topologie $\tau_\mathcal{P}$ auf $(-\epsilon, \epsilon) \times [a,b]$, per Definition der Produktttopologie gerade der Produkttopologie auf $(-\epsilon, \epsilon) \times (a,b)$ entspricht, wenn die beiden Intervalle $(-\epsilon, \epsilon)$ und $(a,b)$ jeweils mit den durch die auf $\mathbb{R}$ erklärten Standardtopologie induzierten Teilraumtopologien ausgestattet sind. Bezeichnen wir die Produkttopologie auf $(-\epsilon, \epsilon) \times (a,b)$ mit $\tau_p$ so gilt weiter, abermals per Definition der Produkttopologie (Betrachte dazu auch den noch kommenden Beweis von Lemma $7.1.$.), dass $\tau_p \subset \tau_{\mathcal{P}}$ gilt. 

Aufgrund der Konstruktion der Produkttopologie liegt es weiterhin Nahe, dass, wenn wir $\mathbb{R}$ mit der Standardtopologie ausstatten, die Produkttopologie auf $\mathbb{R} \times \mathbb{R} = \mathbb{R}^2$, die wir mit $\tau_1$ bezeichnen wollen, äquivalent ist zur Topologie $\tau_2$, bei welcher es sich um die durch die Norm $\| \cdot \|_2$ induzierte Topologie handelt.

Sollte das stimmen, so folgt, da abermals per Konstruktion der Produkttopologie $\tau_p \subset \tau_1$ gilt, dass $\tau_p \subset \tau_2$ ist und damit $\gamma |_{(-\epsilon, \epsilon) \times (a,b)}^{-1}(\mathcal{U}) \in \tau_2$. Auf diese Weise könnten wir also $\phi \circ \gamma |_{(-\epsilon, \epsilon) \times (a,b)}$ als Abbildung verstehen, die auf einer offenen Teilmenge des normierten Raumes $(\mathbb{R}^2, \| \cdot \|_2)$ erklärt ist und in den normierten Raum $(\mathbb{R}^d, \| \cdot \|_{\langle \cdot , \cdot \rangle_{\mathbb{R}^d}})$ abbildet, wobei $\| \cdot \|_{\langle \cdot , \cdot \rangle_{\mathbb{R}^d}}$ die vom $\mathbb{R}^d$-Standardskalarprodukt induzierte Norm meint.

In diesem Fall wäre die Defintion einer glatten Deformation einfach: Die Abbildung $\gamma : (-\epsilon, \epsilon) \times [a,b] \longrightarrow \mathcal{M}$ hieße glatte Deformation der Kurve $\gamma_0$, wenn $\gamma$ eine stetige Deformation der Kurve $\gamma_0$ ist, sodass für alle Karten $(\mathcal{U}, \phi) \in \mathcal{A}$ mit $\gamma |_{(-\epsilon, \epsilon) \times (a,b)}^{-1}(\mathcal{U}) \neq \emptyset$ folgt, dass die Abbildung
\begin{align}
    \phi \circ \gamma |_{(-\epsilon, \epsilon) \times (a,b)} : \gamma |_{(-\epsilon, \epsilon) \times (a,b)}^{-1} (\mathcal{U}) \longrightarrow \mathbb{R}^d
\end{align}
eine glatte Abbildung im herkömmlichen funktionalanalytischen Sinne ist. Die Wohldefiniertheit dieser Definition, sprich die Unabhängigkeit von der Wahl der Karte, würde dann aus der Definition $6.8.$ folgen, da wegen $\tau_p \subset \tau_2$ auch folgen würde, dass $(-\epsilon, \epsilon) \times (a,b) \in \tau_2$ wäre. Da sich nun jede offene Menge $\mathcal{O}$ von $(\mathbb{R}^2, \tau_2)$ zu einer $2$-dimensionalen glatten Mannigfaltigkeit machen lässt, indem man $\mathcal{O}$ mit der Teilraumtopologie 
\begin{align}
    \tau_\mathcal{O} =& \: \{ \mathcal{V} \cap \mathcal{O} \subseteq \mathbb{R}^2 \:\: | \:\: \mathcal{V} \in \tau_2 \} \nonumber \\ =& \: \{ \mathcal{V} \in \tau_2 \:\: | \:\: \mathcal{V} \subseteq \mathcal{O} \}
\end{align}
und der globalen Karten $(\mathcal{O}, \textit{id}_{\mathcal{O}})$ versieht, können wir $\gamma |_{(-\epsilon, \epsilon) \times (a,b)}$ auch als Abbildung zwischen zwei glatten Mannigfaltigkeiten verstehen.

Es bleibt nun noch zu zeigen, dass tatsächlich $\tau_1 = \tau_2$ gilt.

\begin{proposition}
    Sei die Menge $\mathbb{R}$ versehen mit der Standardtopologie. Wir bilden den Produktraum $\mathbb{R} \times \mathbb{R} = \mathbb{R}^2$, versehen mit der von der Standardtopologie induzierten Produkttopologie, welche wir mit $\tau_1$ bezeichnen. Weiter sei $\tau_2$ die vom $\mathbb{R}^2$-Standardskalarprodukt induzierte Standardtopologie des $\mathbb{R}^2$. Dann gilt $\tau_1 = \tau_2$.
\end{proposition}

\begin{proof}
    Um zu zeigen, dass beide Topologien identisch sind, müssen wir einerseits zeigen, dass $\tau_1 \subseteq \tau_2$ gilt, und andererseits, dass $\tau_2 \subseteq \tau_1$ gilt, da auf diese Weise gezeigt ist, dass jede offene Menge aus $\tau_1$ auch eine offene Menge aus $\tau_2$ und umgekehrt ist.
    \begin{itemize}
    
        \item[\textit{i)}] Beginnen wir mit $\tau_1 \subseteq \tau_2$: Sei also $\mathcal{W} \in \tau_1$. Aus der Definition der Produkttopologie folgt, dass für alle $(x,y) \in \mathcal{W}$ offene Intervalle $(a,b) \subseteq \mathbb{R}$ und $(c,d) \subseteq \mathbb{R}$ existieren, sodass 
        \begin{align}
            (x,y) \in (a,b) \times (c,d) \subseteq \mathcal{W}
        \end{align}
        gilt. Da es sich bei den Intervallen $(a,b)$ und $(c,d)$ um offene Intervalle handelt, können wir diese nun verkleinern zu $(\hat{a}, \hat{b}) \subseteq (a,b)$ und $(\hat{c}, \hat{d}) \subseteq (c,d)$, sodass
        \begin{align}
            (x,y) = \bigg( \frac{\hat{a} + \hat{b}}{2} , \frac{\hat{c} + \hat{d}}{2} \bigg)
        \end{align}
        gilt. Weiter definieren wir uns ein $\delta > 0$ durch $\delta = \textit{min} \big\{ \frac{|\hat{a} - \hat{b}|}{2} , \frac{|\hat{c} - \hat{d}|}{2} \big\}$. Dann gilt, dass
        \begin{align}
            \mathcal{U}_\delta ((x,y)) =& \: \{ (z,w) \in \mathbb{R}^2 \:\: | \:\: \| (x,y) - (z,w) \|_2 < \delta \} \nonumber \\ =& \: \{ (z,w) \in \mathbb{R}^2 \:\: | \:\: \sqrt{(x-z)^2 + (y-w)^2} < \delta \} \nonumber \\ \subseteq& \: (\hat{a}, \hat{b}) \times (\hat{c}, \hat{d}) \subseteq \mathcal{W}.
        \end{align}
        Damit können wir also um jedes $(x,y) \in \mathcal{W} \in \tau_1$ eine $\delta$-Umgebung legen, die selbst noch in $\mathcal{W}$ enthalten ist. Dies ist gemäß \eqref{metrische Topologie...} genau die definierende Eigenschaft für offene Menge in $\tau_2$, woraus folgt, dass $\mathcal{W} \in \tau_2$ ist. Da $\mathcal{W} \in \tau_1$ beliebig gewählt war, folgt, dass $\tau_1 \subseteq \tau_2$ ist.

        \item[\textit{ii)}] Betrachten wir nun die Richtung $\tau_2 \subseteq \tau_1$. Sei also $\mathcal{U} \in \tau_2$, d.h. für jeden Punkt $(x,y) \in \mathcal{U}$ existiert ein $\delta_{(x,y)} > 0$, sodass $(x,y) \in \mathcal{U}_{\delta_{(x,y)}} ((x,y)) \subseteq \mathcal{U}$ gilt. Wir betrachten die reellen Intervalle $\big( x - \frac{\delta_{(x,y)}}{\sqrt{2}}, x + \frac{\delta_{(x,y)}}{\sqrt{2}} \big)$ und $\big( y - \frac{\delta_{(x,y)}}{\sqrt{2}} , y + \frac{\delta_{(x,y)}}{\sqrt{2}} \big)$. Sei $(z,w)$ ein beliebiges Element in $\big( x - \frac{\delta_{(x,y)}}{\sqrt{2}}, x + \frac{\delta_{(x,y)}}{\sqrt{2}} \big) \times \big( y - \frac{\delta_{(x,y)}}{\sqrt{2}} , y + \frac{\delta_{(x,y)}}{\sqrt{2}} \big)$. Dann gilt
        \begin{align}
            \| (x,y) - (z,w) \|_2 =& \: \sqrt{(x-z)^2 + (y-w)^2} \nonumber \\ <& \: \sqrt{\bigg( x - x + \frac{\delta_{(x,y)}}{\sqrt{2}} \bigg)^2 + \bigg( y - y + \frac{\delta_{(x,y)}}{\sqrt{2}} \bigg)^2} \nonumber \\ =& \: \sqrt{\frac{\delta_{(x,y)}^2}{2} + \frac{\delta_{(x,y)}^2}{2}} = \delta_{(x,y)},
        \end{align}
        woraus 
        \begin{align}
            \bigg( x - \frac{\delta_{(x,y)}}{\sqrt{2}}, x + \frac{\delta_{(x,y)}}{\sqrt{2}} \bigg) \times \bigg( y - \frac{\delta_{(x,y)}}{\sqrt{2}} , y + \frac{\delta_{(x,y)}}{\sqrt{2}} \bigg) \subseteq& \: \mathcal{U}_{\delta_{(x,y)}} ((x,y)) \nonumber \\ \subseteq& \: \mathcal{U} \label{***********}
        \end{align}
        folgt. Da $\big( x - \frac{\delta_{(x,y)}}{\sqrt{2}}, x + \frac{\delta_{(x,y)}}{\sqrt{2}} \big) \times \big( y - \frac{\delta_{(x,y)}}{\sqrt{2}} , y + \frac{\delta_{(x,y)}}{\sqrt{2}} \big)$ per Definition der Produkttopologie $\tau_1$ eine offene Menge bezüglich $\tau_1$ ist, und wir für jedes $(x,y) \in \mathcal{U}$ ein $\delta_{(x,y)}$ finden können, sodass \eqref{***********} richtig ist, folgt, dass wir $\mathcal{U}$ auch schreiben können als 
        \begin{align}
            \mathcal{U} = \bigcup_{(x,y) \in \mathcal{U}} \bigg( x - \frac{\delta_{(x,y)}}{\sqrt{2}}, x + \frac{\delta_{(x,y)}}{\sqrt{2}} \bigg) \times \bigg( y - \frac{\delta_{(x,y)}}{\sqrt{2}} , y + \frac{\delta_{(x,y)}}{\sqrt{2}} \bigg). \label{555}
        \end{align}
        Aus der Definition einer Topologie folgt nun aber, dass die rechte Seite von \eqref{555} in $\tau_1$ enthalten sein muss, woraus unmittelbar folgt, dass auch $\mathcal{U}$ in $\tau_1$ enthalten sein muss. Da $\mathcal{U} \in \tau_2$ beliebig gewählt war, folgt, dass $\tau_2 \subseteq \tau_1$ gilt.
        
    \end{itemize}
    Damit ist gezeigt, dass $\tau_1 = \tau_2$ gilt.
\end{proof}

Mittels dem Obigem wissen wir damit, was wir unter einer glatten Deformation zu verstehen haben. Damit können wir nun den Begriff der schwachen Geodätischen verallgemeinern:

\begin{definition}
    Sei $(\mathcal{M}, \tau, \mathcal{A}, g)$ eine riemannsche Mannigfaltigkeit und sei $\gamma_0 \in \mathcal{PC}^\infty_{p,q} ([a,b], \mathcal{M})$ eine Kurve mit $a,b \in \mathbb{R}$ und $a<b$, sowie $p,q \in \mathcal{M}$. Dann bezeichnen wir $\gamma_0$ als Geodätische von $\mathbb{L}_{[a,b]} |_{\mathcal{PC}^\infty_{p,q} ([a,b], \mathcal{M})}$, falls für alle glatten Deformationen $\gamma$ von $\gamma_0$, für die die zugehörige Variationsfunktion $F_\gamma$ an der Stelle $\alpha = 0$ differenzierbar ist, gilt, dass
    \begin{align}
        \frac{dF_\gamma}{d\alpha} \bigg|_{\alpha = 0} = 0 \label{glatte Variation}
    \end{align}
    ist.
\end{definition}

Da jede glatte Deformation per Definition auch eine stetige Deformation ist, und für eine Geodätische im Gegensatz zu einer schwachen Geodätischen der Ausdruck \eqref{glatte Variation} nur für glatte Variationen zu gelten braucht, handelt es sich beim Begriff der Geodätischen tatsächlich um eine Verallgemeinerung des Begriffes der schwachen Geodätischen. So ist jede schwache Geodätische automatisch eine Geodätische, aber nicht jede Geodätische muss notwendigerweise eine schwache Geodätische sein.

Man kann sich nun natürlich die Frage stellen, woher die Bezeichnung schwache Geodätische kommt, wenn es sich bei diesem Begriff verglichen mit dem Begriff der Geodätischen, eigentlich um den stärkeren Begriff handelt. Der Grund für diese Bezeichnung liegt in den schwächeren Vorraussetzungen begründet, die man an die Deformationen stellt. Während für eine schwache Geodätische die Deformationen lediglich stetig zu sein brauchen, müssen die Deformationen für eine Geodätische zusätzlich noch glatt sein.  

Im nächsten Abschnitt werden wir sehen, dass wir mittels des Begriffes der glatten Deformation ein leicht zu überprüfendes notwendiges Kriterium für die schwachen Geodätischen herleiten können.


\subsection{Die Geodätengleichung}

Wir wollen nun die schwachen Geodätischen einer riemannschen Mannigfaltigkeit $(\mathcal{M}, \tau, \mathcal{A}, g)$ etwas genauer studieren und ein leicht zu überprüfendes notwendiges Kriterium ableiten, welches für die schwachen Geodätischen einer riemannschen Mannigfaltigkeit immer gilt. 

Mit diesem Kriterium werden wir also in der Lage sein leicht zu überprüfen, ob es sich bei einer gegebenen Kurve tatsächlich um einen möglichen Kandidaten einer schwachen Geodätischen handelt. Darüberhinaus werden wir in einem Spezialfall auch sehen, dass dieses Kriterium verwendet werden kann um die Geodätischen einer riemannschen Mannigfaltigkeit auszurechnen.

Um dieses Werkzeug abzuleiten, müssen wir noch kurz zeigen, dass falls wir eine schwache Geodätische $\gamma_0 : [a,b] \longrightarrow \mathcal{M}$ gegeben haben, die die Punkte $\gamma_0(a) = p$ und $\gamma_0 (b) = q$ miteinander verbindet, die Einschränkung von $\gamma_0$ auf ein beliebiges Intervall $[c,d] \subseteq [a,b]$ abermals eine schwache Geodätische liefert, die aber nun die Punkte $r := \gamma_0(c)$ und $s := \gamma_0 (d)$ miteinander verbindet. 

Anschaulich ist dabei klar, warum das stimmen muss. Nehmen wir nämlich beispielsweise an, dass $\gamma_0$ ein Minimum des Längenfunktionals $\mathbb{L}_{[a,b]}$ darstellt, dann muss die Einschränkung $\gamma_0 |_{[c,d]}$ ein Minimum des Längenfunktionals $\mathbb{L}_[c,d]$ sein, da wir ansonsten eine (stückweise) glatte Kurve $\hat{\gamma}_0 : [c,d] \longrightarrow \mathcal{M}$ mit $\hat{\gamma}_0 (d) = \gamma_0 (c)$ und $\hat{\gamma}_0 (d) = \gamma_0 (d)$ finden können, sodass 
\begin{align}
    \mathbb{L}_{[c,d]} (\hat{\gamma}_0) < \mathbb{L}_{[c,d]} (\gamma_0 |_{[c,d]})
\end{align}
gilt. Für $\mathbb{L}_{[a,b]}(\gamma_0)$ würde dann wegen der Linearität des Integrals folgen, dass
\begin{align}
    \mathbb{L}_{[a,b]}(\gamma_0) =& \: \mathbb{L}_{[a,c]} (\gamma_0 |_{[a,c]}) + \mathbb{L}_{[c,d]} (\gamma_0 |_{[c,d]}) + \mathbb{L}_{[d,b]} (\gamma_0 |_{[d,b]}) \nonumber \\ >& \: \mathbb{L}_{[a,c]} (\gamma_0 |_{[a,c]}) + \mathbb{L}_{[c,d]} (\hat{\gamma}_0 |_{[c,d]}) + \mathbb{L}_{[d,b]} (\gamma_0 |_{[d,b]}) \nonumber \\ =& \: \mathbb{L}_{[a,b]} (\Tilde{\gamma}_0)
\end{align}
mit $\Tilde{\gamma}_0 \in \mathcal{PC}^\infty ([a,b], \mathcal{M})$, definiert durch
\begin{align}
    \tilde{\gamma}_0 (\lambda) = \left\{\begin{array}{ll} \gamma_0 (\lambda), & \lambda \in [a,c] \cup [d,b]  \\
         \hat{\gamma}_0 (\lambda), & \lambda \in (c,d) \end{array}\right. .
\end{align}
In der Menge $\mathcal{PC}^\infty_{p,q} ([a,b], \mathcal{M})$ hätten wir demnach ein Kurve gefunden, welche eine echt kürzere Verbindungskurve zwischen $p$ und $q$ als $\gamma_0$ wäre, im Widerspruch dazu, dass es sich bei $\gamma_0$ bereits um die kürzeste Kurve aus $\mathcal{PC}^\infty_{p,q} ([a,b], \mathcal{M})$ handelt.

Wie zeigen wir diese Aussage nun im Allgemeinen für alle stationären Kurven des Längenfunktionals $\mathbb{L}_{[a,b]}$? Einen Begriff, welchen wir dazu benötigen, sind die sogenannten trivialen Fortsetzungen stetiger Deformationen.

\begin{definition}
    Sei $(\mathcal{M}, \tau, \mathcal{A})$ wieder eine glatte Mannigfaltigkeit und sei $\gamma_0 \in \mathcal{PC}^\infty_{p,q} ([a,b], \mathcal{M})$ mit $a,b \in \mathbb{R}$ und $a<b$ eine Kurve. Sei weiter $[c,d] \subseteq [a,b]$ gegeben und $\gamma_0 |_{[c,d]} := \hat{\gamma}_0$. Darüberhinaus sei $\hat{\gamma} = (\hat{\gamma}_\alpha)_{\alpha \in (-\epsilon, \epsilon)}$ mit $\epsilon > 0$ eine stetige Deformation von $\hat{\gamma}_0$. Dann ist die sogenannte triviale Fortsetzung der stetigen Deformation $\hat{\gamma}$ entlang der Kurve $\gamma_0$ erklärt als die durch den Parameter $\alpha \in (- \epsilon, \epsilon)$ parametrisierte Familie stetiger Kurven $\gamma_\alpha : [a,b] \longrightarrow \mathcal{M}$ mit
    \begin{align}
        \gamma_\alpha (\lambda) = \left\{\begin{array}{ll} \gamma_0 (\lambda), & \lambda \in [a,c) \cup (d,b] \\ \hat{\gamma}_\alpha (\lambda), & \lambda \in [c,d] \end{array} \:\:\:\:\:\:\:\: \forall \alpha \in (-\epsilon, \epsilon).\right.
    \end{align}
\end{definition}

Bei der trivialen Fortsetzung der stetigen Deformation $\hat{\gamma}$ entlang der Kurve $\gamma_0$ handelt es sich also anschaulich um eine Deformation der Kurve $\gamma_0$, bei der nur der Kurvenabschnitt $\gamma_0 ([c,d])$ variiert wird. Im Folgenden wollen wir explizit zeigen, dass es sich bei dieser Art von Deformation tatsächlich auch um eine stetige Deformation handelt.

Bevor wir dies jedoch zeigen können, benötigen wir noch das folgende Lemma:

\begin{lemma}
    Sei $(\mathbb{R}, \tau_{\mathbb{R}})$ die Menge der reellen Zahlen, versehen mit der Standardtopologie, und $(X, \tau)$ ein beliebiger topologischer Raum. Seien weiter $[a,c], [c,d] \subseteq \mathbb{R}$ zwei reelle Teilintervalle mit $a,c,d \in \mathbb{R}$ und $a < c < d$. Wir statten $[a,c]$ und $[c,d]$ jeweils mit der Teilraumtopologie $\tau_{\mathbb{R}, [a,c]}$ bzw. $\tau_{\mathbb{R}, [c,d]}$ aus und betrachten die stetigen Funktionen $f : [a,c] \longrightarrow X$ und $g : [c,d] \longrightarrow X$ mit
    \begin{align}
        f(c) = g(c).
    \end{align}
    Dann ist die Funktion $h : [a,d] \longrightarrow X$, definiert durch 
    \begin{align}
        h (\lambda) = \left\{\begin{array}{ll} f(\lambda) , & \lambda \in [a,c) \\ g (\lambda), & \lambda \in [c,d] \end{array} \:\:\:\:\:\:\:\: \forall \lambda \in [a,d]\right.,
    \end{align}
    eine $(\tau_{\mathbb{R}, [a,d]}, \tau)$-stetige Funktion.
\end{lemma}

\begin{proof}
    Wir beginnen damit die Definitionen der Teilraumtopologien zu wiederholen:
    \begin{align}
        &\tau_{\mathbb{R},[a,c]} := \{ \mathcal{V} \subseteq [a,c] \:\: | \:\: \exists \mathcal{X} \in \tau_\mathbb{R}, \: \textit{sodass} \:\: \mathcal{X} \cap [a,c] = \mathcal{V} \} \\ 
        &\tau_{\mathbb{R},[c,d]} := \{ \mathcal{V}' \subseteq [c,d] \:\: | \:\: \exists \mathcal{X}' \in \tau_\mathbb{R}, \: \textit{sodass} \:\: \mathcal{X}' \cap [c,d] = \mathcal{V}' \} \\ 
        &\tau_{\mathbb{R},[a,d]} := \{ \mathcal{V}'' \subseteq [a,d] \:\: | \:\: \exists \mathcal{X}'' \in \tau_\mathbb{R}, \: \textit{sodass} \:\: \mathcal{X}'' \cap [a,d] = \mathcal{V}'' \}
    \end{align}
    Sei nun $\mathcal{W} \in \tau$ eine beliebige offene Menge aus $\tau$. Wir betrachten die Menge $h^{-1}(\mathcal{W})$:
    \begin{align}
        h^{-1}(\mathcal{W}) = f^{-1}(\mathcal{W}) \cup g^{-1}(\mathcal{W})
    \end{align}
    Da sowohl $f$, als auch $g$ stetig sind, folgt $f^{-1} (\mathcal{W}) \in \tau_{\mathbb{R}, [a,c]}$ und $g^{-1} (\mathcal{W}) \in \tau_{\mathbb{R}, [c,d]}$. Wir betrachten nun zwei Fälle: 
    \begin{itemize}
        \item [\textit{i)}] $c \notin h^{-1}(\mathcal{W})$
        
        \item [\textit{ii)}] $c \in h^{-1}(\mathcal{W})$
    \end{itemize}
    Beachte, dass falls $c \in h^{-1}(\mathcal{W})$ ist, wegen $f(c) = g(c)$ auch sofort folgt, dass $c \in f^{-1}(\mathcal{W})$ und $c \in g^{-1}(\mathcal{W})$ ist.
    
    Betrachten wir nun den ersten Fall $\textit{i)}$. Per Definition der Topologie $\tau_\mathbb{R}$ und der Definition der Teilraumtopologien folgt aus $\mathcal{V} \in \tau_{\mathbb{R}, [a,c]}$ mit $c \notin \mathcal{V}$, dass dann $\mathcal{V} \in \tau_{\mathbb{R}, [a,d]}$ ist und Analoges gilt dann aus Symmetriegründen natürlich auch für $\tau_{\mathbb{R}, [c,d]}$ und $\tau_{\mathbb{R}, [a,d]}$. Die Richtigkeit dieser Behauptung sieht man folgendermaßen ein: Sei $\mathcal{V} \in \tau_{\mathbb{R}, [a,c]}$ mit $c \notin \mathcal{V}$. Dann existiert eine offene Menge $\mathcal{X} \in \tau_\mathbb{R}$ mit $c \notin \mathcal{X}$ und $\mathcal{V} = \mathcal{X} \cap [a,c]$. Sei weiter $\mathcal{N} := \mathcal{X} \setminus (- \infty, c]$. Per Definition von $\tau_\mathbb{R}$ ist $\mathcal{N}$ offen bzgl. $\tau_\mathbb{R}$, da wir um jeden Punkt $\lambda \in \mathbb{N}$ nachwievor eine $\delta$-Umgebung legen können, die vollständig in $\mathcal{N}$ enthalten ist, denn angenommen dem wäre nicht so, d.h. es gebe einen Punkt $\lambda \in \mathcal{N}$, sodass kein $\delta > 0$ existiert, sodass $\mathcal{U}_\delta (\lambda) \subseteq \mathcal{N}$ gilt. Dann gibt es mindestens ein $\mu \in \mathcal{U}_\delta (\lambda)$, welches in $(-\infty, c]$ liegt. Aus der Definition von $\mathcal{U}_\delta (\lambda)$ folgt dann aber, dass $c \in \mathcal{U}_\delta (\lambda)$ ist. Da das für $\delta > 0$ gilt, und $\mathcal{X}$ als offen angenommen wurde, folgt, dass $c \in \mathcal{X}$ ist, im Widerspruch zur Annahme, dass $c \notin \mathcal{V}$ ist. Daraus folgt, dass $\mathcal{N}$ offen bezüglich $\tau_\mathbb{R}$ sein muss. Auf analoge Weise zeigt man, dass die Menge $\mathcal{G} := \mathcal{X} \setminus [c, \infty)$ offen bezüglich $\tau_\mathbb{R}$ ist. Insgesamt lässt sich damit also $\mathcal{X}$ schreiben als 
    \begin{align}
        \mathcal{X} = \mathcal{N} \cup \mathcal{G},
    \end{align}
    wobei $\mathcal{N}$ und $\mathcal{G}$ bezüglich $\tau_\mathbb{R}$ offen sind. Aus der Definition der Menge $\mathcal{N}$ und $\mathcal{G}$ folgt nun, dass 
    \begin{align}
        \mathcal{V} = \mathcal{X} \cap [a,c] = \mathcal{G} \cap [a,c] = \mathcal{G} \cap [a,d] \in \tau_{\mathbb{R}, [a,d]}
    \end{align}
    ist. Da aus $h^{-1}(\mathcal{W})$ folgt, dass $c \notin f^{-1}(\mathcal{W})$ und $c \notin g^{-1}(\mathcal{W})$ ist, und weiter aus der Stetigkeit von $f$ und $g$ folgt, dass $f^{-1}(\mathcal{W}) \in \tau_{\mathbb{R}, [a,c]}$ und $g^{-1}(\mathcal{W}) \in \tau_{\mathbb{R}, [c,d]}$ ist, folgt nach obigem, dass $f^{-1}(\mathcal{W}) \in \tau_{\mathbb{R}, [a,d]}$ und $g^{-1}(\mathcal{W}) \in \tau_{\mathbb{R}, [a,d]}$ ist. Aus der Definition einer Topologie folgt damit, dass $h^{-1}(\mathcal{W}) \in \tau_{\mathbb{R}, [a,d]}$ ist.

    Betrachten wir nun den zweiten Fall \textit{ii)}. Sei also $c \in h^{-1}(\mathcal{W})$. Wir betrachten nun zwei allgemeine Mengen $\mathcal{V} \in \tau_{\mathbb{R}, [a,c]}$ und $\mathcal{U} \in \tau_{\mathbb{R}, [c,d]}$ mit $c \in \mathcal{V}$ und $c \in \mathcal{U}$. Wegen $\mathcal{V} \in \tau_{\mathbb{R}, [a,c]}$ mit $c \in \mathcal{V}$ existiert ein $\mathcal{X} \in \tau_\mathbb{R}$ mit $c \in \mathcal{X}$ und
    \begin{align}
        \mathcal{V} = \mathcal{X} \cap [a,c].
    \end{align}
    Analog existiert wegen $\mathcal{U} \in \tau_{\mathbb{R}, [c,d]}$ und $c \in \mathcal{U}$ ein $\mathcal{X}' \in \tau_\mathbb{R}$ mit $c \in \mathcal{X}'$ und 
    \begin{align}
        \mathcal{U} = \mathcal{X}' \cap [c,d]. 
    \end{align}
    Wegen $\mathcal{X}, \mathcal{X}' \in \tau_\mathbb{R}$ existieren nach der Definition der Topologie $\tau_\mathbb{R}$ Zahlen $\delta_1, \delta_2 > 0$, sodass
    \begin{align}
        \mathcal{U}_{\delta_1}(c) \subseteq \mathcal{X} \:\:\:\: \textit{und} \:\:\:\: \mathcal{U}_{\delta_2}(c) \subseteq \mathcal{X}'.
    \end{align}
    Sei $\delta = \textit{min}\{\delta_1, \delta_2\}$, so gilt $\mathcal{U}_\delta (c) \subseteq \mathcal{X}$ und $\mathcal{U}_\delta(c) \subseteq \mathcal{X}'$. Wir definieren nun zwei neue Mengen durch
    \begin{align}
        &\mathcal{O} := (\mathcal{X} \cap [a,c]) \cup \mathcal{U}_\delta (c), \\
        &\mathcal{K} := (\mathcal{X}' \cap [c,d]) \cup \mathcal{U}_\delta (c).
    \end{align}
    Per Definition gilt $\mathcal{O} \subseteq \mathcal{X}$ und $\mathcal{K} \subseteq \mathcal{X}'$. Weiter sind die Mengen $\mathcal{O}$ und $\mathcal{K}$ offen bezüglich der Toplogie $\tau_\mathbb{R}$, denn betrachten wir beispielsweise die Menge $\mathcal{O}$, so gilt für alle $x \in \mathcal{X} \cap [a,c]$, dass ein $\hat{\delta} > 0$ mit $\hat{\delta} < \delta$ existiert, sodass $\mathcal{U}_{\hat{\delta}}(x) \subseteq \mathcal{O}$. Weiter gilt, dass für alle $x \in \mathcal{U}_\delta (c)$ folgt, da $\mathcal{U}_\delta (c) \in \tau_\mathbb{R}$ ist, dass ein $\Tilde{\delta}>0$ existiert, sodass $\mathcal{U}_{\Tilde{\delta}}(x) \subseteq \mathcal{U}_\delta (c) \subseteq \mathcal{O}$, womit gezeigt ist, dass $\mathcal{O}$ eine offene Menge bezüglich $\tau_\mathbb{R}$ ist. Analog zeigt man $\mathcal{K} \in \tau_\mathbb{R}$. Es gilt nun die folgende Mengengleichheit:
    \begin{align}
        \mathcal{V} \cup \mathcal{U} = \big( [a,c] \cap \mathcal{X} \big) \bigcup \big( [c,d] \cap \mathcal{X}' \big) = \big(\underbrace{ [a,c] \cup [c,d] }_{\textit{$[a,d]$}} \big) \bigcap \big(\underbrace{ \mathcal{O} \cup \mathcal{K} }_{\textit{$\in \tau_\mathbb{R}$}}\big)
    \end{align}
    Mittels dieser Mengengleichheit folgt offensichtlicherweise, dass $\mathcal{V} \cup \mathcal{U} \in \tau_{\mathbb{R}, [a,d]}$. Zeigen wir also noch die Korrektheit dieser Mengengleichung. Dazu zeigen wir zuerst, dass aus $x \in \big( [a,c] \cap \mathcal{X} \big) \bigcup \big( [c,d] \cap \mathcal{X}' \big)$ auch $x \in \big( [a,c] \cup [c,d] \big) \bigcap \big( \mathcal{O} \cup \mathcal{K} \big)$ folgt. Sei also $x \in \big( [a,c] \cap \mathcal{X} \big) \bigcup \big( [c,d] \cap \mathcal{X}' \big)$. Dann ist $x$ enthalten in $[a,c] \cap \mathcal{X}$ oder in $[c,d] \cap \mathcal{X}'$. Es gilt
    \begin{align}
        x \in [a,c] \cap \mathcal{X} \subseteq& \: [a,c] \bigcap \big( [a,c] \cap \mathcal{X} \big) \nonumber \\ \subseteq& \: [a,c] \bigcap \big(\underbrace{ ([a,c] \cap \mathcal{X}) \cup \mathcal{U}_\delta (c) }_{\textit{$\mathcal{O}$}}\big) \nonumber \\ \subseteq& \: \big( [a,c] \cup [c,d]  \big) \bigcap \big( \mathcal{O} \cup \mathcal{K} \big)
    \end{align}
    und 
    \begin{align}
        x \in [c,d] \cap \mathcal{X}' \subseteq& \: [c,d] \bigcap \big( [c,d] \cap \mathcal{X}' \big) \nonumber \\ \subseteq& \: [c,d] \bigcap \big(\underbrace{ ([c,d] \cap \mathcal{X}') \cup \mathcal{U}_\delta (c) }_{\textit{$\mathcal{K}$}}\big) \nonumber \\ \subseteq& \: \big( [a,c] \cup [c,d]  \big) \bigcap \big( \mathcal{O} \cup \mathcal{K} \big)
    \end{align}
    Damit ist die erste Inklusion gezeigt. Sei nun $x \in \big( [a,c] \cup [c,d] \big) \bigcap \big( \mathcal{O} \cup \mathcal{K} \big)$. Wir zeigen, dass daraus $ x \in \big( [a,c] \cap \mathcal{X} \big) \bigcup \big( [c,d] \cap \mathcal{X}' \big)$ folgt. Aus der Bedingung $x \in \big( [a,c] \cup [c,d] \big) \bigcap \big( \mathcal{O} \cup \mathcal{K} \big)$ folgt, dass $x \in [a,c] \cup [c,d]$ und in $\mathcal{O} \cup \mathcal{K}$ liegen muss. Es gelten die folgenden Inklusionen:
    \begin{align}
        x \in [a,c] \cap \mathcal{O} \subseteq& \: [a,c] \cap \mathcal{X} \nonumber \\ \subseteq& \: \big( [a,c] \cap \mathcal{X} \big) \bigcup \big( [c,d] \cap \mathcal{X}' \big), \\
        x \in [c,d] \cap \mathcal{K} \subseteq& \: [c,d] \cap \mathcal{X}' \nonumber \\ \subseteq& \: \big( [a,c] \cap \mathcal{X} \big) \bigcup \big( [c,d] \cap \mathcal{X}' \big), \\
        x \in [a,c] \cap \mathcal{K} \subseteq& \: [a,c] \cap \mathcal{U}_\delta(c) \nonumber \\ \subseteq& \: [a,c] \cap \mathcal{X} \nonumber \\ \subseteq& \:  \big( [a,c] \cap \mathcal{X} \big) \bigcup \big( [c,d] \cap \mathcal{X}' \big),
    \end{align}
    und 
    \begin{align}
        x \in [c,d] \cap \mathcal{O} \subseteq& \: [c,d] \cap \mathcal{U}_\delta(c) \nonumber \\ \subseteq& \: [c,d] \cap \mathcal{X}' \nonumber \\ \subseteq& \:  \big( [a,c] \cap \mathcal{X} \big) \bigcup \big( [c,d] \cap \mathcal{X}' \big).
    \end{align}
    Damit ist die Mengengleichheit gezeigt. Haben wir also zwei Mengen $\mathcal{V} \in \tau_{\mathbb{R}, [a,c]}$ und $\mathcal{U} \in \tau_{\mathbb{R}, [c,d]}$ mit $c \in \mathcal{V}$ und $c \in \mathcal{U}$, so gilt $\mathcal{V} \cup \mathcal{U} \in \tau_{\mathbb{R}, [a,d]}$. Gilt also $c \in h^{-1} (\mathcal{W})$, so folgt damit unmittelbar $h^{-1}(\mathcal{W}) \in \tau_{\mathbb{R}, [a,d]}$. Insgesamt folgt damit also die Stetigkeit von $h$.
\end{proof}

Nun können wir zeigen, dass es sich bei den obigen trivialen Fortsetzungen stetiger Deformationen tatsächlich um stetige Deformationen handelt.

\begin{proposition}
    Sei $(\mathcal{M}, \tau, \mathcal{A})$ eine glatte Mannigfaltigkeit und $\gamma_0 \in \mathcal{PC}^\infty_{p,q} ([a,b], \mathcal{M})$ mit $a,b \in \mathbb{R}$ und $a<b$ eine Kurve, die im Punkt $p \in \mathcal{M}$ beginnt und im Punkt $q \in \mathcal{M}$ endet. Sei weiter $[c,d] \subseteq [a,b]$ gegeben und $\gamma_0 |_{[c,d]} := \hat{\gamma}_0$. Weiter sei $\hat{\gamma} = (\hat{\gamma}_\alpha)_{\alpha \in (-\epsilon, \epsilon)}$ mit $\epsilon > 0$ eine stetige Deformation von $\hat{\gamma}_0$. Dann handelt es sich bei der trivialen Fortsetzung $\gamma = (\gamma_\alpha)_{\alpha \in (-\epsilon, \epsilon)}$ von $\hat{\gamma}$ entlang der Kurve $\gamma_0$ um eine stetige Deformation der Kurve $\gamma_0$.
\end{proposition}

\begin{proof}
    Aus Lemma $7.1.$ folgt sofort, dass für alle $\alpha \in (- \epsilon, \epsilon)$ gilt, dass $\gamma_\alpha \in \mathcal{C}^0 ([a,b], \mathcal{M})$ ist. Weiter folgt, da $\gamma_0 \in \mathcal{PC}^\infty_{p,q} ([a,b], \mathcal{M})$ und $\hat{\gamma}_\alpha \in \mathcal{PC}^\infty_{r,s} ([c,d], \mathcal{M})$ für alle $\alpha \in (-\epsilon, \epsilon)$ mit $r := \gamma_0(c)$ und $s := \gamma_0(d)$, dass $\gamma_\alpha \in \mathcal{PC}^\infty_{p,q} ([a,b], \mathcal{M})$ für alle $\alpha \in (-\epsilon, \epsilon)$ ist. Es bleibt die Stetigkeit von $\gamma$ zu zeigen. Sei dazu $\mathcal{U} \in \tau$ eine beliebige offene Menge der Mannigfaltigkeit $(\mathcal{M}, \tau, \mathcal{A})$. Es können folgende Fälle eintreten:
    \begin{itemize}
    
        \item[\textit{i)}] Gilt $\mathcal{U} \cap \gamma([a,b], [-\epsilon, \epsilon]) = \emptyset$, so folgt, dass $\gamma^{-1}(\mathcal{U}) = \emptyset$ und $\emptyset$ ist per Definition einer Topologie offen bezüglich der Produkttopologie $\tau_{[a,b] \times (-\epsilon, \epsilon)}$.
        
        \item[\textit{ii)}] Gilt $\mathcal{U} \cap \gamma([a,b], (-\epsilon, \epsilon)) \subseteq \gamma([a,b] \setminus [c,d], (-\epsilon, \epsilon))$, so folgt 
        \begin{align}
            \gamma^{-1}(\mathcal{U}) = \gamma_0^{-1}(\mathcal{U}) \times (-\epsilon, \epsilon).
        \end{align} 
        Da $\gamma_0 : [a,b] \longrightarrow \mathcal{M}$ eine stetige Funktion ist, folgt auch in diesem Fall, dass $\gamma^{-1}(\mathcal{U})$ offen bezüglich der Produkttopologie $\tau_{[a,b] \times (-\epsilon, \epsilon)}$ ist.
        
        \item[\textit{iii)}] Es gelte $\mathcal{U} \cap \gamma([a,b], (-\epsilon, \epsilon)) \subseteq \gamma((c,d), (-\epsilon, \epsilon))$. Wegen dem Beweis von Lemma $7.1.$ und der Definition der Produkttopologie folgt
        \begin{align}
            \tau_{(c,d) \times (-\epsilon, \epsilon)} \subseteq \tau_{[c,d] \times (-\epsilon, \epsilon)} \:\: \textit{und} \:\: \tau_{(c,d) \times (-\epsilon, \epsilon)} \subseteq \tau_{[a,b] \times (-\epsilon, \epsilon)}.
        \end{align}
        Damit gilt 
        \begin{align}
            \gamma^{-1}(\mathcal{U}) = \hat{\gamma}^{-1}(\mathcal{U}) \in \tau_{(c,d) \times (-\epsilon, \epsilon)} \subseteq \tau_{[a,b] \times (-\epsilon, \epsilon)},
        \end{align}
        womit $\gamma^{-1}(\mathcal{U})$ bezüglich der Produkttopologie $\tau_{[a,b] \times (-\epsilon, \epsilon)}$ offen ist.
        
        \item[\textit{iv)}] Es gelte $\mathcal{U} \cap \gamma([a,b], (-\epsilon, \epsilon)) \subseteq \gamma((c,d), (-\epsilon, \epsilon)) \cup \gamma([a,b] \setminus [c,d], (-\epsilon, \epsilon))$. Dann gilt
        \begin{align}
            \gamma^{-1}(\mathcal{U}) = \big( \gamma_0 |_{[a,c) \cup (d,b]}^{-1} (\mathcal{U}) \times (-\epsilon, \epsilon) \big) \cup \hat{\gamma}^{-1}(\mathcal{U}).
        \end{align}
        Da $\gamma_0 |_{[a,c) \cup (d,b]}$ nach Lemma $3.3.$ eine $(\tau_{\mathbb{R}, [a,c)\cup (d,b]}, \tau)$-stetige Funktion ist und $\tau_{[a,c)\cup (d,b] \times (-\epsilon, \epsilon)}$ per Konstruktion eine Teilmenge der Produkttopologie $\tau_{[a,b] \times (-\epsilon, \epsilon)}$ ist, folgt, dass $\gamma^{-1}(\mathcal{U})$ offen bezüglich der Produkttopologie $\tau_{[a,b] \times (-\epsilon, \epsilon)}$ ist.
        
        \item[\textit{v)}] Sei $\mathcal{U}$ nun so gewählt, sodass die Mengen $\{c\} \times (-\epsilon, \epsilon)$ oder $\{d\} \times (-\epsilon, \epsilon)$ oder beide in $\gamma^{-1}(\mathcal{U})$ liegen. Dann folgt 
        \begin{align}
            \gamma^{-1}(\mathcal{U}) = \big( \gamma_0 |_{[a,c]}^{-1}(\mathcal{U}) \times (-\epsilon, \epsilon) \big) \cup \hat{\gamma}^{-1}(\mathcal{U}) \cup \big( \gamma_0 |_{[d,b]}^{-1} \times (-\epsilon, \epsilon) \big)
        \end{align}
        mit 
        \begin{align}
            \{c\} \times (-\epsilon, \epsilon) \subseteq  \gamma_0 |_{[a,c]}^{-1}(\mathcal{U}) \times (-\epsilon, \epsilon) \cap \hat{\gamma}^{-1}(\mathcal{U})
        \end{align}
        und/oder
        \begin{align}
            \{d\} \times (-\epsilon, \epsilon) \subseteq  \gamma_0 |_{[c,d]}^{-1}(\mathcal{U}) \times (-\epsilon, \epsilon) \cap \hat{\gamma}^{-1}(\mathcal{U}).
        \end{align}
        Da $\gamma_0 |_{[a,c]}$ und $\gamma_0 |_{[d,b]}$ stetig sind, folgt mittels der Argumente aus dem Beweis von Lemma $7.1.$ und der Definition der Produkttopologie, dass $\gamma^{-1}(\mathcal{U})$ offen bezüglich der Produkttopologie $\tau_{[a,b] \times (-\epsilon, \epsilon)}$ ist.
        
    \end{itemize}
    Insgesamt folgt damit die Stetigkeit der trivialen Fortsetzung $\gamma$ entlang der Kurve $\gamma_0$, womit gezeigt ist, dass es sich bei $\gamma$ um eine stetige Deformation handelt.
\end{proof}

Die Bedeutung des letzten Resultates ist nun die Folgende: Angenommen bei der Kurve $\gamma_0 \in \mathcal{PC}^\infty_{p,q} ([a,b], \mathcal{M})$ handelt es sich um eine Kurve auf der riemannschen Mannigfaltigkeit $(\mathcal{M}, \tau, \mathcal{A}, g)$, die die Punkte $p$ und $q$ aus $\mathcal{M}$ miteinander verbindet. Weiter sei $[c,d]$ mit $c,d \in \mathbb{R}$ und $c < d$ ein beleibiges Teilintervall von $[a,b]$. Dann lässt sich jede stetige Deformation $\hat{\gamma}$ von $\hat{\gamma}_0 := \gamma_0 |_{[c,d]}$ zu einer stetigen Deformation $\gamma$ von $\gamma_0$ fortsetzen, sodass die Fortsetung mit der triviale Fortsetzung von $\hat{\gamma}$ entlang der Kurve $\gamma_0$ übereinstimmt und wir können den folgenden Satz zeigen:

\begin{proposition}
    Sei $(\mathcal{M}, \tau, \mathcal{A}, g)$ eine riemannsche Mannigfaltigkeit und $a,b \in \mathbb{R}$ zwei Zahlen mit $a < b$. Sei weiter $\gamma_0 \in \mathcal{PC}^\infty_{p,q} ([a,b], \mathcal{M})$ eine schwache Geodätische, die den Punkt $p = \gamma(a) \in \mathcal{M}$ mit dem Punkt $q = \gamma(b) \in\mathcal{M}$ verbindet. Sei weiter $[c,d]$ mit $c,d \in \mathbb{R}$ und $c < d$ ein beliebiges Teilintervall von $[a,b]$. Dann ist die Kurve $\hat{\gamma}_0 := \gamma_0 |_{[c,d]} \in \mathcal{PC}^\infty_{r,s} ([c,d], \mathcal{M})$ mit $r = \gamma_0(c)$ und $s = \gamma_0(d)$ eine schwache Geodätische, d.h. eine stationäre Kurve bezüglich des Längenfunktionals $\mathbb{L}_{[c,d]} |_{\mathcal{PC}^\infty_{r,s} ([c,d], \mathcal{M})}$.
\end{proposition}

\begin{proof}
    Ohne Beschränkung der Allgemeinheit nehmen wir an, dass alle Kurven der Klasse $\mathcal{C}^\infty$ zuzuordnen sind. Für Kurven aus der allgemeineren Klasse $\mathcal{PC}^\infty$ ist die Argumentation analog, nur dass man für das Längenfunktional statt des Ausdruckes \eqref{Längenfunktional} den komplizierteren Ausdruck \eqref{Längenfunktional 2} nutzen muss, bei dem die Integrationsgrenzen gegebenfalls vom Deformationsparameter $\alpha$ abhängen, falls man das Längenfunktional bezüglich einer stetigen Deformation $(\gamma_\alpha)_{\alpha \in (-\epsilon, \epsilon)}$ mit $\gamma_\alpha$ aus $\mathcal{PC}^\infty$ für alle $\alpha \in (-\epsilon, \epsilon)$ auswertet, bei welcher die Stellen, an denen die Kurve $\gamma_\alpha$ nicht differenzierbar ist, direkt vom Parameter $\alpha$ abhängig sind.
    
    Sei also nun $\gamma_0 \in \mathcal{C}^\infty_{p,q} ([a,b], \mathcal{M})$ eine schwache Geodätische, die in $p$ beginnt und in $q$ endet und sei $[c,d] \subseteq [a,b]$ beliebig gegeben. Wir betrachten die Kurve $\hat{\gamma}_0 := \gamma_0 |_{[c,d]}$. Da $\gamma_0$ eine schwache Geodätische ist, gilt für beliebige stetige Deformationen $\Tilde{\gamma}$ der Kurve $\gamma_0$, für die das Variationsfunktional $F_{\Tilde{\gamma}}$, definiert durch
    \begin{align}
        F_{\Tilde{\gamma}}(\alpha) = \mathbb{L}_{[a,b]}(\Tilde{\gamma}_\alpha),
    \end{align}
    differenzierbar an der Stelle $\alpha = 0$ ist, dass
    \begin{align}
        \frac{dF_{\Tilde{\gamma}}}{d\alpha} \bigg|_{\alpha = 0} = 0
    \end{align}
    ist. Sei $\hat{\gamma}$ eine beliebige stetige Deformation der Kurve $\hat{\gamma}_0$, sodass das Variationsfunktional $F_{\hat{\gamma}}$ mit $F_{\hat{\gamma}}(\alpha) = \mathbb{L}_{[c,d]}(\hat{\gamma}_\alpha)$ am Punkt $\alpha = 0$ differenzierbar ist. Wir wollen zeigen, dass für derartige stetige Deformationen von $\hat{\gamma}_0$ folgt, dass 
    \begin{align}
        \frac{dF_{\hat{\gamma}}}{d\alpha} \bigg|_{\alpha = 0} = 0
    \end{align}
    ist. Wir betrachten die triviale Fortsetzung der stetigen Deformation $\hat{\gamma}$ entlang der Kurve $\gamma_0$, welche wir mit $\gamma$ bezeichnen wollen. Nach Satz $7.2.$ handelt es sich bei $\gamma$ um eine stetige Deformation von $\gamma_0$, sodass das Variationsfunktional $F_\gamma$ am Punkt $\alpha = 0$ differenzierbar ist, denn es gilt
    \begin{align}
        \frac{dF_{\gamma}}{d\alpha} \bigg|_{\alpha = 0} =& \: \frac{d}{d\alpha} \bigg( \int_a^b \sqrt{g_{\gamma_\alpha (\lambda)} (X_{\gamma_\alpha, \gamma_\alpha (\lambda)}, X_{\gamma_\alpha, \gamma_\alpha (\lambda)}) } \bigg) \,d \lambda \bigg|_{\alpha = 0} \nonumber \\ =& \: \frac{d}{d\alpha} \bigg( \int_a^c \sqrt{g_{\gamma_\alpha (\lambda)} (X_{\gamma_\alpha, \gamma_\alpha (\lambda)}, X_{\gamma_\alpha, \gamma_\alpha (\lambda)}) } \bigg) \,d \lambda \bigg|_{\alpha = 0} \nonumber \\ &+ \: \frac{d}{d\alpha} \bigg( \int_c^d \sqrt{g_{\gamma_\alpha (\lambda)} (X_{\gamma_\alpha, \gamma_\alpha (\lambda)}, X_{\gamma_\alpha, \gamma_\alpha (\lambda)}) } \bigg) \,d \lambda \bigg|_{\alpha = 0} \nonumber \\ &+ \: \frac{d}{d\alpha} \bigg( \int_d^b \sqrt{g_{\gamma_\alpha (\lambda)} (X_{\gamma_\alpha, \gamma_\alpha (\lambda)}, X_{\gamma_\alpha, \gamma_\alpha (\lambda)}) } \bigg) \,d \lambda \bigg|_{\alpha = 0} \nonumber \\ =& \: \frac{d}{d\alpha} \bigg( \int_a^c \sqrt{g_{\gamma_0 (\lambda)} (X_{\gamma_0, \gamma_0 (\lambda)}, X_{\gamma_0, \gamma_0 (\lambda)}) } \bigg) \,d \lambda \bigg|_{\alpha = 0} \nonumber \\ &+ \: \frac{d}{d\alpha} \bigg( \int_c^d \sqrt{g_{\gamma_\alpha (\lambda)} (X_{\gamma_\alpha, \gamma_\alpha (\lambda)}, X_{\gamma_\alpha, \gamma_\alpha (\lambda)}) } \bigg) \,d \lambda \bigg|_{\alpha = 0} \nonumber \\ &+ \: \frac{d}{d\alpha} \bigg( \int_d^b \sqrt{g_{\gamma_0 (\lambda)} (X_{\gamma_0, \gamma_0 (\lambda)}, X_{\gamma_0, \gamma_0 (\lambda)}) } \bigg) \,d \lambda \bigg|_{\alpha = 0} \nonumber \\ =& \: 0 + \frac{d}{d\alpha} \bigg( \int_c^d \sqrt{g_{\gamma_\alpha (\lambda)} (X_{\gamma_\alpha, \gamma_\alpha (\lambda)}, X_{\gamma_\alpha, \gamma_\alpha (\lambda)}) } \bigg) \,d \lambda \bigg|_{\alpha = 0} + 0 \nonumber \\ =& \frac{d}{d\alpha} \bigg( \int_c^d \sqrt{g_{\hat{\gamma}_\alpha (\lambda)} (X_{\hat{\gamma}_\alpha, \hat{\gamma}_\alpha (\lambda)}, X_{\hat{\gamma}_\alpha, \hat{\gamma}_\alpha (\lambda)}) } \bigg) \,d \lambda \bigg|_{\alpha = 0} \nonumber \\ =& \: \frac{dF_{\hat{\gamma}}}{d\alpha} \bigg|_{\alpha = 0}. \label{777}
    \end{align}
    Insgesamt gilt also 
    \begin{align}
        \frac{dF_{\gamma}}{d\alpha} \bigg|_{\alpha = 0} = \frac{dF_{\hat{\gamma}}}{d\alpha} \bigg|_{\alpha = 0}. \label{999999999}
    \end{align}
    Da die rechte Seite nach Vorraussetzung differenzierbar am Punkt $\alpha = 0$ ist, folgt, dass damit auch die linke Seite am Punkt $\alpha = 0$ differenzierbar sein muss. Da $\gamma_0$ eine schwache Geodätische ist, die die Punkte $p$ und $q$ miteinander verbindet, folgt, dass 
    \begin{align}
        \frac{dF_{\gamma}}{d\alpha} \bigg|_{\alpha = 0} = 0
    \end{align}
    ist. Wegen \eqref{999999999} folgt damit auch
    \begin{align}
        \frac{dF_{\hat{\gamma}}}{d\alpha} \bigg|_{\alpha = 0} = 0.
    \end{align}
    Da $\hat{\gamma}$ eine beliebige stetige Deformation von $\hat{\gamma}_0$ ist, für die das Variationsfunktional $F_{\hat{\gamma}}$ an der Stelle $\alpha = 0$ differenzierbar ist, folgt, dass es sich bei $\hat{\gamma}_0$ um eine schwache Geodätische handelt, die im Punkt $r = \gamma(c)$ beginnt und im Punkt $s = \gamma(d)$ endet.  
\end{proof}

\begin{remark}
    Wie in der Einleitung des Beweises von Satz $7.9.$ angerissen, funktioniert die Argumentation des obigen Beweises auch für Kurven $\gamma_0$ und Deformationen $(\hat{\gamma}_\alpha)_{\alpha \in (-\epsilon, \epsilon)}$ aus der Klasse $\mathcal{PC}^\infty$, da wir auch im Falle, wenn das Längenfunktional durch \eqref{Längenfunktional 2} gegeben ist, den Ausdruck $\mathbb{L}_{[a,b]}(\gamma_\alpha)$ mit $\gamma_\alpha \in \mathcal{PC}^\infty ([a,b], \mathcal{M})$ wegen der Linearität des Integrals wie in \eqref{777} aufspalten können als 
    \begin{align}
        \mathbb{L}_{[a,b]}(\gamma_\alpha) =& \: \mathbb{L}_{[a,c]} (\gamma_\alpha |_{[a,c]}) + \mathbb{L}_{[c,d]} (\gamma_\alpha |_{[c,d]}) + \mathbb{L}_{[d,b]}(\gamma_\alpha |_{[d,b]}) \nonumber \\ =& \: \mathbb{L}_{[a,c]} (\gamma_0 |_{[a,c]}) + \mathbb{L}_{[c,d]} (\hat{\gamma}_\alpha |_{[c,d]}) + \mathbb{L}_{[d,b]}(\gamma_0 |_{[d,b]}) \nonumber \\ =& \: \mathbb{L}_{[a,c]} (\gamma_0 |_{[a,c]}) + F_{\hat{\gamma}}(\alpha) + \mathbb{L}_{[d,b]}(\gamma_0 |_{[d,b]}),
    \end{align}
    wobei lediglich der zweite Term auf der rechten Seite noch abhängig von $\alpha$ ist. Die Vereinfachung im Beweis von Satz $7.9.$ wurde nur gemacht, um eine simple Integraldarstellung des Längenfunktionals zu gewährleisten, um anhand dieser leicht und explizit zeigen zu können, warum wir überhaupt das Längenfunktional auf obige Weise aufspalten können.
\end{remark}

Mittels des Resultates aus Satz $7.9.$ sind wir nun auf dem besten Wege die sogenannte Geodätengleichung für (schwache) Geodätische herzuleiten, mit deren Hilfe man einerseits überprüfen kann, ob eine gegebene Kurve tatsächlich einen möglichen Kandidaten einer schwachen Geodätischen auf einer riemannschen Mannigfaltigkeit  $(\mathcal{M}, \tau, \mathcal{A}, g)$ darstellt, und mit deren Hilfe man andererseits, wie wir in einem Spezialfall auch noch sehen werden, die Geodätischen einer riemannschen Mannigfaltigkeit praktisch ausrechnen kann.

Unser nächstes Ziel wird es nun sein, die sogenannten Euler-Lagrange-Gleichungen für Funktionale $\mathcal{S}_{\mathcal{L}, [a,b]} : \mathcal{C}^\infty ([a,b], \mathcal{M}) \longrightarrow \mathcal{M}$ herzuleiten, die von der selben Bauart wie das Längenfunktional $\mathbb{L}_{[a,b]}$ sein sollen. Was damit konkret gemeint ist klären die nächsten Definitionen:

\begin{definition}
    Sei $(\mathcal{M}, \tau, \mathcal{A})$ eine glatte Mannigfaltigkeit und $T \mathcal{M}$ der zur Mannigfaltigkeit zugehörige Tangentialbündel, aufgefasst als glatte Mannigfaltigkeit. Wir nennen eine beliebige Abbildung der Form
    \begin{align}
        \mathcal{L} : T \mathcal{M} \longrightarrow \mathbb{R}
    \end{align}
    Lagrange Funktion.
\end{definition}

\begin{example}
    Wir betrachten die riemannsche Mannigfaltigkeit $(\mathcal{M}, \tau, \mathcal{A}, g)$ und die Funktion $\mathcal{L}$, definiert durch
    \begin{align}
         \mathcal{L} : T \mathcal{M} \longrightarrow& \: \: \mathbb{R} \nonumber\\
            (p, X) \longmapsto& \: \: \sqrt{g_p (X,X)}, \label{Lagrangefunktion für den Längenbegriff}
    \end{align}
    wobei für $(p, X)$ gilt, dass $X \in T_p \mathcal{M}$ ist.
\end{example}

Als Nächstes definieren wir die sogenannte Hebung einer glatte Kurve $\gamma_0 : [a,b] \longrightarrow \mathcal{M}$ auf das Tangentialbündel $T \mathcal{M}$.

\begin{definition}
    Sei $(\mathcal{M}, \tau, \mathcal{A})$ eine glatte Mannigfaltigkeit. Sei weiter $\gamma_0 \in \mathcal{C}^\infty ([a,b], \mathcal{M})$ mit $a,b \in \mathbb{R}$ und $a < b$ eine  glatte Kurve auf $\mathcal{M}$. Wir definieren die sogenannte Hebung von $\gamma_0$ auf das Tangentialbündel $T \mathcal{M}$ durch
    \begin{align}
        X_{\gamma_0} : [a,b] \longrightarrow& \: \: T \mathcal{M} \nonumber\\
            \lambda \longmapsto& \: \: (\gamma_0(\lambda),X_{\gamma_0, \gamma_0(\lambda)}) \equiv X_{\gamma_0, \gamma_0(\lambda)} =: X_\gamma (\lambda).
    \end{align}
\end{definition}

\begin{remark}
    Bezogen auf die Definition $7.9.$ mag man sich mit Blick auf die Definition der Tangentialvektoren fragen, wie genau die Werte $X_\gamma (a)$ und $X_\gamma (b)$ aussehen. Die Idee ist dabei, dass wir für die hier betrachteten Kurven immer annehmen, dass es Tangentialvektoren $X_{\delta,\gamma(a)} \in T_{\gamma(a)} \mathcal{M}$ und $X_{\eta, \gamma(b)} \in T_{\gamma(b)} \mathcal{M}$ gibt, sodass
    \begin{align}
        X_{\delta, \gamma(a)} f = \lim_{h \rightarrow 0+} \frac{(f \circ \gamma)(a + h) - (f \circ \gamma)(a)}{h} \label{rechts}
    \end{align}
    und 
    \begin{align}
        X_{\eta, \gamma(b)} f = \lim_{h \rightarrow 0-} \frac{(f \circ \gamma)(b + h) - (f \circ \gamma)(b)}{h} \label{links}
    \end{align}
    für alle $f \in \mathcal{C}^\infty (\mathcal{M})$ gilt. Dabei meint $\lim_{h \rightarrow 0+}$, dass im obigen Grenzübergang \eqref{rechts} nur Folgen $(h_n)_{n \in \mathbb{N}}$ mit $h_n \xrightarrow{n \rightarrow \infty} 0$ betrachten werden, für die gilt, dass $h_n > 0$ für alle $n \in \mathbb{N}$ ist. Analog bedeutet $\lim_{h \rightarrow 0-}$, dass beim obigen Grenzübergang \eqref{links} nur Folgen $(h_n)_{n \in \mathbb{N}}$ mit $h_n \xrightarrow{n \rightarrow \infty} 0$ betrachtet werden, für die gilt, dass $h_n < 0$ für alle $n \in \mathbb{N}$ ist.
\end{remark}

Mittels der eben definierten Begriffe wie der Lagrangefunktion und der Hebung einer glatten Kurve $\gamma_0 : [a,b] \longrightarrow \mathcal{M}$ sind wir nun in der Lage das sogenannte Wirkungsfunktional $\mathcal{S}_{\mathcal{L}, [a,b]} : \mathcal{C}^\infty ([a,b]. \mathcal{M}) \longrightarrow \mathbb{R}$ zu definieren.

\begin{definition}
    Sei $(\mathcal{M}, \tau, \mathcal{A})$ eine glatte Mannigfaltigkeit und $\gamma_0 : [a,b] \longrightarrow \mathcal{M}$ eine glatte Kurve mit $a,b \in \mathbb{R}$ und $a < b$. Sei weiter $\mathcal{L}$ eine beliebige Lagrangefunktion. Dann ist das Wirkungsfunktion $\mathcal{S}_{\mathcal{L}, [a,b]} : \mathcal{C}^\infty ([a,b], \mathcal{M}) \longrightarrow \mathbb{R}$ bezüglich $\mathcal{L}$ erklärt durch
    \begin{align}
        \mathcal{S}_{\mathcal{L}, [a,b]} (\gamma_0) = \int_a^b (\mathcal{L} \circ X_{\gamma_0}) (\lambda) \,d \lambda.
    \end{align}
\end{definition}

Das eben in Abhängigkeit einer Lagrangefunktion definierte Wirkungsfunktional kann als strukturelle Formalisierung des Längenfunktionals verstanden werden. Wählen wir $\mathcal{L}$ wie in \eqref{Lagrangefunktion für den Längenbegriff}, so folgt, dass $\mathcal{S}_{\mathcal{L}, [a,b]} \equiv \mathbb{L}_{[a,b]}$ ist. 

Sei nun im Folgenden $(\mathcal{U}, \phi)$ eine Karte der $d$-dimensionalen glatten Mannigfaltigkeit $(\mathcal{M}, \tau, \mathcal{A})$ mit $\gamma_0^{-1}(\mathcal{U}) \neq \emptyset$ und $\phi = (x^1, ..., x^d)$. Sei weiter $(T \mathcal{U}, \xi_\phi)$ die zu $(\mathcal{U}, \phi)$ zugehörige Karte auf dem Tangentialbündel. Da $\gamma_0$ eine stetige Kurve ist, ist $\gamma_0^{-1}(\mathcal{U})$ eine offene Menge bezüglich der auf $[a,b]$ erklärten Teilraumtopologie. Damit existiert ein Intervall $[c,d] \subseteq [a,b]$, sodass $\gamma_0([c,d]) \subseteq \mathcal{U}$ und damit $X_{\gamma_0}([c,d]) \subseteq T \mathcal{U}$ gilt. Wir berechnen den Ausdruck $\mathcal{S}_{\mathcal{L}, [c,d]}(\gamma_0 |_{[c,d]})$ bezüglich der Karte $(T \mathcal{U}, \xi_\phi)$:

\begin{align}
    \mathcal{S}_{\mathcal{L}, [c,d]}(\gamma_0 |_[c,d]) =& \: \int_c^d (\mathcal{L} \circ X_{\gamma_0 |_{(c,d)}}) (\lambda) \,d \lambda \nonumber \\ =& \: \int_c^d (\mathcal{L} \circ X_{\gamma_0 }) (\lambda) \,d \lambda \nonumber \\ =& \: \int_c^d ((\mathcal{L} \circ \xi^{-1}_\phi) \circ (\xi_\phi \circ X_{\gamma_0})) (\lambda) \,d \lambda \nonumber \\ =& \: \int_c^d \mathcal{L}_\phi (\xi_\phi(X_{\gamma_0, \gamma_0(\lambda)})) \,d \lambda  \label{++++++++}
\end{align}
In der obigen Gleichung ist dabei $\mathcal{L}_\phi$ gegeben durch $\mathcal{L} \circ \xi^{-1}_\phi$. Mittels der Definition der Karte $(T \mathcal{U}, \xi_\phi)$ aus \eqref{Karte des Tangentialbündels} und der aus \eqref{Vektorkomponente} und \eqref{Basisentwicklung} bekannten Basisentwicklung des Tangentialvektors $X_{\gamma_0, \gamma_0(\lambda)}$ bezüglich der Basis $\big\{ \big( \frac{\partial}{\partial x^1} \big)_{\gamma_0(\lambda)}, ..., \big( \frac{\partial}{\partial x^d} \big)_{\gamma_0(\lambda)} \big\}$, können wir die rechte Seite von \eqref{++++++++} umschreiben als 
\begin{align}
    \int_c^d \mathcal{L}_\phi (((x^1(\gamma_0(\lambda)), ..., x^d(\gamma_0(\lambda))), ((x^1 \circ \gamma_0)'(\lambda) , ...,(x^d \circ \gamma_0)'(\lambda) ))) \,d \lambda
\end{align}
Mit der Setzung $\gamma^{i}_{0, \phi} := x^{i} \circ \gamma_0$ können wir das noch etwas kompakter als 
\begin{align}
    \mathcal{S}_{\mathcal{L}, [c,d]} (\gamma_0 |_{[c,d]}) = \int_c^{d} \mathcal{L}_\phi (((\gamma^{1}_{0, \phi}(\lambda), ..., \gamma^{d}_{0, \phi}(\lambda)), (\gamma^{1\:'}_{0, \phi}(\lambda) , ...,\gamma^{d\:'}_{0, \phi}(\lambda) ))) \,d \lambda \label{Kartendarstellung des Wirkungsfunktionals...}
\end{align}
schreiben. Diese Kartendarstellung der Wirkung $\mathcal{S}_{\mathcal{L}, [c,d]}$ wird sich noch als sehr nützlich in der Herleitung der Euler-Lagrange-Gleichung erweisen.

Um nun die Euler-Lagrange-Gleichungen auch tatsächlich herleiten zu können, benötigen wir noch das folgende wichtige Resultat aus der multivariaten Analysis und der Variationsrechnung, welches oft auch als \textit{Fundamentallemma der Variationsrechnung} bezeichnet wird.

\begin{proposition}
    Sei $d \in \mathbb{N}$ eine natürliche Zahl, $c,d \in \mathbb{R}$ zwei vorgegebene reelle Zahlen mit $c < d$ und $\mathcal{L}_i : [c,d] \longrightarrow \mathbb{R}$ für alle $i \in \{1,...,d\}$ eine stetige Funktion. Gilt der Ausdruck
    \begin{align}
        \int_c^d \bigg( \sum_{i=1}^d \mathcal{L}_i (\lambda) \cdot \eta_i (\lambda)  \bigg) \,d \lambda = 0
    \end{align}
    für alle $d$-Tupel der Form $(\eta_1, ..., \eta_d)$ mit $\eta_i \in \mathcal{C}^\infty_0 ([c,d], \mathbb{R})$ für alle $i \in \{1,...,d\}$, so folgt, dass 
    \begin{align}
        \mathcal{L}_i \equiv 0 \:\:\:\:\:\:\:\: \forall i \in \{1,...,d\} 
    \end{align}
    ist. Dabei bezeichnet die Menge $\mathcal{C}^\infty_0 ([c,d], \mathbb{R})$ die Menge aller Funktionen aus $\mathcal{C}^\infty ([a,b], \mathbb{R})$ mit der Eigenschaft, dass diese auf dem Rand von $[a,b]$, d.h. den Stellen $a$ und $b$, verschwinden.
\end{proposition}

\begin{proof}
    Siehe \cite{sauvigny2013analysis}.
\end{proof}

Mittels des Fundamentallemmas der Variationsrechnung sind wir nun in der Lage für die obigen Funktionale $\mathcal{S}$ die Euler-Lagrange-Gleichungen als notwendiges Kriterium für stationäre Kurve herzuleiten.

\begin{theorem}
    Sei $(\mathcal{M}, \tau, \mathcal{A})$ eine glatte Mannigfaltigkeit, $\mathcal{L}$ eine Lagrangefunktion und $\mathcal{S}_{\mathcal{L}, [a,b]} : \mathcal{C}^\infty ([a,b], \mathcal{M}) \longrightarrow \mathbb{R}$ das zugehörige Wirkungsfunktional über dem Intervall $[a,b]$ mit $a,b \in \mathbb{R}$ und $a<b$. Sei weiter $\gamma_0 \in \mathcal{PC}^\infty_{p,q} ([a,b], \mathcal{M})$ eine stationäre Kurve von $\mathcal{S}_{\mathcal{L}, [a,b]}$, d.h. für jede beliebige stetige Deformation $\gamma = (\gamma_\alpha)_{\alpha \in (- \epsilon, \epsilon)}$ von $\gamma_0$ mit $\epsilon > 0$ gilt, dass 
    \begin{align}
        \frac{d}{d \alpha} \mathcal{S}_{\mathcal{L}, [a,b]} (\gamma_\alpha) \bigg|_{\alpha = 0} = 0
    \end{align}
    ist. Es gelte weiter, dass $(\mathcal{U}, \phi = (x^1, ..., x^d)) \in \mathcal{A}$ eine Karte mit $\gamma_0^{-1}(\mathcal{U}) \neq \emptyset$ ist. Dann existiert wegen der Stetigkeit von $\gamma_0$ ein Teilintervall $[c,d] \subseteq [a,b]$, sodass $\gamma_0 ([c,d]) \subseteq \mathcal{U}$ und $\hat{\gamma}_0 := \gamma_0 |_{[c,d]} \in \mathcal{C}^\infty ([c,d], \mathcal{M})$ ist. Weiter sei die Lagrangefunktion $\mathcal{L}$ noch so gewählt, sodass für alle glatten Deformationen $\hat{\gamma} = (\hat{\gamma}_\alpha)_{\alpha \in (-\epsilon, \epsilon)}$ von $\hat{\gamma}_0$ die folgenden Aussagen gelten:
    \begin{itemize}
    
        \item[\textit{i)}] Sei $\mathcal{L}_\phi = \mathcal{L} \circ \xi_\phi^{-1}$ und $\hat{\gamma}^{i}_{\alpha, \phi} = x^{i} \circ \hat{\gamma}_\alpha$ für alle $i \in \{ 1, ..., d \}$. Dann ist die Funktion 
        \begin{align}
            \mathcal{F}_\mathcal{L} : (-\epsilon, \epsilon) \times [c,d] \longrightarrow& \: \: \mathbb{R} \nonumber\\
            (\alpha, \lambda) \longmapsto& \: \: \mathcal{L}_\phi (((\hat{\gamma}^1_{\alpha, \phi}(\lambda), ..., \hat{\gamma}^d_{\alpha, \phi}(\lambda)), (\hat{\gamma}^{1\:'}_{\alpha, \phi}(\lambda) , ...,\hat{\gamma}^{d\:'}_{\alpha , \phi}(\lambda) )))
        \end{align}
        stetig und stetig partiell differenzierbar, d.h. die partiellen Ableitungen der Funktion $\mathcal{F}_\mathcal{L}$ sind ebenfalls auf $(-\epsilon, \epsilon) \times [c,d]$ stetig. 

        \item[\textit{ii)}] Für die Funktion $\mathcal{F}_\mathcal{L}$ aus \textit{i)} existieren Funktionen $A, B : [c,d] \longrightarrow \mathbb{R}$, sodass für alle $(\alpha, \lambda) \in (-\epsilon, \epsilon) \times [c,d]$ und alle $m \in \{ 1, ..., 2d \}$ die Abschätzungen
        \begin{align}
            |\mathcal{L}_\phi (((\hat{\gamma}^1_{\alpha, \phi}(\lambda), ..., \hat{\gamma}^d_{\alpha, \phi}(\lambda)), (\hat{\gamma}^{1\:'}_{\alpha, \phi}(\lambda) , ...,\hat{\gamma}^{d\:'}_{\alpha , \phi}(\lambda) ))) |\leq A(\lambda) 
        \end{align}
        und 
        \begin{align}
            |\partial_m \mathcal{L}_\phi (((\hat{\gamma}^1_{\alpha, \phi}(\lambda), ..., \hat{\gamma}^d_{\alpha, \phi}(\lambda)), (\hat{\gamma}^{1\:'}_{\alpha, \phi}(\lambda) , ...,\hat{\gamma}^{d\:'}_{\alpha , \phi}(\lambda) ))) |\leq B(\lambda) 
        \end{align}
        gelten. Darüberhinaus gilt noch
        \begin{align}
            \int_c^d A(\lambda) \,d \lambda < \infty \:\:\:\:\: \textit{und} \:\:\:\:\: \int_c^d B(\lambda) \,d \lambda < \infty.
        \end{align}
        
    \end{itemize}
    Unter diesen Vorraussetzungen folgen nun für $n \in \{1, ..., d\}$ die sogenannten Euler-Lagrange-Gleichungen für die Kurve $\hat{\gamma}_0$ auf dem Intervall $[c,d]$:
    \begin{align}
        &\partial_n \mathcal{L}_\phi (((a^1, ... , a^d), (b^1, ..., b^d))) \bigg|_{(((\gamma^1_{0, \phi}(\lambda), ..., \gamma^d_{0, \phi}(\lambda)), (\gamma^{1\:'}_{0, \phi}(\lambda) , ...,\gamma^{d\:'}_{0, \phi}(\lambda) )))} \nonumber \\ =& \: \frac{d}{d \lambda} \bigg( \partial_{d+n} \mathcal{L}_\phi (((a^1, ..., a^d), (b^1, ..., b^d))) \bigg|_{(((\gamma^1_{0, \phi}(\lambda), ..., \gamma^d_{0, \phi}(\lambda)), (\gamma^{1\:'}_{0, \phi}(\lambda) , ...,\gamma^{d\:'}_{0, \phi}(\lambda) )))} \bigg) 
    \end{align}
\end{theorem}

\begin{proof}
    Sei $\gamma_0 : [a,b] \longrightarrow \mathcal{M}$ eine stationäre Kurve des Wirkungsfunktionals $\mathcal{S}_{\mathcal{L}, [a,b]}$ und $(\mathcal{U}, \phi) \in \mathcal{A}$ eine Karte mit $\gamma_0^{-1}(\mathcal{U}) \neq \emptyset$ und $\phi = (x^1, ..., x^d)$. Dann existiert wegen der Stetigkeit von $\gamma_0$ ein abgeschlossenes Intervall $[c,d]$ mit $\gamma_0([c,d]) \subseteq \mathcal{U}$. Wegen der strukturellen Ähnlichkeit des Wirkungsfunktionals und des Längenfunktionals und der Tatsache, dass $\gamma_0$ als stationäre Kurve von $\mathcal{S}_{\mathcal{L}, [a,b]}$ angenommen wurde, ergibt sich für das Wirkungsfunktional wie im Satz $7.3.$, dass $\hat{\gamma}_0$ eine stationäre Kurve bezüglich des Wirkungsfunktionals $\mathcal{S}_{\mathcal{L}, [c,d]}$ ist. Damit gilt also für eine beliebige stetige Deformation $\hat{\gamma} = (\hat{\gamma}_\alpha)_{\alpha \in (-\epsilon, \epsilon)}$ von $\hat{\gamma}_0$, dass
    \begin{align}
        \frac{d}{d \alpha} \mathcal{S}_{\mathcal{L}, [c,d]} (\hat{\gamma}_\alpha) \bigg|_{\alpha = 0} = 0
    \end{align}
    ist. Da das für alle stetigen Deformationen $\hat{\gamma}$ von $\gamma_0$ gilt, gilt dieser Ausdruck also insbesondere auch für alle glatten Deformationen. Sei also ab jetzt $\hat{\gamma}$ eine glatte Deformation von $\gamma_0$. Es gilt mittels der multidimensionalen Kettenregel aus Satz $5.3.$, der Gleichung \eqref{Kartendarstellung des Wirkungsfunktionals...}, der Setzung $\gamma^{i}_{\alpha, \phi} = x^{i} \circ \gamma_\alpha$ für alle $i \in \{1, ..., d\}$ und der Linearität des Integrals, dass
    \begin{align}
        0 =& \: \frac{d}{d \alpha} \mathcal{S}_{\mathcal{L}, [c,d]} (\hat{\gamma}_\alpha) \bigg|_{\alpha = 0} \nonumber \\ =& \: \frac{d}{d \alpha} \bigg( \int_c^d \mathcal{L}_\phi (((\hat{\gamma}^1_{\alpha, \phi}(\lambda), ..., \hat{\gamma}^d_{\alpha, \phi}(\lambda)), (\hat{\gamma}^{1\:'}_{\alpha, \phi}(\lambda) , ...,\hat{\gamma}^{d\:'}_{\alpha , \phi}(\lambda) ))) \,d \lambda \bigg) \bigg|_{\alpha = 0} \nonumber \\ =& \: \int_c^d \frac{d}{d \alpha} \bigg( \mathcal{L}_\phi (((\hat{\gamma}^1_{\alpha , \phi}(\lambda), ..., \hat{\gamma}^d_{\alpha , \phi}(\lambda)), (\hat{\gamma}^{1\:'}_{\alpha , \phi}(\lambda) , ...,\hat{\gamma}^{d\:'}_{\alpha , \phi}(\lambda) ))) \bigg) \bigg|_{\alpha = 0} \;d \lambda \nonumber \\ =& \: \int_c^d \sum_{n = 1}^d \bigg( \partial_n \mathcal{L}_\phi (((\gamma^1_{0, \phi}(\lambda), ..., \gamma^d_{0, \phi}(\lambda)), (\gamma^{1\:'}_{0, \phi}(\lambda) , ...,\gamma^{d\:'}_{0 , \phi}(\lambda) ))) \bigg) \cdot \frac{\hat{\gamma}^n_{0, \phi} (\lambda)}{d \alpha} \,d \lambda \nonumber \\ &+ \: \int_c^d \sum_{n = 1}^d \bigg( \partial_{d + n} \mathcal{L}_\phi (((\gamma^1_{0, \phi}(\lambda), ..., \gamma^d_{0, \phi}(\lambda)), (\gamma^{1\:'}_{0, \phi}(\lambda) , ...,\gamma^{d\:'}_{0 , \phi}(\lambda) ))) \bigg) \cdot \frac{\hat{\gamma}^{n\:'}_{0, \phi} (\lambda)}{d \alpha} \,d \lambda \label{Variation in einer Karte}
    \end{align}
    ist. Wir konnten dabei in der eben durchgeführten Rechnung aufgrund unserer Annahmen an die Funktion $\mathcal{L}$ die Ableitung $\frac{d}{d \alpha}$ in das Integral ziehen \cite{conrad2019differentiating}. Da nun $\hat{\gamma}$ eine glatte Deformation ist, und $x^{i}$ für alle $i \in \{1, ..., d\}$ glatt ist, folgt, dass $x^{i} \circ \hat{\gamma} : (-\epsilon, \epsilon) \times [a,b] \longrightarrow \mathbb{R}$ für alle $i \in \{1, ..., d\}$ ebenfalls glatt ist. Wir können demnach auf den Ausdruck $\frac{\hat{\gamma}^{n\:'}_{\alpha, \phi} (\lambda)}{d \alpha}$ den aus der Analysis wohlbekannten Satz von Schwartz \cite{spivak2018calculus} anwenden, der es uns erlaubt, die Reihenfolge der Ableitungen zu vertauschen. Formal bedeutet das, dass 
    \begin{align}
        \frac{\hat{\gamma}^{n\:'}_{\alpha, \phi} (\lambda)}{d \alpha} = \frac{d}{d \alpha} \hat{\gamma}^{n\:'}_{\alpha, \phi}(\lambda) = \frac{d}{d \alpha} \bigg( \frac{d \hat{\gamma}^n_{\alpha, \phi} (\lambda)}{d \lambda} \bigg) = \frac{d}{d \lambda} \bigg( \frac{d \hat{\gamma}^n_{\alpha, \phi} (\lambda)}{d \alpha} \bigg) \label{mmmmmmmmmmmmmmm}
    \end{align}
    für alle $n \in \{ 1, ..., d \}$ gilt. Beachte, dass wir in \eqref{mmmmmmmmmmmmmmm} ausgenutzt haben, dass die aus der Analysis bekannten gewöhnlichen Differentialoperatoren $\frac{d}{d \lambda}$ und $\frac{d}{d \alpha}$ für $\hat{\gamma}^n_{\alpha, \phi}$ definitionsgemäß mit den partiellen Differentialoperatoren $\frac{\partial}{\partial \lambda}$ und $\frac{\partial}{\partial \alpha}$ zusammen fallen. Anwenden von \eqref{mmmmmmmmmmmmmmm} auf den zweiten Term der rechten Seite von \eqref{Variation in einer Karte} liefert nun
    \begin{align}
        &\int_c^d \sum_{n = 1}^d \bigg( \partial_{d + n} \mathcal{L}_\phi (((\gamma^1_{0, \phi}(\lambda), ..., \gamma^d_{0, \phi}(\lambda)), (\gamma^{1\:'}_{0, \phi}(\lambda) , ...,\gamma^{d\:'}_{0 , \phi}(\lambda) ))) \bigg) \cdot \frac{\hat{\gamma}^{n\:'}_{0, \phi} (\lambda)}{d \alpha} \,d \lambda \nonumber \\ &= \: \int_c^d \sum_{n = 1}^d \bigg( \partial_{d + n} \mathcal{L}_\phi (((\gamma^1_{0, \phi}(\lambda), ..., \gamma^d_{0, \phi}(\lambda)), (\gamma^{1\:'}_{0, \phi}(\lambda) , ...,\gamma^{d\:'}_{0 , \phi}(\lambda) ))) \bigg) \cdot \frac{d}{d \lambda} \bigg( \frac{d \hat{\gamma}^n_{0, \phi} (\lambda)}{d \alpha} \bigg) \,d \lambda \nonumber \\ &= \sum_{n = 1}^d \int_c^d \bigg( \partial_{d + n} \mathcal{L}_\phi (((\gamma^1_{0, \phi}(\lambda), ..., \gamma^d_{0, \phi}(\lambda)), (\gamma^{1\:'}_{0, \phi}(\lambda) , ...,\gamma^{d\:'}_{0 , \phi}(\lambda) ))) \bigg) \cdot \frac{d}{d \lambda} \bigg( \frac{d \hat{\gamma}^n_{0, \phi} (\lambda)}{d \alpha} \bigg) \,d \lambda. \label{Anwendung Schwartz}
    \end{align}
    Für alle $n \in \{ 1, ..., d \}$ können wir nun aufgrund unserer Annahmen an die Funktion $\mathcal{L}_\phi$ den Ausdruck 
    \begin{align}
        \int_c^d \bigg( \partial_{d + n} \mathcal{L}_\phi (((\gamma^1_{0, \phi}(\lambda), ..., \gamma^d_{0, \phi}(\lambda)), (\gamma^{1\:'}_{0, \phi}(\lambda) , ...,\gamma^{d\:'}_{0 , \phi}(\lambda) ))) \bigg) \cdot \frac{d}{d \lambda} \bigg( \frac{d \hat{\gamma}^n_{0, \phi} (\lambda) }{d \alpha} \bigg) \,d \lambda 
    \end{align}
    partiell integrieren \cite{forster6analysis}. Es folgt
    \begin{align}
        &\int_c^d \bigg( \partial_{d + n} \mathcal{L}_\phi (((\gamma^1_{0, \phi}(\lambda), ..., \gamma^d_{0, \phi}(\lambda)), (\gamma^{1\:'}_{0, \phi}(\lambda) , ...,\gamma^{d\:'}_{0 , \phi}(\lambda) ))) \bigg) \cdot \frac{d}{d \lambda} \bigg( \frac{d \hat{\gamma}^n_{0, \phi} (\lambda) }{d \alpha} \bigg) \,d \lambda \nonumber \\ &= \: - \int_c^d \frac{d}{d \lambda} \bigg( \partial_{d + n} \mathcal{L}_\phi (((\gamma^1_{0, \phi}(\lambda), ..., \gamma^d_{0, \phi}(\lambda)), (\gamma^{1\:'}_{0, \phi}(\lambda) , ...,\gamma^{d\:'}_{0 , \phi}(\lambda) ))) \bigg) \cdot \frac{d \hat{\gamma}^n_{0, \phi} (\lambda) }{d \alpha}  \,d \lambda \nonumber \\ & + \: \bigg[ \bigg( \partial_{d + n} \mathcal{L}_\phi (((\gamma^1_{0, \phi}(\lambda), ..., \gamma^d_{0, \phi}(\lambda)), (\gamma^{1\:'}_{0, \phi}(\lambda) , ...,\gamma^{d\:'}_{0 , \phi}(\lambda) ))) \bigg) \cdot  \frac{d \hat{\gamma}^n_{0, \phi} (\lambda) }{d \alpha} \bigg]^d_c. \label{partielle Integration}
    \end{align}
    Der zweite Term auf der rechten Seite von \eqref{partielle Integration} verschwindet nun, da wegen
    \begin{align}
        \hat{\gamma}_{\alpha, \phi}^n (c) = x^{n}(\hat{\gamma}_{\alpha}(c)) = x^{n}(\hat{\gamma}_0 (c))
    \end{align}
    und 
    \begin{align}
        \hat{\gamma}_{\alpha, \phi}^n (d) = x^{n}(\hat{\gamma}_{\alpha}(d)) = x^{n}(\hat{\gamma}_0 (d))
    \end{align}
    für alle $n \in \{1, ..., d\}$ folgt, dass die Ableitung dieser Ausdrücke nach $\alpha$ an den Stellen $\lambda = c$ und $\lambda = d$ verschwinden. Setzen wir nun \eqref{partielle Integration} mit der eben gemachten Beobachtung in die Gleichung \eqref{Anwendung Schwartz} ein, und diese wiederum in \eqref{Variation in einer Karte}, so erhalten wir 
    
    \begin{align}
         0 =& \: \frac{d}{d \alpha} \mathcal{S}_{\mathcal{L}, [c,d]} (\hat{\gamma}_\alpha) \bigg|_{\alpha = 0} \nonumber \\ =& \: \int_c^d \sum_{n = 1}^d \bigg( \bigg( \partial_n \mathcal{L}_\phi (((\gamma^1_{0, \phi}(\lambda), ..., \gamma^d_{0, \phi}(\lambda)), (\gamma^{1\:'}_{0, \phi}(\lambda) , ...,\gamma^{d\:'}_{0 , \phi}(\lambda) ))) \bigg) \cdot \frac{\hat{\gamma}^n_{0, \phi} (\lambda)}{d \alpha} \nonumber \\ &- \: \frac{d}{d \lambda} \bigg( \partial_{d + n} \mathcal{L}_\phi (((\gamma^1_{0, \phi}(\lambda), ..., \gamma^d_{0, \phi}(\lambda)), (\gamma^{1\:'}_{0, \phi}(\lambda) , ...,\gamma^{d\:'}_{0 , \phi}(\lambda) ))) \bigg) \cdot \frac{\hat{\gamma}^n_{0, \phi} (\lambda)}{d \alpha} \bigg) \,d \lambda \nonumber \\ =& \: \int_c^d \sum_{n = 1}^d \bigg( \bigg( \partial_n \mathcal{L}_\phi (((\gamma^1_{0, \phi}(\lambda), ..., \gamma^d_{0, \phi}(\lambda)), (\gamma^{1\:'}_{0, \phi}(\lambda) , ...,\gamma^{d\:'}_{0 , \phi}(\lambda) ))) \bigg) \nonumber \\ &- \: \frac{d}{d \lambda} \bigg( \partial_{d + n} \mathcal{L}_\phi (((\gamma^1_{0, \phi}(\lambda), ..., \gamma^d_{0, \phi}(\lambda)), (\gamma^{1\:'}_{0, \phi}(\lambda) , ...,\gamma^{d\:'}_{0 , \phi}(\lambda) ))) \bigg) \bigg) \cdot \frac{\hat{\gamma}^n_{0, \phi} (\lambda)}{d \alpha}  \,d \lambda. \label{jjjjjjjjjjjjj}
    \end{align}
    Da \eqref{jjjjjjjjjjjjj} für beliebige glatte Deformationen $\hat{\gamma}$ von $\hat{\gamma}_0$ richtig ist, folgt mittels des Fundamentallemmas der Variationsrechnung, dass für alle $n \in \{1, ..., d\}$ auf dem Intervall $[c,d]$ 
    \begin{align}
         &\partial_n \mathcal{L}_\phi (((\gamma^1_{0, \phi}(\lambda), ..., \gamma^d_{0, \phi}(\lambda)), (\gamma^{1\:'}_{0, \phi}(\lambda) , ...,\gamma^{d\:'}_{0 , \phi}(\lambda) ))) \nonumber \\ &- \: \frac{d}{d \lambda} \bigg( \partial_{d + n} \mathcal{L}_\phi (((\gamma^1_{0, \phi}(\lambda), ..., \gamma^d_{0, \phi}(\lambda)), (\gamma^{1\:'}_{0, \phi}(\lambda) , ...,\gamma^{d\:'}_{0 , \phi}(\lambda) ))) \bigg) = 0 
    \end{align}
    gilt. Dies sind gerade die Euler-Lagrange-Gleichungen der Kurve $\hat{\gamma}_0$ auf dem Intervall $[c,d]$.
\end{proof}

Mittels der Euler-Lagrange-Gleichungen können wir nun das folgende wichtige Theorem formulieren, welches ein notwendiges Kriterium für spezielle schwache Geodätische einer riemannschen Mannigfaltigkeit $(\mathcal{M}, \tau, \mathcal{A}, g)$ darstellt,

\begin{theorem}
    Sei $(\mathcal{M}, \tau, \mathcal{A}, g)$ eine riemannsche Mannigfaltigkeit, $a,b \in \mathbb{R}$ zwei reelle Zahlen mit $a<b$ und $\gamma_0 \in \mathcal{PC}^\infty_{p,q} ([a,b], \mathcal{M})$ eine schwache Geodätische, die die Punkte $p$ und $q$ aus $\mathcal{M}$ miteinander verbindet. Wir nehmen darüberhinaus an, dass die schwache Geodätische auf Einheitsgeschwindigkeit parametrisiert ist, d.h. es soll gelten, dass
    \begin{align}
        g_{\gamma_0(\lambda)} (X_{\gamma_0, \gamma_0(\lambda)}, X_{\gamma_0, \gamma_0(\lambda)}) = 1 \:\:\:\:\:\:\:\: \forall \lambda \in [a,b].
    \end{align}
    Sei weiter $(\mathcal{U}, \phi) \in \mathcal{A}$ eine Karte von $\mathcal{M}$ mit $\phi = (x^1, ..., x^d)$ und $\mathcal{U} \cap \gamma_0([a,b]) \neq \emptyset$. Da $\gamma_0$ eine stetige Kurve auf $\mathcal{M}$ ist und folglicherweise $\gamma_0^{-1}(\mathcal{U})$ eine offene Menge bezüglich der auf $[a,b]$ erklärten Teilraumtopologie ist, existiert ein abgeschlossenes Intervall $[c,d] \subseteq [a,b]$, sodass $\gamma_0([c,d]) \subset \mathcal{U}$ und $\gamma |_{[c,d]} \in \mathcal{C}^\infty ([c,d], \mathcal{M})$ gilt. Dann genügt die Einschränkung $\gamma_0 |_{[c,d]}$, welche nach Satz $7.3.$ selbst eine schwache Geodätische ist, den sogenannten Geodätengleichungen
    \begin{align}
        0 = \gamma^{k\:''}_{0, \phi}(\lambda) + \sum_{i,j = 1}^d \Gamma^k_{ij, \phi} (\gamma_{0,\phi}(\lambda)) \cdot \gamma^{i\:'}_{0, \phi} (\lambda) \cdot \gamma^{j\:'}_{0, \phi}(\lambda) \:\:\:\:\:\: \forall k \in \{ 1, ..., d \}
    \end{align}
    mit $\lambda \in [c,d]$. Dabei ist für $i,j,k \in \{ 1, ..., d \}$ und $\lambda \in [c,d]$ $\gamma^k_{0, \phi} := x^k \circ \gamma$, $\gamma_{0,\phi}(\lambda) = (\gamma^1_{0, \phi}(\lambda), ..., \gamma^d_{0, \phi}(\lambda))$ und 
    \begin{align}
        \Gamma^k_{ij, \phi} (\gamma_{0,\phi}(\lambda)) = \frac{1}{2} \sum_{n=1}^d g^{kn}_\phi (\gamma_{0,\phi}(\lambda)) \cdot \bigg(  2 \cdot \partial_j g_{ni, \phi} (\gamma_{0, \phi}(\lambda)) - \partial_n g_{ij, \phi} (\gamma_{0, \phi}(\lambda)) \bigg),
    \end{align}
    wobei $g_{ij, \phi}(\gamma_{0, \phi}(\lambda)) := g_{\gamma_0(\lambda)} \big( \big( \frac{\partial}{\partial x^{i}} \big)_{\gamma_0(\lambda)} , \big( \frac{\partial}{\partial x^j} \big)_{\gamma_0(\lambda)} \big)$ ist und die $g^{ij}_{\phi}(\gamma_{0, \phi}(\lambda))$ die zur Matrix $G_{\gamma_{0}(\lambda)} = (g_{ij, \phi}(\gamma_{0, \phi}(\lambda)))_{i,j \in \{ 1, ..., d \}}$ inverse Matrix bilden (siehe Beweis).
\end{theorem}

\begin{proof}
    Sei also $\gamma_0 : [a,b] \longrightarrow \mathcal{M}$ eine schwache Geodätische, welche auf Einheitsgeschwindigkeit parametrisiert wurde, und sei $[c,d] \subseteq [a,b]$ so gewählt, sodass $\gamma_0([c,d])$ vollständig im Kartengebiet $\mathcal{U}$ der Karte $(\mathcal{U}, \phi)$ mit $\phi = (x^1, ..., x^d)$ enthalten ist. Um nun die Aussage des Theorems zu zeigen, bemerken wir zuerst, dass mittels \eqref{Basisentwicklung}, \eqref{Vektorkomponente} und \eqref{Tensorkomponenten bzgl einer Karte} folgt, dass für alle $\lambda \in [c,d]$
    \begin{align}
        g_{\gamma_0(\lambda)}& ( X_{\gamma_0, \gamma_0(\lambda)} , X_{\gamma_0, \gamma_0(\lambda)} ) \nonumber \\ =& \: g_{\gamma_0(\lambda)} \bigg( \sum_{i=1}^d (x^{i} \circ \gamma_0)'(\lambda) \cdot \bigg( \frac{\partial}{\partial x^{i}} \bigg)_{\gamma_0(\lambda)} , \sum_{j=1}^d (x^{j} \circ \gamma_0)'(\lambda) \cdot \bigg( \frac{\partial}{\partial x^{j}} \bigg)_{\gamma_0(\lambda)} \bigg) \nonumber \\ =& \: \sum_{i,j = 1}^d (x^{i} \circ \gamma_0)'(\lambda) \cdot (x^j \circ \gamma_0)'(\lambda) \cdot g_{\gamma_0(\lambda)} \bigg( \bigg( \frac{\partial}{\partial x^{i}} \bigg)_{\gamma_0(\lambda)} , \bigg( \frac{\partial}{\partial x^{j}} \bigg)_{\gamma_0(\lambda)} \bigg) \nonumber \\ =& \: \sum_{i,j = 1}^d (x^{i} \circ \gamma_0)'(\lambda) \cdot (x^j \circ \gamma_0)'(\lambda) \cdot g_{ij} (\gamma_0(\lambda)) \nonumber \\ =& \: \sum_{i,j = 1}^d (x^{i} \circ \gamma_0)'(\lambda) \cdot (x^j \circ \gamma_0)'(\lambda) \cdot g_{ij} (\phi^{-1}(\phi(\gamma_0(\lambda)))) \nonumber \\ =& \: \sum_{i,j = 1}^d (x^{i} \circ \gamma_0)'(\lambda) \cdot (x^j \circ \gamma_0)'(\lambda) \cdot g_{ij, \phi} (\phi(\gamma_0(\lambda))) \nonumber \\ =& \: \sum_{i,j = 1}^d (x^{i} \circ \gamma_0)'(\lambda) \cdot (x^j \circ \gamma_0)'(\lambda) \cdot g_{ij, \phi} (( x^1(\gamma_0(\lambda)) , ..., x^d(\gamma_0(\lambda)) ))
    \end{align}
    gilt. Dabei ist gemäß \eqref{Tensorkomponenten bzgl einer Karte} $g_{ij} (\gamma_0(\lambda))$ definiert als $g_{\gamma_0(\lambda)} \big( \big( \frac{\partial}{\partial x^{i}} \big)_{\gamma_0(\lambda)}, \big( \frac{\partial}{\partial x^{i}} \big)_{\gamma_0(\lambda)} \big)$ und $g_{ij, \phi} = g_{ij} \circ \phi^{-1}$. Weiter definieren wir die Funktion 
    \begin{align}
        \mathcal{R}_\phi : \phi(\mathcal{U}) \times \mathbb{R}^d \longrightarrow& \: \: \mathbb{R} \nonumber\\
            ((a^1, ..., a^d), (b^1, ..., b^d)) \longmapsto& \: \: \sum_{i, j = 1}^d b^{i} \cdot b^j \cdot g_{ij, \phi} ((a^1, ..., a^d)). \label{431}
    \end{align}
    Dann ergibt sich, dass für $\lambda \in [c,d]$ 
    \begin{align}
        g_{\gamma_0(\lambda)}& ( X_{\gamma_0, \gamma_0(\lambda)} , X_{\gamma_0, \gamma_0(\lambda)} ) \nonumber \\ =& \: \mathcal{R}_\phi ((( x^1(\gamma_0(\lambda)) , ..., x^d(\gamma_0(\lambda)) ), ( (x^1 \circ \gamma_0)'(\lambda) , ..., (x^d \circ \gamma_o)'(\lambda) )))
    \end{align}
    gilt. Mittels der Kurzbezeichnungen $\gamma^{i}_{0, \phi}(\lambda) := (x^{i} \circ \gamma_0)(\lambda))$ für alle $i \in \{ 1, ..., d\}$ ergibt sich damit
    \begin{align}
        g_{\gamma_0(\lambda)}& ( X_{\gamma_0, \gamma_0(\lambda)} , X_{\gamma_0, \gamma_0(\lambda)} ) \nonumber \\ =& \: \mathcal{R}_\phi (((\gamma^1_{0, \phi}(\lambda), ..., \gamma^d_{0, \phi}(\lambda)), (\gamma^{1\:'}_{0, \phi} (\lambda), ..., \gamma^{d\:'}_{0, \phi} (\lambda))))
    \end{align}
    für alle $\lambda \in [c,d]$. Damit ergibt sich für den Integranden des Längenfunktionals 
    \begin{align}
        \sqrt{g_{\gamma_0(\lambda)} ( X_{\gamma_0, \gamma_0(\lambda)} , X_{\gamma_0, \gamma_0(\lambda)} )} =& \: \sqrt{\mathcal{R}_\phi (((\gamma^1_{0, \phi}(\lambda), ..., \gamma^d_{0, \phi}(\lambda)), (\gamma^{1\:'}_{0, \phi} (\lambda), ..., \gamma^{d\:'}_{0, \phi} (\lambda))))} \nonumber \\ =:& \: \mathcal{L}_\phi (((\gamma^1_{0, \phi}(\lambda), ..., \gamma^d_{0, \phi}(\lambda)), (\gamma^{1\:'}_{0, \phi} (\lambda), ..., \gamma^{d\:'}_{0, \phi} (\lambda)))).
    \end{align}
    Wir wenden nun gemäß Theorem $7.1.$ die Euler-Lagrange-Gleichungen auf $\mathcal{L}_\phi$ an. Wir berechnen zuerst den Ausdruck $\partial_n \mathcal{L}_\phi (((a^1, ..., a^d), (b^1, ..., b^d)))$ für $n \in \{ 1, ..., d \}$:
    \begin{align}
        \partial_n \mathcal{L}_\phi (((a^1, ..., a^d),& (b^1, ..., b^d))) \nonumber \\ =& \: \frac{\sum_{i,j=1}^d b^{i} \cdot b^{j} \cdot \partial_n g_{ij, \phi} ((a^1, ..., a^d))}{2 \cdot \sqrt{ \sum_{i,j =1}^d b^{i} \cdot b^{j} \cdot g_{ij, \phi} ((a^1, ..., a^d)) }} \label{qwertzu}
    \end{align}
    Für $(a^1, ..., a^d) = (\gamma^1_{0, \phi}(\lambda), ..., \gamma^d_{0, \phi}(\lambda))$ und $(b^1, ..., b^d) = (\gamma^{1\:'}_{0, \phi}(\lambda), ..., \gamma^{d\:'}_{0, \phi}(\lambda))$ mit $\lambda \in [c,d]$, und der Ausnutzung der Vorraussetzung, dass die Geodätische $\gamma_0$ auf Einheitsgeschwindigkeit parametrisiert ist, liefert \eqref{qwertzu}
    \begin{align}
        \partial_n \mathcal{L}_\phi& (((\gamma^1_{0, \phi}(\lambda), ..., \gamma^d_{0, \phi}(\lambda)), (\gamma^{1\:'}_{0, \phi}(\lambda), ..., \gamma^{d\:'}_{0, \phi}(\lambda)))) \nonumber \\ =& \: \frac{1}{2} \cdot \sum_{i,j = 1}^{d} \gamma^{i\:'}_{0, \phi}(\lambda) \cdot \gamma^{j\:'}_{0, \phi}(\lambda) \cdot \partial_n g_{ij, \phi} ((\gamma^1_{0, \phi}(\lambda), ..., \gamma^d_{0, \phi}(\lambda))). \label{Euler-Lagrange Teil 1 für das Längenfunktional}
    \end{align}
    Als Nächstes berechnen wir den Ausdruck $\partial_{d+n} \mathcal{L}_\phi (((a^1, ..., a^d),(b^1, ..., b^d)))$ für $n \in \{ 1, ..., d \}$:
    \begin{align}
        \partial_{n+d} \mathcal{L}_\phi (((a^1, ..., a^d),& (b^1, ..., b^d))) \nonumber \\ =& \: \frac{\sum_{i,j=1}^d g_{ij, \phi}((a^1, ..., a^d)) \cdot \partial_{d+n} (b^{i} \cdot b^{j}) }{2 \cdot \sqrt{ \sum_{i,j =1}^d b^{i} \cdot b^{j} \cdot g_{ij, \phi} ((a^1, ..., a^d)) }} \label{asdfgh}
    \end{align}
    Mittels der Produktregel der Differentialrechnung \cite{heuser1992lehrbuch} folgt für den Ausdruck $\partial_{d+n} (b^{i} \cdot b^{j})$ 
    \begin{align}
        \partial_{d+n} (b^{i} \cdot b^{j}) =& \: \partial_{d+n}(b^{i}) \cdot b^j + b^{i} \cdot \partial_{d+n} (b^j) \nonumber \\ =& \: \delta^{i}_{n} \cdot b^j + b^{i} \cdot \delta^j_{n},
    \end{align}
    wobei der Ausdruck $\delta^k_{n}$ mit $k,n \in \{ 1, ..., d \}$ das Kronecker-Delta, definiert durch
    \begin{align}
        \delta^k_{n} = \left\{\begin{array}{ll} 1, & k = n \\
         0, & k \neq n \end{array}\right.,
    \end{align}
    bezeichnet. Für den Zähler der rechten Seite von \eqref{asdfgh} folgt damit
    \begin{align}
        \sum_{i,j=1}^d g_{ij, \phi}& ((a^1, ..., a^d)) \cdot \partial_{d+n} (b^{i} \cdot b^{j}) \nonumber \\ =& \: \sum_{i,j=1}^d g_{ij, \phi}((a^1, ..., a^d)) \cdot (\delta^{i}_{n} \cdot b^j + b^{i} \cdot \delta^j_{n}) \nonumber \\ =& \: \sum_{i,j=1}^d g_{ij, \phi}((a^1, ..., a^d)) \cdot \delta^{i}_{n} \cdot b^j + \sum_{i,j=1}^d g_{ij, \phi}((a^1, ..., a^d)) \cdot b^{i} \cdot \delta^j_n \nonumber \\ =& \: \sum_{j = 1}^d g_{nj, \phi} ((a^1, ..., a^d)) \cdot b^j + \sum_{i = 1}^d g_{in, \phi} ((a^1, ..., a^d)) \cdot b^{i}. \label{yxcvb}
    \end{align}
    Da $g$ eine riemannsche Metrik ist und folglich $g_p$ für alle $p \in \mathcal{M}$ symmetrisch ist, folgt, dass
    \begin{align}
        g_{in, \phi} ((a^1, ..., a^d)) = g_{ni, \phi} ((a^1, ..., a^d))  
    \end{align}
    für alle $i,n \in \{1,...,d\}$ und alle $(a^1, ..., a^d) \in \phi(\mathcal{U})$ ist. Mittels dieser Beobachtung und einer Umbenennung des Laufindizes in der zweiten Summe von \eqref{yxcvb} folgt
    \begin{align}
        \sum_{i,j=1}^d g_{ij, \phi} ((a^1, ..., a^d)) \cdot \partial_{d+n} (b^{i} \cdot b^{j}) = 2 \cdot \sum_{j = 1}^d g_{nj, \phi} ((a^1, ..., a^d)) \cdot b^j. \label{äölkjh}
    \end{align}
    Setzen wir \eqref{äölkjh} in \eqref{asdfgh} ein, so erhalten wir 
    \begin{align}
        \partial_{n+d} \mathcal{L}_\phi (((a^1, ..., a^d),& (b^1, ..., b^d))) \nonumber \\ =& \: \frac{\sum_{j = 1}^d g_{nj, \phi} ((a^1, ..., a^d)) \cdot b^j }{\sqrt{ \sum_{i,j =1}^d b^{i} \cdot b^{j} \cdot g_{ij, \phi} ((a^1, ..., a^d)) }} \label{mnbvc}.
    \end{align}
    Für $(a^1, ..., a^d) = (\gamma^1_{0, \phi}(\lambda), ..., \gamma^d_{0, \phi}(\lambda))$ und $(b^1, ..., b^d) = (\gamma^{1\:'}_{0, \phi}(\lambda), ..., \gamma^{d\:'}_{0, \phi}(\lambda))$ mit $\lambda \in [c,d]$, und der abermaligen Ausnutzung der Vorraussetzung, dass die Geodätische $\gamma_0$ auf Einheitsgeschwindigkeit parametrisiert ist, liefert \eqref{mnbvc}
    \begin{align}
        \partial_{n+d} \mathcal{L}_\phi (((\gamma^1_{0, \phi}(\lambda),& ..., \gamma^d_{0, \phi}(\lambda)), (\gamma^{1\:'}_{0, \phi}(\lambda), ..., \gamma^{d\:'}_{0, \phi}(\lambda)))) \nonumber \\ =& \: \sum_{j = 1}^d g_{nj, \phi} ((\gamma^1_{0, \phi}(\lambda), ..., \gamma^d_{0, \phi}(\lambda))) \cdot \gamma^{j\:'}_{0, \phi}(\lambda) \label{üpoiuztr}
    \end{align}
    Wir differenzieren nun den Ausdruck \eqref{üpoiuztr} bezüglich des Kurvenparameters $\lambda$:
    \begin{align}
        \frac{d}{d \lambda} \bigg(\partial_{n+d} \mathcal{L}_\phi& (((\gamma^1_{0, \phi}(\lambda), ..., \gamma^d_{0, \phi}(\lambda)), (\gamma^{1\:'}_{0, \phi}(\lambda), ..., \gamma^{d\:'}_{0, \phi}(\lambda)))) \bigg) \nonumber \\ =& \: \frac{d}{d \lambda} \bigg( \sum_{j = 1}^d g_{nj, \phi} ((\gamma^1_{0, \phi}(\lambda), ..., \gamma^d_{0, \phi}(\lambda))) \cdot \gamma^{j\:'}_{0, \phi}(\lambda) \bigg) \label{qwertüpoiuz}.
    \end{align}
    Mittels der Produktregel der Differentialrechnung, sowie der multidimensionalen Kettenregel aus Satz $5.3.$, folgt für die rechte Seite von \eqref{qwertüpoiuz} 
    \begin{align}
        \frac{d}{d \lambda} \bigg(\sum_{j = 1}^d g_{nj, \phi}& ((\gamma^1_{0, \phi}(\lambda), ..., \gamma^d_{0, \phi}(\lambda))) \cdot \gamma^{j\:'}_{0, \phi}(\lambda) \bigg) \nonumber \\ =& \: \sum_{j,k=1}^d \partial_k g_{nj, \phi} ((\gamma^1_{0, \phi}(\lambda), ..., \gamma^d_{0, \phi}(\lambda))) \cdot \gamma^{k\:'}_{0, \phi} (\lambda) \cdot \gamma^{j\:'}_{0, \phi}(\lambda) \nonumber \\ &+ \: \sum_{j = 1}^d g_{nj, \phi} ((\gamma^1_{0, \phi}(\lambda), ..., \gamma^d_{0, \phi}(\lambda))) \cdot \gamma^{j\:''}_{0, \phi}(\lambda). \label{Euler-Lagrange Teil 2 für das Längenfunktional}
    \end{align}
    Kombinieren wir die Ausdrücke \eqref{Euler-Lagrange Teil 1 für das Längenfunktional} und \eqref{Euler-Lagrange Teil 2 für das Längenfunktional} gemäß den Euler-Lagrange-Gleichungen, so erhalten wir die Gleichung
    \begin{align}
        0 =& \:  \sum_{j = 1}^d g_{nj, \phi} ((\gamma^1_{0, \phi}(\lambda), ..., \gamma^d_{0, \phi}(\lambda))) \cdot \gamma^{j\:''}_{0, \phi}(\lambda) \nonumber \\ &+ \: \sum_{j,k=1}^d \partial_k g_{nj, \phi} ((\gamma^1_{0, \phi}(\lambda), ..., \gamma^d_{0, \phi}(\lambda))) \cdot \gamma^{k\:'}_{0, \phi} (\lambda) \cdot \gamma^{j\:'}_{0, \phi}(\lambda) \nonumber \\ &- \: \frac{1}{2} \sum_{i,j = 1}^d \gamma^{i\:'}_{0, \phi}(\lambda) \cdot \gamma^{j\:'}_{0, \phi}(\lambda) \cdot \partial_n g_{ij, \phi} ((\gamma^1_{0, \phi}(\lambda), ..., \gamma^d_{0, \phi}(\lambda))).
    \end{align}
    Wir schreiben diese Gleichung mittels einer Umbenennung der Laufindizes in der zweiten Summe um zu
    \begin{align}
        0 =& \:  \sum_{j = 1}^d g_{nj, \phi} ((\gamma^1_{0, \phi}(\lambda), ..., \gamma^d_{0, \phi}(\lambda))) \cdot \gamma^{j\:''}_{0, \phi}(\lambda) \nonumber \\ &+ \: \sum_{i,j=1}^d \partial_j g_{ni, \phi} ((\gamma^1_{0, \phi}(\lambda), ..., \gamma^d_{0, \phi}(\lambda))) \cdot \gamma^{j\:'}_{0, \phi} (\lambda) \cdot \gamma^{i\:'}_{0, \phi}(\lambda) \nonumber \\ &- \: \frac{1}{2} \sum_{i,j = 1}^d \gamma^{i\:'}_{0, \phi}(\lambda) \cdot \gamma^{j\:'}_{0, \phi}(\lambda) \cdot \partial_n g_{ij, \phi} ((\gamma^1_{0, \phi}(\lambda), ..., \gamma^d_{0, \phi}(\lambda))).
    \end{align}
    Vereinfachen liefert
    \begin{align}
        0 =& \:  \sum_{j = 1}^d g_{nj, \phi} ((\gamma^1_{0, \phi}(\lambda), ..., \gamma^d_{0, \phi}(\lambda))) \cdot \gamma^{j\:''}_{0, \phi}(\lambda) \nonumber \\ &+ \: \frac{1}{2} \sum_{i,j=1}^d \bigg( 2 \cdot \partial_j g_{ni, \phi} ((\gamma^1_{0, \phi}(\lambda), ..., \gamma^d_{0, \phi}(\lambda))) \nonumber \\ &- \: \partial_n g_{ij, \phi} ((\gamma^1_{0, \phi}(\lambda), ..., \gamma^d_{0, \phi}(\lambda))) \bigg) \cdot \gamma^{i\:'}_{0, \phi} (\lambda) \cdot \gamma^{j\:'}_{0, \phi}(\lambda). \label{Zusammenfügung von Teil 1 und Teil 2}
    \end{align}
    Um diese Gleichung nun noch weiter zu vereinfachen und auf die obige Form aus dem Theorem $7.2.$ zu bringen, bemerken wir, dass es sich bei der Bilinearform $g_{\gamma_0(\lambda)}$ mit $\lambda \in [c,d]$ um eine positiv definite Bilinearform handelt, d.h. es gilt 
    \begin{align}
        g_{\gamma_0(\lambda)} (X, X) \geq 0 \:\:\:\:\:\:\:\: \forall X \in T_{\gamma_0(\lambda)} \mathcal{M} 
    \end{align}
    und 
    \begin{align}
        g_{\gamma_0(\lambda)} (X, X) = 0 \iff X = \mathbf{0},
    \end{align}
    wobei $\mathbf{0}$ der Nullvektor von $T_{\gamma(\lambda)} \mathcal{M}$ ist. Sei $X \in T_p \mathcal{M}$ beliebig, dann lässt sich $X$ in der von $\phi$ induzierten Basis entwickeln zu
    \begin{align}
        X = \sum_{i=1}^d X^{i} \cdot \bigg( \frac{\partial}{\partial x^{i}} \bigg)_{\gamma_0(\lambda)}.
    \end{align}
    Dann gilt 
    \begin{align}
        g_{\gamma_0(\lambda)}(X,X) =& \: \sum_{i,j = 1}^d X^{i} \cdot X^{j} \cdot g_{ij} (\gamma_0(\lambda)) \nonumber \\ =& \: (X^1, ..., X^d) \circ G_{\gamma_0(\lambda)} \circ (X^1, ..., X^d)^T. \label{452}
    \end{align}
    Dabei ist in der rechten Seite von Gleichung \eqref{452} die Verknüpfung $\circ$ die gewöhnliche Matrixmultiplikation und $\cdot^T$ ist die Transpositionsoperation für Matrizen, welche die Zeilen und die Spalten einer Matrix miteinander vertauscht. Weiter ist die $d \times d$-Matrix $G_{\gamma_0(\lambda)}$ gegeben durch
    \begin{align}
        (G_{\gamma_0(\lambda)})_{ij} = g_{ij} (\gamma_0(\lambda))
    \end{align}
    mit $i,j \in \{ 1, ..., d \}$. Wir nutzen nun die Matrix $G_{\gamma_0(\lambda)}$, um die Abbildung 
    \begin{align}
        f_{G_{\gamma_0(\lambda)}} : \mathbb{R}^d \longrightarrow& \: \: \mathbb{R}^d \nonumber\\
            (X^1, ..., X^d) \longmapsto& \: \: (X^1, ..., X^d) \circ G_{\gamma_0(\lambda)} 
    \end{align}
    zu definieren. Statten wir den $\mathbb{R}^d$ wie in Beispiel $4.2.$ wieder mit der komponentenweise Addition und Skalarmultiplikation aus, so ist wegen der Definition der Matrixmultiplikation die Abbildung $f_{G_{\gamma_0(\lambda)}}$ linear. Darüberhinaus ist $f_{G_{\gamma_0(\lambda)}}$ injektiv, denn angenommen das wäre nicht so, so gäbe es nach Lemma $4.1.$ ein $(Y^1, ..., Y^d) \in \mathbb{R}^d$ mit $(Y^1, ..., Y^d) \neq (0, ..., 0)$, sodass 
    \begin{align}
        f_{G_{\gamma_0(\lambda)}} ((Y^1, ..., Y^d)) =& \: (Y^1, ..., Y^d) \circ G_{\gamma_0(\lambda)} \nonumber \\ =& \: (0, ..., 0) 
    \end{align}
    ist. Daraus folgt aber gemäß \eqref{452}, dass
    \begin{align}
        0 = (Y^1, ..., Y^d) \circ G_{\gamma_o(\lambda)} \circ (Y^1, ..., Y^d)^T = g_{\gamma_0(\lambda)}(Y,Y)
    \end{align}
    mit $Y = \sum_{i = 1}^d Y^{i} \big( \frac{\partial}{\partial x^{i}} \big)_{\gamma_0(\lambda)}$. Da per Konstruktion $Y \neq \mathbf{0}$ ist, steht das im Widerspruch zur Vorraussetzung, dass $g_{\gamma_0(\lambda)}$ eine positiv definite Bilinearform ist. Folglich ist $f_{G_{\gamma_0(\lambda)}}$ injektiv und da $f_{G_{\gamma_0(\lambda)}}$ linear ist folgt, mit Lemma $4.2.$ sogar die Bijektivität von $f_{G_{\gamma_0(\lambda)}}$. 

    Damit wissen wir, dass die Abbildung $f_{G_{\gamma_0(\lambda)}}$ eine Umkehrabbildung $h_{G_{\gamma_0(\lambda)}}$ mit
    \begin{align}
        f_{G_{\gamma_0(\lambda)}} \circ h_{G_{\gamma_0(\lambda)}} = \textit{id}_{\mathbb{R}^d} = h_{G_{\gamma_0(\lambda)}} \circ f_{G_{\gamma_0(\lambda)}}
    \end{align}
    besitzt. Da $f_{G_{\gamma_0(\lambda)}}$ linear ist, folgt auch die Linearität der Umkehrabbildung $h_{G_{\gamma_0(\lambda)}}$, denn seien $v, w, \Tilde{v}, \Tilde{w} \in \mathbb{R}^d$ und $\mu, \nu \in \mathbb{R}$ mit $f_{G_{\gamma_0(\lambda)}}(\Tilde{v}) = v$ und $f_{G_{\gamma_0(\lambda)}}(\Tilde{w}) = w$ gegeben, so gilt
    \begin{align}
        h_{G_{\gamma_0(\lambda)}} (\mu \cdot v + \nu \cdot w) =& \: h_{G_{\gamma_0(\lambda)}} (\mu \cdot f_{G_{\gamma_0(\lambda)}}(\Tilde{v}) + \nu \cdot f_{G_{\gamma_0(\lambda)}}(\Tilde{w})) \nonumber \\ =& \: h_{G_{\gamma_0(\lambda)}} (f_{G_{\gamma_0(\lambda)}}(\mu \cdot \Tilde{v} + \nu \cdot \Tilde{w})) \nonumber \\ =& \: \mu \cdot \Tilde{v} + \nu \cdot \Tilde{w} \nonumber \\ =& \: \mu \cdot h_{G_{\gamma_0(\lambda)}} (v) + \nu \cdot h_{G_{\gamma_0(\lambda)}} (w).
    \end{align}
    Sei nun $v = (V^1, ..., V^d) \in \mathbb{R}^d$ beliebig gegeben und sei $e_i \in \mathbb{R}^d$ wieder definiert durch 
    \begin{align}
        e_i := (0, ..., 0, \underbrace{1}_{\textit{$i$-te Stelle}}, 0, ..., 0).
    \end{align}
    Dann gilt wegen der Linearität von $h_{G_{\gamma_0(\lambda)}}$ und der Tatsache, dass $\{e_1, ..., e_d\}$ eine Basis vom $\mathbb{R}^d$ ist, dass 
    \begin{align}
        h_{G_{\gamma_0(\lambda)}} (v) =& \: \sum_{i=1}^d V^{i} \cdot h_{G_{\gamma_0(\lambda)}} (e_i) \nonumber \\ =& \: \sum_{i=1}^d V^{i} \cdot \bigg( \sum_{j = 1}^d g^{ij}(\gamma_0(\lambda)) \cdot e_j \bigg) \nonumber \\ =& \: \sum_{j = 1}^d \bigg( \sum_{i = 1}^d g^{ij} (\gamma_0(\lambda)) \cdot V^{i} \bigg) \cdot e_j \nonumber \\ =& \: (V^1, ..., V^d) \circ G^{-1}_{\gamma_0(\lambda)}
    \end{align}
    ist. Dabei handelt es sich bei dem Ausdruck $g^{ij}_{\gamma_0(\lambda)} \in \mathbb{R}$ erst einmal nur um eine abstrakte Bezeichnung und die Matrix $G^{-1}_{\gamma_0(\lambda)}$ ist dabei definiert durch 
    \begin{align}
        (G^{-1}_{\gamma_0(\lambda)})_{ij} = g^{ij}_{\gamma_0(\lambda)} \:\:\:\:\:\:\:\: \forall i,j \in \{ 1, ..., d \}.
    \end{align}
    Es gilt damit 
    \begin{align}
        f_{G_{\gamma_0(\lambda)}} ( h_{G_{\gamma_0(\lambda)}} (v)) =& \: f_{G_{\gamma_0(\lambda)}}((V^1, ..., V^d) \circ G^{-1}_{\gamma_0(\lambda)}) \nonumber \\ =& \: \big((V^1, ..., V^d) \circ G^{-1}_{\gamma_0(\lambda)} \big) \circ G_{\gamma_0(\lambda)} \nonumber \\ =& \: (V^1, ..., V^d) \circ \big( G^{-1}_{\gamma_0(\lambda)} \circ G_{\gamma_0(\lambda)} \big) \nonumber \\ =& \: (V^1, ..., V^d)
    \end{align}
    und 
    \begin{align}
        h_{G_{\gamma_0(\lambda)}} ( f_{G_{\gamma_0(\lambda)}} (v)) =& \: h_{G_{\gamma_0(\lambda)}}((V^1, ..., V^d) \circ G_{\gamma_0(\lambda)}) \nonumber \\ =& \: \big((V^1, ..., V^d) \circ G_{\gamma_0(\lambda)} \big) \circ G^{-1}_{\gamma_0(\lambda)} \nonumber \\ =& \: (V^1, ..., V^d) \circ \big( G_{\gamma_0(\lambda)} \circ G^{-1}_{\gamma_0(\lambda)} \big) \nonumber \\ =& \: (V^1, ..., V^d),
    \end{align}
    woraus 
    \begin{align}
        G^{-1}_{\gamma_0(\lambda)} \circ G_{\gamma_0(\lambda)} = G_{\gamma_0(\lambda)} \circ G^{-1}_{\gamma_0(\lambda)} = \mathbf{1}_{\mathbb{R}^d}
    \end{align}
    folgt, wobei $\mathbf{1}_{\mathbb{R}^d}$ als $d \times d$-Matrix über das Kronecker-Delta durch
    \begin{align}
        (\mathbf{1}_{\mathbb{R}^d})_{ij} = \delta_i^j \:\:\:\:\:\:\:\: \forall i,j \in \{ 1, ..., d \}
    \end{align}
    definiert ist. D.h., dass die Matrix $G^{-1}_{\gamma_0(\lambda)}$ die Inverse Matrix der sogenannten darstellenden Matrix der Bilinearform $g_{\gamma_0(\lambda)}$ bezüglich der Karte $(\mathcal{U}, \phi)$ ist. In Komponenten folgt damit 
    \begin{align}
        (G^{-1}_{\gamma_0(\lambda)} \circ G_{\gamma_0(\lambda)})_{ki} = \sum_{n = 1}^d g^{kn}(\gamma_0(\lambda)) \cdot g_{nj} (\gamma_0(\lambda)) = \delta_k^{i} \:\:\:\:\:\: \forall k,i \in \{ 1, ..., d \}.
    \end{align}
    Um nun unseren Gleichung \eqref{Zusammenfügung von Teil 1 und Teil 2} zu vereinfachen, multiplizieren wir diese mit $g^{kn}(\gamma_0(\lambda)) = g^{kn} (\phi^{-1}(\phi(\gamma_0(\lambda)))) =: g^{kn}_{\phi} ((\gamma^1_{0, \phi}(\lambda)), ..., \gamma^d_{0, \phi}(\lambda))$ und summieren über $n \in \{1, ..., d\}$. Das liefert mit der Kurzschreibweise $\gamma_{0, \phi}(\lambda) := (\gamma^1_{0, \phi}(\lambda), ..., \gamma^d_{0, \phi}(\lambda))$
    \begin{align}
        0 =& \:  \sum_{j = 1}^d \bigg( \sum_{n = 1}^d g^{kn}_{\phi} (\gamma_{0, \phi}(\lambda)) \cdot g_{nj, \phi} (\gamma_{0, \phi}(\lambda)) \bigg) \cdot \gamma^{j\:''}_{0, \phi}(\lambda) \nonumber \\ &+ \: \frac{1}{2} \sum_{i,j,n=1}^d g^{kn}_\phi (\gamma_{0, \phi}(\lambda)) \cdot \bigg( 2 \cdot \partial_j g_{ni, \phi} (\gamma_{0, \phi}(\lambda)) \nonumber \\ &- \: \partial_n g_{ij, \phi} (\gamma_{0, \phi}(\lambda)) \bigg) \cdot \gamma^{i\:'}_{0, \phi} (\lambda) \cdot \gamma^{j\:'}_{0, \phi}(\lambda) \nonumber \\ =& \: \sum_{j = 1}^d \delta^k_j \cdot \gamma^{j\:''}_{0, \phi}(\lambda) \nonumber + \frac{1}{2} \sum_{i,j,n=1}^d g^{kn}_\phi (\gamma_{0, \phi}(\lambda)) \cdot \bigg( 2 \cdot \partial_j g_{ni, \phi} (\gamma_{0, \phi}(\lambda)) \nonumber \\ &- \: \partial_n g_{ij, \phi} (\gamma_{0, \phi}(\lambda)) \bigg) \cdot \gamma^{i\:'}_{0, \phi} (\lambda) \cdot \gamma^{j\:'}_{0, \phi}(\lambda) \nonumber \\ =& \: \gamma^{k\:''}_{0, \phi}(\lambda) + \frac{1}{2} \sum_{i,j,n=1}^d g^{kn}_\phi (\gamma_{0, \phi}(\lambda)) \cdot \bigg( 2 \cdot \partial_j g_{ni, \phi} (\gamma_{0, \phi}(\lambda)) \nonumber \\ &- \: \partial_n g_{ij, \phi} (\gamma_{0, \phi}(\lambda)) \bigg) \cdot \gamma^{i\:'}_{0, \phi} (\lambda) \cdot \gamma^{j\:'}_{0, \phi}(\lambda) \:\:\:\:\:\:\:\: \forall k \in \{ 1, ..., d \}.
    \end{align}
    Mit der Setzung 
    \begin{align}
        \Gamma^k_{ij, \phi} (\gamma_{0,\phi}(\lambda)) = \frac{1}{2} \sum_{n=1}^d g^{kn}_\phi (\gamma_{0,\phi}(\lambda)) \cdot \bigg(  2 \cdot \partial_j g_{ni, \phi} (\gamma_{0, \phi}(\lambda)) - \partial_n g_{ij, \phi} (\gamma_{0, \phi}(\lambda)) \bigg) \label{Christoffelsymbole}
    \end{align}
    für $i,j,k \in \{1, ..., d \}$ folgt somit die Geodätengleichung
    \begin{align}
        0 = \gamma^{k\:''}_{0, \phi}(\lambda) + \sum_{i,j = 1}^d \Gamma^k_{ij, \phi} (\gamma_{0,\phi}(\lambda)) \cdot \gamma^{i\:'}_{0, \phi} (\lambda) \cdot \gamma^{j\:'}_{0, \phi}(\lambda) \:\:\:\:\:\: \forall k \in \{ 1, ..., d \}.
    \end{align}
    Das zeigt die Behauptung.
\end{proof}

Man beachte an dieser Stelle, dass wenn wir beispielsweise auf einer riemannschen Mannigfaltigkeit $(\mathcal{M}, \tau, \mathcal{A}, g)$ auf der Suche nach einer kürzesten Verbindungskurven $\gamma : [a,b] \longrightarrow \mathcal{M}$ zwischen den zwei Punkten $p = \gamma(a)$ und $q = \gamma(b)$ sind, es uns für gewöhnlich herzlich egal ist, wie diese Verbindungskurve $\gamma$ genau parametrisiert ist, da die eigentliche geometrische Kurve auf $\mathcal{M}$ gegeben ist durch $\gamma([a,b]) \subseteq \mathcal{M}$. Daher hat die obige Einschränkung im Theorem $7.2.$ auch erst einmal keine geometrische Bedeutung.

\begin{remark}
    Die für die Indizes $i,j,k \in \{ 1, ..., d\}$ erklärten Funktionen 
    \begin{align}
        \Gamma^k_{ij, \phi} : \phi (\mathcal{U}) \longrightarrow \mathbb{R}
    \end{align}
    aus Theorem $7.2.$ werden auch Christoffelsymbole genannt. Bei diesen Symbolen handelt es sich um die karteninduzierten Koordinatenfunktionen des sogenannten Levi-Civita-Zusammenhangs, der ein Beispiel eines sogenannten linearen Zusammenhanges $\nabla$ auf $\mathcal{M}$ darstellt, welche man benötigt, um glatte Vektorfelder entlang anderer glatter Vektorfelder abzuleiten \cite{lee2006riemannian}. 
\end{remark}

Am Ende von diesem Abschnitt wollen wir nun noch den Spezialfall betrachten, in welchem wir auf einer riemannschen Mannigfaltigkeit $(\mathcal{M}, \tau, \mathcal{A}, g)$ eine Geodätische $\gamma_0 : [a,b] \longrightarrow \mathcal{M}$ gegeben haben, sodass eine Karte $(\mathcal{U}, \phi)$ mit $\gamma_0([a,b]) \subseteq \mathcal{U}$ existiert. In diesem Fall sagt man auch, dass die Geodätische vollständig in der Karte $(\mathcal{U}, \phi)$ enthalten ist.

\begin{proposition}
    Sei $(\mathcal{M}, \tau, \mathcal{A}, g)$ eine riemannsche Mannigfaltigkeit, $a,b \in \mathbb{R}$ zwei reelle Zahlen mit $a < b$ und $\gamma_0 \in \mathcal{C}^\infty_{p,q}([a,b], \mathcal{M})$ eine auf Einheitsgeschwindigkeit parametrisierte Kurve, die im Punkt $p \in \mathcal{M}$ beginnt und im Punkt $q \in \mathcal{M}$ endet. Weiter existiere eine Karte $(\mathcal{U}, \phi) \in \mathcal{A}$, sodass $\gamma_0([a,b]) \subseteq \mathcal{U}$ gilt. Dann sind folgende Aussagen äquivalent:
    \begin{itemize}
    
        \item[\textit{i)}] $\gamma_0$ ist eine Geodätische.

        \item[\textit{ii)}] Bezüglich der Karte $(\mathcal{U}, \phi)$ gelte für $\gamma_0$ die Geodätengleichung. 
        
    \end{itemize}
\end{proposition}

\begin{proof}
    Folgt direkt aus Theorem $7.2.$ und der Tatsache, dass die Kurve vollständig in einem Kartengebiet enthalten ist.
\end{proof}

Der obige Satz liefert also eine erste Möglichkeit, die Geodätischen einer riemannschen Mannigfaltigkeit tatsächlich auszurechnen.


\subsection{Das Energiefunktional}

In dem vorangegangenen Abschnitt haben wir ein notwendiges Kriterium für die schwachen Geodätischen einer riemannschen Mannigfaltigkeit $(\mathcal{M}, \tau, \mathcal{A}, g)$ kennengelernt. In diesem letzten Abschnitt von Kapitel $7$ wollen wir uns nun noch kurz mit dem sogenannten Energiefunktional auseinandersetzen. 

Der Grund für diese Auseinandersetzung liegt darin begründet, dass, wie wir noch sehen werden, das Längenfunktional und das Energiefunktional in gewisser Weise die selbe notwendige Bedingung bezüglich ihrer stationären Kurven besitzen. Da das Energiefunktional strukturell einfacher gebaut ist als das Längenfunktional, nutzt man Ersteres eher als Ausgangspunkt für verschiedene Anwendungen oder Verallgemeinerungen des (schwachen) Geodätenbegriffes. Im Kapitel $9$ werden wir noch sehen, was wir damit genau meinen.

Beginnen wir nun damit, das sogenannte Energiefunktional zu definieren:

\begin{definition}
    Sei $(\mathcal{M}, \tau, \mathcal{A}, g)$ eine riemannsche Mannigfaltigkeit und $a,b \in \mathbb{R}$ zwei reelle Zahlen mit $a < b$. Wir definieren das sogenannte Energiefunktional $\mathcal{E}_{[a,b]}$ durch
    \begin{align}
        \mathcal{E}_{[a,b]} : \mathcal{C}^\infty ([a,b], \mathcal{M}) \longrightarrow& \: \: \mathbb{R} \nonumber\\
            \gamma \longmapsto& \: \: \mathcal{E}_{[a,b]}(\gamma) := \frac{1}{2} \int_a^b g_{\gamma(\lambda)} ( X_{\gamma, \gamma (\lambda)} , X_{\gamma, \gamma (\lambda)} ) \,d \lambda. \label{Energiefunktional}
    \end{align}
\end{definition}

\begin{remark}
    Bezogen auf Definition $7.11.$ machen wir die folgenden Beobachtungen:
    \begin{itemize}
    
        \item[\textit{i)}] Betrachtet man den Ausdruck \eqref{Energiefunktional}, so gilt abermals, dass es mit Blick auf nicht injektive Kurven $\gamma$ formal korrekter gewesen wäre, wenn wir statt $X_{\gamma, \gamma(\lambda)}$ stattdessen $X_{\gamma, \lambda, \gamma(\lambda)}$ geschrieben hätten. Der übersichtlichkeitshalber nutzen wir aber die obige verkürzte Schreibweise, die Standard in der modernen Differentialgeometrie ist.
        
        \item[\textit{ii)}] Analog zur Bemerkung $7.1.$ ist das Energiefunktional auch für Kurven aus der Menge $\mathcal{PC}^\infty ([a,b], \mathcal{M})$ definierbar. Sei beispielsweise $\gamma \in \mathcal{PC}^\infty ([a,b], \mathcal{M})$, so existiert eine Menge $\{ \lambda_1, \lambda_2, ..., \lambda_n, \lambda_{n+1} \} \subseteq (a,b)$ mit $n \in \mathbb{N}$, $\lambda_1 = a$ und $\lambda_{n+1} = b$, sodass $\gamma |_{(\lambda_i, \lambda_{i+1})} \in \mathcal{C}^\infty ([a,b], \mathcal{M})$ für alle $i \in \{1,...,n\}$ gilt. Dann wäre das Energiefunktional dieser Kurve erklärt als 
        \begin{align}
            \mathcal{E}_{[a,b]} (\gamma) := \frac{1}{2} \sum_{i=1}^n \int_{\lambda_i}^{\lambda_{i+1}} g_{\gamma(\lambda)} ( X_{\gamma, \gamma (\lambda)} , X_{\gamma, \gamma (\lambda)} ) \,d \lambda.
        \end{align}
        Wegen der Linearität des Integrals ist dieser Ausdruck wohldefiniert.
        
    \end{itemize} 
\end{remark}

Beschäftigen wir uns kurz damit, woher das Funktional $\mathcal{E}_{[a,b]}$ seinen Namen hat. Dazu betrachten wir ein physikalisches System, welches beschrieben wird über die klassische Mechanik. 

In diesem physikalischen System bewege sich ein Teilchen entlang der Kurve $\gamma : (-\infty, \infty) \longrightarrow \mathbb{R}^3$ mit $\gamma \in \mathcal{C}^\infty ((-\infty,\infty), \mathbb{R}^3)$, sodass $\gamma(\lambda) \in \mathbb{R}^3$ den Ort des Teilchens im physikalischen Raum zum Zeitpunkt $\lambda \in (-\infty, \infty)$ beschreibt. Der Geschwindigkeitsvektor des Teilchens zum Zeitpunkt $\lambda$ ist dann gegeben durch $X_{\gamma, \gamma(\lambda)}$. Fassen wir den $\mathbb{R}^3$ wie in Bemerkung $6.1.$ als glatte Mannigfaltigkeit auf, die wir mit dem $(0,2)$-Tensorfeld $g$ aus \eqref{spezielle riemannsche Metrik} ausstatten, so ergibt sich die Geschwindigkeit des Teilchens zum Zeitpunkt $\lambda$ zu
\begin{align}
    v(\lambda) = \sqrt{g_{\gamma(\lambda)} (X_{\gamma, \gamma(\lambda)}, X_{\gamma, \gamma(\lambda)})}.
\end{align}
Besitzt das Teilchen die Masse $m$, so ergibt sich gemäß der klassischen Mechanik \cite{goldstein2012klassische} die kinetische Energie des Teilchens zum Zeitpunkt $\lambda$ zu
\begin{align}
    E_{\textit{kin}}(\lambda) =& \: \frac{1}{2} m \cdot v(\lambda)^2 \nonumber \\ =& \: \frac{1}{2} m \cdot g_{\gamma(\lambda)} (X_{\gamma, \gamma(\lambda)}, X_{\gamma, \gamma(\lambda)})
\end{align}
Nehmen wir darüber hinaus an, dass die Masse des Teilchens normiert sei, d.h. es gelte $m = 1$, so ergibt sich die kinetische Energie des Teilchens zum Zeitpunkt $\lambda$ zu
\begin{align}
    E_{\textit{kin}}(\lambda) = \frac{1}{2} \cdot g_{\gamma(\lambda)} (X_{\gamma, \gamma(\lambda)}, X_{\gamma, \gamma(\lambda)}),
\end{align}
was gerade dem Integranten auf der rechten Seite von \eqref{Energiefunktional} entspricht. D.h. der Ausdruck $\mathcal{E}_{[a,b]}(\gamma |_{[a,b]})$ entspricht damit der Integration der kinetischen Energie eines Teilchens normierter Masse über dem Zeitabschnitt $[a,b]$, dessen Bewegung durch die Kurve $\gamma$ beschrieben wird. Wegen diese Beziehung zur (kinetischen) Energie nennt man das Funktional $\mathcal{E}_{[a,b]}$ auch Energiefunktional.

Analog zum Abschnitt $7.3$ können wir nun die stationären Kurven des Energiefunktionals betrachten und wie in Theorem $7.2.$ ein notwendiges Kriterium für diese ableiten. 

\begin{theorem}
    Sei $(\mathcal{M}, \tau, \mathcal{A}, g)$ eine riemannsche Mannigfaltigkeit, $a,b \in \mathbb{R}$ zwei reelle Zahlen mit $a<b$ und $\gamma_0 \in \mathcal{PC}^\infty_{p,q} ([a,b], \mathcal{M})$ eine stationäre Kurve des Energiefunktionals $\mathcal{E}_{[a,b]} |_{\mathcal{PC}^\infty_{p,q} ([a,b], \mathcal{M})}$, die die Punkte $p$ und $q$ miteinander verbindet. Sei weiter $(\mathcal{U}, \phi) \in \mathcal{A}$ eine Karte von $\mathcal{M}$ mit $\phi = (x^1, ..., x^d)$ und $\mathcal{U} \cap \gamma_0([a,b]) \neq \emptyset$. Sei weiter $[c,d] \subseteq [a,b]$ ein beliebiges Teilintervall von $[a,b]$ mit $\gamma_0([c,d]) \subseteq \mathcal{U}$ und $\gamma |_{[c,d]} \in \mathcal{C}^\infty ([c,d], \mathcal{M})$. Dann genügt die Einschränkung $\gamma_0 |_{[c,d]}$, welche nach Satz $7.3.$ selbst wieder eine schwache Geodätische ist, den sogenannten Geodätengleichungen
    \begin{align}
        0 = \gamma^{k\:''}_{0, \phi}(\lambda) + \sum_{i,j = 1}^d \Gamma^k_{ij, \phi} (\gamma_{0,\phi}(\lambda)) \cdot \gamma^{i\:'}_{0, \phi} (\lambda) \cdot \gamma^{j\:'}_{0, \phi}(\lambda) \:\:\:\:\:\: \forall k \in \{ 1, ..., d \} \label{477}
    \end{align}
    mit $\lambda \in [c,d]$. Die in \eqref{477} auftauchenden Größen sind dabei wie in Theorem $7.2.$ erklärt.
\end{theorem}

\begin{proof}
    Der Beweis folgt direkt aus dem Beweis vom Theorem $7.2.$, indem man realisiert, dass für dieses Theorem die Funktion $\mathcal{L}_\phi$, auf die wir die Euler-Lagrange-Gleichungen anwenden wollen, gerade gegeben ist durch die in \eqref{431} erklärte Funktion $\mathcal{R}_{\phi}$. Es gilt für $m \in \{ 1, ..., d \}$ und $\lambda \in [c,d]$
    \begin{align}
        \partial_n \mathcal{L}_\phi (((\gamma^1_{0, \phi}(\lambda)&, ..., \gamma^d_{0, \phi}(\lambda)), (\gamma^{1\:'}_{0, \phi}(\lambda), ..., \gamma^{d\:'}_{0, \phi}(\lambda)))) \nonumber \\ =& \: \sum_{i,j = 1}^d \gamma^{i\:'}_{0, \phi}(\lambda) \cdot \gamma^{j\:'}_{0, \phi}(\lambda) \cdot \partial_n g_{ij, \phi} ((\gamma^1_{0, \phi}(\lambda), ..., \gamma^d_{0, \phi}(\lambda)))
    \end{align}
    und 
    \begin{align}
        \partial_{n+d} \mathcal{L}_\phi (((\gamma^1_{0, \phi}(\lambda),& ..., \gamma^d_{0, \phi}(\lambda)), (\gamma^{1\:'}_{0, \phi}(\lambda), ..., \gamma^{d\:'}_{0, \phi}(\lambda)))) \nonumber \\ =& \: 2 \cdot \sum_{j = 1}^d g_{nj, \phi} ((\gamma^1_{0, \phi}(\lambda), ..., \gamma^d_{0, \phi}(\lambda))) \cdot \gamma^{j\:'}_{0, \phi}(\lambda).
    \end{align}
    Wie im Beweis aus Theorem $7.3.$ ergibt sich daraus die Gültigkeit der Geodätengleichungen für $\gamma_0 |_{[c,d]}$.
\end{proof}


\newpage


\section{Zusammenfassung}

Fassen wir nun zusammen, was wir im Hauptteil dieser Arbeit alles diskutiert haben.

Ausgehend vom Konzept einer Kurve, welches wir mathematisch präzise fassen wollten, haben wir uns in den ersten beiden Kapiteln des Hauptteils, den Kapiteln $2$ und $3$, mit der Etablierung eines Nähebegriffes auf einer abstrakten Menge $X$ beschäftigt. 

Dabei haben wir gesehen, das wir einen Nähebegriff auf einer abstrakten Menge $X$ einführen können, wenn wir diese mit einer abstrakten Struktur namens Topologie versehen, bei der es sich grob gesprochen um eine spezielle Teilmenge der Potenzmenge handelt.

In Kapitel $2$ haben wir uns dabei davon überzeugt, dass mittels des abstrakten Topolgiebegriffes tatsächlich ein Konzept von Nähe auf der Menge $X$ etabliert wird. So waren wir beispielsweise mittels einer Topolgie $\tau \subseteq \mathcal{P}(X)$ in der Lage, auf $X$ einen Konvergenzbegriff für Folgen $(x_n)_{n \in \mathbb{N}}$ einzuführen, in welchem anschaulich ein Nähekonzept codiert ist.

Anschließend haben wir uns im Kapitel $3$ unter anderem mit der Rechtfertigung der genauen Definition einer Topologie auseinandergesetzt. Wir haben uns dazu mit metrischen Räume $(X, d)$ beschäftigt, auf denen wegen der Metrik $d$ ein intuitiver Nähebegriff existiert. In dem wir geklärt haben, in welchem Zusammenhang die metrischen Räume zu den topologischen Räumen stehen, konnten wir begründen, warum die Definition einer Topologie so aussieht, wie sie aussieht, und warum die Definition einer Topologie in gewisser Weise natürlich ist. 

Im Kapitel $4$ und $5$ haben wir uns dann jeweils mit wichtigen Begrifflichkeiten aus dem Bereich der (linearen) Algebra und der (multivariaten) Analysis bzw. der Funktionalanalysis beschäftigt. 

Eines der wichtigsten Ergebnisse im Kapitel $4$ war dabei die Definition der Prähilberträume $(V, \langle \cdot , \cdot \rangle)$, und wie diese mit der Geometrie im Zusammenhang stehen. So haben wir beispielsweise gesehen, dass man in einem Prähilbertraum nicht nur über einen sinnvollen Längenbegriff verfügt, sondern auch einen sinnvollen, wenn auch abstrakten Winkelbegriff besitzt.

So waren wir dann beispielsweise mittels der Prähilberträume in der Lage, ein Modell der klassischen euklidischen Geometrie bereitzustellen, was abermals eine Rechtfertigung dafür lieferte, dass die im Allgemeinen abstrakt gegebenen Längen- und Winkelbegriffe im Prähilbertraum tatsächlich geometrisch sinnvolle sind.

Im Kapitel $6$ haben wir uns dann mit dem Begriff einer endlichdimensionalen Mannigfaltigkeit $(\mathcal{M}, \tau)$ auseinandergesetzt, die für uns von besonderem Interesse sind, da es sich bei diesen Objekten um diejenigen mathematischen Entitäten handelt,auf denen wir den Begriff der (schwachen) Geodätischen definieren wollten.

Mittels der in Kapitel $4$ erarbeiteten Begrifflichkeiten konnten wir eine $d$-dimensionale Mannigfaltigkeit als einen speziellen topologischen Raum einführen, der lokal so \textit{aussieht} wie der mit der Standardgeometrie versehene euklidische Raum $\mathbb{R}^d$. 

Weiter konnten wir wegen dieser lokalen Ähnlichkeit die in Kapitel $5$ erarbeiteten Differenzierbarkeitsbegriffe für Funktionen zwischen normierten Räumen auf die endlichdimensionalen Mannigfalltigkeiten übertragen, was es uns erlaubte, die geometrischen Begriffe des Tangentialvektors und der Tangentialräume einzuführen.

Mittels dieser Begrifflichkeiten waren wir dann wiederum in der Lage, die Begriffe des Tangentialbündels und der glatten Vektorfelder auf $\mathcal{M}$ einzuführen. Ausgehend vom Begriff des glatten Vektorfeldes konnten wir dann die Definition der glatten Tensorfelder motivieren.

Im Kapitel $7$ haben wir uns dann schließlich mit sogenannten riemannschen Mannigfaltigkeiten beschäftigt, bei denen es sich um glatte endlichdimensionale Mannigfaltigfaltigkeiten handelte, die versehen waren mit einem speziellen glatten $(0,2)$-Tensorfeld $g$, welches wir als riemannsche Metrik bezeichneten.

Die riemannsche Metrik $g$ war dabei so definiert, dass diese jeden Tangentialraum der Mannigfaltigkeit zu einem Prähilbertraum machte. Mittels der in Kapitel $4$ erarbeiteten geometrischen Interpration des Prähilbertraumbegriffes konnten wir dann verstehen, wie die riemannsche Metrik $g$ eine Geometrie auf der Mannigfaltigkeit induziert.

Darüberhinaus waren wir nun mit der riemannschen Metrik $g$ dazu in der Lage, einen sinnvollen Längenbegriff für (stückweise) glatte Kurven $\gamma : [a,b] \longrightarrow \mathcal{M}$ mit $a,b \in \mathbb{R}$ und $a < b$ über das Längenfunktional $\mathbb{L}_{[a,b]}$ zu erklären. 

Ausgehend von diesem Längenfunktional konnten wir dann endlich den Begriff einer (schwachen) Geodätischen als die stationären Kurven des Längenfunktionals einführen, bei denen es sich geometrisch unter anderem um die kürzesten Verbindungskurven zweier gegebener Punkte handeln kann. 

Abschließend haben wir uns dann noch damit beschäftigt, wie man die (schwachen) Geodätischen einer riemannschen Mannigfaltigkeit praktisch ausrechnen kann. Dabei waren wir auf die sogenannte Geodätengleichung gestoßen, bei der er sich um ein System gewöhnlicher Differentialgleichungen handelt.


\newpage

\section{Ausblick: Geodätische auf unendlichdimensionalen Mannigfaltigkeiten und das Bildmatching-\\Problem}

Nachdem wir nun im Hauptteil dieser Arbeit den Begriff einer (schwachen) Geodätischen auf einer endlichdimensionalen Mannigfaltigkeit definieren und anschaulich verstehen konnten, kann man sich nun die Frage stellen, ob wir (schwache) Geodätische vom Konzept her auch auf allgemeineren mathematischen Entitäten erklären können. 

Was meinen wir damit genau? Nachdem wir im Hauptteil den Begriff einer (schwachen) Geodätische formal definiert hatten, haben wir uns anschließend darum gekümmert, eine geometrische und anschauliche Vorstellung von diesem Begriff zu erlangen. Dabei haben wir unter anderem gesehen, dass eine (schwache) Geodätische auf einer endlichdimensionalen riemannschen Mannigfaltigkeit $(\mathcal{M}, \tau, \mathcal{A}, g)$ beispielsweise die kürzte Verbindungskurve zweier Punkte $p$ und $q$ aus $\mathcal{M}$ sein kann. Rein konzeptionell handelt es sich also bei einer (schwachen) Geodätischen unter anderem um eine kürzeste Verbindungskurve.

Die Frage, die man sich nun stellen kann ist, ob man beispielsweise ausgehend vom Konzept der kürzesten Verbindungskurven auch einen analogen Geodätenbegriff für allgemeinere mathematischen Entitäten erklären kann. Wie müssten diese allgemeineren mathematischen Entitäten aussehen? Zum einen sollte es sich bei diesen Objekten wieder um topologische Räume $(\mathcal{H}, \rho)$ handeln, denn immerhin benötigen wir einen Kurvenbegriff, und der setzt, wie im Kapitel $2$ diskutiert, einen Stetigkeitsbegriff voraus. 

Darüberhinaus bräuchten wir auf diesen allgemeineren mathematischen Entitäten auch wieder soetwas wie ein Längenfunktional $\mathbb{J}_{[a,b]}$ mit $a,b \in \mathbb{R}$ und $a < b$, denn wenn wir über kürzeste Verbindungskurven sprechen wollen, so benötigen wir natürlich zuerst einmal einen Längenbegriff für eine allgemeine Kurve der Form $\gamma : [a,b] \longrightarrow \mathcal{H}$, damit wir überhaupt wissen, was wir später dann beispielsweise zu minimieren haben. 

Im besten Falle sollte dabei dieser Längenbegriff nicht von einem vollkommen abstrakten Längenfunktional induziert werden, sondern sich strukturell am Längenfunktional endlichdimensionaler riemannscher Mannigfaltigkeiten orientieren, damit wir zum einen überhaupt begründen können, warum wir den Ausdruck $\mathbb{J}_{[a,b]}(\gamma)$ als die Länge der Kurve $\gamma$ bezeichnen, und zum anderen würde dadurch automatisch gewährleistet werden, dass das Funktional $\mathbb{J}_{[a,b]}$ genügend Struktur aufweist, damit nützliche Aussagen über $\mathbb{J}_{[a,b]}$ bewiesen werden können.

Wenn sich aber das Längenfunktional $\mathbb{J}_{[a,b]}$ am Längenfunktional $\mathbb{L}_{[a,b]}$ einer endlichdimensionalen riemannschen Mannigfaltigkeit orientieren soll, so sollten wir hinsichtlich unserer Wahl der allgemeineren mathematischen Entität, auf welcher wir einen Geodätenbegriff entwickeln wollen, fordern, dass sich auch für diese die Begrifflichkeiten aus der multivariaten Analysis oder der Funktionalanalysis übertragen lassen. Immerhin war es schon im Falle des Längenfunktionals $\mathbb{L}_{[a,b]}$ äußerst wichtig gewesen, dass wir auf der Mannigfaltigkeit $(\mathcal{M}, \tau, \mathcal{A}, g)$ in der Lage waren, den Begriff eines Tangentialvektors zu erklären, welchen wir anschaulich mit den Methoden der Analysis definierten.   

Im Prinzip wollen wir also für unsere allgemeinere mathematische Entität im Großen und Ganzen die Struktur einer Mannigfaltigkeit behalten. Das führt nun natürlich unmittelbar auf die Frage, wie man den Mannigfaltigkeitenbegriff sinnvoll verallgemeinern könnte. Da wir bisher immer nur von endlichdimensionalen Mannigfaltigkeiten sprachen, kann man sich nun natürlich die Frage stellen, ob sich der Begriff einer endlichdimensionalen Mannigfaltigkeit sinnvoll in das Unendlichdimensionale übersetzen lässt. Wäre das möglich, so hätten wir zumindestens schon einmal eine mathematische Entität gefunden, die zum einen allgemeiner ist als der Begriff einer endlichdimensionalen Mannigfaltigkeit, und zum anderen zumindestens prinzipiell die Möglichkeit bietet, einen allgemeineren Geodätenbegriff bereitzustellen.

Tatsächlich ist es so, dass es nicht den einen Weg gibt, den Begriff einer endlichdimensionalen Mannigfaltigkeit in das Unendlichdimensionale zu übertragen. Prominente Beispiele unendlichdimensionaler Verallgemeinerungen des endlichdimensionalen Mannigfaltigkeitenbegriffes sind beispielsweise die sogenannten Fréchet-Mannigfaltigkeiten \cite{manoharan1990study}, die sogenannten Banach-Mannigfaltigkeiten \cite{upmeier2011symmetric} und die sogenannten Hilbert-Mannigfaltigkeiten \cite{klingenberg1995riemannian}.

Die grundlegende Idee bezüglich dieser drei verschiedenen Ansätze ist dabei aber mehr oder weniger die selbe. Salopp ist dabei die allgemeine Forderung in allen drei Fällen, dass der betrachtete topologische Raum $(\mathcal{H}, \sigma)$ lokal so \textit{aussieht} wie ein unendlichdimensionaler Vektorraum $E$. Etwas formaler bedeutet das, dass zu jedem Punkt $p \in \mathcal{H}$ eine offene Menge $\mathcal{U} \subseteq \mathcal{H}$ mit $p \in \mathcal{U}$ existiert, sodass $\mathcal{U}$ homöomorph zu einer offenen Menge des unendlichdimensionalen Vektorraums $E$ ist.    

Die lokale Homöomorphie zwischen $(\mathcal{H}, \sigma)$ und $E$ impliziert dabei, dass es sich beim unendlichdimensionalen Vektorraum $E$ nicht nur um einen nackten Vektorraum, sondern tatsächlich um einen mit einer Topologie $\rho$ ausgestatteten Vektorraum handeln muss, da nur in diesem Fall der anschauliche Begriff des Homöomorphismus einen Sinn ergibt. 

Im Falle endlichdimensionaler Mannigfaltigkeiten war dabei die Topologie auf dem Vektorraum $\mathbb{R}^d$, zu der die Mannigfaltigkeit lokal homöomorph war, immer gegeben durch die vom Standardskalarprodukt $\langle \cdot , \cdot \rangle_{\mathbb{R}^d}$ induzierte Topologie, wodurch $\mathbb{R}^d$, wie wir in Satz $4.9.$ gesehen haben, auf natürliche Weise zu einem topologischen Vektorraum wurde. 

Möchte man möglichst viel Struktur aus der endlichdimensionalen Differentialgeometrie in den unendlichdimensionalen Fall übernehmen, so macht es natürlich Sinn zu fordern, dass die auf $E$ gegebene Topologie $\rho$ nicht irgendeine beliebige Topologie ist, sondern so gewählt wird, sodass $(E, \rho)$ gemäß der Definition $4.15.$ ein topologischer Vektorraum ist.

Diese intuitive Forderung an die Topologie des Vektorraumes $E$ ist jeweils als Forderung in den Definitionen der Fréchet-, Banach-, und Hilbert-Mannigfaltigkeiten verbaut, die allesamt mögliche Verallgemeinerungen des Mannigfaltigkeitenbegriffes in das Unendlichdimensionale darstellen. 

So gilt beispielsweise für die Hilbert-Mannigfaltigkeiten salopp, dass es sich beim Vektorraum $(E, \rho)$ um einen vollständigen Prähilbertraum, einen sogenannten Hilbertraum \cite{edwards2012functional}, handeln soll, womit die Topologie $\rho$ auf $E$ durch ein auf $E$ definiertes Skalarprodukt $\langle \cdot , \cdot \rangle$ induziert wird. Für Banach-Mannigfaltigkeiten gilt hingegen grob, dass der Vektorraum $(E, \rho)$ durch einen vollständigen normierten Raum, einen sogenannten Banachraum \cite{edwards2012functional}, gegeben sei, womit die Topologie von einer auf $E$ definierten Norm $\|\cdot\|_E$ induziert wird. In beiden Fällen handelt es sich damit bei $(E, \rho)$ also auf natürliche Weise um einen topologischen Vektorraum.

Legt man aber stattdessen den Begriff einer Fréchet-Mannigfaltigkeit als unendlichdimensionalen Mannigfaltigkeit zu Grunde, so muss in diesem Fall die Topologie $\rho$ auf dem Vektorraum $E$ nicht notwendigerweise von einem Skalarprodukt oder einer Norm induziert sein. Stattdessen fordert man in diesem Fall, dass es sich beim Vektorraum $(E, \rho)$ um einen sogenannten Fréchet-Raum \cite{edwards2012functional} handeln soll, der abermals ein Beispiel eines topologischen Vektorraumes ist. Bei einem Fréchet-Raum handelt es sich dabei salopp um einen topologischen Vektorraum, dessen Topologie unter anderem hausdorff und induziert von einer abzählbaren Familie $(\|\cdot\|_k)_{k \in \mathbb{N}}$ von sogenannten Seminormen sein muss.    

Die Forderung nach der Hausdorff-Eigenschaft ist dabei wichtig, um einen sinnvollen und zur Definition $5.1.$ fast analogen Differenzierbarkeitsbegriff im Fréchet-Raum zu gewährleisten. Immerhin haben wir im Satz $2.1.$ gesehen, das eine angenehme Eigenschaft von Hausdorffräumen es ist, dass deren konvergente Folgen einen eindeutigen Grenzwert besitzen.

Das die Topologie $\rho$ im Fréchet-Raum dabei von einer abzählbaren Familie von Seminormen induziert sein soll, bedeutet, dass eine Teilmenge $\mathcal{D} \subseteq E$ genau dann in $\rho$ liegt, wenn für alle $v \in \mathcal{D}$ ein $K > 0$ und ein $\delta > 0$ existiert, sodass 
\begin{align}
    \{ w \in E \:\: | \:\: \|w - v\|_k < \delta , \:\: \forall k \leq K \} \subseteq \mathcal{D}
\end{align}
gilt. Bei einer auf $E$ erklärten Seminorm $\|\cdot\|$ handelt es sich dabei im wesentlichen um eine Norm, bei der die Forderung nach der positiven Definitheit zu 
\begin{align}
    \|v\| \geq 0 \:\:\:\: \forall v \in E
\end{align}
abgeschwächt wurde.

Besonders in Hinblick auf das in der Einleitung angesprochene Bild-Matching-Problem sind nun Fréchet-Räume und Fréchet-Mannigfaltigkeiten von besonderem Interesse, da man die Menge $\textit{Diff}(\Omega)$, mit $\Omega \subseteq \mathbb{R}^d$, wie in der Einleitung bereits angesprochen, als Fréchet-Mannigfaltigkeit auffassen kann \cite{golubitsky2012stable}. Das erlaubt es uns unter anderem, über glatte oder stückweise glatte Kurven der Form
\begin{align}
    \gamma : [a,b] \longrightarrow \textit{Diff}(\Omega)
\end{align}
zu sprechen. Mittels des Begriffes der glatten Kurve sind wir wiederum in der Lage, auch den Begriff eines Tangentialvektors $X_{\gamma, \gamma(\lambda)}$, mit $\lambda \in (a,b)$, zu erklären. Wie im endlichdimensionalen Fall sind wir nun natürlich auch in der Lage, die Tangentialräume $T_p \textit{Diff}(\Omega)$ für alle $p \in \textit{Diff}(\Omega)$ zu erklären. Ordnen wir nun noch jedem Tangentialraum $T_p \textit{Diff}(\Omega)$ ein Skalarprodukt $\langle \cdot , \cdot \rangle_p$ zu, sodass anschaulich die Funktion 
\begin{align}
            \textit{Diff}(\Omega) \ni p \longmapsto \langle \cdot , \cdot \rangle_p
\end{align}
stetig vom Punkt $p \in \textit{Diff}(\Omega)$ abhängt, so haben wir im Prinzip alle Zutaten beisammen, um ein sinnvolles Längenfunktional auf $\textit{Diff}(\Omega)$ zu erklären , mit welchem wir dann wiederum in der Lage sind, über (schwache) Geodätische auf der Mannigfaltigkeit $\textit{Diff}(\Omega)$ zu sprechen. 

Wie können wir uns das nun für unser in der Einleitung angesprochenes Bildmatching-Problem zunutze machen?

Wie in der Einleitung bereits erwähnt, ist das Ziel hinter dem Bild-Matching-Problem das finden einer Funktion $\phi : \Omega \longrightarrow \Omega$, welche die Bildern $P_1$ und $P_2$ sinnvoll in Beziehung zueinander setzen soll. 

Genauer bedeutet das, dass wenn wir die Menge $\Omega$ wieder als die Bildpixelmenge beider Bilder verstehen wollen, die Funktion $\phi$ uns sagen soll, welcher Pixel auf dem Bild $P_1$ genau welchem Pixel auf dem Bild $P_2$ entspricht. Assoziieren wir dabei, wie in der Einleitung, zum Bild $P_1$ die Funktion $\mathcal{I}_1: \Omega \longrightarrow \mathbb{R}$, und zum Bild $P_2$ die Funktion $\mathcal{I}_2 : \Omega \longrightarrow \mathbb{R}$, so bedeutet diese Forderung, die wir an das Matching $\phi$ stellen, dass 
\begin{align}
    \mathcal{I}_1 = \mathcal{I}_2 \circ \phi
\end{align}
gelten soll. Welche Eigenschaften sollte die Matching-Funktion $\phi$ sonst noch erfüllen? 

Damit $\phi$ überhaupt ein sinnvoller Kandidat für ein Matching zwischen $P_1$ und $P_2$ sein kann, müssen wir natürlich fordern, dass $\phi$ bijektiv, d.h. injektiv und surjektiv, ist. Anschaulich ist klar, warum das eine sinnvolle Forderung ist, denn immerhin soll jeder Bildpixel von $P_1$ genau einem Bildpixel von $P_2$ und umgedreht entsprechen. Darüberhinaus können wir weiter die Menge $\Omega$ mit der Teilraumtopologie $\tau_{\mathbb{R}^d, \Omega}$ der Standardtopologie vom $\mathbb{R}^d$ ausstatten.

Da wir $(\Omega, \tau_{\mathbb{R}^d, \Omega})$ nun als topologischen Raum verstehen können, sind wir in der Lage eine weitere intuitive und sinnvolle Forderung, nämlich die Stetigkeit von $\phi$ und $\phi^{-1}$, zu stellen. Bei dieser Homöomorphismus-Forderung an die Funktion $\phi$ handelt es sich deshalb um eine sinnvolle Forderung, da anschaulich die Bildpunkte nah beieinander liegende Bildpixel von $P_1$ über $\phi$ auch in $P_2$ nah beieinander liegen sollten und umgedreht.

Darüberhinaus können wir nun noch die vereinfachende, aber nicht vollkommen unrealisitische Modellannahme treffen, dass es sich bei der Matching-Funktion $\phi$ nicht nur um einen Homöomorphismus, sondern sogar um einen glatten Diffeomorphismus handeln soll.

Damit können wir also die gesuchte Matching-Funktion auf natürliche Weise als einen Punkt in der unendlichdimensionalen Fréchet-Mannigfaltigkeit $\textit{Diff}(\Omega)$ verstehen. 

Wie in der Einleitung erwähnt, spielen beim sogenannten \textit{Geodesic shooting}-Algorithmus \cite{miller2006geodesic}, welcher auf der hier diskutierten geometrischen Modellierung des Bild-Matching-Problems basiert, die (schwache) Geodätische eine entscheidende Rolle. 

Um also einen Algorithmus wie den \textit{Geodesic shooting}-Algorithmus zu verstehen oder gar herleiten zu können, muss man erst einmal damit beginnen, die (schwachen) Geodötischen auf $\textit{Diff}(\Omega)$ zu studieren und in Erfahrung zu bringen, welche Eigenschaften für diese im Detail gelten, da man sich so dieses Wissen später zu nutze machen kann, um entsprechender Algorithmen zu entwickeln, die auf dem Geodätenbegriff basieren.

Da wir in unserer konkreten geometrischen Modellierung des Bild-Matching-Problems technisch gesehen nur an den (schwachen) Geodätischen von $\textit{Diff}(\Omega)$ interessiert sind, benötigen wir eigentlich keinen allgemeinen Längenbegriff auf $\textit{Diff}(\Omega)$, sondern lediglich einen (möglicherweise einfacheren) Begriff, welcher uns die (schwachen) Geodätischen von $\textit{Diff}(\Omega)$ liefert. Im Abschnitt $7.4$ haben wir genau solch einen Begriff kennengelernt. 

So haben wir im Abschnitt $7.4$ gesehen, dass für endlichdimensionale Mannigfaltigkeiten unter gewissen Vorraussetzungen eine notwendige Stationaritätsbedingung bzgl. des Längenfunktional $\mathbb{L}_{[a,b]}$ und des Energiefunktional $\mathcal{E}_{[a,b]}$ die sogenannten Geodätengleichungen waren. Da es sich bei den Geodätengleichungen um ein System von gewöhnlichen Differentialgleichungen handelt, sind wir mittels dieser in der Lage, die (schwachen) Geodätischen von $\textit{Diff}(\Omega)$ auch tatsächlich zu berechnen.

Da das Energiefunktional das strukturell leichtere Funktional war, nutzen wir im Unendlichdimensionalen gleich von Anfang an das Energiefunktional, um den Begriff der (schwachen) Geodätischen über die stationären Kurven des Energiefunktionals einzuführen.

Bezeichne also $\mathcal{G}_{[a,b]}$ das Energiefunktional auf der unendlichdimensionalen Fréchet-Mannigfaltigkeit $\textit{Diff}(\Omega)$, gegeben durch
\begin{align}
    \mathcal{G}_{[a,b]} : \mathcal{C}^\infty ([a,b], \textit{Diff}(\Omega)) \longrightarrow& \: \: \mathbb{R} \nonumber\\
            \gamma \longmapsto& \: \: \frac{1}{2} \cdot \int_a^b \langle X_{\gamma, \gamma(\lambda)}, X_{\gamma, \gamma(\lambda)} \rangle_{\gamma(\lambda)} \,d \lambda. 
\end{align}
Um aus der Variation von $\mathcal{G}_{[a,b]}$ eine nützliche Form der Euler-Lagrange-Gleichungen im Unendlichdimensionalen herzuleiten, bedient man sich nun eines Trickes, den wir im Folgenden grob umreißen wollen und der sich im wesentlichen an \cite{miller2006geodesic} orientiert. 

Zuerst realisiert man, dass wenn wir die Fréchet-Mannigfaltigkeit $\textit{Diff}(\Omega)$ mit der Funktionenkomposition $\circ$ ausstatten, diese zu einer unendlichdimensionalen Lie-Gruppe wird. Dabei spricht man immer dann von einer Lie-Gruppe, wenn man es mit einer topologischen Gruppe zu tun hat, die gleichzeitig auch eine glatte Mannigfaltigkeit darstellt.

Im Abschnitt $6.4$ haben wir begründet, dass wenn wir den $\mathbb{R}^d$ mit der glatten Struktur aus Bemerkung $6.1.$ ausstatten, wir auf geometrisch sinnvolle Weise jeden Tangentialraum miteinander identifizieren und sinnvoll mit dem $\mathbb{R}^d$ selbst gleichsetzen können.

In unserem unendlichdimensionalen Fall ist nun Ähnliches möglich. So lassen sich in unserem konkreten Fall abermals alle Tangentialräume von $\textit{Diff}(\Omega)$ sinnvoll geometrisch miteinander identifizieren und mit einem unendlichdimensionalen Vektorraum $\mathbf{g}$ gleichsetzen. Der Vektorraum $\mathbf{g}$ ist dabei gegeben als die Menge aller glatten Vektorfelder auf $\Omega$. 

Im Kontext der Lie-Theorie \cite{abbaspour2007basic}, die sich unter anderem mit dem Studium von Lie-Gruppen $(\mathbf{G}, \ast)$ beschäftigt, bezeichnet man den Tangentialraum $T_e \mathbf{G}$ am neutralem Element $e$ von $(\mathbf{G}, \ast)$ oft auch als Lie-Algebra der Lie-Gruppe $(\mathbf{G}, \ast)$. In unserem konkreten Fall, in welchem $(\mathbf{G}, \ast)$ durch $(\textit{Diff}(\Omega), \circ)$ gegeben ist, bezeichnen wir dabei wegen obiger Identifikation den Vektorraum $\mathbf{g}$ als die Lie-Algebra von $(\textit{Diff}(\Omega), \circ)$.

Damit sind wir nun also in der Lage, jeden Tangentialvektor $X$ von $\textit{Diff}(\Omega)$ mit genau einem Element $\mathcal{X}$ aus $\mathbf{g}$ geometrisch sinnvoll zu identifizieren. Das erlaubt es uns wiederum, den in der Definition des Energiefunktionals auftauchenden Ausdruck $\langle \cdot , \cdot \rangle_p$ für alle $p \in \textit{Diff}(\Omega)$ durch ein einzelnes in $\mathbf{g}$ gegebenes Skalarprodukt $\langle \cdot , \cdot \rangle_\mathbf{g}$ zu ersetzen, wodurch sich das Energiefunktional vereinfacht zu
\begin{align}
    \mathcal{G}_{[a,b]}(\gamma) = \frac{1}{2} \cdot \int_a^b \langle \mathcal{X}(\gamma(\lambda)), \mathcal{X}(\gamma(\lambda)) \rangle_\mathbf{g} \,d \lambda.
\end{align}
Bezeichne $\mathbf{g}^*$ den topologischen Dualraum von $\mathbf{g}$, bestehend aus allen linearen und stetigen Funktionalen auf $\mathbf{g}$, so können wir eine Abbildung $L : g \longrightarrow g^*$ definieren, sodass
\begin{align}
    \langle \mathcal{X}, \mathcal{Y} \rangle_\mathbf{g} = L(\mathcal{X}) (\mathcal{Y}) =: (L(\mathcal{X}), \mathcal{Y}) \:\:\:\:\:\:\:\: \forall \mathcal{X}, \mathcal{Y} \in \mathbf{g}
\end{align}
gilt. Dabei bezeichnet man die bilineare Klammer $(\cdot, \cdot)$ oft auch als Dualitätspaarung. Darüberhinaus lassen sich weiter die aus der Lie-Theorie bekannten adjungierten Wirkungen auf $\mathbf{g}$ bzw. auf $\textit{Diff}(\Omega)$ durch
\begin{align}
    \textit{Ad} : \textit{Diff}(\Omega) \times \mathbf{g} \longrightarrow& \: \: \mathbf{g} \nonumber\\
            (\phi, \mathcal{X}) \longmapsto& \: \: \textit{Ad}(\phi, \mathcal{X}) := \textit{Ad}_\phi (\mathcal{X}) := (\phi'(\mathcal{X})) \circ \phi^{-1},  
\end{align}
mit $\textit{Ad}(\phi, \mathcal{X})(x) := \phi'(\phi^{-1}(x)) (\mathcal{X}(\phi^{-1}(x)))$ für alle $x \in \Omega$, und 
\begin{align}
    \textit{ad} : \mathbf{g} \times \mathbf{g} \longrightarrow& \: \: \mathbf{g} \nonumber\\
            (\mathcal{X}, \mathcal{Y}) \longmapsto& \: \: \textit{ad}(\mathcal{X}, \mathcal{Y}) := \textit{ad}_\mathcal{X}(\mathcal{Y}) := \mathcal{X}'(\mathcal{Y}) - \mathcal{Y}'(\mathcal{X}),
\end{align}
mit $(\mathcal{X}'(\mathcal{Y})) (x) := \mathcal{X}'(x) (\mathcal{Y}(x))$ für alle $x \in \Omega$, erklären. Mithilfe dieser Abbildungen ergeben sich die Euler-Lagrange-Gleichungen von $\mathcal{G}_{[a,b]}$ zu
\begin{align}
    \frac{d L(\mathcal{X}(\gamma(\lambda)))}{d \lambda} + \textit{ad}^*_{\mathcal{X}(\gamma(\lambda))} (L (\mathcal{X}(\gamma(\lambda)))) = 0, \label{Euler-Lagrange auf Diff}
\end{align}
wobei $X_{\gamma, \gamma(\lambda)}$ wieder mit $\mathcal{X}(\gamma(\lambda))$ identifiziert wurde und $\textit{ad}^*_{\mathcal{X}(\gamma(\lambda))}$ definiert ist durch
\begin{align}
    (\mathcal{Y}, \textit{ad}_{\mathcal{X}(\gamma(\lambda))}(\mathcal{Z})) = ( \textit{ad}^*_{\mathcal{X}(\gamma(\lambda))} (\mathcal{Y}), \mathcal{Z}) \:\:\:\:\:\:\:\: \forall \mathcal{Y}, \mathcal{Z} \in \mathbf{g}.
\end{align}
Ausgehend von der Gleichung \eqref{Euler-Lagrange auf Diff} kann man dann die verschiedenen Erhaltungsgrößen der (schwachen) Geodätischen von $\textit{Diff}(\Omega)$ studieren, um darauf aufbauend Algorithmen wie den \textit{Geodesic shooting}-Algorithmus zu entwickeln. In \cite{miller2006geodesic} werden einige dieser Erhaltungsgrößen beschrieben.

\newpage

\addcontentsline{toc}{section}{Literatur}

\newpage

\bibliographystyle{plain}
\bibliography{bibliography.bib}
\end{document}